\documentclass[12pt,oneside,a4paper]{book}

\usepackage{perpage}
\usepackage{lipsum}

\newcommand\blfootnote[1]{%
  \begingroup
  \renewcommand\thefootnote{}\footnote{#1}%
  \addtocounter{footnote}{-1}%
  \endgroup
}
\usepackage[utf8]{inputenc}
\usepackage{makeidx}
\makeindex
\usepackage[totoc]{idxlayout}
\usepackage{calligra}
\usepackage{etex}
\usepackage{mathtools}
\usepackage{srcltx}
\usepackage{bussproofs}
\usepackage{mathrsfs}
\usepackage{tikz}
\usetikzlibrary{shapes.geometric}

\usepackage{graphicx}

\usepackage{fancyhdr} % This should be set AFTER setting up the page geometry

\usepackage{amsthm}
\usepackage{amsmath}%
\usepackage{amsfonts}%
\usepackage{amssymb}%
\usepackage{esint}
\usepackage{graphicx}
\usepackage{subfigure}
\usepackage{pstricks}
\usepackage{subfigure}
\usepackage{pst-plot}
\usepackage{array}
\usepackage{tikz}
\usetikzlibrary{matrix,arrows}
\usetikzlibrary{arrows,%
                petri,%
                topaths}%

\theoremstyle{plain}
\newtheorem{claim}{Claim}
\newtheorem{observation}[section]{Observation}
\newtheorem{definition}{Definition}[section]
\newtheorem{example}{Example}[section]

\newtheorem{lemma}{Lemma}[section]

\newtheorem{proposition}{Proposition}[section]

\newtheorem{theorem}{Theorem}[section]
\numberwithin{equation}{section}
\newtheorem*{maintheorem*}{Main Theorem}

\begin{document}

\frontmatter 
\frontmatter

\vspace*{1cm}

\thispagestyle{empty}
\begin{center}
 {\Large\bf A STUDY OF THE INTERRELATION BETWEEN FUZZY TOPOLOGICAL SYSTEMS AND LOGICS}

        \vspace{5cm}

         \small{\bf{Thesis submitted for the degree of\\
        Doctor of Philosophy (Sc.)\\ In Pure Mathematics}}

        \vfill

        \vspace{0.8cm}

        \small{\textit{by}}\\
        \small{\bf{Purbita Jana}}\\
        \vspace{0.8cm}
        \small{\bf{Department of Pure Mathematics\\
        University of Calcutta\\
        2015}}

\end{center}

\tableofcontents
\cleardoublepage
\addcontentsline{toc}{chapter}{Abstract}
\chapter*{\centering Abstract}
The major part of this thesis deals with fuzzy geometric logic and fuzzy geometric logic with graded consequence. The first chapter mainly contains the concept of topological system introduced by S. Vickers in 1989. A topological system is a triple $(X,\models,A)$, where $X$ is a non empty set, $A$ is a frame, and $\models$ is a binary relation between $X$ and $A$. A frame is a lattice which is closed under arbitrary join and finite meet together with the property that binary meet distributes over arbitrary join. This chapter includes almost all possible ground notions which are essential to make this thesis self contained.

In Chapter 2 the notion of fuzzy topological system is introduced and categorical relationship with fuzzy topology and frame is discussed in detail. Also this chapter contains some methodology to make new fuzzy topological systems from the old one.

Chapter 3 provides a generalization of fuzzy topological system which shall be called $\mathscr{L}$-topological system and categorical relationships with appropriate topological space and frame. Furthermore, two ways of constructing subspaces and subsystems of an $\mathscr{L}$-topological space and an $\mathscr{L}$- topological system are respectively provided.  

Chapter 4 deals with the concept of variable basis fuzzy topological space on fuzzy sets and contains a new notion of variable basis fuzzy topological systems whose underlying sets are fuzzy sets. In this chapter categorical relationship between space and system is established.

Chapter 5 contains a different proof of one kind of generalized stone duality, which was done directly by Maruyama \cite{YM}, introducing a notion of $\bar{n}$-fuzzy Boolean system.

The last two chapters, Chapter 6 and Chapter 7 deal the ultimate objective. Chapter 6 deals with the question ``From which logic fuzzy topology can be studied?". To answer this, the notion of fuzzy geometric logic is invented. On top of that a further generalized notions such as fuzzy geometric logic with graded consequence, fuzzy topological spaces with graded inclusion, graded frame and graded fuzzy topological systems came into the picture. Chapter 7 provides categorical relationships among fuzzy topological space with graded inclusion, graded frame and graded fuzzy topological system.

\newpage
\thispagestyle{plain}
\par\vspace*{.35\textheight}{\centering I am dedicating this thesis to my all time superstars, Baba and Ma, for encouraging me to play with Mathematics \par}

\cleardoublepage
\addcontentsline{toc}{chapter}{Acknowledgements}
\chapter*{\centering Acknowledgements}
Pursuing a Ph.D. degree is a beautiful and enjoyable experience. It is just like climbing a high peak, step by step, accompanied with hardship, frustration, encouragement, trust and with so many people's kind help. Though it will not be enough to express my gratitude in words, I would still like to give my thanks to all of them who helped me. 

First and foremost I would like to express my sincere gratitude to my advisor Prof. Mihir K. Chakraborty for the continuous support to my Ph.D. study and related research, for imparting me motivation, and immense knowledge of the subject. His patience and guidance helped me in all the time of research and writing of this thesis. I could not have imagined having a better advisor and mentor for my Ph.D. study.

I am grateful to all my teachers of The Institute for Logic, Language and Computation (ILLC) of University of Amsterdam where I had been a Logic Year student, 2009-2010. I am thankful to all the teachers of Department of Pure Mathematics, University of Calcutta for their encouragement and helpful suggestions during the period of Ph.D. programme. I would like to convey my heartfelt thanks to all of my fellow friends for offering me an ideal environment in which I felt free and could concentrate on my research. I would also express thanks to all the staff of Department of Pure Mathematics, University of Calcutta for their generous support. 

My sincere thanks also go to Prof. N. Raja of TIFR, Mumbai and Dr. Sujata Ghosh of I.S.I., Chennai, who provided me an opportunity to visit their respective institutes and a perfect working place during my stay. I am also very much thankful to Prof. Nitin Nitsure of TIFR, Mumbai for giving me his precious time from his busy schedule to discuss about various fields of mathematics and encouraging me to do mathematics during my stay at TIFR. 

I sincerely show my gratitude to all the members of Calcutta Logic Circle. In this regard I would like to mention that my supervisor, one of the founder of this group, introduced me to this group first. The weekly seminar of this circle helped me to build my knowledge in logic.

My sincere thanks go to all the organizations, Erasmus Mundus, University Grant Commission, National Board of Higher Mathematics, Department of Science and Technology, Indian Academy of Science, Association for Symbolic Logic, which gave me their indispensable generous sponsoring.

Also I am thankful to Prof. Mohua Banerjee of IIT Kanpur, Prof. Anand Pillay of University of Notredam and Prof. Mai Gherke of Paris Diderot University for their helpful comments and suggestions during various conferences.

Last but not the least, I would like to thank my parents (Prabir Kr. Jana and Tanusri Jana), sisters (Pubali and Purna) and brother-in-law (Samit Biswas). Their understanding and love encouraged me to work hard and continue pursuing my Ph.D. degree. My mother's firm and kind hearted personality has affected me to be steadfast and never bend on difficulty. She always lets me know that she is proud of me, which motivates me to work harder and do my best. I thank my friend, Dr. Prosenjit Roy, for all his help and valuable comments in writing this thesis.

\chapter*{List of Papers/Works included in the Thesis}
\addcontentsline{toc}{chapter}{List of Papers/Works included in the Thesis}
\begin{enumerate}
\item P. Jana and M.K. Chakraborty, \emph{Categorical relationships of fuzzy topological systems with fuzzy topological spaces and underlying algebras},  Ann. of Fuzzy Math. and Inform., \textbf{8}, 2014, no. 5, pp. 705--727.

\item P. Jana and M.K. Chakraborty, \emph{On Categorical Relationship among various Fuzzy Topological Systems, Fuzzy Topological Spaces and related Algebraic Structures}, Proceedings of the 13th Asian Logic Conference, World Scientific, 2015, pp. 124--135.

\item P. Jana and M.K. Chakraborty, \emph{Categorical relationships of fuzzy topological systems with fuzzy topological spaces and underlying algebras-II},  Ann. of Fuzzy Math. and Inform., \textbf{10}, 2015, no. 1, pp. 123--137.

\item M. K. Chakraborty and P. Jana, \emph{Fuzzy topology via fuzzy geometric logic with graded consequence}, International Journal of Approximate Reasoning, Elsevier (accepted).
\end{enumerate}

\mainmatter
\chapter{Introduction}
The ultimate objective of this thesis is to develop fuzzy geometric logic and fuzzy geometric logic with graded consequence. The motivation mostly came from the main topic of Vickers book ``Topology via Logic" \cite{SV}, where he introduced a notion of topological system and indicated it's connection with geometric logic. The relationships among topological space, topological system, frame and geometric logic play an important role to study topology through logic (geometric logic). A generalization of topological space to fuzzy topological space was done in \cite{CC, RL} and this concept has been studied extensively and intensively. Naturally the question ``from which logic fuzzy topology can be studied?" comes in mind. If such a logic is obtained what could be its significance?

To answer these questions, as basic steps, we first introduced some notions of fuzzy topological systems and established the interrelation with appropriate topological spaces and algebraic structures. These relationships are studied in categorical framework. As a matter of fact, study of duality takes place as one of the important part of this thesis.

Geometric logic has been discussed in various works such as \cite{GR, PT, SI, SV1, SVI, PJT, SV}. However for our purpose the reference point shall be Vickers' books and papers namely \cite{SV1, SVI, PJT, SV}. The formulae of geometric logic are based on two propositional connectives viz. $\wedge$, the binary conjunction and $\bigvee$, the arbitrary disjunction over arbitrary set of formulae including null set. As a special case the binary disjunction $\vee$ is obtained. Besides, the logic has an existential quantifier $\exists$. It is noteworthy that geometric logic does not have a negation, implication or universal quantifier. Also in this logic sequents of the form $\alpha\vdash\beta$ are derived from a set (may be the null set) of sequents. These special sequents have exactly one formula on either side of the symbol $\vdash$ (turnstile), the intention of the symbol being, as usual, that $\beta$ follows from $\alpha$. The related notion of topological system is a triple $(X,\models,A)$ where $X$ is a non-empty set, $A$ is a frame (c.f. next section) and $\models$ is a binary relation from $X$ to $A$ called `satisfiability'.

Fuzzy topological space was introduced in \cite{CC} and has since been studied extensively \cite{MA, MB, UH, TK, RL, MM}. Fuzzy topological systems have also been defined and studied earlier \cite{MP, AV} with some variation in the definition from the one given in this work. This notion with other names has been studied in categorical framework also \cite{ DM1, SO1, SO2}. Categorical relationship among the categories of fuzzy topological spaces, fuzzy topological systems and frames has been the topic of great interest to which area we have contributed too \cite{MP, MP1, MP2} (vide Chapter 2, 3 and 4). In Chapter 6 another level of generalization (that is, introduction of many-valuedness) shall take place giving rise to graded fuzzy topological systems and fuzzy topological space with graded inclusion. It will also be required to generalize the notion of frame to graded frame.

We now present some motivation behind generalizing the classical notions and geometric logic to multi-valuedness. 

Firstly in \cite{AV} we get an example where the satisfiability relation ($\models$) of a topological system needs a generalization. Taking a physical interpretation of topological systems by Vickers \cite{SV} in which $X$ is a set of computer programs generating bits of 0's and 1's, $A$ consists of the assertions about the sequence of bits generated by computer programs. As an example Vickers took an assertion ``starts 01010", which is true if the sequence of bits generated by a program starts with the bits 01010. For a program, say $x$, which generates an infinite sequence of bits 01010101$\dots$, clearly for this example $x\models \text{starts}\ 01010$. Now if a program, say $x_1$, generates an infinite sequence of bits whose first five bits are not identical with 01010 but similar to 01010 then $x_1$ satisfy the assertion ``starts 01010" to some degree. Hence the concept of fuzzy topological system plays a crucial role to handle this kind of situation. 

Secondly, geometric logic is endowed with an informal observational semantics \cite{SVI}: whether what has been observed does satisfy (match) an assertion or not. In fact, from the stand point of observation, negative and implicational propositions and universal quantification face ontological difficulties. On the other hand arbitrary disjunction needs to be included (cf. \cite{SVI} for an elegant discussion on this issue). Now, observations of facts and assertions about them may corroborate with each other partially. It is a fact of reality and in such cases it is natural to invoke the concept of `satisfiability to some extent or to some degree. As a result the question whether some assertion follows from some other assertion might not have a crisp answer `yes'/`no'. It is likely that in general the `relation of following' or more technically speaking, the consequence relation turnstile ($\vdash$) may be itself many-valued or graded (vide Definition \ref{validg}). For an introduction to the general theory of graded consequence relation we refer to \cite{MK, MSD}. This theory falls within the broad category of fuzzy logic but not exactly the same as that developed in \cite{PH, DM,  VN1, JP}. Thus, we have adopted graded satisfiability as well as graded consequence (c.f. Chapter 6) in this thesis.

Thirdly it has been imperative to link with fuzzy topological systems (and fuzzy topological spaces as a result), a many-valued logic similar to classical topological systems and geometric logic. This goal has been achieved here with the introduction of a general fuzzy geometric logic. In our case, of course, a further generalization has been made by taking the consequence relation as many-valued also and this in turn gives rise to a generalization of the algebraic structure frame to graded frame.

Exploration of relationship between many valued logics ({\L}ukasiewicz n-valued logic {\L}$_n^c $) and fuzzy topology has already been in the agenda \cite{MP, YM}. We, in \cite{MP} (vide Chapter 5), have investigated the relationship between {\L}$_n^c $ and fuzzy topological systems but following \cite{YM}, by {\L}$_n^c $ we have understood {\L}$_n^c $-algebra. On top of that duality between a special kind of fuzzy topological space whose value set is $\{0,\dfrac{1}{n-1},\dfrac{2}{n-1},\dfrac{3}{n-1},......,\dfrac{n-2}{n-1},1\}$ and {\L}$_n^c $ is established. As a matter of fact another proof of duality provided in \cite{YM} comes into picture.

We give below the basic definitions concerning, frame, Boolean algebra, topological space, topological system and category theory.

\subsection*{\centering Frame and Boolean Algebra}
A \textbf{poset} (partially ordered set)\index{poset}\index{partially ordered set} is a set in which a binary relation is reflexive, antisymmetric and transitive. Let us denote it by $P\equiv(P,\leq)$, where $P$ is the set and $\leq\ \subseteq P\times P$ is the above mentioned binary relation, which is known as partial order relation.

An element $l\in P$ is said to be a least upper bound (l.u.b) \index{l.u.b} of $p_1,p_2\in P$ if and only if $p_1\leq l$, $p_2\leq l$ and for any $p\in P$, such that $p_1\leq p$, $p_2\leq p$ implies $l\leq p$. It shall be noted that for any two elements of a poset l.u.b may or may not exist and if exist then it should  be the unique one. Similarly an element $g\in P$ is said to be a greatest lower bound (g.l.b) \index{g.l.b} of $p_1,p_2\in P$ if and only if $g\leq p_1$, $g\leq p_2$ and for any $p\in P$, such that $p\leq p_1$, $p\leq p_2$ implies $p\leq g$. In this case also for any two elements g.l.b may or may not exist and if exist then it should  be the unique one. It shall be noted that l.u.b and g.l.b (if exist) of $p_1,p_2\in P$ are denoted by $p_1\vee p_2$ (join) and $p_1\wedge p_2$ (meet) respectively. If $S$ be any subset of $P$ then arbitrary join (if exist) is represented as $\bigvee S=\bigvee\{s\mid s\in S\}$ and similarly arbitrary meet (if exist) is represented as $\bigwedge S=\bigwedge\{s\mid s\in S\}$. Arbitrary join and arbitrary meet of any set $S\subseteq P$ are also unique if they exist.

A poset $(L,\leq)$ is said to be a \textbf{lattice}\index{lattice} if for each pair of elements of $L$ has least upper bound (l.u.b) and greatest lower bound (g.l.b).
\begin{proposition}\label{lemj}
For any lattice $(L,\leq)$ and $l_1,\ l_2\in L$, $$l_1\leq l_2\Leftrightarrow l_1\wedge l_2=l_1\Leftrightarrow l_1\vee l_2=l_2.$$
\end{proposition}
A lattice $(L,\leq)$ is said to be \textbf{distributive lattice} \index{distributive lattice} if and only if for any $x,y,z\in L$, $x\wedge(y\vee z)=(x\wedge y)\vee (x\wedge z)$ (or $x\vee(y\wedge z)=(x\vee y)\wedge (x\vee z)$) holds.

If a lattice $(L,\leq)$ has a greatest and a least element which are known as top (denoted by $\top$) \index{top element} and bottom \index{bottom element} (denoted by $\bot$) element respectively then it is known as \textbf{bounded lattice} \index{bounded lattice}. It should be noted that $\top=\bigwedge\emptyset$ and $\bot=\bigvee\emptyset$.
A \textbf{complemented lattice} \index{complemented lattice} is a bounded lattice in which every element has a complement i.e. for any $l\in L$, there exist an element $l^c\in L$ such that $l\wedge l^c=\bot$ and $l\vee l^c=\top$.

A \textbf{frame} \index{frame} is a partially ordered set such that every subset has a join, every finite subset has a meet, and a binary meet distributes over an arbitrary join i.e., $$x\wedge\bigvee Y=\bigvee\{x\wedge y :y\in Y\}.$$
A function between two frames is a \textbf{frame homomorphism} \index{frame homomorphism} if and only if it preserves arbitrary join and finite meet.

We observe that any topology of a topological space forms a frame (vide Proposition \ref{tau}).
\begin{proposition}\label{cart}
Cartesian product of two frames is a frame.
\end{proposition}
\begin{proof}
Let $(A,\leq_1)$ and $(B,\leq_2)$ be two frames. To show that $(A\times B,\leq)$ is a frame, where $(a,b)\leq (a_1,b_1)$ if and only if $a\leq_1 a_1$ and $b\leq_2 b_1$. It can be verified that $(A\times B,\leq)$ is a lattice from the definition where $(a,b)\vee (a_1,b_1)=(a\vee a_1,b\vee b_1)$, $(a,b)\wedge (a_1,b_1)=(a\wedge a_1,b\wedge b_1)$ and $\bigvee_i (a_i,b_i)=(\bigvee_i a_i,\bigvee_i b_i)$.\\
It is left to show that $(a,b)\wedge\bigvee_i(a_i,b_i)=\bigvee_i\{(a,b)\wedge (a_i,b_i)\}$. Let us proceed in the following way:
\begin{align*}
(a,b)\wedge\bigvee_i (a_i,b_i) & = (a,b)\wedge (\bigvee_i a_i,\bigvee_i b_i)\\
& = (a\wedge\bigvee_i a_i,b\wedge\bigvee_i b_i)\\
& = (\bigvee_i\{a\wedge a_i\},\bigvee_i\{b\wedge b_i\})\\
& = \bigvee_i(a\wedge a_i,b\wedge b_i)\\
& = \bigvee_i\{(a,b)\wedge (a_i,b_i)\}.
\end{align*}
This completes the proof.
\end{proof}
A \textbf{Boolean algebra} or Boolean lattice \index{Boolean algebra} is a complemented distributive lattice.

A function between two Boolean algebras is a \textbf{Boolean homomorphism} \index{Boolean homomorphism} if and only if it preserves join, meet and complementation.
\subsection*{\centering Topological Space}
Let $X$ be a set. Then a collection of subsets of $X$, say $\tau$, is said to be a \textbf{topology} \index{topology} on $X$ if and only if $\emptyset,\ X\in \tau$, for any $\tau_1,\ \tau_2\in \tau$ implies $\tau_1\cap\tau_2\in \tau$ and for any $\tau_i\in \tau$ implies $\bigcup_i\tau_i\in\tau$. The pair $(X,\tau)$ is known as \textbf{topological space} \index{topological!space} and the elements of $\tau$ are known as \textbf{open sets} \index{open set}.
\begin{proposition}\label{tau}
For any topological space $(X,\tau)$, $(\tau,\subseteq)$ forms a frame.
\end{proposition}
Let $(X,\tau)$, $(Y,\tau')$ be two topological spaces then a function $f:X\longrightarrow Y$ is said to be \textbf{continuous} \index{continuous} if and only if for any $\tau_1'\in\tau'$, $f^{-1}(\tau_1')\in \tau$.

Given a lattice $L$, an $L$-fuzzy subset $\tilde A$ of $X$ is given by the membership function $\tilde A:X\longrightarrow L$.
For each $x\in X$, the value of $\tilde A(x)$ is called the grade of membership of $x$ in the $L$-fuzzy subset $\tilde A$. It will also be denoted by $gr(x \in \tilde A)$.

Let $f$ be a mapping from $X$ to $Y$, and let $\tilde B$ be a $L$-fuzzy subset of $Y$. Then $f^{-1}(\tilde B)$ is a $L$-fuzzy subset \index{$L$-fuzzy!subset} of $X$ given by $f^{-1}(\tilde B)(x)=\tilde B(f(x))$.
Henceforth a $[0,1]$-fuzzy subset will be called simply a fuzzy subset \index{fuzzy!subset}.

Let $X$ be a set, and $\tau$ be a collection of fuzzy subsets of $X$ s.t.
\begin{enumerate}
\item $\tilde\emptyset$ , $\tilde X\in \tau$, where $\tilde\emptyset (x)=0$, for all $x\in X$ and $\tilde X (x)=1$, for all $x\in X$;
\item $\tilde A_i\in\tau$ for $i\in I\Rightarrow\bigcup_{i\in I}\tilde A_i\in \tau$, where $\bigcup_{i\in I}\tilde{A_i}(x)=sup_{i\in I}(\tilde{A_i}(x))$;
\item $\tilde A_1$, $\tilde A_2\in\tau\Rightarrow \tilde A_1\cap\tilde A_2\in\tau$, where $\tilde{A_1}\cap \tilde{A_2}=min\{\tilde{A_1}(x), \tilde{A_2}(x)\}$ (or $\tilde{A_1}(x)\wedge \tilde{A_2}(x)$).
\end{enumerate}
Then $(X,\tau)$ is called a \textbf{fuzzy topological space}\index{fuzzy topological!space}. $\tau$ is called a fuzzy topology over $X$. 
Elements of $\tau$ are called \textbf{fuzzy open sets} \index{fuzzy!open set} of fuzzy topological space $(X,\tau)$.

Let $X$ be a set, and $\tau \subset [0,1]^X$ be a collection of fuzzy subsets of $X$ s.t.
\begin{enumerate}
\item For all $r \in [0,1]$, $\tilde r \in \tau$, where $\tilde r$ is the constant map with value $r$;
\item $\tilde A_i\in\tau$ for $i\in I\Rightarrow\bigcup_{i\in I}\tilde A_i\in \tau$;
\item $\tilde A_1$, $\tilde A_2\in\tau\Rightarrow \tilde A_1\cap\tilde A_2\in\tau$.
\end{enumerate}
This pair $(X,\tau)$ is called a \textbf{fuzzy topological space in the sense of L{\"o}wen} \index{fuzzy topological!space!L{\"o}wen sense} which is of late called ``\textbf{stratified}"\index{stratified}. It is to be noted that when $\tau \subset \bar{n}^X$, where $\bar{n}=\{0,\dfrac{1}{n-1},\dfrac{2}{n-1},\dfrac{3}{n-1},......,\dfrac{n-2}{n-1},1\}$, then $(X,\tau)$ is known to be $\bar{n}$-fuzzy space \index{$\bar{n}$-fuzzy space}.

%\begin{definition}\label{2.10}
%Let $f$ be a mapping from $X$ to $Y$, $\tilde B$ be a fuzzy subset of $Y$. Then $f^{-1}(\tilde B)$ is a fuzzy subset of $X$ given by $f^{-1}(\tilde B)(x)=\tilde B(f(x))$.
%\end{definition}
Let $(X,\tau_1)$ and $(Y,\tau_2)$ be two fuzzy topological spaces. A function $f:X\longrightarrow Y$ is said to be \textbf{fuzzy continuous} \index{fuzzy!continuous} if and only if for every fuzzy open set $\tilde B$ of $Y$, $f^{-1}(\tilde B)$ is an open set of $X$.
\subsection*{\centering Topological System}
A \textbf{topological system}\index{topological!system} \cite{SV} is a triple, $(X,\models,A)$, consisting of a non empty set $X$, a frame $A$ and a binary relation $\models\subseteq  X\times A$ from $X$ to $A$ such that:
\begin{enumerate}
\item for \textbf{any finite subset} $S$ of $A$, $x\models\bigwedge S$ if and only if $x\models a$ for all $a\in S$;
\item for \textbf{any subset} $S$ of $A$, $x\models \bigvee S$ if and only if $x\models a$ for some $a\in S$.
\end{enumerate}
We write $x\models a$ for $(x,a)\in\ \models$ and call $x$ satisfies $a$. We can think the set $X$ as set of objects, $A$ as the set of properties and then $\models$ indicates which object have what properties.

It is to be noted that $\bigwedge S$ is either $a_1\wedge a_2\wedge \dots \wedge a_n$ if $S=\{a_1,a_2,\dots ,a_n\}$ or is $\top$ if $S=\emptyset$.
\begin{proposition}\label{tb}
1. $x\models\top$, for any $x\in X$.\\
2. $x\models\bot$, for no $x\in X$.\\
3. if $x\models a$ and $a\leq b$ then $x\models b$.
\end{proposition}
\begin{proof}
1. $x\models\top$ if and only if $x\models\bigwedge\emptyset$. Hence $x\models a$ for all $a\in\emptyset$, which holds for any $x\in X$.\\
2. $x\models\bot$ if and only if $x\models a$ for some $a\in\emptyset$, clearly which is not possible for any $x\in X$.\\
3. As $A$ is a frame so by Proposition \ref{lemj}, $a\leq b$ if and only if $a\wedge b=a$ holds. So, $x\models a$ if and only if $x\models a\wedge b$. Hence $x\models a$ if and only if $x\models a$ and $x\models b$. Therefore for the given situation if $x\models a$ then $x\models b$. 
\end{proof}
\begin{proposition}
For any topological space $(X,\tau)$, $(X,\in,\tau)$ forms a topological system.
\end{proposition}
\begin{proposition}
For any frame $A$, $(Hom(A,\{\bot,\top\}),\models,A)$, where \\$Hom(A,\{\bot,\top\})$ is the set of frame homomorphisms from $A$ to $\{\bot,\top\}$ and $v\models a$ if and only if $v(a)=\top$, where $v\in Hom(A,\{\bot,\top\})$, forms a topological system.
\end{proposition}
Let us introduce a notion of extent. Let $(X,\models,A)$ be a topological system. Then \textbf{extent} \index{extent} of $a\in A$, denoted by $ext(a)$, is defined by $ext(a)=\{x\in X\mid x\models a\}$. Loosely speaking, extent of a property $a$ is the collection of all those objects in $X$, which have the property $a\in A$. 
\begin{proposition}
Let $(X,\models,A)$ be a topological system. Then $ext(A)=\{ext(a)\}_{a\in A}$ forms a topology on $X$ i.e., $(X,ext(A))$ is a topological space.
\end{proposition}
\begin{proof}
As $\bot\in A$, $ext(\bot)=\{x\in X\mid x\models \bot\}=\emptyset\in ext(A)$, using Proposition \ref{tb}.
Also for $\top\in A$, $$ext(\top)=\{x\in X\mid x\models \top\}=X\in ext(A),$$ using Proposition \ref{tb}.
Let $ext(a_1)$, $ext(a_2)\in ext(A)$. Then 
\begin{align*}
ext(a_1)\cap ext(a_2) & = \{x\in X\mid x\models a_1\}\cap\{x\in X\mid x\models a_2\}\\
& = \{x\in X\mid x\models a_1\ and\ x\models a_2\}\\
& = \{x\in X\mid x\models a_1\wedge a_2\}\\
& = ext(a_1\wedge a_2).
\end{align*}
As $a_1,\ a_2\in A$, $a_1\wedge a_2\in A$ and so, $ext(a_1\wedge a_2)\in ext(A)$. Therefore $ext(a_1)\cap ext(a_2)\in ext(A)$.
Similarly for $ext(a_i)\in ext(A)$, $\bigcup_iext(a_i)\in ext(A)$.
\end{proof}
A \textbf{continuous map} \index{continuous map!between!topological systems} from the topological system $(X,\models,A)$ to the topological system $(Y,\models',B)$ is a pair of maps $(f,g)$ where $f$ is a set map from $X$ to $Y$, $g$ is a frame homomorphism from $B$ to $A$, satisfying the condition viz. $x\models g(b)$ if and only if $f(x)\models' b$ for any $x\in X$ and $b\in B$. 
\subsection*{\centering Category Theory}
\begin{definition}[Categories] \cite{AJ} A \textbf{category} \index{category} is a quadruple \\$\mathbb{A}=(O,hom,id,\circ)$ consisting of 
\begin{enumerate}
\item A class $O$, whose members are called $\mathbb{A}-objects$,
\item For each pair $(A,B) $ of $\mathbb{A}-objects$, a set $hom(A,B)$, whose members are called $\mathbb{A}-morphisms$ from $A$ to $B$ [The statement $f\in hom(A,B)$ is expressed most graphically by using arrows],
\item For each $\mathbb{A}-object$ $A$, a morphism $id_A:A\longrightarrow A$ called $\mathbb{A}-identity$ on $A$,
\item A composition law associating with each $\mathbb{A}-morphism$ $f:A\longrightarrow B$ and each $\mathbb{A}-morphism$ $g:B\longrightarrow C$ an $\mathsf{A}-morphism$ $g\circ f:A\longrightarrow C$, called the composite of $f$ and $g$, subject to the following conditions:
\subitem (a) composition is associative i.e. for morphisms $f:A\longrightarrow B,$ $g:B\longrightarrow C$ and $h:C\longrightarrow D$ the equation $h\circ (g\circ f)=(h\circ g)\circ f$ holds,
\subitem (b) $\mathbb{A}-identities$ act as identities with respect to composition i.e. for $\mathbb{A}-morphisms$ $f:A\longrightarrow B,$ we have $id_B\circ f=f$ and $f\circ id_A=f$,
\subitem (c) the sets $hom(A,B)$ are pairwise disjoint. 
\end{enumerate}
\end{definition}
\begin{definition} For any category $\mathbb{A}=(O,hom_\mathbb{A},id,\circ)$ the \textbf{dual (or opposite) category} \index{dual category}\index{opposite category} of $\mathbb{A}$ is the category $\mathbb{A}^{op}=(O,hom_{\mathbb{A}^{op}},id,\circ^{op})$ where \\$hom_{\mathbb{A}^{op}}(A,B)=hom_\mathbb{A}(B,A)$ and $f\circ^{op} g=g\circ f.$ (Thus $\mathbb{A}$ and $\mathbb{A}^{op}$ have the same objects and, except for their direction, the same morphisms). 
\end{definition}
\begin{example}
If $\mathbb{A}=(X,\leq)$ is a pre-ordered class, then $\mathbb{A}^{op}=(X,\geq)$.
\end{example}
\begin{definition}[Functors] If $\mathbb{A}$ and $\mathbb{B}$  are categories, then the \textbf{functor} \index{functor} $ F$ from $\mathbb{A}$ to $\mathbb{B}$ is a function that assigns to each $\mathbb{A}-object$ $A$, a $\mathbb{B}-object$ $F(A)$, and to each $\mathbb{A}-morphism$ $f:A\longrightarrow A'$, a $\mathbb{B}-morphism$ $F(f):F(A)\longrightarrow F(A')$, in such a way that
\begin{enumerate}
\item $F-$ preserves compositions, i.e. $F(f\circ g)=F(f)\circ F(g)$ whenever $f\circ g$ is defined and
\item $F-$ preserve identity morphisms, i.e. $F(id_A)=id_{F(A)}$ for each $\mathbb{A}-object$ A.
\end{enumerate}
\end{definition}
\begin{example}
For any category $\mathbb{A}$, there is the identity functor $id_A: \mathbb{A}\longrightarrow \mathbb{A}$ defined by $id_A(f:A\longrightarrow B)=f:A\longrightarrow B$.
\end{example}
\begin{definition}[Natural Transformation] Let $F,\ G:\mathbb{A}\longrightarrow \mathbb{B}$ be two functors. A \textbf{natural transformation} \index{natural transformation} $\tau$ from $F$ to $G$ (denoted by $\tau:F\longrightarrow G$) is a function that assigns to each $\mathbb{A}-object$ $A$, a $\mathbb{B}-morphism$ $\tau_A:FA\longrightarrow GA$ in such a way that the following naturality condition holds: for each $\mathbb{A}-morphism$ $f:A\longrightarrow A',$ the square
\begin{center}
\begin{tikzpicture}[description/.style={fill=white,inner sep=2pt}] 
    \matrix (m) [matrix of math nodes, row sep=2.5em, column sep=2.5em]
    { FA&&GA  \\
         FA'&&GA'\\ }; 
    \path[->,font=\scriptsize] 
        (m-1-1) edge node[auto] {$\tau_A$} (m-1-3)
        (m-1-1) edge node[auto] {$Ff$} (m-2-1)
        (m-2-1) edge node[auto,swap] {$\tau_{A'}$} (m-2-3)
        (m-1-3) edge node[auto,swap] {$Gf$} (m-2-3)
       % (m-1-5) edge node[auto,swap] {$j$} (m-1-3)
       % (m-1-5) edge node[auto] {$\psi$} (m-2-3)
        %(m-1-3) edge node[auto] {$G(\hat{f})$} (m-2-3)
         ;
\end{tikzpicture}
\end{center}
commutes.
\end{definition}
Let $G:\mathbb{A}\longrightarrow \mathbb{B}$ be a functor, and let $B$ be a $\mathbb{B}$-object.
\begin{definition}[$G$-structured arrow and $G$-costructured arrow]\label{2.1} The concept of $G$-structured arrow and $G$-costructured arrow are defined as follows.
\begin{enumerate}
\item  A $G$-$\mathbf{structured\ arrow\ with\ domain}$ \index{$G$-structured arrow} $B$ is a pair $(f,A)$ consisting of an $\mathbb{A}$-object $A$ and a $\mathbb{B}$-morphism $ f:B\longrightarrow GA$.
\item  A $G$-$\mathbf{costructured\ arrow\ with\ codomain}$ \index{$G$-costructured arrow} $B$ is a pair $(A,f)$ consisting of an $\mathbb{A}$-object $ A$ and a $\mathbb{B}$-morphism $ f:GA\longrightarrow B$.
\end{enumerate}
\end{definition}
\begin{definition}[$G$-universal arrow and $G$-couniversal arrow]\label{2.2}
$G$-universal arrow and $G$-couniversal arrow are defined as follows:
\begin{enumerate}
\item A $G$-structured arrow $(g,A)$ with domain $B$ is called $G$-$\mathbf{universal}$ \index{$G$-universal arrow} for $B$ provided that for each $G$-structured arrow $(g',A')$ with domain $B$, there exists a unique $\mathbb{A}$-morphism $\hat f:A \longrightarrow A'$ with $g'=G(\hat f)\circ g$. i.e., s.t. the triangle
\begin{center}
\begin{tikzpicture}[description/.style={fill=white,inner sep=2pt}] 
    \matrix (m) [matrix of math nodes, row sep=2.5em, column sep=2.5em]
    {B & &GA  \\
        & & GA' \\ }; 
    \path[->,font=\scriptsize] 
        (m-1-1) edge node[auto] {$g$} (m-1-3)
        (m-1-1) edge node[auto,swap] {$g'$} (m-2-3)
       % (m-1-5) edge node[auto,swap] {$j$} (m-1-3)
        %(m-1-5) edge node[auto] {$\psi$} (m-2-3)
        (m-1-3) edge node[auto] {$G\hat f$} (m-2-3);
\end{tikzpicture}
\end{center}
commutes.
\\ We can also represent the above statement by the following diagram.
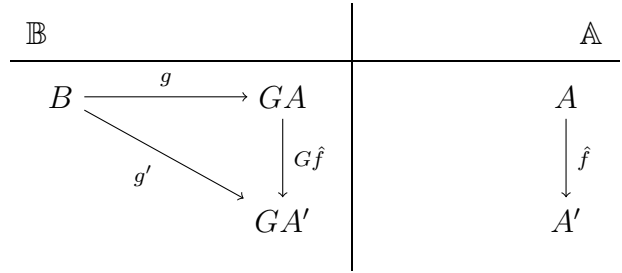
\begin{center}    
\begin{tabular}{ l | r  } 
$\mathbb{B}$ & $\mathbb{A}$\\
\hline
{\begin{tikzpicture}[description/.style={fill=white,inner sep=2pt}] 
    \matrix (m) [matrix of math nodes, row sep=2.5em, column sep=2.5em]
    {B & &GA  \\
        & & GA' \\ }; 
    \path[->,font=\scriptsize] 
        (m-1-1) edge node[auto] {$g$} (m-1-3)
        (m-1-1) edge node[auto,swap] {$g'$} (m-2-3)
       % (m-1-5) edge node[auto,swap] {$j$} (m-1-3)
        %(m-1-5) edge node[auto] {$\psi$} (m-2-3)$%
        (m-1-3) edge node[auto] {$G\hat f$} (m-2-3);
\end{tikzpicture}}  &  {\begin{tikzpicture}[description/.style={fill=white,inner sep=2pt}] 
    \matrix (m) [matrix of math nodes, row sep=2.5em, column sep=2.5em]
    {& &A  \\
        & &A' \\ }; 
    \path[->,font=\scriptsize]

        %(m-1-5) edge node[auto,swap] {$j$} (m-1-3)
        %(m-1-5) edge node[auto] {$\psi$} (m-2-3)
        (m-1-3) edge node[auto] {$\hat f$} (m-2-3);
\end{tikzpicture}}\\
\end{tabular}
\end{center}
The diagram above indicates the fact that $g:B\longrightarrow GA$ is the $G$-universal arrow provided that for given $g':B\longrightarrow GA'$ there exist a unique $\mathbb{A}-morphisim$ $\hat f:A\longrightarrow A'$ s.t. the triangle commutes.
\item A $G$-costructured arrow $(A,g)$ with codomain $B$ is called $G$-$\mathbf{couniversal}$ \index{$G$-couniversal arrow} for $B$ provided that for each $G$-costructured arrow $(A',g')$ with codomain $B$, there exists a unique $\mathbb{A}$-morphism $\hat f:A' \longrightarrow A$ with $g'=g\circ G(\hat f)$. i.e., s.t. the triangle
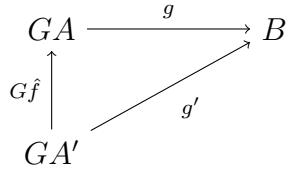
\begin{center}
\begin{tikzpicture}[description/.style={fill=white,inner sep=2pt}] 
    \matrix (m) [matrix of math nodes, row sep=2.5em, column sep=2.5em]
    { GA&&B  \\
         GA'\\ }; 
    \path[->,font=\scriptsize] 
        (m-1-1) edge node[auto] {$g$} (m-1-3)
        (m-2-1) edge node[auto] {$G\hat f$} (m-1-1)
        (m-2-1) edge node[auto,swap] {$g'$} (m-1-3)
        %(m-1-1) edge node[auto,swap] {$f$} (m-2-3)
       % (m-1-5) edge node[auto,swap] {$j$} (m-1-3)
       % (m-1-5) edge node[auto] {$\psi$} (m-2-3)
        %(m-1-3) edge node[auto] {$G(\hat{f})$} (m-2-3)
         ;
\end{tikzpicture}
\end{center}
commutes.
\\ We can also represent the above statement by the following diagram.
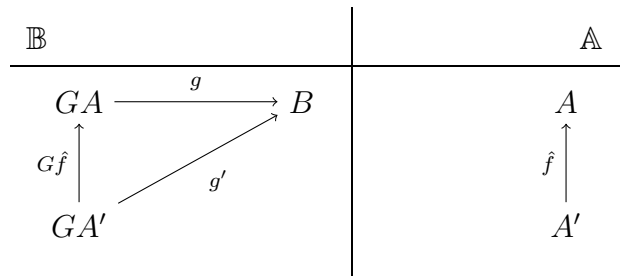
\begin{center}    
\begin{tabular}{ l | r  } 
$\mathbb{B}$ & $\mathbb{A}$\\
\hline
{\begin{tikzpicture}[description/.style={fill=white,inner sep=2pt}] 
    \matrix (m) [matrix of math nodes, row sep=2.5em, column sep=2.5em]
    { GA&&B  \\
         GA'\\ }; 
    \path[->,font=\scriptsize] 
        (m-1-1) edge node[auto] {$g$} (m-1-3)
        (m-2-1) edge node[auto] {$G\hat f$} (m-1-1)
        (m-2-1) edge node[auto,swap] {$g'$} (m-1-3)
        %(m-1-1) edge node[auto,swap] {$f$} (m-2-3)
       % (m-1-5) edge node[auto,swap] {$j$} (m-1-3)
       % (m-1-5) edge node[auto] {$\psi$} (m-2-3)
        %(m-1-3) edge node[auto] {$G(\hat{f})$} (m-2-3)
         ;
\end{tikzpicture}}  &  {\begin{tikzpicture}[description/.style={fill=white,inner sep=2pt}] 
    \matrix (m) [matrix of math nodes, row sep=2.5em, column sep=2.5em]
    {& &A  \\
        & &A' \\ }; 
    \path[->,font=\scriptsize]

        %(m-1-5) edge node[auto,swap] {$j$} (m-1-3)
        %(m-1-5) edge node[auto] {$\psi$} (m-2-3)
        (m-2-3) edge node[auto] {$\hat f$} (m-1-3);
\end{tikzpicture}}\\
\end{tabular}
\end{center}
The diagram above indicates the fact that $g:GA\longrightarrow B$ is the $G$-couniversal arrow provided that for given $g':GA'\longrightarrow B'$ there exist a unique $\mathbb{A}-morphisim$ $\hat f:A'\longrightarrow A$ s.t. the triangle commutes.
\end{enumerate}
\end{definition}
\begin{definition}[Left Adjoint and Right Adjoint]\label{2.3}
Left Adjoint and Right Adjoint are defined as follows.
\begin{enumerate}
\item A functor $G:\mathbb{A}\longrightarrow \mathbb{B}$ is said to be $\mathbf{left\ adjoint}$ \index{left adjoint} provided that  for every $\mathbb{B}$-object $B$, there exists a $G$-couniversal arrow with codomain $B$.
\\As a consequence, there exists a natural transformation $\eta:id_A\longrightarrow FG$ ($id_A$ is the identity morphism from $A$ to $A$), where  $F:\mathbb{B}\longrightarrow \mathbb{A}$ is a functor s.t. for given $f:A\longrightarrow FB$ there exists a unique $\mathbb{B}$-morphism $\hat f:GA\longrightarrow B$ s.t. the triangle
\begin{center}
\begin{tikzpicture}[description/.style={fill=white,inner sep=2pt}] 
    \matrix (m) [matrix of math nodes, row sep=2.5em, column sep=2.5em]
    { A&&FGA  \\
       &&FB\\ }; 
    \path[->,font=\scriptsize] 
        (m-1-1) edge node[auto] {$\eta_A$} (m-1-3)
        (m-1-1) edge node[auto,swap] {$f$} (m-2-3)
        (m-1-3) edge node[auto] {$F\hat f$} (m-2-3)
        %(m-1-1) edge node[auto,swap] {$f$} (m-2-3)
       % (m-1-5) edge node[auto,swap] {$j$} (m-1-3)
       % (m-1-5) edge node[auto] {$\psi$} (m-2-3)
        %(m-1-3) edge node[auto] {$G(\hat{f})$} (m-2-3)
         ;
\end{tikzpicture}
\end{center}
commutes.
\\ This $\eta$ is called the unit \index{unit} of the adjunction.
\\ Hence, we have the diagram of unit \index{diagram of!unit} as follows:
\begin{center}    
\begin{tabular}{ l | r  } 
$\mathbb{A}$ & $\mathbb{B}$\\
\hline
{\begin{tikzpicture}[description/.style={fill=white,inner sep=2pt}] 
    \matrix (m) [matrix of math nodes, row sep=2.5em, column sep=2.5em]
    {A & &FGA  \\
        & &FB \\ }; 
    \path[->,font=\scriptsize] 
        (m-1-1) edge node[auto] {$\eta$} (m-1-3)
        (m-1-1) edge node[auto,swap] {$f$} (m-2-3)
       % (m-1-5) edge node[auto,swap] {$j$} (m-1-3)
        %(m-1-5) edge node[auto] {$\psi$} (m-2-3)
        (m-1-3) edge node[auto] {$F\hat f$} (m-2-3);
\end{tikzpicture}}  &  {\begin{tikzpicture}[description/.style={fill=white,inner sep=2pt}] 
    \matrix (m) [matrix of math nodes, row sep=2.5em, column sep=2.5em]
    {& &GA  \\
        & &B \\ }; 
    \path[->,font=\scriptsize]

        %(m-1-5) edge node[auto,swap] {$j$} (m-1-3)
        %(m-1-5) edge node[auto] {$\psi$} (m-2-3)
        (m-1-3) edge node[auto] {$\hat f$} (m-2-3);
\end{tikzpicture}}\\
\end{tabular}
\end{center}
\item A functor $G:\mathbb{A}\longrightarrow \mathbb{B}$ is said to be $\mathbf{right\ adjoint}$ \index{right adjoint} provided that  for every $\mathbb{B}$-object $B$, there exists a $G$-universal arrow with domain $B$.
\\ From the definition above, it follows that there exists a natural transformation $\xi:FG\longrightarrow id_A$ ($id_A$ is the identity morphism from $A$ to $A$), where $F:\mathbb{B}\longrightarrow\mathbb{A}$ is a functor s.t. for given $f':FB\longrightarrow A$, there exists a unique $\mathbb{B}$-morphism $\hat f:B\longrightarrow GA$ s.t the triangle
\begin{center}
\begin{tikzpicture}[description/.style={fill=white,inner sep=2pt}] 
    \matrix (m) [matrix of math nodes, row sep=2.5em, column sep=2.5em]
    { FGA&&A  \\
         FB\\ }; 
    \path[->,font=\scriptsize] 
        (m-1-1) edge node[auto] {$\xi_A$} (m-1-3)
        (m-2-1) edge node[auto] {$F\hat f$} (m-1-1)
        (m-2-1) edge node[auto,swap] {$f'$} (m-1-3)
        %(m-1-1) edge node[auto,swap] {$f$} (m-2-3)
       % (m-1-5) edge node[auto,swap] {$j$} (m-1-3)
       % (m-1-5) edge node[auto] {$\psi$} (m-2-3)
        %(m-1-3) edge node[auto] {$G(\hat{f})$} (m-2-3)
         ;
\end{tikzpicture}
\end{center}
commutes.
\\ This $\xi$ is called the co-unit \index{co-unit} of the adjunction.
\\ Hence, we have the diagram of co-unit \index{diagram of!co-unit} as follows:
\begin{center}    
\begin{tabular}{ l | r  } 
$\mathbb{A}$ & $\mathbb{B}$\\
\hline
{\begin{tikzpicture}[description/.style={fill=white,inner sep=2pt}] 
    \matrix (m) [matrix of math nodes, row sep=2.5em, column sep=2.5em]
    { FGA&&A  \\
         FB\\ }; 
    \path[->,font=\scriptsize] 
        (m-1-1) edge node[auto] {$\xi$} (m-1-3)
        (m-2-1) edge node[auto] {$F\hat f$} (m-1-1)
        (m-2-1) edge node[auto,swap] {$f'$} (m-1-3)
        %(m-1-1) edge node[auto,swap] {$f$} (m-2-3)
       % (m-1-5) edge node[auto,swap] {$j$} (m-1-3)
       % (m-1-5) edge node[auto] {$\psi$} (m-2-3)
        %(m-1-3) edge node[auto] {$G(\hat{f})$} (m-2-3)
         ;
\end{tikzpicture}}  &  {\begin{tikzpicture}[description/.style={fill=white,inner sep=2pt}] 
    \matrix (m) [matrix of math nodes, row sep=2.5em, column sep=2.5em]
    {& &GA  \\
        & &B \\ }; 
    \path[->,font=\scriptsize]

        %(m-1-5) edge node[auto,swap] {$j$} (m-1-3)
        %(m-1-5) edge node[auto] {$\psi$} (m-2-3)
        (m-2-3) edge node[auto] {$\hat f$} (m-1-3);
\end{tikzpicture}}\\
\end{tabular}
\end{center}
\end{enumerate}
\end{definition} 
\chapter{Category of Fuzzy Topological Systems}
\section{Introduction}\blfootnote{The results of this chapter appear in {\bf \cite{MP} P. Jana and M.K. Chakraborty: \textit{Categorical relationships of fuzzy topological systems with fuzzy topological spaces and underlying algebras}, Ann. of Fuzzy Math. and Inform., \textbf{8}, 2014, no. 5, pp. 705--727.}}
 A topological system\index{topological!system} is a mathematical concept introduced in \cite{SV} in the year 1989. It is a triple $(X,\models , A)$, where $X$ is a non-empty set, $A$ is a frame and $\models $ is a binary relation from $X$ to $A$. $x\models a$ where $x\in X$ and $a\in A$ is read as `$x$ satisfies $a$'. Vickers introduced the notion of topological system in the context of the so-called geometric logic, which was further studied in \cite{PJT}. Topological systems and frames form the categories $\mathbf{TopSys}$ and $\mathbf{Frm}$ respectively. These are linked by two adjoint functors \cite{SV}.

Now, the notion of satisfaction may be graded. From various standpoints it is reasonable to assume that in some situations $x$ satisfies $a$ to some extent or to a degree. In other words, the binary relation $\models $ may be a fuzzy relation \cite{LZ}. The extent to which $x$ satisfies $a$ shall be denoted by $grade(x\models a)$ or simply $gr(x\models a)$, which is an element of some suitable value set. The value set in fuzzy literature has been generalized from the unit interval [0,1] in \cite{LZ} to a lattice  \cite{JG}. In this chapter, the value set shall be chosen as the unit interval [0,1] with the natural ordering. Thus, we arrive at the notion of a {\it fuzzy topological system}. In fact, by a fuzzy topological system, we shall understand a triple $(X,\models , A)$, where $X$ is a non-empty set, $A$ is an algebra, which is at least a  lattice with arbitrary join and finite meet, and $\models $ is a fuzzy relation from $X$ to $A$, grades of relatedness being assigned from a suitable value set. In this chapter we shall consider the fuzzy topological system which has a frame $A$ and [0,1] as the value set. An attempt in this direction was taken in \cite{AV}, but there are differences with our approach which will be shown later. 

There exists a huge literature on fuzzy topological spaces, \cite{MA,MB,CC,UH,RL,SR} to mention only a few. This area of study appeared immediately after fuzzy set theory was introduced by Lotfi Zadeh in 1965. One is naturally inclined to study the relationships between fuzzy topological systems and existing fuzzy topological spaces, particularly with respect to the viewpoint of categorical duality and equivalence. As has been observed by Vickers \cite{SV}, topological spaces make a special kind of topological system, and we shall see here that fuzzy topological spaces make a special kind of fuzzy topological system.

Generalization of topological systems to fuzzy topological systems is important from another angle too. In first-order logic, (semantic) consequence relation is defined in terms of satisfaction. When the satisfaction relation is fuzzy, the corresponding consequence relation may be either crisp or fuzzy. In the first case, we get fuzzy logic and in the second case, logic of graded consequence \cite{MK,MSD}. The logic of graded consequence falls within the broad category of fuzzy logic, but is marked by its distinction from the variety of fuzzy logics in \cite{PH,VN,JP} with respect to the nature of logical consequence \cite{MSD}. Thus, fuzzy topological systems may be viewed as abstract generalizations of fuzzy logics with graded consequence. However, we shall not delve into this issue in this chapter -- this is included in Chapter 6 in details. We shall study some categorical relationships between the categories of fuzzy topological systems and of fuzzy topological spaces. Duality between a subcategory  of fuzzy topological systems and category of frames shall be established. Another subcategory of fuzzy topological systems will be shown to be equivalent to the category of fuzzy topological spaces.

We have also defined sum and product of fuzzy topological systems. These definitions are different from those given in \cite{AV}, and we consider them to be more appropriate in the fuzzy context where the satisfaction relation is graded.
\section{Categories: [0,1]-Top, [0,1]-TopSys, Frm and their interrelationships}
\subsection{Categories}
\subsection*{$\mathbf{[0,1]}$-$\mathbf{Top}$}
\begin{definition}[Fuzzy topological space]\label{3.1_1t}
 Let $X$ be a set and $\tau$ be a collection of fuzzy subsets of $X$ s.t.
\begin{enumerate}
\item $\tilde\emptyset$, $\tilde X\in \tau$, where $\tilde\emptyset (x)=0$, for all $x\in X$ and $\tilde X (x)=1$, for all $x\in X$;
\item $\tilde A_i\in\tau$ for $i\in I\ \text{implies}\ \bigcup_{i\in I}\tilde A_i\in \tau$, where $\bigcup_{i\in I}\tilde{A_i}(x)=sup_{i\in I}(\tilde{A_i}(x))$;
\item $\tilde A_1$, $\tilde A_2\in\tau\ \text{implies}\ \tilde A_1\cap\tilde A_2\in\tau$, where $\tilde{A_1}\cap \tilde{A_2}=min\{\tilde{A_1}(x), \tilde{A_2}(x)\}$.
\end{enumerate}
Then $(X,\tau)$ is called a \textbf{fuzzy topological space}\index{fuzzy topological!space}. $\tau$ is called a \textbf{fuzzy topology} over $X$. 
\end{definition}
Elements of $\tau$ are called \textbf{fuzzy open sets}\index{fuzzy!open sets} of fuzzy topological space $(X,\tau)$.
\begin{definition}[Fuzzy continuous map]\label{3.2_1t} Let $(X,\tau_1)$ and $(Y,\tau_2)$ be two fuzzy topological spaces. A function $f:X\longrightarrow Y$ is said to be \textbf{fuzzy continuous}\index{fuzzy!continuous} if and only if for every fuzzy open set $\tilde B$ of $Y$, $f^{-1}(\tilde B)$ is an open set of $X$.
\end{definition}
The identity map\index{identity!map} $id_X:X\longrightarrow X$ is fuzzy continuous map.
\begin{lemma}\label{3.4_1t}
Let $(X,\tau_1)$, $(Y,\tau_2)$, $(Z,\tau_3)$ be fuzzy topological spaces and $f:X\longrightarrow Y$, $g:Y\longrightarrow Z$ be fuzzy continuous. Then $g\circ f:X\longrightarrow Z$ is fuzzy continuous. 
\end{lemma}
\begin{proof}
Let $\tilde{T}\in \tau_3$. As $g:Y\longrightarrow Z$ is fuzzy continuous, $g^{-1}(\tilde{T})\in \tau_2$. Now $f:X\longrightarrow Y$ is fuzzy continuous and $g^{-1}(\tilde{T})\in \tau_2$. So, $f^{-1}(g^{-1}(\tilde{T}))\in \tau_1$ and consequently $(g\circ f)^{-1}(\tilde{T})\in\tau_1$. Thereby the composition map\index{composition!map} $g\circ f:X\longrightarrow Z$ is fuzzy continuous.
\end{proof}
\begin{lemma}\label{3.5_1t}
Let $(X,\tau_1)$, $(Y,\tau_2)$, $(Z,\tau_3)$, $(U,\tau_4)$ be fuzzy topological spaces, and $f:X\longrightarrow Y$, $g:Y\longrightarrow Z$, $h:Z\longrightarrow U$ be fuzzy continuous. Then $h\circ (g\circ f)=(h\circ g)\circ f$ holds.
\end{lemma}
\begin{proof}
$h\circ (g\circ f)(x)=h((g\circ f)(x))=h(g(f(x)))=(h\circ g)(f(x))=((h\circ g)\circ f)(x)$.
\end{proof}
\begin{lemma}\label{3.6_1t}
For the fuzzy continuous map $f:X\longrightarrow Y$, it holds that $ id_Y\circ f=f$ and $f\circ id_X=f$.
\end{lemma}
\begin{proof}
Proof is straight forward.
\end{proof}
From Lemmas \ref{3.4_1t}, \ref{3.5_1t} and \ref{3.6_1t} we get the following theorem.
\begin{theorem}\label{3.7_1t}
Fuzzy topological spaces together with continuous maps form the category $\mathbf{[0,1]}$-$\mathbf{Top}$\index{category!$\mathbf{[0,1]}$-$\mathbf{Top}$} \cite{SR}.
\end{theorem}
\subsection*{$\mathbf{[0,1]}$-$\mathbf{TopSys}$}
\begin{definition}[Fuzzy topological system]\label{3.1_1}
 A \textbf{fuzzy topological system}\index{fuzzy topological!system} is a triple $(X,\models ,A)$, where $X$ is a non-empty set, $A$ is a frame and $\models$ is a fuzzy relation from $X$ to $A$ such that
 \begin{enumerate}%[$\bullet$]
\item if $S$ is a \textbf{finite subset} of $A$, then
$gr(x\models \bigwedge S) = inf\{ gr(x\models s)\mid s\in S\}$;
\item if $S$ is   \textbf{any subset} of $A$, then
$gr(x\models \bigvee S)=sup\{ gr(x\models s)\mid s\in S\}$.
\end{enumerate}
\end{definition}
\begin{proposition}
For any $x\in X$, $gr(x\models \top)=1$ and $gr(x\models \bot)=0$, where $\top$ and $\bot$ denotes the top and the bottom elements of the frame $A$, respectively.
\end{proposition}
\begin{proof}It will be enough to check the following:
(1) $gr(x\models\top)=gr(x\models\bigwedge\emptyset)=inf\{gr(x\models s)\mid s\in\emptyset\}=inf \ \emptyset=1$.
(2) $gr(x\models\bot)=gr(x\models\bigvee\emptyset)=sup\{gr(x\models s)\mid s\in\emptyset\}=sup \ \emptyset=0$.
\end{proof}
In \cite{AV} a fuzzy topological system\index{fuzzy topological!system} is defined by the following conditions viz.
\begin{enumerate}
\item if $S$ is a  \textbf{finite subset} of $A$, then $gr(x\models \bigwedge S)\leq gr(x\models s)$ for all $s\in S$;
\item if $S$ is \textbf{any subset} of $A$, then $gr(x\models \bigvee S)\leq gr(x\models s)$ for some $s\in S$;
\item $gr(x\models \top)=1$ and $gr(x\models \bot)=0$ for all $x\in X$.
\end{enumerate}  
The above conditions allow one to get a fuzzy topological space from a fuzzy topological system, but not conversely. However, using Definition \ref{3.1_1} of fuzzy topological system, we can produce a functor($Ext$) from the category of fuzzy topological systems to the category of fuzzy topological spaces and also a functor($J$) from the category of fuzzy topological spaces to fuzzy topological systems. Not only that, we can show that these two functors are adjoint. 
Definition \ref{3.1_1} is a natural one and its advantages over the other definition would be clear in  the sequel.
\begin{definition}[Spatial]\label{spatial}
A fuzzy topological system $(X,\models,A)$ is said to be \textbf{spatial}\index{spatial} if and only if (for any $x\in X$, $gr(x\models a)=gr(x\models b)$) imply ($a=b$), for any $a,\ b\in A$.
\end{definition}
\begin{definition}[Localic]\label{localic}
A fuzzy topological system is \textbf{localic}\index{localic} if and only if for any $x, \ y\in X$, $x\neq y$ imply $gr(x\models a)\neq gr(y\models a)$, for some $a\in A$.
\end{definition}
\begin{definition}\label{3.2_1} Let $D=(X,\models ,A)$ and $E=(Y,\models ',B)$ be fuzzy topological systems. A \textbf{continuous map}\index{continuous map!between!fuzzy topological systems} $f:D\longrightarrow E$ is a pair $(f_1,f_2)$ where,
\begin{enumerate}
\item $f_1:X\longrightarrow Y$ is a function;
\item $f_2:B\longrightarrow A$ is a frame homomorphism;
\item $gr(x\models f_2(b))=gr(f_1(x)\models' b)$, for all $x\in X$ and $b\in B$.
\end{enumerate}
\end{definition}
\begin{definition}\label{3.3_1}
Let $D=(X,\models,A)$ be a fuzzy topological system. The \textbf{identity map}\index{identity!map} $I_D:D\longrightarrow D$ is a pair $(I_1,I_2)$ where
$I_1: X  \longrightarrow X$ and $I_2:A  \longrightarrow A$.

\end{definition}
Let $D=(X,\models',A)$, $E=(Y,\models'',B)$, $F=(Z,\models''',C)$. Let $(f_1,f_2):D\longrightarrow E$ and $(g_1,g_2):E\longrightarrow F$ be continuous maps\index{continuous!map}. The \textbf{composition}\index{composition} $(g_1,g_2)\circ (f_1,f_2):D\longrightarrow F$ is a pair $(g_1\circ f_1,f_2\circ g_2)$, where $g_1\circ f_1:X\longrightarrow Z$ and
$f_2\circ g_2:C\longrightarrow A.$
\begin{lemma}\label{3.4_1}
$(g_1,g_2)\circ (f_1,f_2):D\longrightarrow F$ is continuous, where $(f_1,f_2):D\longrightarrow E$, $(g_1,g_2):E\longrightarrow F$ are continuous. 
\end{lemma}
\begin{proof}
We only show that $gr(x\models' f_2\circ g_2(c))=gr(g_1\circ f_1(x)\models''' c)$.
\begin{align*}
gr(x\models' f_2\circ g_2(c))& =gr(x\models' f_2(g_2(c)))\\
& =gr(f_1(x)\models'' g_2(c))\tag{as $(f_1,f_2)$ is continuous}\\%\label{eq:as $(f_1,f_2)$ is continuous}\\
& =gr(g_1(f_1(x))\models''' c) \tag {as $(g_1,g_2)$ is continuous}\\%\label{eq:as $(g_1,g_2)$ is continuous}\\
& =gr(g_1\circ f_1(x)\models''' c).
\end{align*}
This finishes the proof.\end{proof}
\begin{lemma}\label{3.5_1}
Let $D,\ E,\ F,\ G$ be fuzzy topological systems, and $f:D\longrightarrow E$, $g:E\longrightarrow F$, $h:F\longrightarrow G$ are continuous maps. Then $h\circ (g\circ f)=(h\circ g)\circ f$ holds.
\end{lemma}
\begin{proof}
$h\circ (g\circ f)=(h_1,h_2)\circ((g_1,g_2)\circ (f_1,f_2))=(h_1,h_2)\circ (g_1\circ f_1,f_2\circ g_2)=(h_1\circ(g_1\circ f_1),(f_2\circ g_2)\circ h_2)=((h_1\circ g_1)\circ f_1, f_2\circ (g_2\circ h_2))=(h_1\circ g_1,g_2\circ h_2)\circ (f_1,f_2)=((h_1,h_2)\circ (g_1,g_2))\circ (f_1,f_2)=(h\circ g)\circ f$.
\end{proof}
\begin{lemma}\label{3.6_1}
For the continuous map $f:D\longrightarrow E$, it holds that $ id_E\circ f=f$ and $f\circ id_D=f$.
\end{lemma}
\begin{proof}
Proof is straight forward.
\end{proof}
From Lemmas \ref{3.4_1}, \ref{3.5_1} and \ref{3.6_1} we get the following theorem.
\begin{theorem}\label{3.7_1}
Fuzzy topological systems together with continuous maps form the category $\mathbf{[0,1]}$-$\mathbf{TopSys}$\index{category!$\mathbf{[0,1]}$-$\mathbf{TopSys}$}.
\end{theorem}
It is known that frames together with frame homomorphisms form the category $\mathbf{Frm}$\index{category!$\mathbf{Frm}$} \cite{AJ}.
We shall now investigate the interrelations among the categories $\mathbf{[0,1]}$-$\mathbf{TopSys}$, $\mathbf{[0,1]}$-$\mathbf{Top}$ and $\mathbf{Frm}$.
\subsection{Functors}
In this subsection, we define various functors required to prove our desired results.
\subsection*{Functor $Ext$ from $\mathbf{[0,1]}$-$\mathbf{TopSys}$ to $\mathbf{[0,1]}$-$\mathbf{Top}$}\index{functor!Ext}
\begin{definition}\label{3.8_1}
Let $(X,\models,A)$ be a fuzzy topological system. For each $a \in  A$, its \textbf{extent}\index{extent} in $(X,\models, A)$ is a mapping $ext(a)$ from $X$ to $[0,1]$ given by $ext(a)(x)=gr(x\models a)$.
%i.e. $ext(a):X\longrightarrow [0,1]$ such that $ext(a)(x)=gr(x\models a)$\\
Also $ext(A)=\{ext(a)\}_{a\in A}$.
\end{definition}
\begin{lemma}\label{3.9_1}
$ext(A)$ forms a fuzzy topology \cite{CC} on $X$.
\end{lemma}
\begin{proof}
Let $ext(a_1),\ ext(a_2)\in ext(A)$. Then,
\begin{align*}
(ext(a_1)\cap ext(a_2))(x) & =min \{ext(a_1)(x),ext(a_2)(x)\}\\
& =min \{gr(x\models a_1),gr(x\models a_2)\}\\
& =gr(x\models a_1\wedge a_2)\\
& =ext(a_1\wedge a_2)(x).
\end{align*}
So, $ ext(a_1)\cap ext(a_2)\in ext(A)$. 

Similarly, $\bigcup_i ext(a_i)=ext(\bigvee_i a_i)\in ext(A)$, for $i\in I $ (an index set), i.e., $\bigcup_i ext(a_i)\in ext(A).$

Lastly we have $\tilde X (x)=1 =gr(x\models \top) =ext(\top)(x). \ \text{So},\ \tilde X =ext(\top)\in ext(A)$ and, $\tilde \emptyset (x)=0 =gr(x\models \bot) =ext(\bot)(x)$. So, $\tilde \emptyset =ext(\bot)\in ext(A)$.
\end{proof}
As a consequence $(X,ext(A))$ forms a fuzzy topological space.
\begin{lemma}\label{3.10_1}
If $(f_1,f_2):(X,\models ',A)\longrightarrow (Y,\models '',B)$ is continuous then \\$f_1:(X,ext(A))\longrightarrow (Y,ext(B))$ is fuzzy continuous.
\end{lemma}
\begin{proof}
$(f_1,f_2):(X,\models ',A)\longrightarrow (Y,\models '',B)$ is continuous.\\
So we have,
\begin{equation}\label{o}
gr(x\models 'f_2(b))=gr(f_1(x)\models ''b), \hspace{3mm} \textrm{for all}  \ x\in X,  \ b\in B.
\end{equation}
Now,
\begin{align*} 
(f_1^{-1}(ext(b)))(x) & = ext(b)(f_1(x))\\
& = gr(f_1(x)\models ''b)\\
& = gr(x\models 'f_2(b)),  \hspace{3mm} by\  \eqref{o}\\ 
& = ext(f_2(b))(x).
\end{align*}
So, $f_1^{-1}(ext(b))=ext(f_2(b))\in ext(A)$.
Therefore, $f_1$ is a fuzzy continuous map from $(X,ext(A))$ to $(Y,ext(B))$.
\end{proof}
\begin{definition}\label{3.11_1}
$\mathbf{Ext}$ is a functor from $\mathbf{[0,1]}$-$\mathbf{TopSys}$ to $\mathbf{[0,1]}$-$\mathbf{Top}$ defined as follows.\\
$Ext$ acts on an object $(X,\models ',A)$ as $Ext(X,\models ',A)=(X,ext(A))$ and on a morphism $(f_1,f_2)$ as $Ext(f_1,f_2)=f_1$.
\end{definition}
The above two Lemmas \ref{3.9_1} and \ref{3.10_1} show that $Ext$ is a functor.
%The diagram below express the above fact-
%\begin{center}
%\begin{tabular}{ M{4cm} |M{4cm}} 

%\tikzmark{a}{$\mathbf{[0,1]}$-$\mathbf{TopSys}$} & \tikzmark{1}{$\mathbf{[0,1]}$-$\mathbf{Top}$} \\ \hline 
%{
%\begin{center}\begin{tikzpicture}[description/.style={fill=white,inner sep=2pt}] 
%    \matrix (m) [matrix of math nodes, row sep=2.5em, column sep=2.5em]
%    {& &(X,\models ',A)  \\
%        & &(Y, \models '',B) \\ }; 
%    \path[->,font=\scriptsize] 

        %(m-1-5) edge node[auto,swap] {$j$} (m-1-3)
        %(m-1-5) edge node[auto] {$\psi$} (m-2-3)
%        (m-1-3) edge node[auto] {$(f_1,f_2)$} (m-2-3);
%\end{tikzpicture}\end{center}}  & {\begin{center}\begin{tikzpicture}[description/.style={fill=white,inner sep=2pt}] 
%    \matrix (m) [matrix of math nodes, row sep=2.5em, column sep=2.5em]
%    {& &(X,ext(A))  \\
%        & &(Y, ext(B)) \\ }; 
%    \path[->,font=\scriptsize] 

        %(m-1-5) edge node[auto,swap] {$j$} (m-1-3)
        %(m-1-5) edge node[auto] {$\psi$} (m-2-3)
%        (m-1-3) edge node[auto] {$f_1$} (m-2-3);
%\end{tikzpicture}\end{center}}\\
%\end{tabular}

%\end{center}
%\link{a}{1}
\subsection*{Functor $J$ from $\mathbf{[0,1]}$-$\mathbf{Top}$ to $\mathbf{[0,1]}$-$\mathbf{TopSys}$}
\begin{definition}\label{3.12_1}\index{functor!J}
$\mathbf{J}$ is a functor from $\mathbf{[0,1]}$-$\mathbf{Top}$ to $\mathbf{[0,1]}$-$\mathbf{TopSys}$ defined as follows.\\
$J$ acts on an object $(X,\tau)$ as $J(X,\tau)=(X, \in ,\tau)$ where $gr(x\in \tilde{T})=\tilde{T}(x)$ for $\tilde{T}$ in $\tau$ and on a morphism $f$ as $J(f)=(f,f^{-1})$.
\end{definition}
%The diagram below express the above facts-
%\begin{center}     
%\begin{tabular}{ M{4cm} |M{4cm}} 
%\tikzmark{q}{$\mathbf{[0,1]}$-$\mathbf{Top}$} & \tikzmark{2}{$\mathbf{[0,1]}$ $\mathbf{TopSys}$}\\
%\hline
%{\begin{tikzpicture}[description/.style={fill=white,inner sep=2pt}] 
%    \matrix (m) [matrix of math nodes, row sep=2.5em, column sep=2.5em]
%    {& &(X,\tau_1)  \\
%        & &(Y, \tau_2) \\ }; 
%    \path[->,font=\scriptsize] 

        %(m-1-5) edge node[auto,swap] {$j$} (m-1-3)
        %(m-1-5) edge node[auto] {$\psi$} (m-2-3)
%        (m-1-3) edge node[auto] {$f$} (m-2-3);
%\end{tikzpicture}}  &  {\begin{tikzpicture}[description/.style={fill=white,inner sep=2pt}] 
%    \matrix (m) [matrix of math nodes, row sep=2.5em, column sep=2.5em]
%    {& &(X,\in ,\tau_1)  \\
%        & &(Y,\in ,\tau_2) \\ }; 
%    \path[->,font=\scriptsize] 

        %(m-1-5) edge node[auto,swap] {$j$} (m-1-3)
        %(m-1-5) edge node[auto] {$\psi$} (m-2-3)
%        (m-1-3) edge node[auto] {$(f,f^{-1})$} (m-2-3);
%\end{tikzpicture}}\\
%\end{tabular}
%\link{q}{2}
%\end{center}
\begin{lemma}\label{3.13_1}
$(X,\in , \tau)$ is a fuzzy topological system.
\end{lemma}
\begin{proof}
It will be enough to show that $gr(x\in \bigcup_i \tilde{T_i})=sup_i\{ gr(x\in \tilde{T_i})\}$ and $gr(x\in \tilde{T_1}\cap \tilde{T_2})=inf\{ gr(x\in \tilde{T_1}),gr(x\in \tilde{T_2})\}$.
$$(1)\ gr(x\in \bigcup_i \tilde{T_i})  =(\bigcup_i \tilde{T_i})(x)
=sup_i\{\tilde{T_i}(x)\}
 =sup_i\{ gr(x\in \tilde{T_i})\}.$$
\begin{align*}
(2)\ gr(x\in \tilde{T_1}\cap \tilde{T_2}) & =(\tilde{T_1}\cap \tilde{T_2})(x)\\
& =min\{ \tilde{T_1}(x), \tilde{T_2}(x)\}\\
& =min\{ gr(x\in \tilde{T_1}), gr(x\in \tilde{T_2})\}\\
& =inf\{ gr(x\in \tilde{T_1}), gr(x\in \tilde{T_2})\}.
\end{align*}
This finishes the proof.
\end{proof}
\begin{lemma}\label{3.14_1}
$J(f)=(f,f^{-1})$ is continuous provided $f$ is fuzzy continuous.
\end{lemma}
\begin{proof}
We have $f:(X,\tau_1)\longrightarrow (Y,\tau_2)$ and $(f,f^{-1}):(X,\in ,\tau_1)\longrightarrow (Y,\in ,\tau_2)$.
We have to show  $$ gr(x\in f^{-1}(\tilde{T_2}))=gr(f(x)\in \tilde{T_2}), \  \textrm{for}\  \tilde{T_2}\in \tau_2.$$
Now, $$gr(x\in f^{-1}(\tilde{T_2}))=(f^{-1}(\tilde{T_2}))(x)=\tilde{T_2}(f(x))=gr(f(x)\in \tilde{T_2}).$$
Hence $J(f)=(f,f^{-1})$ is continuous.
\end{proof}
So $ J$ is a functor from $\mathbf{[0,1]}$-$\mathbf{Top}$ to $\mathbf{[0,1]}$-$\mathbf{TopSys}$.
\subsection*{Functor $fm$ from $\mathbf{[0,1]}$-$\mathbf{TopSys}$ to $\mathbf{Frm^{op}}$}
\begin{definition}\label{3.15_1}\index{functor!fm}
$\mathbf{fm}$ is a functor from $\mathbf{[0,1]}$-$\mathbf{TopSys}$ to $\mathbf{Frm^{op}}$ defined as follows.\\
$fm$ acts on an object $(X,\models ,A)$ as $fm(X,\models ,A)=A$ and on a morphism $(f_1,f_2)$ as $fm(f_1,f_2)=f_2$.
\end{definition}
 It is easy to see that $fm$ is a functor indeed.
%  \\ The diagram below express the above facts-
%  \begin{center}     
%\begin{tabular}{ l | r  } 
%$\mathsf{Fuzzy}$ $\mathsf{Top}$ $\mathsf{System}$ --------&---$\xrightarrow{fm}$---------- $\mathsf{Frm^{op}}$\\
%\hline
%{\begin{tikzpicture}[description/.style={fill=white,inner sep=2pt}] 
%    \matrix (m) [matrix of math nodes, row sep=2.5em, column sep=2.5em]
%    {& &(X,\models ',A)  \\
%        & &(Y, \models '',B) \\ }; 
%    \path[->,font=\scriptsize] 

        %(m-1-5) edge node[auto,swap] {$j$} (m-1-3)
        %(m-1-5) edge node[auto] {$\psi$} (m-2-3)
%        (m-1-3) edge node[auto] {$(f_1,f_2)$} (m-2-3);
%\end{tikzpicture}}  &  {\begin{tikzpicture}[description/.style={fill=white,inner sep=2pt}] 
%    \matrix (m) [matrix of math nodes, row sep=2.5em, column sep=2.5em]
%    {& &A  \\
%        & &B \\ }; 
%    \path[->,font=\scriptsize] 

        %(m-1-5) edge node[auto,swap] {$j$} (m-1-3)
        %(m-1-5) edge node[auto] {$\psi$} (m-2-3)
%        (m-1-3) edge node[auto] {$f_2$} (m-2-3);
%\end{tikzpicture}}\\
%\end{tabular}
%\end{center}
\subsection*{Functor $S$ from $\mathbf{Frm^{op}}$ to $\mathbf{[0,1]}$-$\mathbf{TopSys}$}
\begin{definition}\label{3.16_1}\index{functor!S}
Let $A$ be a frame, $Hom(A,[0,1])=\{ frame$ $ hom$ $v:A\longrightarrow [0,1]\}$.
\end{definition}
\begin{lemma}\label{3.17_1}
$(Hom(A,[0,1]),\models_*,A)$, where $A$ is a frame and $gr(v\models_* a)=v(a)$, is a fuzzy topological system.
\end{lemma}
\begin{proof}
It will be enough to show that $gr(v\models_* \bigvee_i a_i)=sup_i\{gr(v\models_* a_i)\}$ for $a_i\in A$ and $gr(v\models_* a\wedge b)=inf\{ gr(v\models_* a),gr(v\models_* b)\}$ for $a,\ b\in A$.
$$(1)\ gr(v\models_* \bigvee_i a_i) =v(\bigvee_i a_i)
=sup_i\{v(a_i)\}
 =sup_i\{gr(v\models_* a_i)\}.$$
\begin{align*}
(2)\ gr(v\models_* a\wedge b) & =v(a\wedge b)\\
& =min\{ v(a), v(b)\}\\
& =min\{ gr(v\models_* a),gr(v\models_* b)\}\\
& =inf\{ gr(v\models_* a),gr(v\models_* b)\}.
\end{align*}This completes the proof.
\end{proof}
\begin{lemma}\label{3.18_1}
If $f:B\longrightarrow A$ is a frame homomorphism then $$(\_\circ f,f):(Hom(A,[0,1]),\models_*,A)\longrightarrow (Hom(B,[0,1]),\models_*,B)$$is continuous.
\end{lemma}
\begin{proof}
Let us show that $gr(v\models_* f(b))=gr(v\circ f\models b)$ for $v\in Hom(A,[0,1])$ and $b\in B$. Now, $gr(v\models_*f(b))=v(f(b))=v\circ f(b)=gr(v\circ f\models b).$
\end{proof}
Recall that morphisms in $\mathbf{Frm^{op}}$ are morphisms of $\mathbf{Frm}$ but acting in opposite direction. $\mathbf{Frm^{op}}$ is also known as $\mathbf{Loc}$ \cite{PJS}, i.e. the category of locale.
\begin{definition}\label{3.19_1}
$\mathbf{S}$ is a functor from $\mathbf{Frm^{op}}$ to $\mathbf{[0,1]}$-$\mathbf{TopSys}$ defined as follows.\\
$S$ acts on an object $A$ as $S(A)=(Hom(A,[0,1]),\models_*,A)$ and on a morphism $f$ as $S(f)=(\_\circ f,f)$.
\end{definition}
Previous two Lemmas \ref{3.17_1} and \ref{3.18_1} show that $S$ is indeed a functor.
%The diagram below express the above fact-
%\begin{center}     
%\begin{tabular}{ l | r  } 
%$\mathsf{Frm^{op}}$ -----&---------$\xrightarrow{S}$---------------- $\mathsf{Fuzzy}$ $\mathsf{Top}$ $\mathsf{System}$\\
%\hline
%{\begin{tikzpicture}[description/.style={fill=white,inner sep=2pt}] 
%    \matrix (m) [matrix of math nodes, row sep=2.5em, column sep=2.5em]
%    {& & A \\
%        & &B \\ }; 
%    \path[->,font=\scriptsize] 

        %(m-1-5) edge node[auto,swap] {$j$} (m-1-3)
        %(m-1-5) edge node[auto] {$\psi$} (m-2-3)
%        (m-1-3) edge node[auto] {$f$} (m-2-3);
%\end{tikzpicture}}  &  {\begin{tikzpicture}[description/.style={fill=white,inner sep=2pt}] 
%    \matrix (m) [matrix of math nodes, row sep=2.5em, column sep=2.5em]
%    {& &(Hom(A,[0,1]),\models_* ,A)  \\
%        & &(Hom(B,[0,1]),\models_* ,B) \\ }; 
%    \path[->,font=\scriptsize] 

        %(m-1-5) edge node[auto,swap] {$j$} (m-1-3)
        %(m-1-5) edge node[auto] {$\psi$} (m-2-3)
%        (m-1-3) edge node[auto] {$(\_\circ f,f)$} (m-2-3);
%\end{tikzpicture}}\\
%\end{tabular}
%\end{center}
%where $gr(v\models_* a)=v(a)$\\
%and $Hom(A,[0,1])=\{ frame$ $hom$ $v:A\longrightarrow [0,1]\}$
\begin{lemma}\label{3.20_1}
$Ext$ is the right adjoint to the functor $J$.
\end{lemma}
\begin{proof}
It is possible to prove the theorem by presenting the co-unit of the adjunction.
Recall that $J(X,\tau)=(X,\in, \tau)$ and $Ext(X,\models,A)=(X,ext(A))$.\\
So, $J(Ext(X,\models,A))=(X,\in,ext(A))$.

Let us draw the diagram of co-unit.
\begin{center}
\begin{tabular}{ l | r } 
$\mathbf{[0,1]}$-$\mathbf{TopSys}$ & $\mathbf{[0,1]}$-$\mathbf{Top}$\\
\hline
 {\begin{tikzpicture}[description/.style={fill=white,inner sep=2pt}] 
    \matrix (m) [matrix of math nodes, row sep=2.5em, column sep=2.5em]
    { J(Ext(X,\models,A))&&(X,\models,A)  \\
         J(Y,\tau ') \\ }; 
    \path[->,font=\scriptsize] 
        (m-1-1) edge node[auto] {$\xi_X$} (m-1-3)
        (m-2-1) edge node[auto] {$J(f)(\equiv(f_1,f_1^{-1}))$} (m-1-1)
        (m-2-1) edge node[auto,swap] {$\hat f (\equiv(f_1,f_2))$} (m-1-3)
        %(m-1-1) edge node[auto,swap] {$f$} (m-2-3)
       % (m-1-5) edge node[auto,swap] {$j$} (m-1-3)
       % (m-1-5) edge node[auto] {$\psi$} (m-2-3)
        %(m-1-3) edge node[auto] {$G(\hat{f})$} (m-2-3)
         ;
\end{tikzpicture}} & {\begin{tikzpicture}[description/.style={fill=white,inner sep=2pt}] 
    \matrix (m) [matrix of math nodes, row sep=2.5em, column sep=2.5em]
    { Ext(X,\models,A)  \\
         (Y,\tau ') \\ }; 
    \path[->,font=\scriptsize] 
        (m-2-1) edge node[auto,swap] {$f(\equiv f_1)$} (m-1-1)
       
         ;
\end{tikzpicture}} \\ 
\end{tabular}
\end{center}

Hence co-unit is defined by $\xi_X=(id_X,ext^*)$. That is,

\begin{center}
\begin{tikzpicture}[description/.style={fill=white,inner sep=2pt}] 
    \matrix (m) [matrix of math nodes, row sep=2.5em, column sep=2.5em]
    {(X,\in ,ext(A))&&(X,\models,A),  \\
          }; 
    \path[->,font=\scriptsize] 
        (m-1-1) edge node[auto] {$\xi_X$} (m-1-3)
        (m-1-1) edge node[auto,swap] {$(id_X,ext^*)$} (m-1-3)
         ;
\end{tikzpicture}
\end{center}
where $ext^*$ is a mapping from $A$ to $ext(A)$ such that $ext^*(a)=ext(a)$, for all $a\in A$.
\begin{claim}\label{3.21_1}$(id_X,ext^*):J(Ext(X,\models ,A))\longrightarrow (X,\models ,A)$ is a continuous map of fuzzy topological system.
\end{claim}
$\mathit{Proof\ of\ the\ Claim.}$
Let us show that, $ext^*(a)(x)=gr(id_X(x)\models a)$.
We have, $ext(a)(x)=gr(x\models a)$
So, $ext^*(a)(x)=gr(id_X(x)\models a)$.\ \ \ \openbox 

Let us define $f$ as follows.
Given $(f_1,f_2):J(Y,\tau')\longrightarrow (X,\models,A)$, then $f=f_1$.
It suffices to show that the diagram on the left commutes. We have $J(f)=(f_1,f_1^{-1})$. Hence,
\begin{align*}
 (f_1,f_2) & =\xi_X\circ J(f)\\
& =(id_X,ext^*)\circ (f_1,f_1^{-1}) \tag{as $J(f)=(f_1,f_1^{-1})$}\\% \label {eq: as $J(f)=(f_1,f_1^{-1})$}\\
& =(id_X\circ f_1,f_1^{-1}\circ ext^*).
\end{align*}
Clearly $id_X\circ f_1=f_1$. Now we will show that $f_2=f_1^{-1}\circ ext^*$.
As $(id_X,ext^*)$ is continuous, $ext^*(a)(x)=gr(x\models a)$, i.e., $ext^*(a)=a$.
Hence $f_1^{-1}ext^*(a)  =f_1^{-1}(a)=f_2(a).$
Therefore $f_1^{-1}ext^*(a) =f_1^{-1}(a)=f_2(a)$   
(as $(f_1,f_2)$ is continuous so for $a\in A, f_2(a)(x)=gr(f_1(x)\models a)$, i.e., $f_2(a)=f_1^{-1}(a)$).
Hence $$\xi_X(\equiv(id_X,ext^*)):J(Ext(X,\models,A))\longrightarrow (X,\models,A)$$ is the co-unit, consequently $Ext$ is the right adjoint to the functor $J$.
\end{proof}
Diagram of the unit of the above adjunction is as follows.

\begin{center}
\begin{tabular}{ l | r  } 
 $\mathbf{[0,1]}$-$\mathbf{Top}$ & $\mathbf{[0,1]}$-$\mathbf{TopSys}$ \\
\hline
 {\begin{tikzpicture}[description/.style={fill=white,inner sep=2pt}] 
    \matrix (m) [matrix of math nodes, row sep=2.5em, column sep=2.5em]
    {(X,\tau) & &Ext(J(X, \tau))  \\
        & & Ext(Y,\models ,B) \\ }; 
    \path[->,font=\scriptsize] 
        (m-1-1) edge node[auto] {$\eta_X(\equiv id_X)$} (m-1-3)
        (m-1-1) edge node[auto,swap] {$\hat{f}(\equiv f_1)$} (m-2-3)
       % (m-1-5) edge node[auto,swap] {$j$} (m-1-3)
        %(m-1-5) edge node[auto] {$\psi$} (m-2-3)
        (m-1-3) edge node[auto] {$ext(f)(\equiv f_1)$} (m-2-3);
\end{tikzpicture}} &  {\begin{tikzpicture}[description/.style={fill=white,inner sep=2pt}] 
    \matrix (m) [matrix of math nodes, row sep=2.5em, column sep=2.5em]
    {& &J(X,\tau)  \\
        & &(Y,\models ,B) \\ }; 
    \path[->,font=\scriptsize]

        %(m-1-5) edge node[auto,swap] {$j$} (m-1-3)
        %(m-1-5) edge node[auto] {$\psi$} (m-2-3)
        (m-1-3) edge node[auto] {$f(\equiv(f_1,f_1^{-1}))$} (m-2-3);
\end{tikzpicture}}\\
\end{tabular}
\end{center}
%Hence we have $(\eta_X,\xi_X):J\dashv Ext:FBSy_n\longrightarrow FBS_n$.

\begin{observation}\label{o1}
If a fuzzy topological system $(X,\models,A)$ is spatial then the co-unit $\xi_X$ becomes a natural isomorphism.
\end{observation}
\begin{observation}\label{o2}
For any fuzzy topological space $(X,\tau)$, the unit $\eta_X$ is a natural isomorphism.  
\end{observation}
Observation \ref{o1} and Observation \ref{o2} gives the following theorem.
\begin{theorem}
Category of spatial fuzzy topological systems is equivalent to the category $\mathbf{[0,1]}$-$\mathbf{Top}$.
\end{theorem}
\begin{lemma}\label{3.22_1}
fm is the left adjoint to the functor $S$.
\end{lemma}
\begin{proof}
It is possible to prove the theorem by presenting the unit of the adjunction.
Recall that  $S(B)=(Hom(B,[0,1]),\models_*,B)$, where $gr(v\models_* a)=v(a)$, and $fm(X,\models,A)=A$.
Hence, $$S(fm(X,\models ,A))=(Hom(A,[0,1]),\models_* ,A).$$

\begin{center}
\begin{tabular}{ l | r  } 
 $\mathbf{[0,1]}$-$\mathbf{TopSys}$ & $\mathbf{Frm^{op}}$ \\
\hline
 {\begin{tikzpicture}[description/.style={fill=white,inner sep=2pt}] 
    \matrix (m) [matrix of math nodes, row sep=2.5em, column sep=2.5em]
    {(X,\models ,A) & &S(fm(X, \models ,A))  \\
        & & S(B) \\ }; 
    \path[->,font=\scriptsize] 
        (m-1-1) edge node[auto] {$\eta_A$} (m-1-3)
        (m-1-1) edge node[auto,swap] {$f(\equiv(f_1,f_2))$} (m-2-3)
       % (m-1-5) edge node[auto,swap] {$j$} (m-1-3)
        %(m-1-5) edge node[auto] {$\psi$} (m-2-3)
        (m-1-3) edge node[auto] {$S\hat{f}$} (m-2-3);
\end{tikzpicture}} &  {\begin{tikzpicture}[description/.style={fill=white,inner sep=2pt}] 
    \matrix (m) [matrix of math nodes, row sep=2.5em, column sep=2.5em]
    {& &fm(X,\models ,A)  \\
        & &B \\ }; 
    \path[->,font=\scriptsize]

        %(m-1-5) edge node[auto,swap] {$j$} (m-1-3)
        %(m-1-5) edge node[auto] {$\psi$} (m-2-3)
        (m-1-3) edge node[auto] {$\hat{f}(\equiv f_2)$} (m-2-3);
\end{tikzpicture}}\\
\end{tabular}
\end{center}
Then unit is defined by $\eta_A =(p*,id_A)$. That is,
\begin{center}
 \begin{tikzpicture}[description/.style={fill=white,inner sep=2pt}] 
    \matrix (m) [matrix of math nodes, row sep=2.5em, column sep=2.5em]
    {(X,\models ,A) & &S(fm(X, \models ,A)),  \\
        }; 
    \path[->,font=\scriptsize] 
        (m-1-1) edge node[auto] {$\eta_A$} (m-1-3)
        (m-1-1) edge node[auto,swap] {$(p^*,id_A)$} (m-1-3)
        %(m-1-1) edge node[auto,swap] {} (m-2-3)
       % (m-1-5) edge node[auto,swap]
        ;
\end{tikzpicture}
\end{center}
where, 
\begin{align*}
  p^* \colon X &\longrightarrow Hom(A,[0,1])\\
  x &\longmapsto p_x\colon A\longrightarrow [0,1]
\end{align*}
such that $p_x(a)=gr(x\models a)$.
\begin{claim}\label{3.23_1} For each $x\in X,$ $p_x:A\longrightarrow [0,1]$ is a frame homomorphism.
\end{claim}
$\mathit{Proof\ of\ the\ claim.}$
$p_x(a_1\wedge a_2)=gr(x\models a_1\wedge a_2)=min\{gr(x\models a_1),gr(x\models a_2)\}=min\{p_x(a_1),p_x(a_2)\}$.
$p_x(\bigvee_i a_i)=gr(x\models \bigvee_i a_i)=sup_i\{gr(x\models a_i)\}=sup_i\{p_x(a_i)\}$.\ \ \openbox

\begin{claim}\label{3.24_1} $(p^*,id_A):(X,\models,A)\longrightarrow S(fm(X,\models,A)) $ is a continuous map of fuzzy topological system.
\end{claim}
$\mathit{Proof\ of\ the Claim.}$
Here it will be enough to show that $gr(x\models id_A(a))=gr(p^*(x)\models_* a)$.
Now, $gr(p^*(x)\models_* a)=p^*(x)(a)=p_x(a)=gr(x\models a)=gr(x\models id_A(a))$.\ \ \openbox 

Let us define $\hat{f}$ as follows:
$(f_1,f_2):(X,\models ,A)\longrightarrow (Hom(B,[0,1]),\models_*,B)$
\\then $\hat{f}=f_2$ (as $f_2$ is the frame homomorphism).
Recall that $S(\hat{f})=(\_\circ f_2,f_2)$.

Now it suffices to show that the triangle on the left commute. Recall that $S(\hat{f})=(-\circ f_2,f_2)$.
To show, $(f_1,f_2)=S(\hat{f})\circ \eta_A=(\_\circ f_2,f_2)\circ (p^*,id_A)=((\_\circ f_2)\circ p^*,id_A\circ f_2)$.
Clearly $f_2=id_A\circ f_2.$

It is only left to show that $f_1=(\_\circ f_2)\circ p^*$, i.e., for $x\in X,$ $f_1(x)=(\_\circ f_2)\circ p^*(x)=(\_\circ f_2)\circ p_x=p_x\circ f_2.$
Now for all $b\in B$, $p_x\circ f_2(b)=gr(x\models f_2(b))=gr(f_1(x)\models_*b)=f_1(x)(b).$
Hence,
$$\eta_A(\equiv(p^*,id_A)):(X,\models ,A)\longrightarrow S(fm(X,\models ,A))$$
is the unit, consequently $fm$ is the left adjoint to the functor $S$.
\end{proof}
Diagram of the co-unit of the above adjunction is as follows.

\begin{center}
\begin{tabular}{ l | r } 
$\mathbf{Frm^{op}}$ & $\mathbf{[0,1]}$-$\mathbf{TopSys}$\\
\hline
 {\begin{tikzpicture}[description/.style={fill=white,inner sep=2pt}] 
    \matrix (m) [matrix of math nodes, row sep=2.5em, column sep=2.5em]
    { fm(S(A))&& A  \\
         fm(Y,\models ,B) \\ }; 
    \path[->,font=\scriptsize] 
        (m-1-1) edge node[auto] {$\xi_A (\equiv id_A)$} (m-1-3)
        (m-2-1) edge node[auto] {$fm(f)(\equiv f')$} (m-1-1)
        (m-2-1) edge node[auto,swap] {$\hat f (\equiv f')$} (m-1-3)
        %(m-1-1) edge node[auto,swap] {$f$} (m-2-3)
       % (m-1-5) edge node[auto,swap] {$j$} (m-1-3)
       % (m-1-5) edge node[auto] {$\psi$} (m-2-3)
        %(m-1-3) edge node[auto] {$G(\hat{f})$} (m-2-3)
         ;
\end{tikzpicture}} & {\begin{tikzpicture}[description/.style={fill=white,inner sep=2pt}] 
    \matrix (m) [matrix of math nodes, row sep=2.5em, column sep=2.5em]
    { S(A)  \\
         (Y,\models ,B) \\ }; 
    \path[->,font=\scriptsize] 
        (m-2-1) edge node[auto,swap] {$f(\equiv(\_\circ f',f'))$} (m-1-1)
       
         ;
\end{tikzpicture}} \\ 
\end{tabular}
\end{center}

\begin{observation}\label{o3}
If a fuzzy topological system $(X,\models,A)$ is localic then the unit $\eta_A$ becomes a natural isomorphism.
\end{observation}
\begin{observation}\label{o4}
For any frame $A$, the co-unit $\xi_A$ is a natural isomorphism.  
\end{observation}
Observation \ref{o3} and Observation \ref{o4} gives the following theorem.
\begin{theorem}
Category of localic fuzzy topological systems is dually equivalent to the category of frames.
\end{theorem}
\begin{theorem}\label{3.25_1}
$Ext\circ S$ is the right adjoint to the functor $fm\circ J$.
\end{theorem}
\begin{proof}
Follows from the combination of the adjoint situations in Lemmas \ref{3.20_1} and \ref{3.22_1}.
\end{proof}
The obtained functorial relationships can be illustrated by the following diagram:
  
\begin{center}
\begin{tikzpicture}
\node (C) at (0,3) {$\mathbf{[0,1]}$-$\mathbf{TopSys}$};
\node (A) at (-2,0) {$\mathbf{[0,1]}$-$\mathbf{Top}$};
\node (B) at (2,0) {$\mathbf{Frm^{op}}$};
%\node at (0,0) {\rotatebox{270}{$\Rightarrow$}};
\path[->,font=\scriptsize ,>=angle 90]
(A) edge [bend left=15] node[above] {$fm\circ J$} (B);
\path[<-,font=\scriptsize ,>=angle 90]
(A)edge [bend right=15] node[below] {$Ext \circ S$} (B);
\path[->,font=\scriptsize ,>=angle 90]
(A) edge [bend left=20] node[above] {$J$} (C);
\path[<-,font=\scriptsize ,>=angle 90]
(A)edge [bend right=20] node[above] {$Ext$} (C);
\path[->,font=\scriptsize, >=angle 90]
(C) edge [bend left=20] node[above] {$fm$} (B);
\path[<-,font=\scriptsize, >=angle 90]
(C)edge [bend right=20] node[above] {$S$} (B);
\end{tikzpicture}
\end{center}

The obtained adjunction between the categories of fuzzy topological spaces and locales is a fuzzification of the well-known $\mathbf{Top}$-$\mathbf{Loc}$ adjunction \cite{PJS}.
\section{Sum and Product of fuzzy topological systems}
\begin{definition}\label{4.1_1}
Let $\{D_{\lambda}\}$ be a family of fuzzy topological systems, where $D_{\lambda}\equiv (X_{\lambda},\models_{\lambda},A_{\lambda})$. The fuzzy topological sum\index{fuzzy topological!sum} $\sum D_{\lambda} =(X,\models^*,A)$ is defined by,\\
1. $X=\bigcup X_{\lambda}$, union of the sets;\\
2. $A=\prod A_{\lambda}$, the Cartesian product of the frames;\\
3. $gr(z\models^*<a_{\lambda}>)=\bigvee_{\lambda}(\tilde{X_{\lambda}}(z)\wedge gr(z\models_{\lambda}a_{\lambda}))$, where $\tilde{X_{\lambda}}$ is the membership function of the set $X_{\lambda}$(characteristic function).
\end{definition}
It should be noted that the Cartesian product of frames is a frame (c.f. Proposition \ref{cart}).
\begin{lemma}\label{4.2_1}
A sum of fuzzy topological systems is a fuzzy topological system.
\end{lemma}
\begin{proof}
$\prod A_{\lambda}$ is a frame. It will be enough to show that\\
1. $gr(z\models^*<a_{\lambda}>\wedge <b_{\lambda}>)=gr(z\models^*<a_{\lambda}>)\wedge gr(z\models^*<b_{\lambda}>)$ and\\
2. $gr(z\models^*\bigvee_i<a_{\lambda}^i>)=\bigvee_igr(z\models^*<a_{\lambda}^i>)$. Let us proceed in the following way.

$1.\ gr(z\models^*<a_{\lambda}>\wedge <b_{\lambda}>)\\ =\bigvee_\lambda (z\in X_{\lambda}\wedge gr(z\models_{\lambda} a_{\lambda}\wedge b_{\lambda}))\\
 =\bigvee_\lambda(z\in X_{\lambda}\wedge gr(z\models_{\lambda}a_{\lambda})\wedge gr(z\models_{\lambda}b_{\lambda}))\\
 =\bigvee_\lambda((z\in X_{\lambda}\wedge gr(z\in a_{\lambda}))\wedge (z\in X_{\lambda}\wedge gr(z\models_{\lambda}b_{\lambda})))\\
 =\bigvee_\lambda(z\in X_{\lambda}\wedge gr(z\in a_{\lambda}))\wedge \bigvee_\lambda(z\in X_{\lambda}\wedge gr(z\models_{\lambda}b_{\lambda}))\\
 =gr(z\models^*<a_{\lambda}>)\wedge gr(z\models^*<b_{\lambda}>).$
\begin{align*}
2.\ gr(z\models^*\bigvee_i<a_{\lambda}^i>) & =\bigvee_{\lambda}(\tilde{X_{\lambda}}(z)\wedge gr(z\models_{\lambda}\bigvee_i a_{\lambda}^i))\\
& =\bigvee_{\lambda}(\tilde{X_{\lambda}}(z)\wedge (\bigvee_i gr(z\models_{\lambda} a_{\lambda}^i)))\\
& =\bigvee_{\lambda}(\bigvee_i (\tilde{X_{\lambda}}(z)\wedge gr(z\models_{\lambda} a_{\lambda}^i)))\\
& =\bigvee_i (\bigvee_{\lambda}(\tilde{X_{\lambda}}(z)\wedge gr(z\models_{\lambda} a_{\lambda}^i)))\\
& =\bigvee_igr(z\models^*<a_{\lambda}^i>).
\end{align*}
This completes the proof.
\end{proof}
\begin{definition}\label{4.3_1}
Let $A$ and $B$ be two frames. The tensor product (coproduct)\index{tensor product}\index{product} $A\otimes B$ is the frame presented as follows:
\begin{multline*}
$$Fr(a\otimes b:a\in A, b\in B\mid  
\bigwedge_i(a_i\otimes b_i)=(\bigwedge_i a_i)\otimes (\bigwedge_i b_i),
\bigvee_i(a_i\otimes b)\\=(\bigvee_i a_i)\otimes b,
\bigvee_i(a\otimes b_i)=a\otimes (\bigvee_i b_i)).$$
\end{multline*}
We also define two injections $i_A:A\longrightarrow A\otimes B$ and $i_B:B\longrightarrow A\otimes B$ by\\
$i_A(a)=a\otimes 1$ and $i_B(b)=1\otimes b.$
\end{definition}
\begin{lemma}\label{4.4_1}
Every element of $A\otimes B$ can be written in the form $\bigvee_i(a_i\otimes b_i)$ for some $a_i\in A$ and $b_i\in B.$
\end{lemma}
\begin{definition}\label{4.5_1}
Let $D=(X,\models,A),E=(Y,\models,B)$ be fuzzy topological systems. The fuzzy topological product $D\times E=(Z,\models^*,C)$, where 
\\(i) $Z=X\times Y$ is the Cartesian product;\\
(ii) $C=A\otimes B$ is the tensor product of frames;\\
(iii) $gr((x,y)\models^*\bigvee_i(a_i\otimes b_i))=\bigvee_i(gr(x\models a_i)\wedge gr(y\models b_i))$.
\end{definition}
\begin{lemma}\label{4.6_1}
Product of two fuzzy topological systems is a fuzzy topological system.
\end{lemma}
\begin{proof}
It will be enough to show that $gr((x,y)\models^*\bigvee_i(a_i\otimes b_i)\wedge \bigvee_j(c_j\otimes d_j))=gr((x,y)\models^*\bigvee_i(a_i\otimes b_i))\wedge gr((x,y)\models^*\bigvee_j(c_j\otimes d_j))$ and $gr((x,y)\models^*\bigvee_j(\bigvee_i(a_i^j\otimes b_i^j)))=\bigvee_j gr((x,y)\models^*\bigvee_i(a_i^j\otimes b_i^j))$. Let us proceed in the following way.
\begin{align*}
1.\ & gr((x,y)\models^*\bigvee_i(a_i\otimes b_i)\wedge \bigvee_j(c_j\otimes d_j))\\ 
& =gr((x,y)\models^*\bigvee_{i,j}(a_i\wedge c_j)\otimes (b_i\wedge d_j))\\
& =\bigvee_{i,j}(gr(x\models a_i\wedge c_j)\wedge gr(y\models b_i\wedge d_j))\\
& =\bigvee_{i,j}(gr(x\models a_i)\wedge gr(x\models c_j)\wedge gr(y\models b_i)\wedge gr(y\models d_j))\\
& =\bigvee_i(gr(x\models a_i)\wedge gr(y\models b_i))\wedge \bigvee_j(gr(x\models c_j)\wedge gr(y\models d_j))\\
& =gr((x,y)\models^*\bigvee_i(a_i\otimes b_i))\wedge gr((x,y)\models^*\bigvee_j(c_j\otimes d_j)).
\end{align*}
\begin{align*}
2.\ gr((x,y)\models^*\bigvee_j(\bigvee_i(a_i^j\otimes b_i^j)))
& =gr((x,y)\models^*\bigvee_i(\bigvee_j(a_i^j\otimes b_i^j)))\\
& =\bigvee_i(\bigvee_j(gr(x\models a_i^j)\wedge gr(y\wedge b_i^j)))\\
& =\bigvee_j\bigvee_i(gr(x\models a_i^j)\wedge gr(y\wedge b_i^j))\\
& =\bigvee_j gr((x,y)\models^*\bigvee_i(a_i^j\otimes b_i^j)).
\end{align*}
This completes the proof.
\end{proof}
This section provides two ways to construct new fuzzy topological system from given fuzzy topological systems. The sum and product of fuzzy topological systems are defined in such a way that they corroborate with co-product and product of fuzzy topological systems respectively in the categorical frame work.

\chapter{Category of L-valued Fuzzy Topological Systems over Fuzzy Sets}
\section{Introduction}\blfootnote{The results of this chapter appear in {\bf \cite{MP1} P. Jana and M.K. Chakraborty, \emph{On Categorical Relationship among various Fuzzy Topological Systems, Fuzzy Topological Spaces and related Algebraic Structures}, Proceedings of the 13th Asian Logic Conference, World Scientific, 2015, pp. 124--135.} and {\bf \cite{MP2} P. Jana and M.K. Chakraborty, \emph{Categorical relationships of fuzzy topological systems with fuzzy topological spaces and underlying algebras-II},  Ann. of Fuzzy Math. and Inform., \textbf{10}, 2015, no. 1, pp. 123--137.}}
This chapter provides a generalization of the notions of $[0,1]$-$\mathbf{Top}$, $[0,1]$-$\mathbf{TopSys}$ (described in Chapter 2), $\mathbf{Frm^{op}}$ and their categorical relationships. First of all instead of working on fuzzy topological space on ordinary set we deal with fuzzy topological space on fuzzy set, fuzzy topological system whose underlying set is a fuzzy set. Secondly we generalize the value set too i.e. instead of working on the value set [0,1] we take $L$ which is a frame. It is to be noted that in \cite{MP1}, a special case of the work done in this chapter taking the value set as [0,1] has been dealt with. Furthermore, two ways of constructing subspaces and subsystems of 
an $\mathscr{L}$-topological space and an $\mathscr{L}$-topological system respectively are provided.
\section{Categories: $\mathscr{L}$-Top, $\mathscr{L}$-TopSys and Loc}
\subsection{Categories}
\subsection*{$\mathscr{L}$-Top}
\begin{definition}[$\mathscr{L}$-topological space]\label{ltop}
Let $L$ be a frame, $(X,\tilde{A})$ be a $L$-fuzzy set\index{$L$-fuzzy!set}, and $\tau$ be a collection of $L$-fuzzy subsets \index{$L$-fuzzy!subset} of $(X,\tilde{A})$ such that
\begin{enumerate}
\item $(X,\tilde\emptyset)$, $(X,\tilde{A})\in \tau$, where $\tilde\emptyset (x)=0$, for all $x\in X$;
\item $(X,\tilde A_i)\in\tau$ for $i\in I\ \text{implies}\ (X,\bigcup_{i\in I}\tilde A_i)\in \tau$, where $\bigcup_{i\in I}\tilde{A_i}(x)=sup_{i\in I}(\tilde{A_i}(x))$, for all $x\in X$;
\item $(X,\tilde A_1)$, $(X,\tilde A_2)\in\tau\ \text{implies}\ (X,\tilde A_1\cap\tilde A_2)\in\tau$, where for all $x\in X$, $\tilde{A_1}\cap \tilde{A_2}=min\{\tilde{A_1}(x), \tilde{A_2}(x)\}$.
\end{enumerate}
Then $(X,\tilde{A},\tau)$ is called a $\mathscr{L}$\textbf{-topological space}\index{$\mathscr{L}$- topological!space}. The elements of $\tau$ are known as \textbf{fuzzy open sets}\index{fuzzy!open set}.
\end{definition}
\begin{definition}[Proper function]\label{prop}\cite{MB}
A \textbf{proper function} \index{proper function} $f$ from $(X,\tilde{A})$ to $(Y,\tilde{B})$ is a relation from $(X,\tilde{A})$ to $(Y,\tilde{B})$ such that $\forall x\in \mid \tilde{A}\mid$, $\exists \ unique\ y\in \mid\tilde{B}\mid$ for which  $\tilde{A}(x)=f(x,y)$ and $f(x,y')=0_L$ if $y'\neq y\in\mid\tilde{B}\mid$, where $0_L$ is the least element of the frame $L$, $\mid\tilde{A}\mid=\{ x\in X : \tilde{A}(x)>0_L\}$ and $\mid\tilde{B}\mid=\{ y\in Y : \tilde{B}(y)>0_L\}$. For a fixed $x\in\mid\tilde{A}\mid$, we will denote that $unique\ y\in\mid\tilde{B\mid}$ by $f(x)$.
\end{definition}
\begin{definition}[Identity proper function]\label{id}
Let $(X,\tilde{A})$ be a $L$-fuzzy set. The map $i_{\tilde{A}}:X\times X\longrightarrow L$ is said to be 
an \textbf{identity proper function}\index{identity!proper function} if and only if  $i_{\tilde{A}}(x,x)=\tilde{A}(x)$, for any $x\in X$ and 
$i_{\tilde{A}}(x,x')= 0_L $(the least element of the frame $L)$, when $x\neq x'$ in $X$.
\end{definition}
\begin{definition}[Fuzz-top continuous]\label{3.2_1l} Let $(X,\tilde{A},\tau_1)$ and $(Y,\tilde{B},\tau_2)$ be two $\mathscr{L}$-topological spaces. A proper function $f:(X,\tilde{A})\longrightarrow (Y,\tilde{B})$ is said to be \textbf{fuzz-top continuous}\index{fuzz-top continuous} if and only if for every fuzzy open set $(Y,\tilde B_1)$ of $(Y,\tilde{B},\tau_2)$, $(X,f^{-1}(\tilde B_1))$ is an open set of $(X,\tilde{A},\tau_1)$, where $f^{-1}(\tilde B_1):X\longrightarrow L$ is defined as $f^{-1}(\tilde B_1)(a)=sup_{b\in\mid \tilde{B}\mid}\{min\{f(a,b),\tilde{B_1}(b)\}\}$, for all $a\in\mid\tilde{A}\mid$.
\end{definition}
\begin{proposition}\label{pic}
The identity proper map\index{identity!proper map} $id_X:X\longrightarrow X$ is fuzz-top continuous map and if $f:(X,\tilde{A})\longrightarrow (Y,\tilde{B})$, $g:(Y,\tilde{B})\longrightarrow (Z,\tilde{C})$ are proper functions then $g\circ f:(X,\tilde{A})\longrightarrow (Z,\tilde{C})$ is a proper function and $(g\circ f)^{-1}(\tilde{C_1})=f^{-1}(g^{-1}(\tilde{C_1}))$, where $(Z,\tilde{C_1})$ is an $L$-fuzzy subset of $(Z,\tilde{C})$.
\end{proposition}
\begin{lemma}\label{3.4_1l}
Let $(X,\tilde{A},\tau_1)$, $(Y,\tilde{B},\tau_2)$, $(Z,\tilde{C},\tau_3)$ be $\mathscr{L}$-topological spaces and $f:(X,\tilde{A})\longrightarrow (Y,\tilde{B})$, $g:(Y,\tilde{B})\longrightarrow (Z,\tilde{C})$ be fuzz-top continuous. Then $g\circ f:(X,\tilde{A})\longrightarrow (Z,\tilde{C})$ is fuzz-top continuous. 
\end{lemma}
\begin{proof}
Let $(Z,\tilde{C_1})\in \tau_3$ then $(g\circ f)^{-1}(\tilde{C_1}):X\longrightarrow L$ and $(g\circ f)^{-1}(\tilde{C_1})=f^{-1}(g^{-1}(\tilde{C_1}))$. As $g:(Y,\tilde{B})\longrightarrow (Z,\tilde{C})$ is fuzz-top continuous, $g^{-1}(\tilde{C_1})\in \tau_2$. Now $f:(X,\tilde{A})\longrightarrow (Y,\tilde{B})$ is fuzz-top continuous and $g^{-1}(\tilde{C_1})\in \tau_2$. So, $f^{-1}(g^{-1}(\tilde{C_1}))\in \tau_1$ and consequently $(g\circ f)^{-1}(\tilde{C_1})\in\tau_1$. Thereby the composition map\index{composition!map} $g\circ f:(X,\tilde{A})\longrightarrow (Z,\tilde{C})$ is fuzz-top continuous.
\end{proof}
\begin{lemma}\label{3.5_1l}
Let $(X,\tilde{A},\tau_1)$, $(Y,\tilde{B},\tau_2)$, $(Z,\tilde{C},\tau_3)$, $(U,\tilde{D},\tau_4)$ be $\mathscr{L}$-topological spaces, and $f:(X,\tilde{A})\longrightarrow (Y,\tilde{B})$, $g:(Y,\tilde{B})\longrightarrow (Z,\tilde{C})$, $h:(Z,\tilde{C})\longrightarrow (U,\tilde{D})$ be fuzz-top continuous. Then $h\circ (g\circ f)=(h\circ g)\circ f$ holds.
\end{lemma}
\begin{lemma}\label{3.6_1l}
For the fuzz-top continuous map $f:(X,\tilde{A})\longrightarrow (Y,\tilde{B})$, it holds that $id_Y\circ f=f$ and $f\circ id_X=f$.
\end{lemma}
From Lemmas \ref{3.4_1l}, \ref{3.5_1l} and \ref{3.6_1l} we get the following theorem.
\begin{theorem}\label{3.7_1l}
$\mathscr{L}$- topological spaces together with fuzz-top continuous maps form the category $\mathscr{L}$-$\mathbf{Top}$\index{category!$\mathscr{L}$-$\mathbf{Top}$}.
\end{theorem}
It is to note that the category $\mathscr{L}$-$\mathbf{Top}$ is denoted by Fuz$-$Top $(L)$ in \cite{MB} when $L$ is a complete Heyting algebra instead of a frame.
\subsection*{$\mathscr{L}$-$\mathbf{TopSys}$}
\begin{definition}[$\mathscr{L}$-topological system]\label{2}
An \textbf{$\mathscr{L}$-topological system} (where $L$ is a frame) is a quadruple $(X, \tilde{A} ,\models ,P)$, where 
$(X, \tilde{A})$ is a non-empty $L$ valued fuzzy set ($L$-fuzzy set), $P$ is a frame and $\models$ is an $L$- fuzzy 
relation from $X$ to $P$ such that
 \begin{enumerate}%[$\bullet$]
\item $gr(x\models p)\in L$;
\item $gr(x\models p)\leq \tilde{A}(x)$;
\item if $S$ is a \textbf{finite subset} of $P$, then
$gr(x\models \bigwedge S) = inf\{ gr(x\models s):s\in S\}$;
\item if $S$ is  \textbf{any subset} of $P$, then
$gr(x\models \bigvee S)=sup\{ gr(x\models s):s\in S\}$.
\end{enumerate} 
\end{definition}
\underline{Note 1:} Because of condition 2, $\models$ is a fuzzy relation on the $L$-fuzzy set $(X,\tilde{A})$ \cite{MM}.

\underline{Note 2:} The notion of topological system introduced in \cite{SV} was defined by crisp set and crisp relation whereas 
the notion of fuzzy topological system defined in \cite{MP} consists of crisp set and fuzzy relation. In our new 
setting the notion of $\mathscr{L}$-topological system is defined by $L$-fuzzy set and $L$-fuzzy relation.

The notion of continuous map between these $\mathscr{L}$-topological systems is defined as follows:
\begin{definition}\label{3}
 Let $D=(X,\tilde{A}, \models ,P)$, $E=(Y,\tilde{B}, \models ',Q)$ be $\mathscr{L}$-topological systems. A \textbf{continuous map} \index{continuous map!between!$\mathscr{L}$-topological systems} $f:D\longrightarrow E$ is a pair $(f_1,f_2)$ where,
\begin{enumerate}
\item $f_1:(X,\tilde{A})\longrightarrow (Y,\tilde{B})$ is a proper function (Definition \ref{prop}) from 
$(X,\tilde{A})$ to $(Y,\tilde{B})$;
\item $f_2:Q\longrightarrow P$ is a frame homomorphism and
\item $gr(x\models f_2(q))=gr(f_1(x)\models' q)$, for all $x\in X$ and $q\in Q$.
\end{enumerate}
\end{definition}
Let us define identity map and composition of two maps as follows:
\begin{definition}\label{4}
Let $D=(X,\tilde{A},\models,P)$ be an $\mathscr{L}$-topological system. The \textbf{identity map} \index{identity map!of!$\mathscr{L}$-topological system} 
$I_D:D\longrightarrow D$ is a pair $(I_1,I_2)$, where
%\begin{align*}
$I_1:(X,\tilde{A})\longrightarrow (X,\tilde{A})$ is an identity proper function and $I_2:P\longrightarrow P \ is\ an\ identity\ morphism\ of\ P$.
%\end{align*}

Let $D=(X,\tilde{A},\models',P)$, $E=(Y,\tilde{B},\models'',Q)$, $F=(Z,\tilde{C},\models''',R)$. Let $(f_1,f_2):D\longrightarrow E$ and $(g_1,g_2):E\longrightarrow F$ be continuous maps. The \textbf{composition} 
$(g_1,g_2)\circ (f_1,f_2):D\longrightarrow F$ is a pair $(g_1\circ f_1,f_2\circ g_2)$, where $g_1\circ f_1:(X,\tilde{A})\longrightarrow (Z,\tilde{C})$ and
$f_2\circ g_2:R\longrightarrow P$.
\end{definition}
\begin{lemma}\label{3.4_1'}
$(g_1,g_2)\circ (f_1,f_2):D\longrightarrow F$ is continuous, where $(f_1,f_2):D\longrightarrow E$, $(g_1,g_2):E\longrightarrow F$ are continuous. 
\end{lemma}
\begin{proof}
From Proposition \ref{pic} it is clear that $g_1\circ f_1:(X,\tilde{A})\longrightarrow (Z,\tilde{C})$ is a proper function. We show that $gr(x\models' f_2\circ g_2(r))=gr(g_1\circ f_1(x)\models''' r)$.
\begin{align*}
gr(x\models' f_2\circ g_2(r))& =gr(x\models' f_2(g_2(r)))\\
& =gr(f_1(x)\models'' g_2(r))\tag{as $(f_1,f_2)$ is continuous}\\%\label{eq:as $(f_1,f_2)$ is continuous}\\
& =gr(g_1(f_1(x))\models''' r) \tag {as $(g_1,g_2)$ is continuous}\\%\label{eq:as $(g_1,g_2)$ is continuous}\\
& =gr(g_1\circ f_1(x)\models''' r).
\end{align*}
This competes the proof.
\end{proof}
\begin{lemma}\label{3.5_1'}
Let $D,\ E,\ F,\ G$ be $\mathscr{L}$-topological systems and $f:D\longrightarrow E$, $g:E\longrightarrow F$, $h:F\longrightarrow G$ are continuous maps. Then $h\circ (g\circ f)=(h\circ g)\circ f$ holds.
\end{lemma}
\begin{proof}
$h\circ (g\circ f)=(h_1,h_2)\circ((g_1,g_2)\circ (f_1,f_2))=(h_1,h_2)\circ (g_1\circ f_1,f_2\circ g_2)=(h_1\circ(g_1\circ f_1),(f_2\circ g_2)\circ h_2)=((h_1\circ g_1)\circ f_1, f_2\circ (g_2\circ h_2))=(h_1\circ g_1,g_2\circ h_2)\circ (f_1,f_2)=((h_1,h_2)\circ (g_1,g_2))\circ (f_1,f_2)=(h\circ g)\circ f$.
\end{proof}
\begin{lemma}\label{3.6_1'}
For the continuous map $f:D\longrightarrow E$, it holds that $\ id_E\circ f=f$ and $f\circ id_D=f$.
\end{lemma}
\begin{proof}
Proof is straight forward.
\end{proof}
From Lemmas \ref{3.4_1'}, \ref{3.5_1'} and \ref{3.6_1'} we get the following theorem.
\begin{theorem}\label{3.7_1'}
$\mathscr{L}$-topological systems together with continuous maps form the category $\mathscr{L}$-$\mathbf{TopSys}$\index{category!$\mathscr{L}$-$\mathbf{TopSys}$}.
\end{theorem}

\begin{definition}\label{Frm}
Frames together with frame homomorphisms form the category $\mathbf{Frm}$ \cite{MP}.
\end{definition}
The opposite category of frame is known as the category of locale and denoted by $\mathbf{Frm}^{op}$ or 
$\mathbf{Loc}$.
\subsection{Functors}
The interrelation among the categories: $\mathscr{L}$-$\mathbf{TopSys}$, $\mathscr{L}$-$\mathbf{Top}$, $\mathbf{Loc}$ 
via some suitable functors shall be established. 
\subsection*{Functor $Ext_L$ from $\mathscr{L}$-$\mathbf{TopSys}$ to $\mathscr{L}$-$\mathbf{Top}$}\index{functor!$Ext_L$}
\begin{definition}\label{6}
Let $(X,\tilde{A},\models,P)$ be an $\mathscr{L}$- topological system. For each $p\in P$, its \textbf{extent$_L$}\index{extent$_L$} 
in $(X,\tilde{A},\models,P)$ is given by $ext_L(p)=(X,ext_L^*(p))$ where $ext_L^*(p)$ is a mapping from $X$ to $L$ defined 
by $ext_L^*(p)(x)=gr(x\models p)$ for all $x\in X$.
That is, $ext_L^*(p):X\longrightarrow L$ such that $ext_L^*(p)(x)=gr(x\models p)$ for all $x\in X$.
Also $ext_L(P)=\{(X,ext_L^*(p))\}_{p\in P}=(X,ext_L^*P)$ where $ext_L^*P=\{ext_L^*p\}_{p\in P}$.
\end{definition}
\begin{lemma}\label{3.9_1l}
$(X,\tilde{A},ext_L(P))$ is an $\mathscr{L}$-topological space.
\end{lemma}
\begin{proof}
Let $(X,ext_L^*(p_1)),\ (X,ext_L^*(p_2))\in ext_L(P)$.
Now,
\begin{align*}
(ext_L^*(p_1)\cap ext_L^*(p_2))(x) & =min \{ext_L^*(p_1)(x),ext_L^*(p_2)(x)\}\\
& =min \{gr(x\models p_1),gr(x\models p_2)\}\\
& =gr(x\models p_1\wedge p_2)\\
& =ext_L^*(p_1\wedge p_2)(x).
\end{align*}
As $p_1,\ p_2\in P,\ p_1\wedge p_2\in P$ and hence $(X,ext_L^*(p_1\wedge p_2))\in ext_L(P)$. Consequently $(X,ext_L^*(p_1)\cap ext_L^*(p_2))\in ext_L(P)$, for any $x\in X$.

Similarly, $\bigcup_i ext_L^*(p_i)=ext_L^*(\bigvee_i p_i)$ and hence $(X,\bigcup_i ext_L^*(p_i))\in ext_L(P)$, for $i\in I$ (an index set).
\end{proof}
\begin{lemma}\label{3.10_1l}
If $(f_1,f_2):(X,\tilde{A},\models ',P)\longrightarrow (Y,\tilde{B},\models '',Q)$ is continuous then $f_1:(X,\tilde{A},ext_L(P))\longrightarrow (Y,\tilde{B},ext_L(Q))$ is fuzz-top continuous.
\end{lemma}
\begin{proof}
$(f_1,f_2):(X,\tilde{A},\models ',P)\longrightarrow (Y,\tilde{P},\models '',Q)$ is continuous.
So we have,
\begin{equation}\label{k}
 gr(x\models 'f_2(q))=gr(f_1(x)\models ''q) \hspace{3mm}   \textrm{for all}\  x\in X, \ q\in Q.
 \end{equation}
 Hence, 
\begin{align*} 
(f_1^{-1}(ext_L(q)))(x) & = ext_L(q)(f_1(x))\\
& = gr(f_1(x)\models ''q)\\
& = gr(x\models 'f_2(q)),   \hspace{3mm} by  \ \eqref{k}\\
& = ext_L(f_2(q))(x).
\end{align*}
So, $(X,f_1^{-1}(ext_L(q)))=(X,ext_L(f_2(q)))\in ext_L(Q)$.
Therefore, $f_1$ is a fuzzy continuous map from $(X,\tilde{A},ext_L(P))$ to $(Y,\tilde{B},ext_L(Q))$.
\end{proof}
Now the functor $Ext_L$ is defined as follows:
\begin{definition}\label{7}
$\mathbf{Ext_L}$ is a functor from $\mathscr{L}$-$\mathbf{TopSys}$ to $\mathscr{L}$-$\mathbf{Top}$ defined thus.\\
$Ext_L$ acts on the object $(X,\tilde{A},\models ',P)$ as $Ext_L(X,\tilde{A},\models ',P)=(X,\tilde{A},ext_L(P))$ and on 
the morphism $(f_1,f_2)$ as $Ext_L(f_1,f_2)=f_1$.
\end{definition}  
Lemmas \ref{3.9_1l} and \ref{3.10_1l} gives the guarantee that $Ext_L$ is indeed a functor.
\subsection*{Functor $J_L$ from $\mathscr{L}$-$\mathbf{Top}$ to $\mathscr{L}$-$\mathbf{TopSys}$}
\begin{definition}\label{8}\index{functor!$J_L$}
$\mathbf{J_L}$ is a functor from  $\mathscr{L}$-$\mathbf{Top}$ to $\mathscr{L}$-$\mathbf{TopSys}$ defined thus. $J_L$ acts
on the object $(X,\tilde{A},\tau)$ as $J_L(X,\tilde{A},\tau)=(X,\tilde{A},\in ,\tau)$ where ``$\in$" is an $L$- fuzzy 
relation such that $gr(x\in \tilde{T})=\tilde{T}(x)$ for $\tilde{T} \in \tau$ and on the morphism $f$ as 
$J_L(f)=(f,f^{-1})$.
\end{definition}
\begin{lemma}\label{3.13_1'}
$(X,\tilde{A},\in , \tau)$ is an $\mathscr{L}$-topological system.
\end{lemma}
\begin{proof}
By straight forward method one can show that 
$gr(x\in \tilde{T_1}\cap \tilde{T_2})=inf\{ gr(x\in \tilde{T_1}),gr(x\in \tilde{T_2})\}$ and $gr(x\in \bigcup_i \tilde{T_i})= sup_i\{ gr(x\in \tilde{T_i})\}$.
\begin{align*}
(1)\ gr(x\in \tilde{T_1}\cap \tilde{T_2}) & =(\tilde{T_1}\cap \tilde{T_2})(x)\\
& =min\{ \tilde{T_1}(x), \tilde{T_2}(x)\}\\
& =min\{ gr(x\in \tilde{T_1}), gr(x\in \tilde{T_2})\}\\
& =inf\{ gr(x\in \tilde{T_1}), gr(x\in \tilde{T_2})\}.
\end{align*}
\begin{align*}
(2) \ gr(x\in \bigcup_i \tilde{T_i})  =(\bigcup_i \tilde{T_i})(x)
=sup_i\{\tilde{T_i}(x)\}
 =sup_i\{ gr(x\in \tilde{T_i})\}.
\end{align*}
\end{proof}
\begin{lemma}\label{3.14_1'}
$(f,f^{-1})$ is continuous provided $f$ is fuzz-top continuous.
\end{lemma}
Therefore $ J_L$ is a functor from $\mathscr{L}$-$\mathbf{Top}$ to $\mathscr{L}$-$\mathbf{TopSys}$. To make a connection between $\mathscr{L}$-$\mathbf{TopSys}$ and $\mathbf{Loc}$ (opposite category of $\mathbf{Frm}$)
we introduce two functors viz. $Lo_L$, $S_L$.
\subsection*{Functor $Lo_L$ from $\mathscr{L}$-$\mathbf{TopSys}$ to $\mathbf{Loc}$}
\begin{definition}\label{9}\index{functor!$Lo_L$}
$\mathbf{Lo_L}$ is a functor from $\mathscr{L}$-$\mathbf{TopSys}$ to $\mathbf{Loc}$ defined thus. $Lo_L$ acts on the 
object $(X,\tilde{A},\models ,P)$ as $Lo_L(X,\tilde{A},\models , P)=P$ and on the morphism $(f_1,f_2)$ as 
$Lo_L(f_1,f_2)=f_2$.
\end{definition}
It is clear that $Lo_L$ is indeed a functor.
\subsection*{Functor $S_L$ from $\mathbf{Loc}$ to $\mathscr{L}$-$\mathbf{TopSys}$}
\begin{definition}\label{10}
$\mathbf{S_L}$ is a functor from $\mathbf{Loc}$ to $\mathscr{L}$-$\mathbf{TopSys}$ defined thus. \\$S$ acts on the object 
$P$ as $S(P)=(Hom(P,L),\tilde{P},\models_*,P)$, where 
$Hom(P,L)=\{ frame\ hom\ v:P\longrightarrow L\}$, $\tilde{P}(v)=\bigvee_{p\in P}v(p)$ and $gr(v\models_* p)=v(p)$, 
and on the morphism $f$ as $S(f)=(\_\circ f,f)$.
\end{definition}
The following two Lemmas ensures that $S_L$ is indeed a functor.
\begin{lemma}\label{3.17_1l}
$(Hom(P,L),\tilde{P},\models_*,P)$ is an $\mathscr{L}$-topological system.
\end{lemma}
\begin{proof}It is enough to check the following:
\begin{enumerate}
\item $gr(v\models_* p)=v(p)\in L$.
\item $gr(v\models_*p)=v(p)\leq \bigvee_{p\in P}v(p)=\tilde{P}(v)$.
\item $gr(v\models_* p_1\wedge p_2) = v(p_1\wedge p_2) = min\{ v(p_1), v(p_2)\} = min\{ gr(v\models_* p_1),gr(v\models_* p_2)\} = inf\{ gr(v\models_* p_1),gr(v\models_* p_2)\}$.
\item $gr(v\models_* \bigvee_i p_i) = v(\bigvee_i p_i) = sup_i\{v(p_i)\} = sup_i\{gr(v\models_* p_i)\}$.
\end{enumerate}
This completes the proof.
\end{proof}
\begin{lemma}\label{3.18_1l}
If $f:Q\longrightarrow P$ is a frame homomorphism then $(\_\circ f,f):(Hom(P,L),\tilde{P},\models_*,P)\longrightarrow (Hom(Q,L),\tilde{Q},\models_*,Q)$ is continuous.
\end{lemma}
Finally it can be shown that the following theorems hold.
\begin{theorem}\label{ext_L}
$Ext_L$ is the right adjoint to the functor $J_L$.
\end{theorem}
\begin{proof}
We will prove the theorem by presenting the co-unit of the adjunction. Recall that 
$J_L(X,\tilde{A},\tau)=(X,\tilde{A},\in \tau)$ and $Ext_L(X,\tilde{A},\models,P)=(X,\tilde{A},ext_L(P))$.
Hence, 
$$J_L(Ext_L(X,\tilde{A},\models,P))=(X,\tilde{A},\in,ext_L(P)).$$
Let us draw the diagram of co-unit.
\begin{center}
\begin{tabular}{ l | r } 
$\mathscr{L}$-$\mathbf{TopSys}$ & $\mathscr{L}$- $\mathbf{Top}$\\
\hline
 {\begin{tikzpicture}[description/.style={fill=white,inner sep=2pt}] 
    \matrix (m) [matrix of math nodes, row sep=2.5em, column sep=2.3em]
    { J_L(Ext_L(X,\tilde{A},\models, P))&&(X,\tilde{A},\models, P)  \\
         J_L(Y,\tilde{B}, \tau ') \\ }; 
    \path[->,font=\scriptsize] 
        (m-1-1) edge node[auto] {$\xi_X$} (m-1-3)
        (m-2-1) edge node[auto] {$J_L(f)(\equiv (f_1,f_1^{-1}))$} (m-1-1)
        (m-2-1) edge node[auto,swap] {$\hat f (\equiv (f_1,f_2))$} (m-1-3)
        %(m-1-1) edge node[auto,swap] {$f$} (m-2-3)
       % (m-1-5) edge node[auto,swap] {$j$} (m-1-3)
       % (m-1-5) edge node[auto] {$\psi$} (m-2-3)
        %(m-1-3) edge node[auto] {$G(\hat{f})$} (m-2-3)
         ;
\end{tikzpicture}} & {\begin{tikzpicture}[description/.style={fill=white,inner sep=2pt}] 
    \matrix (m) [matrix of math nodes, row sep=2.5em, column sep=2.5em]
    { Ext_L(X,\tilde{A}, \models, P)  \\
         (Y,\tilde{B}, \tau ') \\ }; 
    \path[->,font=\scriptsize] 
        (m-2-1) edge node[auto,swap] {$f(\equiv f_1)$} (m-1-1)
       
         ;
\end{tikzpicture}} \\ 
\end{tabular}
\end{center}
Let us define co-unit by $\xi_X=(id_X,ext_L')$. That is,
\begin{center}
\begin{tikzpicture}[description/.style={fill=white,inner sep=2pt}] 
    \matrix (m) [matrix of math nodes, row sep=2.5em, column sep=2.5em]
    { (X,\tilde{A},\in ,ext_L(P))&&(X,\tilde{A},\models, P)  \\
          }; 
    \path[->,font=\scriptsize] 
        (m-1-1) edge node[auto] {$\xi_X$} (m-1-3)
        (m-1-1) edge node[auto,swap] {$(id_X,ext_L')$} (m-1-3)
         ;
\end{tikzpicture}
\end{center}
 where $ext_L'$ is a mapping from $P$ to $ext_L(P)$ such that, $ext_L'(p)=(X,ext_L^*(p))$ for all $p\in P$.
It can be shown that $$(id_X,ext_L'):J_L(Ext_L(X,\tilde{A},\models , P))\longrightarrow (X,\tilde{A},\models , P)$$ is indeed 
a continuous map of $\mathscr{L}$-topological system, as follows.
According to the definition $ext_L^*(p)(x)=gr(x\models p)$. Hence $ext_L'(p)(x)=gr(x\models p)$. Consequently $gr(x\in ext'(p))=gr(id_X(x)\models p)$.

Now define $f$ as follows.
Given $(f_1,f_2):J_L(Y,\tilde{B},\tau')\longrightarrow (X,\tilde{A},\models, P)$, then $f=f_1$.
Now we will prove that the diagram on the left commutes.
Here $J_L(f)=(f_1,f_1^{-1})$ and $$(f_1,f_2)=\xi_X\circ J_L(f)=(id_X,ext_L')\circ (f_1,f_1^{-1})=(id_X\circ f_1,f_1^{-1}\circ ext_L').$$
%We have to show
%\begin{align*}
% (f_1,f_2) & =\xi_X\circ J(f)\\
%& =(id_X,ext')\circ (f_1,f_1^{-1})\\
%& =(id_X\circ f_1,f_1^{-1}\circ ext')
%\end{align*}
Clearly $id_X\circ f_1=f_1$.
Also we have $f_1^{-1}ext_L'(p)= f_1^{-1}(p)= f_2(p)$.
So, $f_2=f_1^{-1}\circ ext_L'$.
%\begin{align*}
%f_1^{-1}ext'(p) & = f_1^{-1}(p)\\
%& = f_2(p)
%\end{align*}
Hence $$\xi_X(\equiv(id_X,ext_L')):J_L(Ext_L(X,\tilde{A},\models, P))\longrightarrow (X,\tilde{A},\models, P)$$ is the co-unit, 
consequently $Ext_L$ is the right adjoint to the functor $J_L$.
\end{proof}
\begin{theorem}\label{loc_L}
$Lo_L$ is the left adjoint to the functor $S_L$.
\end{theorem}
\begin{proof}
We will prove the theorem by presenting the unit of the adjunction.
Recall that  $S_L(Q)=(Hom(Q,L),\tilde{Q},\models_*,Q)$, where $gr(v\models_* q)=v(q)$.
Hence,
$$S_L(Lo_L(X,\tilde{A},\models, P))=(Hom(P,L),\tilde{P},\models_*, P).$$
\begin{center}
\begin{tabular}{ l | r  } 
 $\mathscr{L}$-$\mathbf{TopSys}$ &$\mathbf{Loc}$ \\
\hline
 {\begin{tikzpicture}[description/.style={fill=white,inner sep=2pt}] 
    \matrix (m) [matrix of math nodes, row sep=2.5em, column sep=2.5em]
    {(X,\tilde{A},\models , P) & &S_L(Lo_L(X,\tilde{A}, \models , P))  \\
        & & S_L(Q) \\ }; 
    \path[->,font=\scriptsize] 
        (m-1-1) edge node[auto] {$\eta$} (m-1-3)
        (m-1-1) edge node[auto,swap] {$f(\equiv (f_1,f_2))$} (m-2-3)
       % (m-1-5) edge node[auto,swap] {$j$} (m-1-3)
        %(m-1-5) edge node[auto] {$\psi$} (m-2-3)
        (m-1-3) edge node[auto] {$S_L(\hat{f})$} (m-2-3);
\end{tikzpicture}} &  {\begin{tikzpicture}[description/.style={fill=white,inner sep=2pt}] 
    \matrix (m) [matrix of math nodes, row sep=2.5em, column sep=2.5em]
    { Lo_L(X,\tilde{A},\models , P)  \\
         Q \\ }; 
    \path[->,font=\scriptsize]

        %(m-1-5) edge node[auto,swap] {$j$} (m-1-3)
        %(m-1-5) edge node[auto] {$\psi$} (m-2-3)
        (m-1-1) edge node[auto] {$\hat{f}(\equiv f_2)$} (m-2-1);
\end{tikzpicture}}
\end{tabular}
\end{center}
Then unit is defined by $\eta =(p^*,id_P)$. That is,
\begin{center}
 \begin{tikzpicture}[description/.style={fill=white,inner sep=2pt}] 
    \matrix (m) [matrix of math nodes, row sep=2.5em, column sep=2.5em]
    {(X,\tilde{A},\models , P) & &S(Lo(X,\tilde{A}, \models , P)),  \\
        }; 
    \path[->,font=\scriptsize] 
        (m-1-1) edge node[auto] {$\eta$} (m-1-3)
        (m-1-1) edge node[auto,swap] {$(p^*,id_P)$} (m-1-3)
        %(m-1-1) edge node[auto,swap] {} (m-2-3)
       % (m-1-5) edge node[auto,swap]
        ;
\end{tikzpicture}
\end{center}
where $p^*\colon (X,\tilde{A})\longrightarrow (Hom(P,L),\tilde{P})$, s.t. for any $x\in \mid\tilde{A}\mid$, $p^*(x)$ is a mapping
from $P$ to $L$ and $p^*(x)(p)=gr(x\models p)$.
We can show that, $$(p^*,id_P):(X,\tilde{A},\models,P)\longrightarrow S(Lo_L(X,\tilde{A},\models,P)) $$ is a continuous map
of $\mathscr{L}$-topological system in the following way.

It will be enough to show that $gr(x\models id_P(p))=gr(p^*(x)\models p)$.\\
We have $gr(x\models p)=p^*(x)(p)=gr(p^*(x)\models p)$.

Let us define $\hat{f}$ as follows
$$(f_1,f_2):(X,\tilde{A},\models , P)\longrightarrow (Hom(P,L),\tilde{P},\models_*,P)$$
then $\hat{f}=f_2$, as $f_2$ is the frame homomorphism.

Recall that $S_L(\hat{f})=(-\circ f_2,f_2)$.
Now we have to show that the triangle on the left commute.
%Recall that $S(\hat{f})=(-\circ f_2,f_2)$\\
We have to show that, $$(f_1,f_2)=(-\circ f_2,f_2)\circ (p^*,id_P)=((-\circ f_2)p^*,id_P\circ f_2).$$
Clearly $f_2=id_P\circ f_2.$
It is only left to show that $f_1=(-\circ f_2)p^* =p_x\circ f_2.$
We have for all $q\in Q$,
$$
p^*(x)\circ f_2(q)  =p^*(x)(f_2(q))
 =gr(x\models f_2(q))
 =gr(f_1(x)\models_* q)
 =f_1(x)(q).
 $$
So, $(\_\circ f_2)p^*=f_1$.
\end{proof}
\begin{theorem}\label{fin}
$Ext_L\circ S_L$ is the right adjoint to the functor $Lo_L\circ J_L$.
\end{theorem}
\begin{proof}
Proof follows from Theorem \ref{ext_L} and Theorem \ref{loc_L}.
\end{proof}
The obtained functorial relationships can be illustrated by the following diagram:
  
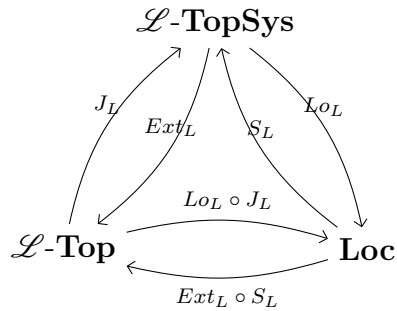
\begin{center}
\begin{tikzpicture}
\node (C) at (0,3) {$\mathscr{L}$-$\mathbf{TopSys}$};
\node (A) at (-2,0) {$\mathscr{L}$-$\mathbf{Top}$};
\node (B) at (2,0) {$\mathbf{Loc}$};
%\node at (0,0) {\rotatebox{270}{$\Rightarrow$}};
\path[->,font=\scriptsize ,>=angle 90]
(A) edge [bend left=15] node[above] {$Lo_L\circ J_L$} (B);
\path[<-,font=\scriptsize ,>=angle 90]
(A)edge [bend right=15] node[below] {$Ext_L \circ S_L$} (B);
\path[->,font=\scriptsize ,>=angle 90]
(A) edge [bend left=20] node[above] {$J_L$} (C);
\path[<-,font=\scriptsize ,>=angle 90]
(A)edge [bend right=20] node[above] {$Ext_L$} (C);
\path[->,font=\scriptsize, >=angle 90]
(C) edge [bend left=20] node[above] {$Lo_L$} (B);
\path[<-,font=\scriptsize, >=angle 90]
(C)edge [bend right=20] node[above] {$S_L$} (B);
\end{tikzpicture}
\end{center}
\section{Subcategories of $\alpha$-cuts}
In this section we will construct two kinds of subsystems of some $\mathscr{L}$-topological systems. 
To do so here we will deal with some special kind of subsets of the fuzzy set. We will construct the 
subsets using the concept of $\alpha$-cut and fuzzy $\alpha$-cut of fuzzy set respectively. For the notion of
classical $\alpha$-cut of a fuzzy set we refer to \cite{KY}. Here we introduce a notion of fuzzy $\alpha$-cut of a fuzzy set.
\begin{definition}[$\alpha$-cut of a fuzzy set]\label{1'}
Let $(X,\tilde{A})$ be an $L$-fuzzy set. Then for $\alpha \in L$, where $L$ is a frame, the \textbf{$\alpha$-cut} of 
$(X,\tilde{A})$ is the ordinary set $\{ x\in X\mid \tilde{A}(x)\geq \alpha \}$.
\end{definition} 
\begin{definition}[Strict $\alpha$-cut of a fuzzy set]\label{1''}
Let $(X,\tilde{A})$ be an $L$-fuzzy set. Then for $\alpha \in L$, where $L$ is a frame, the \textbf{strict $\alpha$-cut} of 
$(X,\tilde{A})$ is the ordinary set $\{ x\in X\mid \tilde{A}(x)> \alpha \}$.
\end{definition} 
\begin{definition}[Fuzzy $\alpha$-cut of a fuzzy set]\label{2'}
Let $(X,\tilde{A})$ be an $L$-fuzzy set. Then for $\alpha \in L$, where $L$ is a frame, the \textbf{fuzzy $\alpha$-cut} of 
$(X,\tilde{A})$ is the fuzzy subset $(X,\tilde{A_{\alpha}})$ such that $\tilde{A_{\alpha}}$ is defined as follows:
\begin{center}
$\tilde{A_{\alpha}}(x)$ $=\begin{cases}
        \tilde{A}(x) & \emph{if $\tilde{A}(x)\geq \alpha $} \\
        0_L & \emph{otherwise}.\end{cases}$
\end{center}        
\end{definition} 
\subsection{$\mathbf{TopSys}_\alpha$}
Let $(X,\tilde{A},\models ,P)$ be an $\mathscr{L}$-topological system. Let us consider the triple 
$(\{x\in X\mid \tilde{A}(x)>\alpha\},\models_\alpha,P)$, where $\models_\alpha\subseteq X\times P$ such that 
$x\models_\alpha p$ iff $gr(x\models p)>\alpha$. It can be shown that the triple forms a topological system and 
consequently a subsystem of $\mathscr{L}$-topological system. 

Thus we get subsystems for each $\alpha<1_L$ ($1_L$ is the top element of {L}). Now for $\alpha>\alpha '$, 
$\{x\in X\mid \tilde{A}(x)>\alpha\}$ is a subset of $\{x\in X\mid \tilde{A}(x)>\alpha '\}$ and hence 
$(\{x\in X\mid \tilde{A}(x)>\alpha\},\models_\alpha,P)$ is a subsystem of 
$(\{x\in X\mid \tilde{A}(x)>\alpha '\},\models_\alpha ',P)$. Hence we get chains of subsystems of 
$\mathscr{L}$-topological system.

The restriction of a continuous function between two $\mathscr{L}$-topological systems is a continuous function between 
corresponding subsystems for each $\alpha\in L$. 

The above subsystems for a fixed $\alpha< 1_L$($\in L$) together with continuous maps form a category called 
$\mathbf{TopSys}_\alpha$, which is a subcategory of $\mathscr{L}$-$\mathbf{TopSys}$. Here we will get chains of 
subcategories of $\mathscr{L}$-$\mathbf{TopSys}$ by changing the values of $\alpha$.

It may be noted that for a linear $L$ we will get only one chain.
\subsection{$\mathbf{Top}_\alpha$ vs. $\mathbf{TopSys}_\alpha$}
Let $(X,\tilde{A},\tau)$ be an $\mathscr{L}$-topological space and take strict $\alpha$-cut of $(X,\tilde{A})$ i.e., 
$\{x\in X\mid \tilde{A}(x)>\alpha\}$. Let $\tau_\alpha$ be defined by 
$\tau_\alpha=\{\{x\in X\mid T_i(x)>\alpha\}\mid \tilde{T_i}\in\tau\}$. Then 
$(\{x\in X\mid \tilde{A}(x)>\alpha\},\tau_\alpha)$ also form a topological space and called the topological subspace.

For any fixed $\alpha\in L$, topological subspaces together with continuous maps forms a category, which is a 
subcategory of the category $\mathbf{Top}$, called $\mathbf{Top}_\alpha$.

By routine check it can be shown that the restriction of the functors ($Ext$ and $J$) between $\mathbf{TopSys}$ and 
$\mathbf{Top}$ are adjoint functors between $\mathbf{TopSys}_\alpha$ and $\mathbf{Top}_\alpha$, for each $\alpha\in L$.
\subsection{$\mathscr{L}$-$\mathbf{TopSys}_\alpha$}
Let $(X,\tilde{A},\models ,P)$ be an $\mathscr{L}$-topological system. Let us consider the quadruple 
$(X,\tilde{A_\alpha},\models_\alpha,P)$, where $(X,\tilde{A_\alpha})$ is the fuzzy $\alpha$-cut of 
$(X,\tilde{A})$, $\models_\alpha$ is an $L$-fuzzy relation between $X$ and $L$ such that 
$gr(x\models_\alpha p)=gr(x\models p)$ for $\tilde{A}(x)\geq \alpha$ and $gr(x\models_\alpha p)=0_L$ for 
$\tilde{A}(x)<\alpha$. It can be shown that the quadruple forms an $\mathscr{L}$-topological system and consequently 
an $\mathscr{L}$-topological subsystem of $\mathscr{L}$-topological system. Hence for each $\alpha\in L$ we will get 
$\mathscr{L}$-topological subsystems. Furthermore as for $\alpha>\alpha'$, $(X,\tilde{A_\alpha})$ is a fuzzy subset 
of $(X,\tilde{A_{\alpha'}})$, we will get chains of such $\mathscr{L}$-topological subsystems.

A continuous map between two $\mathscr{L}$-topological subsystems, say $(X,\tilde{A_\alpha},\models_\alpha, P)$ and 
$(Y,\tilde{B_\alpha},\models_\alpha',Q)$, is the restriction of a continuous map between the 
$\mathscr{L}$-topological systems $(X,\tilde{A},\models, P)$ and $(Y,\tilde{B},\models',Q)$.

It can be shown that for fixed $\alpha \in L$, $\mathscr{L}$-topological subsystems together with the above mentioned 
continuous maps form a subcategory of $\mathscr{L}$-$\mathbf{TopSys}$. Thus we will get chains of such subcategories 
of $\mathscr{L}$-$\mathbf{TopSys}$ by changing the values of $\alpha$ in $L$. For a fixed $\alpha\in L$, let us call 
the corresponding subcategory by $\mathscr{L}$-$\mathbf{TopSys}_\alpha$.
\subsection{$\mathscr{L}$-$\mathbf{Top}_\alpha$ vs. $\mathscr{L}$-$\mathbf{TopSys}_\alpha$}
Let $(X,\tilde{A},\tau)$ be an $\mathscr{L}$-topological space and take fuzzy $\alpha$-cut of $(X,\tilde{A})$ i.e., 
$(X,\tilde{A_\alpha})$. Let $\tau'$ be defined by 
$\tau'=\{(X,\tilde{T'})\mid \tilde{T'}=\tilde{A_\alpha}\cap \tilde{T}, \tilde{T}\in\tau\}$. 
Then $(X,\tilde{A_\alpha},\tau')$ also form an $\mathscr{L}$-topological space and called the 
$\mathscr{L}$-topological subspace.

For any fixed $\alpha\in L$, $\mathscr{L}$-topological subspaces together with continuous maps form a category, which 
is a subcategory of the category $\mathscr{L}$-$\mathbf{Top}$, called $\mathscr{L}$-$\mathbf{Top}_\alpha$.

By routine check it can be shown that the restriction of the functors ($Ext_L$ and $J_L$) between 
$\mathscr{L}$-$\mathbf{TopSys}$ and $\mathscr{L}$-$\mathbf{Top}$ are adjoint functors between 
$\mathscr{L}$-$\mathbf{TopSys}_\alpha$ and $\mathscr{L}$-$\mathbf{Top}_\alpha$, for each $\alpha\in L$.

\chapter{Category of Variable Basis Fuzzy Topological Systems over Fuzzy Sets}
\section{Introduction}\blfootnote{The results of this chapter are based on {\bf \cite{MP2} P. Jana and M.K. Chakraborty, \emph{Categorical relationships of fuzzy topological systems with fuzzy topological spaces and underlying algebras-II},  Ann. of Fuzzy Math. and Inform., \textbf{10}, 2015, no. 1, pp. 123--137.}}
In this chapter, we proceed with the concept of variable basis fuzzy topological spaces on 
fuzzy sets and propose the notion of variable basis topological systems whose underlying sets are fuzzy sets. Solovyov and Rodabough worked on variable basis fuzzy topological spaces \cite{SR} and systems \cite{SO1} over crisp sets. Hence their case becomes a particular case of ours. Here we are able to establish adjunction between space and system, i.e., spatialization is achieved but adjunction between space and algebra or system and algebra is still unsettled. In the particular case, i.e., in \cite{SO1} Solovyov also left this case as an open question.

\section{Categories: FuzzTop, FuzzTopSys}
\subsection{Categories}
\subsection*{FuzzTop}
\begin{definition}[Fuzz topological space]\cite{MB}\label{futop}
Let $(X, L, \tilde{A})$ be a Fuzz-object (i.e. $X$ is a non-empty set, $L$ is a frame, $\tilde{A}$ is a map from $X$ to $L$) and $\tau$ be a collection of maps from $X$ into $L$ such that:
\begin{enumerate}
\item if $U\in\tau$, $U(x)\leq \tilde{A}(x)$, for all $x\in X$;
\item $\tilde{\emptyset}$ and $\tilde{A}$ are in $\tau$, where $\tilde{\emptyset}:X\longrightarrow L$ is such that 
$\tilde{\emptyset}(x)=0_{L}$, for all $x\in X$, where $0_{L}$ is the least element of the frame $L$;
\item $\tau$ is closed under finite infima and arbitrary suprema. Then $(X, L, \tilde{A},\tau)$ is a \textbf{Fuzz topological space} \index{Fuzz topological!space}.
\end{enumerate}
\end{definition}
\begin{definition}[FuzzTop]\cite{MB}\label{05} 
The category $\mathbf{FuzzTop}$ \index{category!FuzzTop} is defined thus.
\begin{itemize}
\item The objects are Fuzz topological spaces.
\item The morphisms are pairs $(f,\phi):(X,L,\tilde{A},\tau)\longrightarrow (X_1,L_1,\tilde{A_1},\tau_1)$ satisfying 
the following properties:\\
a) $\phi$ is a relation from $L$ to $L_1$ such that $\phi^{-1}:L_1\longrightarrow L$ is a frame homomorphism.\\
b) $f:X\times X_1\longrightarrow L$ is a map such that $f(x,y)\leq \tilde{A}(x)\wedge \phi^{-1}\tilde{A_1}(y)$, for 
all $x\in X$, $y\in X_1$ and there exist unique $b$ in $\mid\tilde{A_1}\mid$ with $f(a,b')=\tilde{A}(a)$ for $b'=b$ in 
$\mid\tilde{A_1}\mid$, otherwise $f(a,b')=0_L$.\\
c) If $V\in \tau_1$, then $U\in \tau$ where $U(x)=\bigvee_{y\in X_1}[f(x,y)\wedge \phi^{-1}V(y)]$, for all $x\in X$.

The pair $(f,\phi)$ is known as $\mathbf{FuzzTop}$ morphism \index{$\mathbf{FuzzTop}$! morphism}.
\item If $(f,\phi):(X,L,\tilde{A},\tau)\longrightarrow (X_1,L_1,\tilde{A_1},\tau_1)$ and 
$(g,\psi):(X_1,L_1,\tilde{A_1},\tau_1)\longrightarrow (X_2,L_2,\tilde{A_2},\tau_2)$ are morphisms in
$\mathbf{FuzzTop}$, their composite $(g,\psi)\circ (f,\phi)$ is that of the $\mathbf{FuzzTop}$ morphisms 
$(f,\phi):(X,L,\tilde{A})\longrightarrow (X_1,L_1,\tilde{A_1})$ and 
$(g,\psi):(X_1,L_1,\tilde{A_1})\longrightarrow (X_2,L_2,\tilde{A_2})$, viz. $(g\circ f,\psi\phi)$ with 
$g\circ f:X\times X_2\longrightarrow L$ as $$g\circ f(x,z)=\bigvee_{y\in X_1}[f(x,y)\wedge \phi^{-1}g(y,z)],\ \text{for all}\ x\in X,\ z\in X_2$$ and $\psi\phi$ is the relational composite of $\phi$ and $\psi$ (c.f. Proposition \ref{comf}).
\item The identity on $(X,L,\tilde{A},\tau)$ is the pair $(i_A,i_L)$ such that $i_A$ is the proper identity function on $(X,L,\tilde{A})$ and $i_L$ is the identity map on $L$. 
This is a $\mathbf{FuzzTop}$ morphism  can be proved (c.f. Proposition \ref{idf}).
\end{itemize}
\end{definition}
\begin{proposition}\label{comf}
Let $(f,\phi):(X,L,\tilde{A},\tau)\longrightarrow (X_1,L_1,\tilde{A_1},\tau_1)$ and 
$(g,\psi):(X_1,L_1,\tilde{A_1},\tau_1)\longrightarrow (X_2,L_2,\tilde{A_2},\tau_2)$ be morphisms in $\mathbf{FuzzTop}$. Then their composite $(g,\psi)\circ (f,\phi)$ is also a $\mathbf{FuzzTop}$ morphism. 
\end{proposition}
\begin{proof}
(a) $\psi\phi:L\longrightarrow L_2$ gives $$(\psi\phi)^{-1}=\phi^{-1}\circ\psi^{-1}:L_2\longrightarrow L.$$ Now by given hypothesis $\phi:L\longrightarrow L_1$ such that $\phi^{-1}:L_1\longrightarrow L$ is a frame homomorphism and $\psi:L_1\longrightarrow L_2$ such that $\psi^{-1}:L_2\longrightarrow L_1$ is a frame homomorphism. Hence $\psi\phi:L\longrightarrow L_2$ such that $\phi^{-1}\circ\psi^{-1}:L_2\longrightarrow L$ is a frame homomorphism, i.e., $\psi\phi:L\longrightarrow L_2$ such that $(\psi\phi)^{-1}:L_2\longrightarrow L$ is a frame homomorphism.\\
(b) Given that $f:X\times X_1\longrightarrow L$ such that $$f(x,x_1)\leq \tilde{A}(x)\wedge\phi^{-1}\tilde{A_1}(x_1)$$ and $g:X_1\times X_2\longrightarrow L_1$ such that $$g(x_1,x_2)\leq \tilde{A_1}(x_1)\wedge\psi^{-1}\tilde{A_2}(x_2).$$ Now $g\circ f:X\times X_2\longrightarrow L$ such that $$g\circ f(x,x_2)=\bigvee_{x_1\in X_1}[f(x,x_1)\wedge \phi^{-1}g(x_1,x_2)].$$ Now for any $x_1\in X_1$, 
\begin{multline*}
$$f(x,x_1)\wedge\phi^{-1}g(x_1,x_2)\leq \tilde{A}(x)\wedge\phi^{-1}\tilde{A_1}(x_1)\wedge \phi^{-1}(\tilde{A_1}(x_1)\wedge\psi^{-1}\tilde{A_2}(x_2))=\\ \tilde{A}(x)\wedge\phi^{-1}(\tilde{A_1}(x_1))\wedge\phi^{-1}(\psi^{-1}(\tilde{A_2}(x_2)))\leq \tilde{A}(x)\wedge (\psi\phi)^{-1}(\tilde{A_2}(x_2)).$$ 
\end{multline*}
Therefore $$g\circ f(x,x_2)=\bigvee_{x_1\in X_1}[f(x,x_1)\wedge \phi^{-1}g(x_1,x_2)]\leq \tilde{A}(x)\wedge (\psi\phi)^{-1}(\tilde{A_2}(x_2)).$$
Rest part can be done by routine check.\\
(c) Let us take $\mathbf{FuzzTop}$ morphisms $(f,\phi):(X,L,\tilde{A},\tau)\longrightarrow (X_1,L_1,\tilde{A_1},\tau_1)$ and $(g,\psi):(X_1,L_1,\tilde{A_1},\tau_1)\longrightarrow (X_2,L_2,\tilde{A_2},\tau_2)$. So, if $V\in\tau_1$, then $U\in\tau$ where $$U(x)=\bigvee_{x_1\in X_1}[f(x,x_1)\wedge\phi^{-1}V(x_1)]$$ and if $W\in\tau_2$, then $V\in\tau_1$ where $$V(x_1)=\bigvee_{x_2\in X_2}[g(x_1,x_2)\wedge\psi^{-1}W(x_2)].$$
It is left to show that $(g\circ f,\psi\phi):(X,L,\tilde{A},\tau)\longrightarrow (X_2,L_2,\tilde{A_2},\tau_2)$ is $\mathbf{FuzzTop}$ morphism, i.e., if $W\in\tau_2$, then $U'\in \tau$ where $$U'(x)=\bigvee_{x_2\in X_2}[g\circ f(x,x_2)\wedge (\psi\phi)^{-1}W(x_2)].$$ Now,
\begin{align*}
U'(x) & = \bigvee_{x_2\in X_2}[g\circ f(x,x_2)\wedge (\psi\phi)^{-1}W(x_2)]\\
& = \bigvee_{x_2\in X_2}[\bigvee_{x_1\in X_1}[f(x,x_1)\wedge \phi^{-1}g(x_1,x_2)]\wedge \phi^{-1}(\psi^{-1}W(x_2))]\\
& = \bigvee_{x_2\in X_2}\bigvee_{x_1\in X_1}[f(x,x_1)\wedge \phi^{-1}g(x_1,x_2)\wedge \phi^{-1}(\psi^{-1}W(x_2))]\\
& = \bigvee_{x_1\in X_1}[f(x,x_1)\wedge \bigvee_{x_2\in X_2}[\phi^{-1}g(x_1,x_2)\wedge \phi^{-1}(\psi^{-1}W(x_2))]]\\
& = \bigvee_{x_1\in X_1}[f(x,x_1)\wedge \phi^{-1}(\bigvee_{x_2\in X_2}[g(x_1,x_2)\wedge (\psi^{-1}W(x_2))]]\\
& = \bigvee_{x_1\in X_1}[f(x,x_1)\wedge \phi^{-1}V(x_1)]\\
& = U(x).
\end{align*}
Hence $U'=U\in\tau$.
\end{proof}
\begin{proposition}\label{idf}
For any Fuzz topological space $(X,L,\tilde{A},\tau)$, the pair $(i_A,i_L)$ defined as above is a $\mathbf{FuzzTop}$ morphism. 
\end{proposition}
\begin{proof}
(a) $i_L$ is the identity map on $L$ hence $i_L:L\longrightarrow L$ is a relation from $L$ to $L$ such that $i_L^{-1}=i_L:L\longrightarrow L$ is a frame homomorphism.
\\
(b) As $i_A$ is an identity proper function on $(X,L,\tilde{A})$, $$i_A(x_1,x_2)=0_L\leq \tilde{A}(x)\wedge i_L(\tilde{A}(x_2)),\ \text{for}\ x_1\neq x_2$$ and $$i_A(x_1,x_1)=\tilde{A}(x_1)=\tilde{A}(x_1)\wedge i_L(\tilde{A}(x_1)).$$
(c) Let $V\in\tau$, then $$U(x)=\bigvee_{x'\in X}[i_A(a,x')\wedge i_L^{-1}(V(x'))]=i_A(a,a)\wedge V(a)=\tilde{A}(a)\wedge V(a).$$ Now $V\in \tau$, so $V(x)\leq\tilde{A}(x)$ for any $x$. Hence $V(a)\leq\tilde{A}(a)$ gives $$\tilde{A}(a)\wedge V(a)=V(a).$$ Therefore $U(x)=V(a)\in \tau$. 
\end{proof}
\subsection*{FuzzTopSys}
\begin{definition}[Fuzz topological system]\label{02}
A \textbf{Fuzz topological system} \index{Fuzz topological system} is a quintuple $(X, L, \tilde{A} ,\models ,P)$, where $(X, L, \tilde{A})$ is a 
Fuzz-object \cite{MB} (i.e., $X$ is a non-empty set, $L$ is a frame, $\tilde{A}$ is a map from $X$ to $L$), $P$ is a 
frame and $\models $ is an $L$- fuzzy relation between $X$ and $P$ such that
 \begin{enumerate}%[$\bullet$]
\item $gr(x\models p)\in L$;
\item $gr(x\models p)\leq \tilde{A}(x)$;
\item if $S$ is a \textbf{finite subset} of $P$, then
$gr(x\models \bigwedge S) = inf\{ gr(x\models s):s\in S\}$;
\item if $S$ is  \textbf{any subset} of $P$, then
$gr(x\models \bigvee S)=sup\{ gr(x\models s):s\in S\}$.
\end{enumerate} 
\end{definition}
\underline{Note 1:} Because of condition 2, $\models$ is a fuzzy relation on the $L$-fuzzy set $(X,\tilde{A})$ \cite{MM}.

\underline{Note 2:} The value set $L$ of $\mathscr{L}$-topological system (c.f Chapter 3) is fixed but in the case of Fuzz topological system the 
value set $L$ may vary. Thus we can consider an $\mathscr{L}$-topological system as a specific instant of Fuzz 
topological system.

\underline{Note 3:} In \cite{MB}, to define Fuzz-object $(X,L,\tilde{A})$, $L$ is taken as a complete Heyting algebra whereas in 
our work we consider $L$, a frame.

The notion of continuous map between these Fuzz topological systems is defined as follows:
\begin{definition}\label{03}
 Let $D=(X, L, \tilde{A}, \models ,P)$ and $E=(Y, M, \tilde{B}, \models ',Q)$ be Fuzz topological systems. A 
 \textbf{continuous map} \index{continuous map!between!Fuzz topological systems} $f:D\longrightarrow E$ is a triple $(f, \phi , g)$ where,
\begin{enumerate}
\item $(f, \phi):(X, L, \tilde{A})\longrightarrow (Y, M, \tilde{B})$ such that \\
a) $\phi$ is a relation from $L$ to $M$ such that $\phi^{-1}:M\longrightarrow L$ is a map preserving finite meet and 
arbitrary join,\\
b) $f:X\times Y\longrightarrow L$ is a map such that $f(x,y)\leq \tilde{A}(x)\wedge \phi^{-1}\tilde{B}(y)$, for all 
$x\in X$, $y\in Y$ and for any $a$ in $\mid \tilde{A}\mid $, there exist unique $b$ in $\mid\tilde{B}\mid$ with 
$f(a,b)=\tilde{A}(a)$ and $f(a,b')=0_L$ with $b'\neq b$ in $\mid \tilde{B}\mid$.
\item $g:Q\longrightarrow P$ is a frame homomorphism and
\item $gr(x\models g(q))=\bigvee_{y\in Y} [\phi^{-1}(gr(y\models' q))\wedge f(x,y)]$, for all $x\in X$ and $q\in Q$.
\end{enumerate}
\end{definition}
Let us define identity map and composition of two maps.
\begin{definition}\label{04}
Let $D=(X, L, \tilde{A},\models,P)$ be a Fuzz topological system. The \textbf{identity map} \index{identity map!of!Fuzz topological system} $I_D:D\longrightarrow D$ 
is a triple $(I_A,I_L,I_P)$ defined by
%\begin{align*}
$I_A:X\times X\longrightarrow L$ s.t. $I_A(x_1,x_2) =\tilde{A}(x)$ iff $x_1=x_2$, otherwise 
$I_A(x_1,x_2) =0_L$, $I_L:L\longrightarrow L$ is identity morphism of $L$ and $I_P:P\longrightarrow P$ is identity 
morphism of $P$.
%\end{align*}

Let $D=(X, L, \tilde{A},\models,P)$, $E=(Y, M, \tilde{B},\models',Q)$, $F=(Z, N, \tilde{C},\models'',R)$. Let 
$(f, \phi ,g):D\longrightarrow E$ and $(f_1, \phi_1, g_1):E\longrightarrow F$ be continuous maps. The 
\textbf{composition} $(f_1, \phi_1, g_1)\circ (f,\phi,g):D\longrightarrow F$ is defined by
$f_1\circ f:X\times Z\longrightarrow L$ such that 
$f_1\circ f(x,z)=\bigvee_{y\in Y}[f(x,y)\wedge \phi^{-1}(f_1(y,z))]$ for all $x\in X$ and $z\in Z$, $\phi_1 \phi$ the relation composite of $\phi_1$, $\phi$, and $g\circ g_1:R\longrightarrow P$ is a frame homomorphism. Hence $(f_1 ,\phi_1 ,g_1)\circ (f,\phi ,g)=(f_1\circ f,\phi_1\phi,g\circ g_1)$.
\end{definition}
\begin{lemma}\label{3.4_1F}
$(g_1,g_2)\circ (f_1,f_2):D\longrightarrow F$ is continuous, where $(f_1,f_2):D\longrightarrow E$, $(g_1,g_2):E\longrightarrow F$ are continuous. 
\end{lemma}
\begin{proof}
It will be enough to show that $$gr(x\models g\circ g_1(r))=\bigvee_{z\in Z}[(\phi_1\phi)^{-1}(gr(z\models''r))\wedge f_1\circ f(x,z)],$$ where it is given that $$gr(x\models g(q))=\bigvee_{y\in Y}[\phi^{-1}(gr(y\models'q))\wedge f(x,y)]$$ and $$gr(y\models' g_1(r))=\bigvee_{z\in Z}[\phi_1^{-1}(gr(z\models''r))\wedge f_1(y,z)].$$Now,
\begin{align*}
gr(x\models g\circ g_1(r)) & = gr(x\models g(g_1(r)))\\
& = \bigvee_{y\in Y}[\phi^{-1}(gr(y\models' g_1(r)))\wedge f(x,y)]\\
& = \bigvee_{y\in Y}[\phi^{-1}(\bigvee_{z\in Z}[\phi_1^{-1}(gr(z\models''r))\wedge f_1(y,z)])\wedge f(x,y)]\\ 
& = \bigvee_{y\in Y}[\bigvee_{z\in Z}[\phi^{-1}(\phi_1^{-1}(gr(z\models''r)))\wedge \phi^{-1}(f_1(y,z))]\wedge f(x,y)]\\
& = \bigvee_{y\in Y}[\bigvee_{z\in Z}[(\phi_1\phi )^{-1}(gr(z\models''r ))\wedge \phi^{-1}(f_1(y,z))\wedge f(x,y)]]\\
& = \bigvee_{z\in Z}[\bigvee_{y\in Y}[(\phi_1\phi )^{-1}(gr(z\models''r ))\wedge \phi^{-1}(f_1(y,z))\wedge f(x,y)]]\\
& = \bigvee_{z\in Z}[(\phi_1\phi )^{-1}(gr(z\models''r ))\wedge \bigvee_{y\in Y}[\phi^{-1}(f_1(y,z))\wedge f(x,y)]]\\
& = \bigvee_{z\in Z}[(\phi_1\phi)^{-1}(gr(z\models''r ))\wedge f_1\circ f(x,z)].
\end{align*}
This completes the proof.
\end{proof}
\begin{lemma}\label{3.5_1F}
The identity map defined as above $I_D:D\longrightarrow D$ is continuous.
\end{lemma}
\begin{proof}
By the definition $I_D=(I_A,I_L,I_P)$ satisfies the first two conditions of continuity. The rest part of the proof of the lemma is as follows.
\begin{align*}
\bigvee_{x'\in X}[I_L^{-1}(gr(x'\models p))\wedge I_A(x,x')] &
=I_L^{-1}(gr(x\models p))\wedge I_A(x,x)\\
& =gr(x\models p)\wedge\tilde{A}(x)\\
& =gr(x\models p)\ \text{, as}\ gr(x\models p)\leq\tilde{A}(x)\\
& =gr(x\models I_P(p)).
\end{align*}
\end{proof}
\begin{lemma}\label{3.6_1F}
For the continuous map $f:D\longrightarrow E, I_E\circ f=f$ and $f\circ I_D=f$.
\end{lemma}
\begin{proof}
Follows from the definition.
\end{proof}
From Lemmas \ref{3.4_1F}, \ref{3.5_1F} and \ref{3.6_1F} we get the following theorem.
\begin{theorem}\label{3.7_1F}
Fuzz topological systems together with continuous maps form the category $\mathbf{FuzzTopSys}$\index{category!$\mathbf{FuzzTopSys}$}.
\end{theorem}
The above mentioned continuous map is known as $\mathbf{FuzzTopSys}$ morphism\index{$\mathbf{FuzzTopSys}$ morphism}. 
\subsection{Functors}
The interrelation between the categories $\mathbf{FuzzTopSys}$ and $\mathbf{FuzzTop}$ is now established via some 
suitable functors.
\subsection*{Functor $Ext_F$ from $\mathbf{FuzzTopSys}$ to $\mathbf{FuzzTop}$}\index{functor!$Ext_F$}
 First we propose a notion of extent.
\begin{definition}\label{06}
Let $(X,L,\tilde{A},\models ,P)$ be a Fuzz topological system and $p\in P$. For each $p$, its \textbf{extent$_F$} \index{extent$_F$} in 
$(X,L,\tilde{A},\models ,P)$ is given by $ext_F(p):X\longrightarrow L$ such that $ext_F(p)(x)=gr(x\models p)$ for any 
$x\in X$, and $ext_F(P)=\{ext_F(p)\}_{p\in P}$.
\end{definition}
Now the functor $Ext_F$ is defined as follows:
 \begin{definition}\label{07}
$\mathbf{Ext_F}$ \index{functor!$Ext_F$} is a functor from $\mathbf{FuzzTopSys}$ to $\mathbf{FuzzTop}$ defined thus.
$Ext_F$ acts on the object $(X,L,\tilde{A},\models ',P)$ as $Ext_F(X,L,\tilde{A},\models ',P)=(X,L,\tilde{A},ext_F(P))$ and 
on the morphism $(f,\phi ,g)$ as $Ext_F(f,\phi ,g)=(f,\phi)$.
\end{definition} 
\begin{lemma}\label{3.9_1F}
$(X,L,\tilde{A},ext_F(P))$ is a Fuzz topological space.
\end{lemma}
\begin{proof}
Clearly if $ext_F(p)\in ext_F(P)$ then $$ext_F(p)(x)=gr(x\models p)\leq \tilde{A}(x).$$
Let $ext_F(p_1),\ ext_F(p_2)\in ext_F(P)$. Then
\begin{align*}
(ext_F(p_1)\cap ext_F(p_2))(x) & =min \{ext_F(p_1)(x),ext_F(p_2)(x)\}\\
& =min \{gr(x\models p_1),gr(x\models p_2)\}\\
& =gr(x\models p_1\wedge p_2)\\
& =ext_F(p_1\wedge p_2)(x)
\end{align*}
As $p_1,p_2\in P$, $p_1\wedge p_2\in P$ and hence $ext_F(p_1\wedge p_2)\in ext_F(P)$. \\
Similarly, $\bigcup_i ext_F(p_i)=ext_F(\bigvee_i p_i)\in ext_F(P)$, for $i\in I$(an index set).
\end{proof}
\begin{lemma}\label{3.10_1F}
$(f,\phi)$ is $\mathbf{FuzzTop}$ morphism provided $(f,\phi,g)$ is $\mathbf{FuzzTopSys}$ morphism.
\end{lemma}
\begin{proof}
It will be enough to show that if $ext_F(q)\in ext_F(Q)$ then $ext_F(p)\in ext_F(P)$, where $$ext_F(p)(x)=\bigvee_{x_1\in X_1}[f(x,x_1)\wedge \phi^{-1}ext_F(q)(x_1)].$$ Let us proceed in the following way.
\begin{align*}
ext_F(p)(x) & = \bigvee_{x_1\in X_1}[f(x,x_1)\wedge \phi^{-1}ext_F(q)(x_1)]\\
& = \bigvee_{x_1\in X_1}[f(x,x_1)\wedge \phi^{-1}(gr(x_1\models q))]\\
& = gr(x\models g(q))\\
& = ext_F(g(q))(x).
\end{align*}
Hence $ext_F(p)=ext_F(g(q))\in ext_F(P)$.
\end{proof}
Lemmas \ref{3.9_1F} and \ref{3.10_1F} shows that $Ext_F$ is indeed a functor.
\subsection*{Functor $J_F$ from $\mathbf{FuzzTop}$ to $\mathbf{FuzzTopSys}$} 
\begin{definition}\label{08}
$\mathbf{J_F}$ \index{functor!$J_F$} is a functor from $\mathbf{FuzzTop}$ to $\mathbf{FuzzTopSys}$ defined thus. $J_F$ acts on the object 
$(X,L,\tilde{A},\tau)$ as $J_F(X,L,\tilde{A},\tau)=(X,L,\tilde{A},\in ,\tau)$ where $gr(x\in U)=U(x)$ 
for $U \in \tau$ and on the morphism $(f,\phi)$ as $J_F((f,\phi))=(f,\phi,f^{-1}_{\phi})$, where 
$f^{-1}_{\phi}V(x)=\bigvee_{x_1\in X_1}[f(x,x_1)\wedge \phi^{-1}V(x_1)]$ for all $x\in X$ and 
$V\in\tau_1$.
\end{definition}
\begin{lemma}\label{3.13_1F}
$(X,L,\tilde{A},\in , \tau)$ is an Fuzz topological system.
\end{lemma}
\begin{proof}
Let us check the following.
\begin{enumerate}
\item $gr(x\in U)=U(x)\in L$, for $U\in \tau$ and $x\in X$.
\item $gr(x\in U)=U(x)\leq \tilde{A}(x)$, for $U\in \tau$ and $x\in X$.
\item $gr(x\in U_1\cap U_2)=(U_1\cap U_2)(x)=min\{U_1(x),U_2(x)\}=inf\{gr(x\in U_1),gr(x\in U_2)\}$, for any $U_1,U_2\in \tau$ and $x\in X$.
\item $gr(x\in\bigcup_i U_i)=(\bigcup_iU_i)(x)=sup_i\{U_i(x)\}=sup_i\{gr(x\in U_i)\}$ for $U_i$'s$\in \tau$ and $x\in X$.
\end{enumerate}
This completes the proof.
\end{proof}
\begin{lemma}\label{3.14_1F}
$(f,\phi,f_\phi^{-1})$ is $\mathbf{FuzzTopSys}$ morphism if $(f,\phi)$ is $\mathbf{FuzzTop}$ morphism.
\end{lemma}
\begin{proof}
It will be enough to show that $f_\phi^{-1}$ is a frame homomorphism and $$gr(x\in f_\phi^{-1}(V))=\bigvee_{x_1\in X_1}[\phi^{-1}gr(x\in V)\wedge f(x,x_1)].$$
We have,
\begin{align*}
f_\phi^{-1}(U_1\cap U_2)(x) & = \bigvee_{x_1\in X_1}[f(x,x_1)\wedge \phi^{-1}(U_1\cap U_2)(x_1)]\\
& = \bigvee_{x_1\in X_1}[f(x,x_1)\wedge \phi^{-1}(U_1(x_1)\wedge U_2(x_1))]\\
& = \bigvee_{x_1\in X_1}[f(x,x_1)\wedge \phi^{-1}(U_1(x_1))\wedge \phi^{-1}(U_2(x_1))]\\
& = \bigvee_{x_1\in X_1}[f(x,x_1)\wedge \phi^{-1}(U_1(x_1))]\wedge \bigvee_{x_1\in X_1}[f(x,x_1)\wedge \phi^{-1}(U_2(x_1))]\\
& = f_\phi^{-1}(U_1)(x)\wedge f_\phi^{-1}( U_2)(x)\\
& = (f_\phi^{-1}(U_1)\cap f_\phi^{-1}(U_2))(x).
\end{align*}
Therefore $f_\phi^{-1}(U_1\cap U_2)=f_\phi^{-1}(U_1)\cap f_\phi^{-1}(U_2)$ and similarly $f_\phi^{-1}(\bigcup_iU_i)=\bigcup_if_\phi^{-1}(U_i)$. Hence $f_\phi^{-1}$ is a frame homomorphism.\\
Let us proceed in the following way for the rest part.
\begin{align*}
gr(x\in f_\phi^{-1}(V)) & = f_\phi^{-1}(V)(x)\\
& = \bigvee_{x_1\in X_1}[f(x,x_1)\wedge \phi^{-1}V(x) ]\\ 
& = \bigvee_{x_1\in X_1}[\phi^{-1}gr(x\in V)\wedge f(x,x_1)].
\end{align*}
This finishes the proof.
\end{proof}
So $ J_F$ is a functor from $\mathbf{FuzzTop}$ to $\mathbf{FuzzTopSys}$.
\begin{theorem}\label{ext1}
$Ext_F$ is the right adjoint to the functor $J_F$.
\end{theorem}
\begin{proof}
We will prove the theorem by presenting the co-unit of the adjunction. Recall that 
$J_F(X,L,\tilde{A},\tau)=(X,L,\tilde{A},\in \tau)$ and $Ext_F(X,L,\tilde{A},\models,P)=(X,L,\tilde{A},ext_F(P))$. So,
$$J_F(Ext_F(X,L,\tilde{A},\models,P))=(X,L,\tilde{A},\in,ext_F(P)).$$
Let us draw the diagram of co-unit.
\begin{center}
\begin{tabular}{ l | r } 
$\mathbf{FuzzTopSys}$ & $\mathbf{FuzzTop}$\\
\hline
 {\begin{tikzpicture}[description/.style={fill=white,inner sep=2pt}] 
    \matrix (m) [matrix of math nodes, row sep=2.5em, column sep=0.5em]
    { J_F(Ext_F(X,L,\tilde{A},\models, P))&&(X,L,\tilde{A},\models, P)  \\
         J_F(Y,M,\tilde{B}, \tau ') \\ }; 
    \path[->,font=\scriptsize] 
        (m-1-1) edge node[auto] {$\xi_X$} (m-1-3)
        (m-2-1) edge node[auto] {$J_F(f,\phi)(\equiv (f,\phi,f_{\phi}^{-1}))$} (m-1-1)
        (m-2-1) edge node[auto,swap] {$(f,\phi,g)$} (m-1-3)
        %(m-1-1) edge node[auto,swap] {$f$} (m-2-3)
       % (m-1-5) edge node[auto,swap] {$j$} (m-1-3)
       % (m-1-5) edge node[auto] {$\psi$} (m-2-3)
        %(m-1-3) edge node[auto] {$G(\hat{f})$} (m-2-3)
         ;
\end{tikzpicture}} & {\begin{tikzpicture}[description/.style={fill=white,inner sep=2pt}] 
    \matrix (m) [matrix of math nodes, row sep=2.5em, column sep=0.1em]
    { Ext_F(X,L,\tilde{A}, \models, P)  \\
         (Y,M,\tilde{B}, \tau ') \\ }; 
    \path[->,font=\scriptsize] 
        (m-2-1) edge node[auto,swap] {$(f,\phi)$} (m-1-1)
       
         ;
\end{tikzpicture}} \\ 
\end{tabular}
\end{center}
Let us define co-unit by $\xi_X=(i_A,i_L,ext_F^*)$\\
i.e.
\begin{center}
\begin{tikzpicture}[description/.style={fill=white,inner sep=2pt}] 
    \matrix (m) [matrix of math nodes, row sep=2.5em, column sep=2.5em]
    { (X,L,\tilde{A},\in ,ext_F(P))&&(X,L,\tilde{A},\models, P)  \\
          }; 
    \path[->,font=\scriptsize] 
        (m-1-1) edge node[auto] {$\xi_X$} (m-1-3)
        (m-1-1) edge node[auto,swap] {$(i_A,i_L,ext_F^*)$} (m-1-3)
         ;
\end{tikzpicture}
\end{center}
 where $i_A:X\times X\longrightarrow L$, $i_L:L\longrightarrow L$ and $ext_F^*$ is a mapping from $P$ to $ext_F(P)$ such 
 that, $ext_F^*(p)=ext_F(p)$ for all $p\in P$.
It can be shown that $$(i_A,i_L,ext_F^*):J_F(Ext_F(X,L,\tilde{A},\models , P))\longrightarrow (X,L,\tilde{A},\models , P)$$ 
is indeed a continuous map of Fuzz topological system as follows.
Now we will prove that the diagram on the left commutes.
Here $J_F(f,\phi)=(f,\phi,f_{\phi}^{-1})$ and 
$(f,\phi,g)=\xi_X\circ J_F(f,\phi)=(i_A,i_L,ext^*)\circ (f,\phi,f_{\phi}^{-1})=(i_A\circ f,i_L\circ \phi,f_{\phi}^{-1}\circ ext_F^*)$.\\
We have $i_A\circ f:Y\times X\longrightarrow M$ such that
$$i_A\circ f(y,x)=\bigvee_{x'\in X}[f(y,x')\wedge \phi^{-1}(i_A(x',x))]=f(y,x)\wedge \phi^{-1}\tilde{A}(x).$$
Now $$f(y,x)\leq \tilde{B}(y)\wedge \phi^{-1}\tilde{A}(x)\leq \phi^{-1}\tilde{A}(x).$$
Hence $f(y,x)\wedge \phi^{-1}\tilde{A}(x)=f(y,x)$.
Therefore $i_A\circ f(y,x)=f(y,x)$ and consequently $i_A\circ f=f$.
Clearly $i_L\circ \phi=\phi$.
Now as $(i_A,i_L,ext_F^*)$ is continuous so $$gr(x\in ext_F^*(p))=\bigvee_{x'\in X}[i_L^{-1}(gr(x'\models p))\wedge i_A(x,x')].$$
Hence, $$ext_F^*(p)(x)=gr(x\models p)\wedge \tilde{A}(x)=gr(x\models p).$$
Now $(f,\phi,g)$ is continuous hence $$gr(y\in g(p))=\bigvee_{x\in X}[\phi^{-1}(gr(x\models p))\wedge f(y,x)].$$
So, $$g(p)(y)=\bigvee_{x\in X}[\phi^{-1}(gr(x\models p))\wedge f(y,x)].$$
Hence 
\begin{multline*}
$$f_{\phi}^{-1}ext_F(p)(y)=\bigvee_{x\in X}[f(y,x)\wedge \phi^{-1}(ext_F(p)(x))]=\\ \bigvee_{x\in X}[\phi^{-1}(gr(x\models p))\wedge f(y,x)]=g(p)(y).$$
\end{multline*}
Hence $f_{\phi}^{-1}\circ ext_F^*(p)(y)=g(p)(y)$.
So, $g=f_{\phi}^{-1}\circ ext_F^*$.
%\begin{align*}
%f_1^{-1}ext'(p) & = f_1^{-1}(p)\\
%& = f_2(p)
%\end{align*}
Therefore, $$\xi_X(\equiv (i_A,i_L,ext_F^*)):J_F(Ext_F(X,L,\tilde{A},\models, P))\longrightarrow (X,L,\tilde{A},\models, P)$$ is the 
co-unit, consequently $Ext_F$ is the right adjoint to the functor $J_F$.
\end{proof}
The obtained functorial relationships can be illustrated by the following diagram:
  
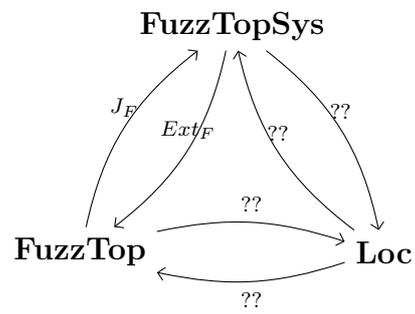
\begin{center}
\begin{tikzpicture}
\node (C) at (0,3) {$\mathbf{FuzzTopSys}$};
\node (A) at (-2,0) {$\mathbf{FuzzTop}$};
\node (B) at (2,0) {$\mathbf{Loc}$};
%\node at (0,0) {\rotatebox{270}{$\Rightarrow$}};
\path[->,font=\scriptsize ,>=angle 90]
(A) edge [bend left=15] node[above] {$??$} (B);
\path[<-,font=\scriptsize ,>=angle 90]
(A)edge [bend right=15] node[below] {$??$} (B);
\path[->,font=\scriptsize ,>=angle 90]
(A) edge [bend left=20] node[above] {$J_F$} (C);
\path[<-,font=\scriptsize ,>=angle 90]
(A)edge [bend right=20] node[above] {$Ext_F$} (C);
\path[->,font=\scriptsize, >=angle 90]
(C) edge [bend left=20] node[above] {$??$} (B);
\path[<-,font=\scriptsize, >=angle 90]
(C)edge [bend right=20] node[above] {$??$} (B);
\end{tikzpicture}
\end{center}
Note that in this chapter localification i.e existence of adjunction between FuzzTopSys and  Loc$\times$Loc is not 
established. This can be considered as an interesting open question.
\chapter{Category of $\bar{n}$-Fuzzy Boolean Systems}
\section{Introduction}\blfootnote{The results of this chapter appear in {\bf \cite{MP} P. Jana and M.K. Chakraborty: \textit{Categorical relationships of fuzzy topological systems with fuzzy topological spaces and underlying algebras}, Ann. of Fuzzy Math. and Inform., \textbf{8}, 2014, no. 5, pp. 705--727.}}
It may be recalled that the celebrated Stone duality theorem \cite{PJS} of algebraic logic states that there exists a categorical duality between Boolean algebras and zero dimensional compact Hausdorff spaces \cite{PJS,JM}. In \cite{YM}, Maruyama showed a duality between zero dimensional compact Kolmogorov fuzzy topological spaces and the algebras of {\L}ukasiewicz $n$-valued logic which generalizes the Stone duality in a way. In this section duality between the category of one kind of fuzzy topological systems and {\L}$_n^c$-algebras (an {\L}$_n^c $-algebra is an $ MV_n$-algebra enriched by constants \cite{YM}) shall be established. This category will be shown to be equivalent to the category of $\bar{n}$-fuzzy topological spaces, which are Kolmogorov, compact and zero-dimensional. As a consequence, duality between the category of these topological spaces and the category of {\L}$_n^c$-algebras will be established. This result constitutes another proof of the duality proved in \cite{YM}.

\section{$\bar n$-Fuzzy Boolean system, {\textbf{\L}}$\mathbf{_n^c}$\textbf{-Alg}, $\bar n$-Fuzzy Boolean Space and their interrelationships}

We first give the definitions of the related notions, which deal with {\L}$_n^c$-algebras. 

\subsection{{\L}$_n^c$-algebras}

%$\L_n^c$-algebra is considered as $\mathbf{MV}_n$-algebra enriched by constants.
\begin{definition}[$\mathbf{MV}$-algebra]\cite{PM}
An $\mathbf{MV}$\textbf{-algebra} \index{$\mathbf{MV}$-algebra} is an algebra $\mathcal{A}=(A,\oplus,*,^{\bot},0,1)$ of type $(2,2,1,0,0)$, where $(A,\oplus,0)$ is a commutative monoid and for any $x,y,z\in A$ the following axioms are satisfied:
\begin{flalign*}
1.\ & x\oplus 1=1, & 2.\ & (x^{\bot})^{\bot}=x, & 3.\ & 0^{\bot}=1,\\
4.\ & (x^{\bot}\oplus y)^{\bot}\oplus y=(y^{\bot}\oplus x)^{\bot}\oplus x, & 5.\ & x*y=(x^{\bot}\oplus y^{\bot})^{\bot}.
\end{flalign*}
\end{definition}
The original definition of $\mathbf{MV}$-algebra was given in \cite{CH}. The present definition is by Mangani \cite{PM} which we take from \cite{AI}.
 %where the following axioms are satisfied: for every $x,y,z\in A$\\
%\begin{flalign*}
%1. & x\oplus y=y\oplus x & 1'. & x*y=y*x\\
%2. & x\oplus (y\oplus z)=(x\oplus y)\oplus z & 2'. & x*(y*z)=(x*y)*z\\
%3. & x\oplus (x)^{\bot}=1 & 3'. & x*x^{\bot}=0\\
%4. & x\oplus 1=1 & 4'. & x*0=0\\
%5. & x\oplus 0=x & 5'. & x*1=x\\
%6. &(x\oplus y)^{\bot}=x^{\bot}*y^{\bot} & 6'. & (x*y)^{\bot}=x^{\bot}\oplus y^{\bot}\\
%7. & (x^{\bot})^{\bot}=x\\
%8. & 0^{\bot}=1\\
%9. & x\vee y=y\vee x & 9'. & x\wedge y=y\wedge x\\
%10. & x\vee (y\vee z)=(x\vee y)\vee z & 10'. & x\wedge (y\wedge z)=(x\wedge y)\wedge z\\
%11. & x\oplus (y\wedge z)=(x\oplus y)\wedge (x\oplus z) & 11'. & x*(y\vee z)=(x*y)\vee (x*z)
%\end{flalign*}
%where $\ x\vee y=(x*y^{\bot})\oplus y$\\
%and $\ x\wedge y=(x^{\bot}\vee y^{\bot})^{\bot}=(x\oplus y^{\bot})*y$
%\end{definition}
\begin{definition}\label{5.1}
For any $m\in \mathbb{N}$, we define \\
(i) $0x=0$ and $(m+1)x=mx\oplus x$.\\
(ii) $x^0=1$ and $x^{m+1}=x^m*x$.
\end{definition}
\begin{definition}[$\mathbf{MV}_n$-algebra]\label{5.2}\cite{GR}
An $\mathbf{MV}_n$\textbf{-algebra} \index{$\mathbf{MV}_n$-algebra} $(n\geq 2)$ is an $\mathbf{MV}$-algebra $\mathcal{A}=(A,\oplus,*,^{\bot},0,1)$ whose operations fullfil the additional axioms
\begin{flalign*}
1.\ & (n-1)x\oplus x=(n-1)x, & 1'.\ & x^{(n-1)}*x=x^{(n-1)},
\end{flalign*}
and if $n\geq 4$  the  axioms
\begin{flalign*}
2.\ & [(jx)*(x^{\bot}\oplus[(j-1)x]^{\bot})]^{n-1}=0, & 2'.\ & (n-1)[x^j\oplus (x^{\bot}*[x^{j-1}]^{\bot})]=1, 
\end{flalign*}
where  $1<j<n-1$ and $j$ does not divide $n-1$.
\end{definition}
\begin{definition}\label{5.3}
\cite{YM} $\bar{n}$ denotes the set $\{0,\dfrac{1}{n-1},\dfrac{2}{n-1},\dfrac{3}{n-1},......,\dfrac{n-2}{n-1},1\}$ \\equipped with all constants $r\in \bar{n}$ and the operations $(\wedge,\vee,*,\oplus,\rightarrow,^{\bot})$ are defined as follows:\\
$x\wedge y=min(x,y),\ x\vee y=max(x,y),\ x*y=max(0,x+y-1),\ x\oplus y=min(1,x+y),\ x\rightarrow y=min(1,1-(x-y)),\ x^{\bot}=1-x$ and $0$-ary operations (i.e. constants) $r\in \bar{n}$.
\end{definition}
\begin{definition}[{\L}$_n^c$-algebra]\label{5.4}
An \textbf{{\L}$_n^c$-algebra} \index{{\L}$_n^c$-algebra} is an $MV_n$ algebra enriched by $n$ constants \cite{YM}. That is, it is an $MV_n$ algebra $\mathcal{A}=(A,\wedge,\vee,*,\oplus,\rightarrow,^{\bot},0,1)$ in which the algebra $\bar{n}$ is embedded.
\end{definition}
We shall denote the counterparts of $\dfrac{1}{n-1},\dfrac{2}{n-1},\dfrac{3}{n-1},...,\dfrac{n-2}{n-1}$ in $A$ by these tokens and $(A, \wedge, \vee, *, \oplus, \rightarrow, ^{\bot}, 0, \dfrac{1}{n-1}, \dfrac{2}{n-1}, \dfrac{3}{n-1}, . . . , \dfrac{n-2}{n-1}, 1)$ denotes a general {\L}$_n^c$-algebra. We note that $\bar{n}$ is an {\L}$_n^c$-algebra and also that every {\L}$_n^c$-algebra is a frame.
\begin{definition}[{\L}$_n^c$-homomorphism]\label{5.5}\cite{YM}
An \textbf{{\L}$_n^c$-homomorphism} \index{{\L}$_n^c$-homomorphism} is a function between two {\L}$_n^c$-algebras which preserves the operations.
\end{definition}
The following results w.r.t. {\L}$_n^c$-algebras may be obtained in \cite{YM}.
\begin{proposition}\label{5.6}\cite{YM}
\begin{enumerate}
\item Let $A$ be an {\L}$_n^c$-algebra. If $a,b\in A$ are idempotent\index{idempotent}, i.e., $a*a=a,$ $b*b=b$, then $a*b=a\wedge b$ and $a\oplus b=a\vee b$.
\item Let $A$ be an {\L}$_n^c$-algebra and $r\in \bar n$. There is an idempotent term $T_r(x)$ with one variable $x$ such that, for any homomorphism $v:A\longrightarrow \bar n$ and any $x\in A$, the following holds:\\
(i) $v(T_r(x))=1$ iff $v(x)=r$;\\
(ii) $v(T_r(x))=0$ iff $v(x)\neq r$. \\
Any homomorphism from one {\L}$_n^c$-algebra to another preserves the operation $T_r(-).$
\item Let $A$ be an {\L}$_n^c$-algebra and $a_i\in A$ for $i\in I$ and $I$ is a finite set. Then, $T_1(\bigvee_{i\in I}a_i)=\bigvee_{i\in I}T_1(a_i)$ and $T_1(\bigwedge_{i\in I}a_i)=\bigwedge_{i\in I}T_1(a_i)$.
\item Let $A$ be an {\L}$_n^c$-algebra. For any $a,b\in A$, the following holds:\\
$\bigwedge_{r\in \bar n}(T_r(a)\leftrightarrow T_r(b))\leq a\leftrightarrow b$, where $a\leftrightarrow b\equiv (a\rightarrow b)\wedge (b\rightarrow a)$.
\item Let $A$ be an {\L}$_n^c$-algebra and $r\in \bar n$. There is a term $S_r(x)$ with one variable $x$ such that for any homomorphism $v:A\longrightarrow \bar n$, the following two conditions hold:\\
(i) $v(S_r(x))=r$ iff $v(x)=1$;\\
(ii) $v(S_r(x))=0$ iff $v(x)\neq 1$.\\
Any homomorphism preserves the operation $S_r(-).$
\end{enumerate}
\end{proposition}
\begin{definition}[$\bar n$-filter]\label{5.7}\cite{YM}
Let $A$ be an {\L}$_n^c$-algebra. A non-empty subset $F$ of $A$ is called an $\mathbf{\bar n}$\textbf{-filter} \index{$\bar n$-filter} of $A$ iff $F$ is an upper set and is closed under $*$. 

An $\bar n$-filter $F$ of $A$ is called \textbf{proper} $\mathbf{\bar n}$\textbf{-filter} iff $F\neq A.$

A proper $\bar n$-filter $P$ of $A$ is \textbf{prime} iff for any $a,\ b\in A$, $a\vee b\in P$ implies either $a\in P$ or $b\in P$.
\end{definition}
\begin{definition}[Finite intersection property]\label{5.8}\cite{YM}
Let $A$ be an {\L}$_n^c$-algebra. A subset $X$ of $A$ has \textbf{finite intersection property} (f.i.p.)\index{finite intersection property}\index{f.i.p.} with respect to $*$ iff for any non-empty subset $\{a_1,a_2,....,a_n\}$, $a_1*......*a_n\neq 0$.
\end{definition}
\begin{proposition}\label{5.9}\cite{YM}
\begin{enumerate}
\item Let $A$ be an {\L}$_n^c$-algebra and $F$ an $\bar{n}$-filter of $A$. Let $b\in A$ be such that $b\notin F$. Then there is a prime $\bar{n}$-filter $P$ of $A$ such that $F\subset P$ and $b\notin P$.
\item Let $A$ be an {\L}$_n^c$-algebra and $X$ be a subset of $A$. If $X$ has f.i.p. with respect to $*$, then there is a prime $\bar n$-filter $P$ of $A$ with $X\subset P.$
\item Let $A$ be an {\L}$_n^c$-algebra. For a prime $\bar n$-filter $P$ of $A$, define $v_P:A\longrightarrow \bar n$ by $v_P(a)=r\Leftrightarrow T_r(a)\in P$. Then, $v_P$ is a bijection from the set of all prime $\bar n$-filters of $A$ to the set of all homomorphisms from $A$ to $\bar n$ with $v_P^{-1}(\{ 1\})=P.$
\end{enumerate}
\end{proposition}

\subsection{$\bar{n}$-valued fuzzy topology}

An $\bar{n}$-fuzzy set on a set $X$ is defined as a function from $X$ to $\bar{n}$.
Let $\tilde{A}$, $\tilde{B}$ be two $\bar{n}$-fuzzy sets. Operations on $\bar{n}$-fuzzy sets are defined pointwise in the usual way. %Let us define $\bar{n}$-fuzzy set\\
%$\tilde{A} \wedge \lambda$ on $X$ as $(\tilde{A}\wedge \lambda)(x)=\tilde{A} (x)\wedge \lambda (x)$\\
%$\tilde{A} \vee \lambda$ on $X$ as $(\tilde{A}\vee \lambda)(x)=\tilde{A} (x)\vee \lambda (x)$\\
%$\tilde{A} * \lambda$ on $X$ as $(\tilde{A} * \lambda)(x)=\tilde{A} (x) * \lambda (x)$\\
%$\tilde{A} \oplus \lambda$ on $S$ as $(\tilde{A}\oplus \lambda)(x)=\tilde{A} (x)\oplus \lambda (x)$\\
%$\tilde{A} \rightarrow \lambda$ on $X$ as $(\tilde{A}\rightarrow \lambda)(x)=\tilde{A} (x)\rightarrow \lambda (x)$\\
%$(\tilde{A})^{\bot}$ on $X$ as $(\tilde{A})^{\bot}(x)=(\tilde{A} (x))^{\bot}$\\
Let $X$, $Y$ be sets and $f$ a function from $X$ to $Y$. For an $\bar{n}$-fuzzy set $\tilde{A}$ on $X$, define the direct image $f(\tilde{A}):Y\longrightarrow \bar{n}$ of $\tilde{A}$ under $f$ by
$f(\tilde{A})(y)=\bigvee \{ \tilde{A} (x):x \in f^{-1}(\{ y\})\}$ for $y\in Y$.
For $f:X\longrightarrow Y$ and an $\bar{n}$-fuzzy set $\tilde{B}$ on $Y$, define the inverse image $f^{-1}(\tilde{B}):X\longrightarrow \bar{n}$ of $\tilde{B}$ under $f$ by $f^{-1}(\tilde{B})=\tilde{B} \circ f$.\\
Note: $f^{-1}$ commutes with $\bigvee$, i.e., $f^{-1}(\bigvee_{i\in I}\tilde{B}_i)=\bigvee_{i\in I}f^{-1}(\tilde{B}_i)$ for $\bar{n}$-fuzzy sets $\tilde{B}_i$ on $Y$. These definitions are already given in the introduction.

For sets $X$ and $Y$, $Y^X$ denotes the set of all functions from $X$ to $Y$. By $r$ we shall also denote the constant function with value $r\in \bar{n}$.
%\begin{definition}\label{5.10}\cite{YM}
%Let $\tau$ be a collection of $\bar{n}$-fuzzy subsets of a set $X$, such that the following conditions hold.\\
%For a set $X$ and a $\tau$ of $\bar{n}$-fuzzy subsets of $X$, $(X,\tau)$ is called an $\bar{n}$-fuzzy topological space if and only if the following conditions hold:\\
%1. $r\in \tau$ for any $r\in \bar{n}$.\\ %where $r$ is the constant function from $X$ to $\bar n$ whose value is always $r$.\\
%2. if $\mu_1$, $\mu_2\in \tau$ then $\mu_1\wedge\mu_2\in \tau$\\
%3. if $\mu_i\in \tau$ for $i\in I$ then $\bigvee_{i\in I}\mu_i\in \tau$\\
%We call $\tau$ the $\bar{n}$-fuzzy topology of $(S,\tau)$ and an element of $\tau$ an open $\bar{n}$-fuzzy set on $(S,\tau)$.\\
%Then the pair $(X,\tau)$ is the L{\"o}wen-type $\bar{n}$-fuzzy topological space.
%\end{definition}
%Clearly this is L{\"o}wen's definition of fuzzy topological space.
The following definitions are standard ones restated from \cite{YM}.
\begin{definition}[Discrete $\bar{n}$-fuzzy topology]\label{5.11}\cite{YM}
For a set $X$, $\bar{n}^X$ is called the \textbf{discrete $\bar{n}$-fuzzy topology}\index{discrete $\bar{n}$-fuzzy topology} on $X$. $(X,\bar{n}^X)$ is called a discrete $\bar{n}$-fuzzy topological space.
\end{definition}
%\begin{definition}\label{5.12}\cite{YM}
%Let $(X,\tau_1)$, $(Y,\tau_2)$ be $\bar{n}$-fuzzy topological spaces. Then $f:X \longrightarrow Y$ is continuous iff for any open $\bar{n}$-fuzzy set $\mu$ on $(Y,\tau_2)$, $f^{-1}(\mu)$(i.e. $\mu \circ f$) is an open $\bar{n}$-fuzzy set on $(X,\tau_1)$.
%\end{definition}
\begin{definition}[Open basis]\label{5.13}\cite{YM}
Let $(X,\tau)$ be an $\bar{n}$-fuzzy topological space. Then an \textbf{open basis}\index{open basis} $\mathcal{B}$ of $(X,\tau)$ is a subset of $\tau$ such that the following holds:\\
(i) $\mathcal{B}$ is closed under finite meets;\\
(ii) for any $\tilde{A} \in \tau$, there are $\tilde{A}_i\in \mathcal{B}$ for $i\in I$ such that $\tilde{A}=\bigvee_{i\in I}\tilde{A}_i$.
\end{definition}
\begin{definition}[Kolmogorov space]\label{5.14}\cite{YM}
An $\bar{n}$-fuzzy topological space $(X,\tau)$ is \textbf{Kolmogorov}\index{Kolmogorov space} iff for any $x_1,x_2\in X$ with $x_1\neq x_2$, there is an open $\bar{n}$-fuzzy set $\tilde{A}$ on $(X,\tau)$ with $\tilde{A} (x_1)\neq \tilde{A} (x_2)$.
\end{definition}
\begin{definition}[Hausdorff space]\label{5.15}\cite{YM}
An $\bar{n}$-fuzzy topological space $(X,\tau)$ is \textbf{Hausdorff} \index{Hausdorff space} iff for any $x_1,x_2\in X$ with $x_1\neq x_2$, there are $r\in \bar{n}$ and open $\bar{n}$-fuzzy sets $\tilde{A_1},$ $\tilde{A_2}$ on $(X,\tau)$ such that $\tilde{A_1} (x_1)\geq r$, $\tilde{A_2} (x_2)\geq r$ and $\tilde{A_1}\wedge \tilde{A_2} <r$.
\end{definition}
\begin{definition}[Compact]\label{5.16}\cite{YM}
Let $(X,\tau)$ be an $\bar{n}$-fuzzy topological space. An $\bar{n}$-fuzzy set $\tilde{A}$ on $(X,\tau)$ is \textbf{compact} \index{compact} iff if $\tilde{A}\leq\bigvee_{i\in I}\tilde{A}_i$ for open $\bar{n}$-fuzzy sets $\tilde{A}_i$ on $X$, then there is a finite subset $J$ of $I$ s.t. $\tilde{A} \leq \bigvee_{i\in J}\tilde{A}_i$.
\end{definition}
$(X,\tau)$ is compact iff, if $\mathbf{1}=\bigvee_{i\in I}\tilde{A}_i$ for open $\bar{n}$-fuzzy sets $\tilde{A}_i$ on $X$, then there is a finite subset $J$ of $I$ s.t. $\mathbf{1}= \bigvee_{i\in J}\tilde{A}_i$, where $\mathbf{1}$ is the constant map taking every elements to $1$.
%Let $\mathbf{1}$ denote the constant function on $X$ whose value is always $1$. Then, $(X,\tau)$ is compact iff, if $\mathbf{1}=\bigvee_{i\in I}\mu_{i\in I}$ for open $\bar{n}$-fuzzy set $\mu_i$ on $(X,\tau)$, then there is a finite subset $J$ of $I$ s.t. $\mathbf{1}=\bigvee_{i\in J}\mu_i$.
%\begin{definition}
%Let $(S,\theta)$ be an $\bar{n}$-fuzzy space. Define $\theta^*=\{\mu^{-1}(\{1\}):\mu\in\theta\}$.\\
%Then $S^*$ denotes a topological space $(S,\theta^*)$.
%\end{definition} 

In \cite{YM}, Maruyama established the duality between {\L}$_n^c$-algebra and $\bar{n}$-fuzzy Boolean space, which is a zero dimensional, compact, Kolmogorov $\bar{n}$-fuzzy topological space, via suitable functors. Maruyama's work can be depicted as follows:
\begin{center}
\begin{tikzpicture}
%\node (C) at (0,3) {$\mathbf{FBSys_n}$};
\node (A) at (-2.5,0) {$\mathbf{FBS_n}$};
\node (B) at (2.5,0) {({\textbf{\L}}$\mathbf{_n^c}$\textbf{-Alg})$^{op}$};
%\node at (0,0) {\rotatebox{270}{$\Rightarrow$}};
\path[->,font=\scriptsize ,>=angle 90]
(A) edge [bend left=15] node[above] {$Cont$} (B);
\path[<-,font=\scriptsize ,>=angle 90]
(A) edge [bend right=15] node[below] {$Spec$} (B);
%\path[->,font=\scriptsize ,>=angle 90]
%(A) edge [bend left=20] node[above] {$J$} (C);
%\path[<-,font=\scriptsize ,>=angle 90]
%(A) edge [bend right=20] node[above] {$Ext$} (C);
%\path[->,font=\scriptsize ,>=angle 90]
%(C) edge [bend left=20] node[above] {$Lag$} (B);
%\path[<-,font=\scriptsize ,>=angle 90]
%(C) edge [bend right=20] node[above] {$S$} (B);
\end{tikzpicture}
\end{center} 
In our work we introduce the notion of $\bar{n}$-fuzzy Boolean system. Consequently duality between $\bar{n}$-fuzzy Boolean system and {\L}$_n^c$-algebra, as well as equivalence of $\bar{n}$-fuzzy Boolean system and $\bar{n}$-fuzzy Boolean space are established. As a result duality between $\bar{n}$-fuzzy Boolean space and {\L}$_n^c$-algebra is shown.
\subsection{Categories} 

\subsection*{$\bar n$-Fuzzy Boolean systems}
\begin{definition}[$\bar n$-fuzzy Boolean system]\label{5.17} \index{$\bar n$-fuzzy Boolean system}
 An \textbf{$\bar n$-fuzzy Boolean system} is a triple $(X,\models ,A)$ where $X$ is a non empty set, $A$ is an {\L}$_n^c$-algebra and $\models$ is an $n$-valued fuzzy relation from $X$ to $A$ such that 
 \begin{enumerate}
\item $gr(x\models a*b)=max(0,gr(x\models a)+gr(x\models b)-1)$;
\item $gr(x\models a^{\bot})=1-gr(x\models a)$;
\item $gr(x\models r)=r$ for all $r\in \bar n$;
\item $x_1\neq x_2\Rightarrow gr(x_1\models a)\neq gr(x_2\models a)$ for some $a\in A$.
\end{enumerate}
\end{definition}
It turns out that an $\bar{n}$-fuzzy Boolean system is an $\bar{n}$-fuzzy topological system with certain additional conditions, as $a\vee b=(a*b^{\bot})\oplus b$ and $a\wedge b=(a\oplus b^{\bot})*b$.
\begin{definition}\label{5.18} Let $D=(X,\models ,A)$ and $E=(Y,\models ',B)$ be $\bar n$-fuzzy Boolean systems. A \textbf{continuous map}\index{continuous map!between!fuzzy Boolean systems} $f:D\longrightarrow E$ is a pair $(f_1,f_2)$, where
\begin{enumerate}
\item $f_1:X\longrightarrow Y$ is a function;

\item $f_2:B\longrightarrow A$ is {\L}$_n^c$-homomorphism;

\item $gr(x\models f_2(b))=gr(f_1(x)\models' b)$, for all $x\in X$, $b\in B.$
\end{enumerate}
\end{definition}
\begin{definition}\label{5.19}
Let $D=(X,\models,A)$ be $\bar n$-fuzzy Boolean system. The identity map $I_D:D\longrightarrow D$ is the pair $(I_1,I_2)$ of identity maps, where $I_1:X\longrightarrow  X$ and $I_2:A\longrightarrow  A$.

Let $D=(X,\models',A)$, $E=(Y,\models'',B)$, $F=(Z,\models''',C)$. Let $(f_1,f_2):D\longrightarrow E$ and $(g_1,g_2):E\longrightarrow F$ be continuous maps. The composition $(g_1,g_2)\circ (f_1,f_2):D\longrightarrow F$ is defined by $(g_1,g_2)\circ (f_1,f_2)=(g_1\circ f_1,f_2\circ g_2)$, where $g_1\circ f_1:X\longrightarrow  Z$ and
$f_2\circ g_2:C\longrightarrow  A$.
\end{definition}
%\begin{lemma}\label{5.20}
%$(g_1,g_2)\circ (f_1,f_2):D\longrightarrow F$ is continuous. 
%\end{lemma}
%\begin{proof}
%It is enough to show that $gr(x\models' f_2\circ g_2(c))=gr(g_1\circ f_1(x)\models''' c)$.
%\\ We have,
%\begin{equation*}
%\begin{align*}
%gr(x\models' f_2\circ g_2(c))& =gr(x\models' f_2(g_2(c)))\\
%& =gr(f_1(x)\models'' g_2(c))\tag{as $(f_1,f_2)$ is continuous}\\%\label{eq:as $(f_1,f_2)$ is continuous}\\
%& =gr(g_1(f_1(x))\models''' c )\tag {as $(g_1,g_2)$ is continuous}\\%\label{eq:as $(g_1,g_2)$ is continuous}\\
%& =gr(g_1\circ f_1(x)\models''' c)
%\end{align*}
%\end{equation*}
%\end{proof}
%Thus we arrive at the following theorem. 
%\begin{theorem}\label{5.21}
%$\bar n$-fuzzy Boolean Systems together with continuous functions forms a category denoted $\mathbf{FBSys_n}$.
%\end{theorem}
$\bar n$-fuzzy Boolean Systems together with continuous functions forms a category denoted by $\mathbf{FBSys_n}$ \index{category!$\mathbf{FBSys_n}$}.
\begin{definition}\label{5.22}
A continuous map $f:D(\equiv(X,\models ,A))\longrightarrow E(\equiv(Y,\models ' ,B))$ is a homeomorphism if and only if there is a map $g:E\longrightarrow D$ such that $g\circ f=id_D$ and $f\circ g=id_E$.

When there is a homeomorphism \index{homeomorphism} from $D$ to $E$ then we will call that $D$ and $E$ are homeomorphic.
\end{definition}
This means that homeomorphic systems are structurally equivalent, i.e., (assuming that $f=(f_1,f_2)$ and $g=(g_1,g_2)$).
\begin{itemize}
\item $X$ and $Y$ are in bijective correspondence i.e there exists
 a bijection between $X$ and $Y$;
\item $A$ and $B$ are isomorphic {\L}$_n^c$-algebras;
\item $gr(x\models f_2(b))=gr(f_1(x)\models ' b)$ and $gr(g_1(y)\models a)=gr(y\models ' g_2(b))$.
\end{itemize}
Note that since $X$ and $Y$ are in bijective correspondence, the last condition reduces to $gr(x\models f_2(b))=gr(f_1(x)\models ' b)$.
\subsection*{The category {\textbf{\L}}$\mathbf{_n^c}$\textbf{-Alg}}
%\begin{theorem}\label{5.23}
{\L}$_n^c$-algebras together with {\L}$_n^c$-homomorphisms form the category {\textbf{\L}}$\mathbf{_n^c}$\textbf{-Alg} \index{{\textbf{\L}}$\mathbf{_n^c}$\textbf{-Alg}}.
%\end{theorem}
\subsection*{$\bar n$-fuzzy Boolean spaces}
Let $\bar{n}$ be equipped with the discrete $\bar{n}$-fuzzy topology.
\begin{definition}\label{5.24}\cite{YM}
For an $\bar{n}$-fuzzy topological space, $Cont(X,\tau)$ is defined as the set of all continuous functions from $X$ to $\bar{n}$. We endow $Cont(X,\tau)$ with the operations $(\wedge,\vee,*,\oplus,\rightarrow,^{\bot},0,\dfrac{1}{n-1},\dfrac{2}{n-1},...,
\dfrac{n-2}{n-1},1)$ defined point wise. %For $f$, $g\in Cont(X,\tau)$, define\\
%\begin{align*}
%(f\wedge g)(x)=f(x)\wedge g(x)\\
%(f\vee g)(x)=f(x)\vee g(x)\\
%(f*g)(x)=f(x)*g(x)\\
%(f\oplus g)(x)=f(x)\oplus g(x)\\
%(f\rightarrow g)(x)=f(x)\rightarrow g(x)\\
%(f)^{\bot}(x)=(f(x))^{\bot}
%\end{align*}
%$r\in \bar{n}$ is defined as the constant function on $(X,\tau)$ whose value is always $r$.\\
\end{definition}
\begin{lemma}\label{5.25}\cite{YM}
Let $(X,\tau)$ be an $\bar{n}$-fuzzy topological space. Then $Cont(X,\tau)$ is closed under the operations $(\wedge,\vee,*,\oplus,\rightarrow,^{\bot},0,\dfrac{1}{n-1},\dfrac{2}{n-1},\dfrac{3}{n-1},.....,1)$.
\end{lemma}
\begin{definition}[Zero dimensional]\label{5.26}\cite{YM}
 An $\bar n$-fuzzy topological space $(X,\tau)$ is called \textbf{zero-dimensional} \index{zero-dimensional} if and only if $Cont(X,\tau)$ forms an open basis of $(X,\tau)$.
\end{definition}
\begin{definition}[$\bar{n}$-fuzzy Boolean space]\label{5.27}\cite{YM}
 An $\bar n$-fuzzy topological space $(X,\tau)$ is called an \textbf{$\bar{n}$-fuzzy Boolean space}\index{$\bar{n}$-fuzzy Boolean space} iff $(X,\tau)$ is zero dimensional (c.f. Definition \ref{5.26}), compact (c.f. Definition \ref{5.16}) and Kolmogorov (c.f. Definition \ref{5.14}).
\end{definition}
In \cite{YM}, the category of $\bar{n}$-fuzzy Boolean spaces and continuous functions has been defined and called $\mathbf{FBS_n}$ \index{category!$\mathbf{FBS_n}$}.
 
We shall now consider interrelation among the categories $\mathbf{FBSys_n}$, $\mathbf{FBS_n}$ and {\textbf{\L}}$\mathbf{_n^c}$\textbf{-Alg}. 

 \subsection{Functors}
 
 \subsection*{Functor $Ext_B$ from $\mathbf{FBSys_n}$ to $\mathbf{FBS_n}$}\index{functor!$Ext_B$} \index{functor!$Ext_B$}
 \begin{definition}\label{5.28}
 Let $(X,\models , A)$ be an $\bar{n}$-fuzzy Boolean system. For each $a\in A$, its \textbf{extent$_B$}\index{extent$_B$} in $(X,\models ,A)$ is a mapping $ext_B(a)$ from $X$ to $\bar n$ given by $ext_B(a)(x)=gr(x\models a)$.
 
In the set $ext_B(A)=\{ext_B(a):a\in A\}$, the operations $(\vee,\wedge,*,\oplus,\rightarrow,^{\bot},0,\dfrac{1}{n-1},\dfrac{2}{n-1},...,
\dfrac{n-2}{n-1},1)$ are defined pointwise. Thus $ext_B$ is a homomorphism from $A$ to $\bar{n}^X$.
 \end{definition}
\begin{lemma}\label{5.29}
 $(X,ext_B(A))$ is compact.
 \end{lemma}
\begin{proof}
Let us assume that $\mathbf{1}=\bigvee_{i\in I} ext_B(a_i)$ for some $a_i\in A$, where $\mathbf{1}$ is the constant function defined on $X$, whose value is always 1.
Now, we have, 
$$\mathbf{1}=T_1\circ \mathbf{1}= T_1\circ \bigvee_{i\in I} ext_B(a_i)=\bigvee_{i\in I} T_1\circ ext_B(a_i)=\bigvee_{i\in I} ext_B(T_1(a_i)).$$
So, $$\mathbf{0}=(\bigvee_{i\in I} ext_B(T_1(a_i)))^{\bot}=\bigwedge_{i\in I}ext_B((T_1(a_i))^{\bot}).$$ Therefore, $$0=\mathbf{0}(x)=(\bigwedge_{i\in I}ext_B((T_1(a_i))^{\bot}))(x),\ \text{for all}\ x.$$ Hence for a fixed $x$, $$0= (\bigwedge_{i\in I}ext_B((T_1(a_i))^{\bot}))(x)= \bigwedge_{i\in I}ext_B((T_1(a_i))^{\bot})(x)= \bigwedge_{i\in I}gr(x\models (T_1(a_i))^{\bot}).$$ Let there be a homomorphism $v:A\longrightarrow \bar{n}(\subseteq A)$ such that $v((T_1(a_i))^{\bot})=1$ for all $i\in I$. Then $$gr(x\models v((T_1(a_i))^{\bot}))=v((T_1(a_i))^{\bot})=1,$$ as $v((T_1(a_i))^{\bot})\in \bar{n}$ for all $i\in I$. Therefore, $$\bigwedge_{i\in I} gr(x\models v((T_1(a_i))^{\bot}))=1,$$ which is a contradiction. So, there is no homomorphism $v:A\longrightarrow \bar{n}$ such that $v((T_1(a_i))^{\bot})=1$ for all $i\in I$. By Proposition \ref{5.9}.3 there is no prime $\bar n$-filter of A which contains $\{ (T_1(a_i))^{\bot}:i\in I\}$. So by Proposition \ref{5.9}.2 $\{(T_1(a_i))^{\bot}:i\in I\}$ does not have f.i.p. with respect to $*$ and so there is a finite subset $\{ i_1,.....,i_m\}$ of $I$ s.t.
$$(T_1(a_{i_1}))^{\bot}*.....*(T_1(a_{i_m}))^{\bot}=0,$$ which yields $$T_1(a_{i_1})\oplus ......\oplus T_1(a_{i_m})=1.$$ Now $T_1(a_{i_k})$ is idempotent and therefore, for any $k\in \{ 1,....,m\}$ we have $$T_1(a_{i_1})\vee .....\vee T_1(a_{i_m})=1.$$ That is, $T_1(a_{i_1}\vee ....\vee a_{i_m})=1$.
Since, $T_1(x)\leq x$, so $a_{i_1}\vee .....\vee a_{i_m}=1$. Therefore $ext_B(a_{i_1}\vee .....\vee a_{i_m})=1$.
\end{proof}
\begin{lemma}\label{5.30}
 $(X,ext_B(A))$ is Kolmogorov.
 \end{lemma}
\begin{proof}
Let $x_1\neq x_2$. Then there is $a\in A$ such that $$gr(x_1\models a)\neq gr(x_2\models a),$$ That is, $$ext_B(a)(x_1)\neq ext(a)(x_2).$$
\end{proof}
\begin{lemma}\label{5.31}
$(X,ext_B(A))$ is zero-dimensional.
\end{lemma}
\begin{proof}
Here we have to show that $Cont(X,ext_B(A))$ forms an open basis of $(X, ext_B(A))$. For $f,\ g\in Cont(X,ext_B(A))$, $$(f\wedge g)(x)=f(x)\wedge g(x).$$ Hence for $a\in A$ if $ext_B(a)\in Cont(X,ext_B(A)),$ then $Cont(X,\tau)$ is zero-dimensional. Let $\tilde{A}$ be an $\bar n$-fuzzy set on $\bar n$. We have,
\begin{align*}
((ext_B(a))^{-1}(\tilde{A} ))(x) & = (\tilde{A} \circ ext_B(a))(x)\\
& = \tilde{A} (ext_B(a)(x))\\
& = \bigvee_{r\in \bar n} (S_{\tilde{A} (r)}\circ T_r)(ext_B(a)(x))\\
& = \bigvee_{r\in \bar n} (S_{\tilde{A} (r)}\circ T_r)(gr(x\models a))\\
& = \bigvee_{r\in \bar n} (S_{\tilde{A} (r)}(T_r(gr(x\models a))))\\
& = \bigvee_{r\in \bar n} S_{\tilde{A} (r)} (gr(x\models T_r(a)))\\
& = \bigvee_{r\in \bar n} gr(x\models S_{\tilde{A} (r)}(T_r(a)))\\
& = gr(x\models \bigvee_{r\in \bar n} S_{\tilde{A} (r)}(T_r(a))).
\end{align*} 
Therefore, $$ext_B(\bigvee_{r\in \bar n} S_{\tilde{A} (r)}(T_r(a)))\in ext_B(A).$$
Hence, $ext_B(a)\in Cont(X,ext_B(A))$.
\end{proof}
From Lemmas \ref{5.29}, \ref{5.30}, \ref{5.31} we get the following theorem.
\begin{theorem}\label{5.32}
Let $(X,\models ,A)$ be an $\bar{n}$-fuzzy Boolean system. Then $(X,ext_B(A))$ is an $\bar{n}$-fuzzy Boolean space. 
\end{theorem}
\begin{definition}\label{5.34}
$\mathbf{Ext_B}$\index{functor!$Ext_B$} is a functor from $\mathbf{FBSys_n}$ to $\mathbf{FBS_n}$ defined as follows.\\
$Ext_B$ acts on an object $(X,\models ',A)$ as $Ext_B(X,\models ',A)=(X,ext_B(A))$ where\\ $ext_B(A)=\{ext_B(a):a\in A\}$ and on a morphism $(f_1,f_2)$ as $Ext_B(f_1,f_2)=f_1$.
\end{definition}
Theorem \ref{5.32} and the fact that ``if $(f_1,f_2)$ is continuous then $f_1$ is $\bar n$-fuzzy continuous" shows that $Ext_B$ \index{functor!$Ext_B$} is a functor.
%The diagram below express the above fact-
%\begin{center}     
%\begin{tabular}{ l | r  } 
%$FBSy_n$ & $FBS_n$\\
%\hline
%{\begin{tikzpicture}[description/.style={fill=white,inner sep=2pt}] 
%    \matrix (m) [matrix of math nodes, row sep=2.5em, column sep=2.5em]
%    {& &(X,\models,A)  \\
%        & &(Y, \models ',B) \\ }; 
%    \path[->,font=\scriptsize] 

        %(m-1-5) edge node[auto,swap] {$j$} (m-1-3)
        %(m-1-5) edge node[auto] {$\psi$} (m-2-3)
%        (m-1-3) edge node[auto] {$(f_1,f_2)$} (m-2-3);
%\end{tikzpicture}}  &  {\begin{tikzpicture}[description/.style={fill=white,inner sep=2pt}] 
%    \matrix (m) [matrix of math nodes, row sep=2.5em, column sep=2.5em]
%    {& &(X,ext(A))  \\
%        & &(Y,ext(B)) \\ }; 
%    \path[->,font=\scriptsize] 

        %(m-1-5) edge node[auto,swap] {$j$} (m-1-3)
        %(m-1-5) edge node[auto] {$\psi$} (m-2-3)
%        (m-1-3) edge node[auto] {$f_1$} (m-2-3);
%\end{tikzpicture}}\\
%\end{tabular}
%\end{center}
\subsection*{Functor $J_B$ from $\mathbf{FBS_n}$ to $\mathbf{FBSys_n}$}\index{functor!$J_B$}
\begin{lemma}\label{5.35} \index{functor!$J_B$}
Let $(X,\tau)$ be an $\bar{n}$-fuzzy Boolean space. Then $(X, \in , Cont(X,\tau))$ is an $\bar n$-fuzzy Boolean system.
\end{lemma}
\begin{proof}
Here $X$ is a set and $Cont(X,\tau)$ is an {\L}$_n^c$-algebra, where the operations of $Cont(X,\tau)$ are defined point wise.
It will be enough to show that\\
1. $gr(x\in t_1*t_2)=max(0,gr(x\in t_1)+gr(x\in t_2)-1)$,\\
2. $gr(x\in t_1^{\bot})=1-gr(x\in t_1)$,\\
3. $gr(x\in r)=r$ for all $r\in \bar n$ and \\
4. $x_1\neq x_2\Rightarrow gr(x_1\in t)\neq gr(x_2\in t)$ for some $t\in Cont(X,\tau)$ holds.

Let us verify the above mentioned points.

1. $gr(x\in t_1*t_2)  =t_1*t_2(x)=t_1(x)*t_2(x)=max(0,gr(x\in t_1)+gr(x\in t_2)-1)$.

2. $gr(x\in t_1^{\bot})  =t_1^{\bot}(x)=1-t_1(x)=1-gr(x\in t_1)$.

3. $gr(x\in r)=r(x)=r$ for all $r\in \bar n$.

4. As $(X,\tau)$ is Kolmogorov and zero-dimensional, for $x_1\neq x_2$, there exists $t\in Cont(X,\tau)$ such that $t(x_1)\neq t(x_2)$.
Hence, we have, $x_1\neq x_2\Rightarrow gr(x_1\in t)\neq gr(x_2\in t)$ for some $t\in Cont(X,\tau)$, which completes the proof.
\end{proof}
\begin{lemma}\label{5.36}
$(f,\_\circ f)$ is continuous provided that $f$ is $\bar{n}$ fuzzy continuous.
\end{lemma}
\begin{proof}
Here $f:X\longrightarrow Y$ is a function and clearly $\_\circ f:Cont(Y,\tau_2)\longrightarrow Cont(X,\tau_1)$ is an {\L}$_n^c$-hom.
It suffices to show that $gr(x\in t_2\circ f)=gr(f(x)\in t_2)$.
Now, $gr(x\in t_2\circ f)  =t_2\circ f(x)=t_2(f(x))=gr(f(x)\in t_2)$.
\end{proof}
\begin{definition}\label{5.37}
$\mathbf{J_B}$\index{functor!$J_B$} is a functor from $\mathbf{FBS_n}$ to $\mathbf{FBSys_n}$ defined thus.\\
$J_B$ acts on an object $(X,\tau)$ as $J_B(X,\tau)=(X, \in ,Cont(X,\tau))$ where $gr(x\in t)=t(x)$ and on a morphism $f$ as $J_B(f)=(f,\_\circ f)$.
\end{definition}
Lemmas \ref{5.35} and \ref{5.36} shows that $J_B$ is a functor.
%The diagram below express the above facts-
%\begin{center}     
%\begin{tabular}{ l | r  } 
%$FBS_n$ & $FBSy_n$\\
%\hline
%{\begin{tikzpicture}[description/.style={fill=white,inner sep=2pt}] 
%    \matrix (m) [matrix of math nodes, row sep=2.5em, column sep=2.5em]
%    {& &(X,\tau_1)  \\
%        & &(Y,\tau_2) \\ }; 
%    \path[->,font=\scriptsize] 

        %(m-1-5) edge node[auto,swap] {$j$} (m-1-3)
        %(m-1-5) edge node[auto] {$\psi$} (m-2-3)
%        (m-1-3) edge node[auto] {$f$} (m-2-3);
%\end{tikzpicture}}  &  {\begin{tikzpicture}[description/.style={fill=white,inner sep=2pt}] 
%    \matrix (m) [matrix of math nodes, row sep=2.5em, column sep=2.5em]
%    {& &(X,\in ,Cont(X,\tau_1))  \\
%        & &(Y,\in ,Cont(Y,\tau_2)) \\ }; 
%    \path[->,font=\scriptsize] 

        %(m-1-5) edge node[auto,swap] {$j$} (m-1-3)
        %(m-1-5) edge node[auto] {$\psi$} (m-2-3)
%        (m-1-3) edge node[auto] {$(f,\_\circ f)$} (m-2-3);
%\end{tikzpicture}}\\
%\end{tabular}
%\end{center}
\subsection*{Functor $\mathbf{Lag}$ from $\mathbf{FBSys_n}$ to ({\textbf{\L}}$\mathbf{_n^c}$\textbf{-Alg})$^{op}$}\index{functor!$Lag$}
\begin{definition}\label{5.38}
$\mathbf{Lag}$ \index{functor!$Lag$} is a functor from $\mathbf{FBSys_n}$ to ({\textbf{\L}}$\mathbf{_n^c}$\textbf{-Alg})$^{op}$ defined to act on an object $(X,\models ,A)$ as $Lag(X,\models ,A)=A$ and on a morphism $(f_1,f_2)$ as $Lag(f_1,f_2)=f_2$. 
\end{definition}
%It is routine to see that $Lag$ is a functor.
%  \\ The diagram below express the above facts-
%  \begin{center}     
%\begin{tabular}{ l | r  } 
%$FBSy_n$ & $(\hcancel{\mathbf{L}}_n^c-Alg)^{op}$\\
%\hline
%{\begin{tikzpicture}[description/.style={fill=white,inner sep=2pt}] 
%    \matrix (m) [matrix of math nodes, row sep=2.5em, column sep=2.5em]
%    {& &(X,\models ',A)  \\
%        & &(Y, \models '',B) \\ }; 
%    \path[->,font=\scriptsize] 

        %(m-1-5) edge node[auto,swap] {$j$} (m-1-3)
        %(m-1-5) edge node[auto] {$\psi$} (m-2-3)
%        (m-1-3) edge node[auto] {$(f_1,f_2)$} (m-2-3);
%\end{tikzpicture}}  &  {\begin{tikzpicture}[description/.style={fill=white,inner sep=2pt}] 
%    \matrix (m) [matrix of math nodes, row sep=2.5em, column sep=2.5em]
%    {& &A  \\
%        & &B \\ }; 
%    \path[->,font=\scriptsize] 

        %(m-1-5) edge node[auto,swap] {$j$} (m-1-3)
        %(m-1-5) edge node[auto] {$\psi$} (m-2-3)
%        (m-1-3) edge node[auto] {$f_2$} (m-2-3);
%\end{tikzpicture}}\\
%\end{tabular}
%\end{center}
\subsection*{Functor $S_B$ from ({\textbf{\L}}$\mathbf{_n^c}$\textbf{-Alg})$^{op}$ to $\mathbf{FBSys_n}$}\index{functor!$S_B$}
\begin{definition}\label{5.39}
Let $A$ be an {\L}$_n^c$-algebra, $Hom(A,\bar{n})=\{ ${\L}$_n^c$ $ hom$ $v:A\longrightarrow \bar{n}\}$.
\end{definition}
\begin{lemma}\label{5.40}
Let $A$ be an {\L}$_n^c$-algebra. Then $(Hom(A,\bar{n}),\models_*,A)$ is an $\bar{n}$-fuzzy Boolean system.
\end{lemma}
\begin{proof}
Clearly $Hom(A,\bar{n})$ is a set, $A$ is a frame. It suffices to show that\\
(i) $gr(v\models a*b)=max(0,gr(v\models_* a)+gr(v\models_* b)-1)$,\\
(ii) $gr(v\models_* a^{\bot})=1-gr(v\models_* a)$,\\
(iii) $gr(v\models_* r)=v(r)=r$ for all $r\in \bar n$ and\\
(iv) $v_1\neq v_2\Rightarrow gr(v_1\models a)\neq gr(v_2\models a)$ for some $a\in A$ holds.

Let us verify the above mentioned points.

(i) $gr(v\models_* a*b) =v(a*b)=v(a)*v(b)=max(0,v(a)+v(b)-1)=max(0,gr(v\models_* a)+gr(v\models_* b)-1)$.

Similarly we can show (ii) and (iii).

(iv) As, $v_1\neq v_2$ so we have, $v_1(a)\neq v_2(a)$ for some $a\in A$.
Hence $v_1\neq v_2\Rightarrow gr(v_1\models a)\neq gr(v_2\models a)$ for some $a\in A$.
\end{proof}
It is easy to see that $(\_\circ f,f)$ is continuous provided $f$ is an {\L}$_n^c$-hom.
\begin{definition}\label{5.41}
$\mathbf{S_B}$ \index{functor!$S_B$} is a functor from ({\textbf{\L}}$\mathbf{_n^c}$\textbf{-Alg})$^{op}$ to $\mathbf{FBSys_n}$ defined as follows.\\
$S_B$ acts on an object $A$ as $S_B(A)=(Hom(A,\bar{n}),\models_*,A)$ and on a morphism $f$ as $S_B(f)=(\_\circ f,f)$, where $gr(v\models_* a)=v(a)$.
\end{definition}
The above fact and Lemma \ref{5.40} shows that $S_B$ is a functor.
%The diagram below express the above fact-
%\begin{center}     
%\begin{tabular}{ l | r  } 
%$(\hcancel{\mathbf{L}}_n^c-Alg)^{op}$ & $FBSy_n$\\
%\hline
%{\begin{tikzpicture}[description/.style={fill=white,inner sep=2pt}] 
%    \matrix (m) [matrix of math nodes, row sep=2.5em, column sep=2.5em]
%    {& & A \\
%        & &B \\ }; 
%    \path[->,font=\scriptsize] 

        %(m-1-5) edge node[auto,swap] {$j$} (m-1-3)
        %(m-1-5) edge node[auto] {$\psi$} (m-2-3)
%        (m-1-3) edge node[auto] {$f$} (m-2-3);
%\end{tikzpicture}}  &  {\begin{tikzpicture}[description/.style={fill=white,inner sep=2pt}] 
%    \matrix (m) [matrix of math nodes, row sep=2.5em, column sep=2.5em]
%    {& &(Hom(A,\bar{n}),\models_* ,A)  \\
%        & &(Hom(B,\bar{n}),\models_* ,B) \\ }; 
%    \path[->,font=\scriptsize] 

        %(m-1-5) edge node[auto,swap] {$j$} (m-1-3)
        %(m-1-5) edge node[auto] {$\psi$} (m-2-3)
%        (m-1-3) edge node[auto] {$(\_\circ f,f)$} (m-2-3);
%\end{tikzpicture}}\\
%\end{tabular}
%\end{center}
%where $gr(v\models_* a)=v(a)$\\
%and $Hom(A,\bar{n})=\{ \hcancel{\mathbf{L}}_n^c$ $hom$ $v:A\longrightarrow \bar{n}\}$
\begin{theorem}\label{5.42}
$Ext_B$ is the right adjoint to the functor $J_B$.
\end{theorem}
\begin{proof}
We will prove the theorem presenting the co-unit of the adjunction.\\
Let us draw the diagram of co-unit.
\begin{center}
\begin{tabular}{ l | r } 
$\mathbf{FBSys_n}$ & $\mathbf{FBS_n}$\\
\hline
 {\begin{tikzpicture}[description/.style={fill=white,inner sep=2pt}] 
    \matrix (m) [matrix of math nodes, row sep=2.5em, column sep=2.5em]
    { J_B(Ext_B(X,\models,A))&&(X,\models,A)  \\
         J_B(Y,\tau ') \\ }; 
    \path[->,font=\scriptsize] 
        (m-1-1) edge node[auto] {$\xi_X$} (m-1-3)
        (m-2-1) edge node[auto] {$J_B(f)(\equiv(f_1,f_1^{-1}))$} (m-1-1)
        (m-2-1) edge node[auto,swap] {$\hat f (\equiv(f_1,f_2))$} (m-1-3)
        %(m-1-1) edge node[auto,swap] {$f$} (m-2-3)
       % (m-1-5) edge node[auto,swap] {$j$} (m-1-3)
       % (m-1-5) edge node[auto] {$\psi$} (m-2-3)
        %(m-1-3) edge node[auto] {$G(\hat{f})$} (m-2-3)
         ;
\end{tikzpicture}} & {\begin{tikzpicture}[description/.style={fill=white,inner sep=2pt}] 
    \matrix (m) [matrix of math nodes, row sep=2.5em, column sep=2.5em]
    { Ext_B(X,\models,A)  \\
         (Y,\tau ') \\ }; 
    \path[->,font=\scriptsize] 
        (m-2-1) edge node[auto,swap] {$f(\equiv f_1)$} (m-1-1)
       
         ;
\end{tikzpicture}} \\ 
\end{tabular}
\end{center}
Recall that $J_B(X,\tau)=(X,\in , Cont(X,\tau))$ and $Ext_B(X,\models,A)=(X,ext_B(A))$.\\
$So,$ $$J_B(Ext_B(X,\models,A))=(X,\in,Cont(X,ext_B(A))).$$
Co-unit is defined by,
\begin{center}
\begin{tikzpicture}[description/.style={fill=white,inner sep=2pt}] 
    \matrix (m) [matrix of math nodes, row sep=2.5em, column sep=2.5em]
    { (X,\in ,Cont(X,ext_B(A)))&&(X,\models,A)  \\
          }; 
    \path[->,font=\scriptsize] 
        (m-1-1) edge node[auto] {$\xi_X$} (m-1-3)
        (m-1-1) edge node[auto,swap] {$(id_X,ext_B^*)$} (m-1-3)
         ;
\end{tikzpicture}
\end{center}
 where $ext_B^*(a)=ext_B(a)$.
\begin{lemma}\label{5.43}
$(id_X,ext_B^*):J_B(Ext_B(X,\models ,A))\longrightarrow (X,\models ,A)$ is a continuous map of $\bar{n}$-fuzzy Boolean system.
\end{lemma}
$\mathit{Proof.}$
To establish the lemma it is enough to show that $$gr(x\in ext_B^*(a))=gr(id_X(x)\models a).$$ That is, $$gr(x\in ext_B(a))=gr(x\models a).$$
Let us define $f$ as follows:
Given $$(f_1,f_2):J_B(Y,\tau ')\longrightarrow (X,\models,A),$$ then $f=f_1$.

Now we will prove that the diagram on the left commutes. In other words,
\begin{align*}
 (f_1,f_2) & =\xi_X\circ J_B(f)\\
& =(id_X,ext_B^*)\circ (f_1,f_1^{-1}) \tag{as $J_B(f)=(f_1,f_1^{-1})$}\label {eq: as $J_B(f)=(f_1,f_1^{-1})$}\\
& =(id_X\circ f_1,f_1^{-1}\circ ext_B^*).
\end{align*}
Clearly $id_X\circ f_1=f_1$.
It is left to show that $f_2=f_1^{-1}\circ ext_B^*$.

As $(id_X,ext_B^*)$ is continuous, $ext_B^*(a)(x)=gr(x\models a)$, i.e., $ext_B^*(a)=a$. Hence $f_1^{-1}ext^*(a)  =f_1^{-1}(a)  =f_2(a)$, as $(f_1,f_2)$ is continuous.
Therefore $$\xi_X(\equiv(id_X,ext_B^*)):J_B(Ext_B(X,\models,A))\longrightarrow (X,\models,A)$$ is the co-unit, consequently $Ext_B$ is the right adjoint to the functor $J_B$.
\end{proof}
Diagram of the unit of the above adjunction is as follows.
\begin{center}
\begin{tabular}{ l | r  } 
 $\mathbf{FBS_n}$ & $\mathbf{FBSys_n}$ \\
\hline
 {\begin{tikzpicture}[description/.style={fill=white,inner sep=2pt}] 
    \matrix (m) [matrix of math nodes, row sep=2.5em, column sep=2.5em]
    {(X,\tau) & &Ext_B(J_B(X, \tau))  \\
        & & Ext_B(Y,\models ,B) \\ }; 
    \path[->,font=\scriptsize] 
        (m-1-1) edge node[auto] {$\eta_X$} (m-1-3)
        (m-1-1) edge node[auto,swap] {$\hat{f}(\equiv f_1)$} (m-2-3)
       % (m-1-5) edge node[auto,swap] {$j$} (m-1-3)
        %(m-1-5) edge node[auto] {$\psi$} (m-2-3)
        (m-1-3) edge node[auto] {$ext_B(f)(\equiv f_1)$} (m-2-3);
\end{tikzpicture}} &  {\begin{tikzpicture}[description/.style={fill=white,inner sep=2pt}] 
    \matrix (m) [matrix of math nodes, row sep=2.5em, column sep=2.5em]
    {& &J_B(X,\tau)  \\
        & &(Y,\models ,B) \\ }; 
    \path[->,font=\scriptsize]

        %(m-1-5) edge node[auto,swap] {$j$} (m-1-3)
        %(m-1-5) edge node[auto] {$\psi$} (m-2-3)
        (m-1-3) edge node[auto] {$f(\equiv(f_1,f_1^{-1}))$} (m-2-3);
\end{tikzpicture}}\\
\end{tabular}
\end{center}
%Hence we have $(\eta_X,\xi_X):J\dashv Ext:FBSy_n\longrightarrow FBS_n$.
\begin{theorem}\label{5.44}
The category $\mathbf{FBSys_n}$ is equivalent to the category $\mathbf{FBS_n}$.
\end{theorem}
\begin{proof}
We have two natural transformations $\xi$, $\eta$ such that for an $\bar{n}$-fuzzy Boolean system $(X,\models ,A)$, $$\xi_X:J_B(Ext_B(X,\models ,A))\longrightarrow (X,\models , A)$$ and for an $\bar{n}$-fuzzy Boolean space $(X,\tau)$,
$$\eta_X:(X,\tau)\longrightarrow Ext_B(J_B(X,\tau)).$$
Now it suffices to show that $\xi$, $\eta$ are natural isomorphisms.
\begin{claim}
$\xi_X$ is an isomorphism.
\end{claim}
\textit{Proof\ of\ the\ Claim.}
We have, $\xi_X:J_B(Ext_B(X,\models ,A))\longrightarrow (X,\models , A)$, i.e., $\xi_X:(X,\in ,ext_B(A))\longrightarrow (X,\models , A)$. Now $\xi_X$ is a natural transformation between two $\bar{n}$-fuzzy Boolean systems. Hence we have to show that $\xi_X$ is a homeomorphism. Here $\xi_X=(id_X,ext_B^*).$
We have to show that $A$ and $ext_B(A)$ are isomorphic {\L}$_n^c$-algebras, i.e., $ext_B^*:A\longrightarrow ext_B(A)$ is an isomorphism.
We already have $ext_B^*$ is an {\L}$_n^c$-homomorphism as $ext_B^*(a)=ext_B(a)$ for all $a\in A$.
Next we will show that $ext_B^*$ is injective.
Let $ext_B(a)=ext_B(b)$ for $a,\ b\in A$. So, for all $x\in X$ $ext_B(a)(x)=ext_B(b)(x)$, i.e., $gr(x\models a)=gr(x\models b)$.
Therefore, for any $r\in \bar{n}$ and $x\in X$, $gr(x\models T_r(a))=gr(x\models T_r(b))$ (follows from Proposition \ref{5.6}.2).
Now using Proposition \ref{5.9}.3, we have for any prime filter $P$ of $A$ and any $r\in \bar{n}$, $T_r(a)\in P$ if and only if $T_r(b)\in P$.
We claim that $T_r(a)=T_r(b)$ for any $r\in \bar{n}$. If possible let $T_r(a)\neq T_r(b)$ for some $r\in \bar{n}$. Take $T_r(a)\nleq T_r(b)$ without loss of generality. Let $F=\{a_1\in A:T_r(a)\leq a_1\}$. As $T_r(a)$ is idempotent, so $F$ is an $\bar{n}$-filter of $A$. Clearly $T_r(b)\notin F.$ Using Proposition \ref{5.9}.1, we have that there is a prime filter $P$ of $A$ such that $F\subset P$ and $T_r(b)\notin P$. Now $F\subset P$, so $T_r(a)\in P$, which is a contradiction.
Therefore, $T_r(a)=T_r(b)$ for any $r\in \bar{n}$. So, $\bigwedge_{r\in\bar{n}}T_r(a)\longleftrightarrow T_r(b)=1$, but $\bigwedge_{r\in\bar{n}}T_r(a)\longleftrightarrow T_r(b)\leq a\longleftrightarrow b$   (Proposition \ref{5.6}.4)
It follows that $a=b$ and therefore, $ext^*$ is injective.

It is easy to see that $ext_B^*$ is onto.
As a consequence, we get that $A$ and $ext_B(A)$ are isomorphic.

Lastly, we have $gr(x\in ext_B(a))=ext_B(a)(x)=gr(x\models a)$.
Hence $\xi_X$ is an isomorphism and consequently $\xi$ is a natural isomorphism.\ \ \openbox

 Now, it is left to show that $\eta$ is a natural isomorphism. We have $$\eta_X:(X,\tau)\longrightarrow (X,ext_B(Cont(X,\tau))).$$ Defining $\eta_X$ as $\eta_X(x)(t)=t(x)$ for $x\in X$ and $t\in Cont(X,\tau)$, it is possible to show that $\eta$ is a natural isomorphism. 
\end{proof}
\begin{theorem}\label{5.45}
$Lag$ is a left adjoint to the functor $S_B$.
\end{theorem}
\begin{proof}
We will prove the theorem presenting the unit of the adjunction.
Let us draw the diagram of unit:
\begin{center}
\begin{tabular}{ l | r  } 
 $\mathbf{FBSys_n}$ & ({\textbf{\L}}$\mathbf{_n^c}$\textbf{-Alg})$^{op}$ \\
\hline
 {\begin{tikzpicture}[description/.style={fill=white,inner sep=2pt}] 
    \matrix (m) [matrix of math nodes, row sep=2.5em, column sep=2.5em]
    {(X,\models ,A) & &S_B(Lag(X, \models ,A))  \\
        & & S_B(B) \\ }; 
    \path[->,font=\scriptsize] 
        (m-1-1) edge node[auto] {$\eta_A$} (m-1-3)
        (m-1-1) edge node[auto,swap] {$f(\equiv(f_1,f_2))$} (m-2-3)
       % (m-1-5) edge node[auto,swap] {$j$} (m-1-3)
        %(m-1-5) edge node[auto] {$\psi$} (m-2-3)
        (m-1-3) edge node[auto] {$S_B(\hat{f})$} (m-2-3);
\end{tikzpicture}} &  {\begin{tikzpicture}[description/.style={fill=white,inner sep=2pt}] 
    \matrix (m) [matrix of math nodes, row sep=2.5em, column sep=2.5em]
    {& &Lag(X,\models ,A)  \\
        & &B \\ }; 
    \path[->,font=\scriptsize]

        %(m-1-5) edge node[auto,swap] {$j$} (m-1-3)
        %(m-1-5) edge node[auto] {$\psi$} (m-2-3)
        (m-1-3) edge node[auto] {$\hat{f}(\equiv f_2)$} (m-2-3);
\end{tikzpicture}}\\
\end{tabular}
\end{center}
Recall that  $S_B(B)=(Hom(B,\bar{n}),\models_*,B)$, where $gr(v\models_* b)=v(b)$.
Hence $$S_B(Lag(X,\models ,A))=(Hom(A,\bar{n}),\models_* ,A).$$ Then unit is defined by, 
\begin{center}
 \begin{tikzpicture}[description/.style={fill=white,inner sep=2pt}] 
    \matrix (m) [matrix of math nodes, row sep=2.5em, column sep=2.5em]
    {(X,\models ,A) & &S_B(Lag(X, \models ,A))  \\
        }; 
    \path[->,font=\scriptsize] 
        (m-1-1) edge node[auto] {$\eta_A$} (m-1-3)
        (m-1-1) edge node[auto,swap] {$(p^*,id_A)$} (m-1-3)
        %(m-1-1) edge node[auto,swap] {} (m-2-3)
       % (m-1-5) edge node[auto,swap]
        ;
\end{tikzpicture}
\end{center}
where,
\begin{align*}
  p^* \colon X &\longrightarrow Hom(A,\bar{n})\\
  x &\longmapsto p_x\colon A\longrightarrow \bar{n}
\end{align*}
such that $p_x(a)=gr(x\models a)$.
\begin{lemma}\label{5.46} For each $x\in X,$ $p_x:A\longrightarrow \bar{n}$ is an {\L}$_n^c$ homomorphism.
\end{lemma}
$\mathit{Proof.}$
Proof is straight froward.

\begin{lemma}\label{5.47} $(p^*,id_A):(X,\models,A)\longrightarrow S_B(Lag(X,\models,A)) $ is a continuous map of $\bar{n}$-fuzzy Boolean system.
\end{lemma}
$\mathit{Proof.}$
It will be enough to show that $$gr(x\models id_A(a))=gr(p^*(x)\models_* a),$$ That is, $$gr(p^*(x)\models_* a)=gr(x\models a).$$
As we have $gr(p^*(x)\models_* a)=p^*(x)(a)$ and also we have the fact that 
$p^*(x)(a)=gr(x\models a)$ hence the lemma is proved. 

Let us define $\hat{f}$ as follows
$(f_1,f_2):(X,\models ,A)\longrightarrow (Hom(B,\bar{n}),\models_*,B)$
then $\hat{f}=f_2$ (as $f_2$ is the {\L}$_n^c$ homomorphism).
Now we have to show that the triangle on the left commute. In other words 
$$(f_1,f_2)=S_B(\hat{f})\circ \eta_A=(-\circ f_2,f_2)\circ (p^*,id_A)=((-\circ f_2)p^*,id_A\circ f_2).$$
Clearly $f_2=id_A\circ f_2.$

It is only left to show $f_1=(-\circ f_2)p^* $, i.e., for $x\in X$, $$f_1(x)=(\_\circ f_2)p^*(x)=p_x\circ f_2.$$
We have for all $b\in B$, $$p_x\circ f_2(b)=p_x(f_2(b))=gr(x\models f_2(b))=gr(f_1(x)\models_* b)=f_1(x)(b).$$
Hence $p_x\circ f_2=f_1(x)$, i.e., $(\_\circ f_2)p^*(x)=f_1(x)$.
So, $(\_\circ f_2)p^*=f_1$.
\end{proof}
Diagram of the co-unit of the above adjunction is as follows.
\begin{center}
\begin{tabular}{ l | r } 
({\textbf{\L}}$\mathbf{_n^c}$\textbf{-Alg})$^{op}$ & $\mathbf{FBSys_n}$\\
\hline
 {\begin{tikzpicture}[description/.style={fill=white,inner sep=2pt}] 
    \matrix (m) [matrix of math nodes, row sep=2.5em, column sep=2.5em]
    { Lag(S_B(A))&& A  \\
         Lag(Y,\models ,B) \\ }; 
    \path[->,font=\scriptsize] 
        (m-1-1) edge node[auto] {$\xi_A$} (m-1-3)
        (m-2-1) edge node[auto] {$Lag(f)(\equiv f')$} (m-1-1)
        (m-2-1) edge node[auto,swap] {$\hat f (\equiv f')$} (m-1-3)
        %(m-1-1) edge node[auto,swap] {$f$} (m-2-3)
       % (m-1-5) edge node[auto,swap] {$j$} (m-1-3)
       % (m-1-5) edge node[auto] {$\psi$} (m-2-3)
        %(m-1-3) edge node[auto] {$G(\hat{f})$} (m-2-3)
         ;
\end{tikzpicture}} & {\begin{tikzpicture}[description/.style={fill=white,inner sep=2pt}] 
    \matrix (m) [matrix of math nodes, row sep=2.5em, column sep=2.5em]
    { S_B(A)  \\
         (Y,\models ,B) \\ }; 
    \path[->,font=\scriptsize] 
        (m-2-1) edge node[auto,swap] {$f(\equiv(\_\circ f',f'))$} (m-1-1)
       
         ;
\end{tikzpicture}} \\ 
\end{tabular}
\end{center}
%Hence we have $(\eta_A,\xi_A):Lag\dashv S:(\hcancel{\mathbf{L}}_n^c-Alg)^{op}\longrightarrow FBSy_n.$
\begin{theorem}\label{5.48}
The category {\textbf{\L}}$\mathbf{_n^c}$\textbf{-Alg} is dually equivalent to the category $\mathbf{FBSys_n}$.
\end{theorem}
\begin{proof}
We have two natural transformations $\xi$, $\eta$ such that for an {\L}$_n^c-algebra$ $A$,
$$\xi_A:Lag(S_B(A))\longrightarrow A$$ and for an $\bar{n}$-fuzzy Boolean system $(X,\models , A)$, $$\eta_A:(X,\models ,A)\longrightarrow S_B(Lag(X,\models ,A)).$$
%Now we have to show that $\xi$ and $\eta$ are natural isomorphisms.\\
It is easy to see that $\xi$ is a natural isomorphism.
\begin{claim}
$\eta_A:(X,\models ,A)\longrightarrow (Hom(A,\bar{n}),\models_*,A)$ is an isomorphism. 
\end{claim}%We have\\
%$\eta_A:(X,\models ,A)\longrightarrow (Hom(A,\bar{n}),\models_*,A)$.\\
\textit{Proof\ of \ the \ Claim.}
$\eta_A$ is a natural transformation between two $\bar{n}$-fuzzy Boolean systems. So we will show that $\eta_A=(p^*,id_A)$ is a homeomorphism.%Here $\eta_A=(p^*,id_A)$ 
We show that $p^*:X\longrightarrow Hom(A,\bar{n})$ is a bijection, i.e., $X$ and $Hom(A,\bar{n})$ are in bijective correspondence.
First we show that $p^*$ is injective. Take $x_1,x_2\in X$ such that $x_1\neq x_2$. 
It follows that  for some $a\in A$, $$gr(x_1\models a)\neq gr(x_2\models a)$$ (using Definition \ref{5.17}), i.e., $p_{x_1}(a)\neq p_{x_2}(a)$ for some $a\in A$. Hence $p^*(x_1)(a)\neq p^*(x_2)(a)$ for some $a\in A$. As a consequence, $p^*$ is injective.
From the construction of $p^*$, it is clear that $p^*$ is onto.
So $X$ and $Hom(A,\bar{n})$ are isomorphic.
%\item Clearly $A$ and $A$ are isomorphic frame.
Lastly we have $$gr(x\models a)=p_x(a)=gr(p^*(x)\models_* a).$$
Hence, $\eta_A$ is an isomorphism.\ \ \openbox 

Therefore, ({\textbf{\L}}$\mathbf{_n^c}$\textbf{-Alg})$^{op}$ is equivalent to the category $\mathbf{FBSys_n}$. As a consequence, {\textbf{\L}}$\mathbf{_n^c}$\textbf{-Alg} is dually equivalent to the category $\mathbf{FBSys_n}$.
\end{proof}
\begin{theorem}\label{5.49}
$Ext_B\circ S_B$ is a right adjoint to the functor $Lag\circ J_B$.
\end{theorem}
\begin{proof}
Follows from the combination of the adjoint situations in Theorems \ref{5.42}, \ref{5.45}.
\end{proof}
\begin{theorem}\label{5.50}
{\textbf{\L}}$\mathbf{_n^c}$\textbf{-Alg} is dually equivalent to the category $\mathbf{FBS_n}$.
\end{theorem}
\begin{proof}
Follows as a combination of the equivalences of Theorems \ref{5.44}, \ref{5.48}.
\end{proof}
The obtained functorial relationships can be depicted as follows:
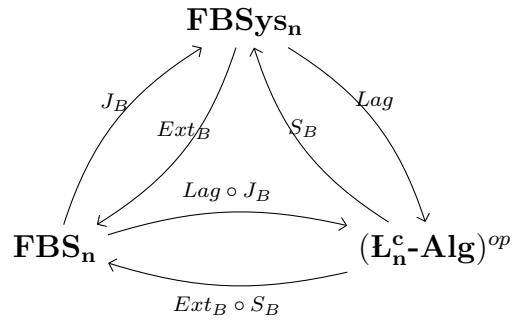
\begin{center}
\begin{tikzpicture}
\node (C) at (0,3) {$\mathbf{FBSys_n}$};
\node (A) at (-2.5,0) {$\mathbf{FBS_n}$};
\node (B) at (2.5,0) {({\textbf{\L}}$\mathbf{_n^c}$\textbf{-Alg})$^{op}$};
%\node at (0,0) {\rotatebox{270}{$\Rightarrow$}};
\path[->,font=\scriptsize ,>=angle 90]
(A) edge [bend left=15] node[above] {$Lag\circ J_B$} (B);
\path[<-,font=\scriptsize ,>=angle 90]
(A) edge [bend right=15] node[below] {$Ext_B \circ S_B$} (B);
\path[->,font=\scriptsize ,>=angle 90]
(A) edge [bend left=20] node[above] {$J_B$} (C);
\path[<-,font=\scriptsize ,>=angle 90]
(A) edge [bend right=20] node[above] {$Ext_B$} (C);
\path[->,font=\scriptsize ,>=angle 90]
(C) edge [bend left=20] node[above] {$Lag$} (B);
\path[<-,font=\scriptsize ,>=angle 90]
(C) edge [bend right=20] node[above] {$S_B$} (B);
\end{tikzpicture}
\end{center}
The duality which is obtained now between $\mathbf{FBS_n}$ and {\textbf{\L}}$\mathbf{_n^c}$\textbf{-Alg} can be proved to be same as the duality established in \cite{YM}.
\chapter{Fuzzy Geometric Logic with Graded Consequence}
\section{Introduction}\blfootnote{The results of this chapter are in {\bf \cite{MP3} M.K. Chakraborty and P. Jana: \textit{Fuzzy topological space via fuzzy logic with graded consequence}, International Journal of Approximate Reasoning, Elsevier (accepted).}}
In \cite{SV}, the relationship of geometric logic and topological system was mentioned. Vickers mentioned that the satisfaction relation of the topological system $(X,\models , A)$ matches logic of finite observations or geometric logic. The aim of this chapter is to introduce fuzzy geometric logic with graded consequence. Three closely related notions namely, graded fuzzy topological system, fuzzy topology with graded inclusion and graded frame shall also be introduced. As a matter of fact all these notions are interwoven. Mathematically speaking, these may be called generalizations of corresponding classical notions to the many-valued context. But many-valuedness has taken place in two layers as will be apparent in the sequel.
In the first-order logic (semantic) consequence relation is defined in terms of satisfaction. When the satisfaction relation is fuzzy the corresponding consequence relation may be either crisp or fuzzy. In the first case, we get many valued logic \cite{LB} and in the second case logic of graded consequence \cite{MK}. 

It is known that if $(X,\models , A)$ is a topological system then $\models$ matches the logic of finite observations or geometric logic \cite{SV, PJT}. In this chapter we will introduce fuzzy geometric logic and fuzzy geometric logic with graded consequence and connect them with appropriate fuzzy topological systems. Soundness of both the logics shall also be investigated.

\section{Fuzzy Geometric Logic}\index{fuzzy geometric logic}
The \textbf{alphabet} \index{alphabet} of the language $\mathscr{L}$ of fuzzy geometric logic comprises of the connectives $\wedge$, $\bigvee$, the existential quantifier $\exists$, parentheses $)$ and $($, as well as:
\begin{itemize}
\item countably many individual constants $c_1,\ c_2,\dots$;
\item denumerably many individual variables $x_1,\ x_2,\dots$;
\item propositional constants $\top$, $\bot$;
\item for each $i>0$, countably many $i$-place predicate symbols $p^i_j$'s, including at least the $2$-place symbol ``$=$" for identity;
\item for each $i>0$, countably many $i$-place function symbols $f^i_j$'s.
\end{itemize}
\begin{definition}[Term]\label{term}\index{term}
\textbf{Terms} are recursively defined in the usual way as follows:
\begin{itemize}
\item every constant symbol $c_i$ is a term;
\item every variable $x_i$ is a term;
\item if $f_j$ is an $i$-place function symbol, and $t_1,t_2,\dots,t_i$ are terms then\\ $f^i_jt_1t_2\dots t_i$ is a term;
\item nothing else is a term.
\end{itemize}
\end{definition}
\begin{definition}[Geometric formula]\label{wff}\index{geometric formula}
\textbf{Geometric formulae} are recursively defined as follows:
\begin{itemize}
\item $\top$, $\bot$ are geometric formulae;
\item if $p_j$ is an $i$-place predicate symbol, and $t_1,t_2,\dots,t_i$ are terms then\\ $p^i_jt_1t_2\dots t_i$ is a geometric formula;
\item if $t_i$, $t_j$ are terms then $(t_i=t_j)$ is a geometric formula;
\item if $\phi$ and $\psi$ are geometric formulae then $(\phi\wedge\psi)$ is a geometric formula;
%\item if $\phi$ and $\psi$ are geometric formulae then $(\phi\vee\psi)$ is a geometric formula;
\item if $\phi_i$'s ($i\in I$) are geometric formulae then $\bigvee\{\phi_i\}_{i\in I}$ is a geometric formula, when $I=\{1,2\}$ then the above formula is written as $\phi_1\vee \phi_2$;
\item if $\phi$ is a geometric formula and $x_i$ is a variable then $\exists x_i\phi$ is a geometric formula;
\item nothing else is a geometric formula.
\end{itemize}
\end{definition}
\textbf{Note:} Definitions of bound or free variables, and complexity of formulae are the standard ones as used in first order logic.
\begin{definition}\label{t}
$t[t'/ x]$ is the result of replacing $t'$ for every occurrence of $x$ in $t$, defined recursively as follows:
\begin{itemize}
\item if $t$ is $c_i$ or $x_i$ other than $x$ then $t[t'/x]$ is $t$;
\item if $t$ is $x$ then $t[t'/x]$ is $t'$;
\item if $t$ is $f^i_jt_1t_2\dots t_i$ then $t[t'/x]$ is $f^i_jt_1[t'/x]t_2[t'/x]\dots t_i[t'/x]$.
\end{itemize}
\end{definition}
\begin{definition}\label{phi}
$\phi[t/x]$ is the result of replacing $t$ for every free occurrence of $x$ in $\phi$, defined recursively as follows:
\begin{itemize}
\item if $\phi$ is $p^i_jt_1t_2\dots t_i$ then $\phi[t/x]$ is $p^i_jt_1[t/x]t_2[t/x]\dots t_i[t/x]$;
\item if $\phi$ is $(t_i=t_j)$ then $\phi[t/x]$ is $(t_i[t/x]=t_j[t/x])$;
\item if $\phi$ is $\phi_1\wedge\phi_2$ then $\phi[t/x]$ is $\phi_1[t/x]\wedge\phi_2[t/x]$;
\item if $\phi$ is $\phi_1\vee\phi_2$ then $\phi[t/x]$ is $\phi_1[t/x]\vee\phi_2[t/x]$;
\item if $\phi$ is $\bigvee\{\phi_i\}_{i\in I}$ then $\phi[t/x]$ is $\bigvee\{\phi_i[t/x]\}_{i\in I}$;
\item if $\phi$ is $\top$ or $\bot$ then $\phi[t/x]$ is $\top$ or $\bot$ respectively;
\item if $\phi$ is $\exists x_i\psi$ ($x_i$ is other than $x$) then $\phi[t/x]$ is $\exists x_i\psi[t/x]$;
\item if $\phi$ is $\exists x\psi$ then $\phi[t/x]$ is $\phi$ (i.e. $\exists x\psi$).
\end{itemize}
\end{definition}
\begin{definition}[Interpretation]\label{interpretation}\index{interpretation}
An \textbf{interpretation} $I$ consists of
\begin{itemize}
\item a set $D$, called the domain of interpretation;
\item an element $I(c_i)\in D$ for each constant $c_i$;
\item a function $I(f^i_j):D^i\longrightarrow D$ for each function symbol $f^i_j$;
\item a fuzzy relation $I(p^i_j):D^i\longrightarrow [0,1]$ for each predicate symbol $p^i_j$, i.e., it is a fuzzy subset of $D^i$.
\end{itemize}
\end{definition}
\begin{definition}[Graded Satisfiability]\label{sat}\index{graded satisfiability}
Let $s$ be a sequence over $D$. Let $s=(s_1,s_2,\dots)$ be a sequence over $D$ where $s_1,s_2,\dots$ are all elements of $D$. Let $d$ be an element of $D$. Then $s(d/x_i)$ is the result of replacing $i$'th coordinate of $s$ by $d$, i.e., $s(d/x_i)=(s_1,s_2,\dots,s_{i-1},d,s_{i+1},\dots)$. Let $t$ be a term. Then $s$ assigns an element $s(t)$ of $D$ as follows:
\begin{itemize}
\item if $t$ is the constant symbol $c_i$ then $s(c_i)=I(c_i)$;
\item if $t$ is the variable $x_i$ then $s(x_i)=s_i$;
\item if $t$ is the function symbol $f^i_jt_1t_2\dots t_i$ then $$s(f^i_jt_1t_2\dots t_i)=I(f^i_j)(s(t_1),s(t_2),\dots,s(t_i)).$$
\end{itemize}
Now we define grade of satisfiability of $\phi$ by $s$ written as $gr(s\ \emph{sat}\ \phi)$, where $\phi$ is a geometric formula, as follows:
\begin{itemize}
\item $gr(s\ \emph{sat}\ p^i_jt_1t_2\dots t_i)=I(p^i_j)(s(t_1),s(t_2),\dots,s(t_i))$;
\item $gr(s\ \emph{sat}\ \top)=1$;
\item $gr(s\ \emph{sat}\ \bot)=0$;
\item $gr(s\ \emph{sat}\ t_i=t_j)$ $=\begin{cases}
        1 & \emph{if $s(t_i)=s(t_j)$} \\
        0 & \emph{otherwise};
    \end{cases}$
\item $gr(s\ \emph{sat}\ \phi_1\wedge\phi_2)=gr(s\ \emph{sat}\ \phi_1)\wedge gr(s\ \emph{sat}\ \phi_2)$;
\item $gr(s\ \emph{sat}\ \phi_1\vee\phi_2)=gr(s\ \emph{sat}\ \phi_1)\vee gr(s\ \emph{sat}\ \phi_2)$;
\item $gr(s\ \emph{sat}\ \bigvee\{\phi_i\}_{i\in I})=sup\{gr(s\ \emph{sat}\ \phi_i)\mid i\in I\}$;
\item $gr(s\ \emph{sat}\ \exists x_i\phi)=sup\{gr(s(d/ x_i)\ \emph{sat}\ \phi)\mid d\in D\}$.
\end{itemize}
\end{definition}
 In $[0,1]$, we have used $\wedge$ and $\vee$ to mean min and max respectively - a convention that will be followed throughout.
The expression $\phi\vdash\psi$, where $\phi$, $\psi$ are wffs, is called a sequent\index{sequent}. We now define satisfiability of a sequent.
\begin{definition}\label{valid}
1. $s$ \emph{sat} $\phi\vdash\psi$ iff $gr(s\ \emph{sat}\ \phi)\leq gr(s\ \emph{sat}\ \psi)$.\\
2. $\phi\vdash\psi$ is valid in $I$ iff $s$ \emph{sat} $\phi\vdash\psi$ for all $s$ in the domain of $I$.\\
3. $\phi\vdash\psi$ is universally valid iff it is valid in all interpretations.
\end{definition}
\begin{theorem}[Local Determination A]\label{LDI}\index{local determination}
Let $I$ be an interpretation and $t$ be a term. If the sequences $s$ and $s'$ are such that they agree on the variables occurring in the term $t$ then $s(t)=s'(t)$.
\end{theorem}
\begin{proof}
When $t$ is a constant $c$ then $s(c)=I(c)=s'(c)$.

When $t$ is a variable $x_i$ then $s(x_i)=s_i=s_i'=s'(x_i)$.

Let the theorem hold for the terms $t_1,t_2,...,t_i$. Then the theorem holds for the term $f^i_jt_1t_2....t_i$ as the following holds.
\begin{align*} 
s(f^i_jt_1t_2...t_i) & =I(f^i_j)(s(t_1),s(t_2),...s(t_i))\\& =I(f^i_j)(s'(t_1),s'(t_2),...s'(t_i))\\
& =s'(f^i_jt_1t_2...t_i).
\end{align*}
This completes the proof.
\end{proof}
\begin{theorem}[Local Determination B]\label{LDII}\index{local determination}
Let $I$ be an interpretation and $\phi$ be a geometric formula. If the sequences $s$ and $s'$ are such that they agree on the free variables occurring in $\phi$ then $gr(s\ \emph{sat}\ \phi)=gr(s'\ \emph{sat}\ \phi)$.
\end{theorem}
\begin{proof}
We will prove it using mathematical induction on the complexity of $\phi(=n)$.\\
Base Case: When $n=0$ then $\phi$ is any one of the following formulae $$p^i_jt_1t_2...t_i,\ (t_i=t_j),\ \top ,\ \bot.$$
When $\phi$ is $p^i_jt_1t_2...t_i$:
\begin{align*}
gr(s\ \text{sat}\ \phi) & = gr(s\ \text{sat}\ p^i_jt_1t_2...t_i)\\
& = I(p^i_j)(s(t_1),s(t_2),...,s(t_i))\\
& = I(p^i_j)(s'(t_1),s'(t_2),...,s'(t_i)),\ \ \ \ \text{by Theorem \ref{LDI}}\\
& = gr(s'\ \text{sat}\ p^i_jt_1t_2...t_i)\\ 
& = gr(s'\ \text{sat}\ \phi).
\end{align*}
When $\phi$ is $(t_i=t_j)$:
\begin{align*}
gr(s\ \text{sat}\ \phi) & = gr(s\ \text{sat}\ (t_i=t_j))\\
&  =\begin{cases}
        1 & \text{if}\ \ s(t_i)=s(t_j) \\
        0 & \text{otherwise}
    \end{cases}\\
&  =\begin{cases}
        1 & \text{if}\ \ s'(t_i)=s'(t_j) \\
        0 & \text{otherwise}
    \end{cases}\\
& = gr(s'\ \text{sat}\ (t_i=t_j))\\
& = gr(s'\ \text{sat}\ \phi).
\end{align*}
When $\phi$ is $\top$:
\begin{align*}
gr(s\ \text{sat}\ \phi) & = gr(s\ \text{sat}\ \top)\\
& = 1\\
& = gr(s'\ \text{sat}\ \top)\\
& = gr(s'\ \text{sat}\ \phi).
\end{align*}
When $\phi$ is $\bot$:
\begin{align*}
gr(s\ \text{sat}\ \phi) & = gr(s\ \text{sat}\ \bot)\\
& = 0\\
& = gr(s'\ \text{sat}\ \bot)\\
& = gr(s'\ \text{sat}\ \phi).
\end{align*}
Let the theorem hold for $n\leq m$.

Let $\phi$ be a geometric formula of complexity $m+1$ then, $\phi$ is either $\phi_1\wedge\phi_2$ or $\phi_1\vee\phi_2$ or $\exists x_i\psi$, where $\phi_1,\ \phi_2,\ \psi$ are of complexity less than equal to $m$.\\
When $\phi$ is $\phi_1\wedge\phi_2$:
\begin{align*}
gr(s\ \text{sat}\ \phi) & = gr(s\ \text{sat}\ \phi_1\wedge\phi_2)\\
& = gr(s\ \text{sat}\ \phi_1)\wedge gr(s\ \text{sat}\ \phi_2)\\
& = gr(s'\ \text{sat}\ \phi_1)\wedge gr(s'\ \text{sat}\ \phi_2)\\
& = gr(s'\ \text{sat}\ \phi_1\wedge\phi_2)\\
& = gr(s'\ \text{sat}\ \phi).
\end{align*}
When $\phi$ is $\bigvee\{ \phi_i\}_{i\in I}$:
\begin{align*}
gr(s\ \text{sat}\ \phi) & = gr(s\ \text{sat}\ \bigvee\{ \phi_i\}_{i\in I})\\
& = sup\{gr(s\ \text{sat}\ \phi_i)\mid i\in I\}\\
& = sup\{gr(s'\ \text{sat}\ \phi_i)\mid i\in I\}\\
& = gr(s'\ \text{sat}\ \bigvee\{ \phi_i\}_{i\in I})\\
& = gr(s'\ \text{sat}\ \phi).
\end{align*}
When $\phi$ is $\exists x_i\psi$:
\begin{align*}
gr(s\ \text{sat}\ \phi) & = gr(s\ \text{sat}\ \exists x_i\psi)\\
& = sup\{gr(s(d/x_i)\ \text{sat}\ \psi)\mid d\in D\}\\
& = sup\{gr(s'(d/x_i)\ \text{sat}\ \psi)\mid d\in D\}\\
& = gr(s'\ \text{sat}\ \exists x_i\psi)\\
& = gr(s'\ \text{sat}\ \phi).
\end{align*}
This completes the proof.
\end{proof}
\begin{theorem}[Substitution Theorem]\label{substitution}\index{substitution theorem}
Let $D$ be the domain of interpretation $I$:
\begin{enumerate}
\item Let $t$ and $t'$ be terms then for every sequence $s$ over $D$, $$s(t[t'/x_k])=s(s(t')/x_k)(t).$$
\item Let $\phi$ be a geometric formula and $t$ be a term. For every sequence $s$ over $D$, $gr(s\ \emph{sat}\ \phi[t/x_k])=gr(s(s(t)/x_k)\ \emph{sat}\ \phi)$.
\end{enumerate}
\end{theorem}
\begin{proof}
Let us proceed in the following way.

1. \ When $t$ is a constant $c$ then, $$s(t[t'/x_k])=s(c[t'/x_k])=s(c)=I(c)$$ and $$s(s(t')/x_k)(t)=s(s(t')/x_k)(c)=s(c)=I(c).$$ Hence in this case $s(t[t'/x_k])=s(s(t')/x_k)(t)$.

When $t$ is a variable $x_i(\neq x_k)$ then,
$$s(t[t'/x_k])=s(x_i[t'/x_k])=s(x_i)=s_i$$ and
$$s(s(t')/x_k)(t)=s(s(t')/x_k)(x_i)=s(x_i)=s_i.$$
Hence in this case $s(t[t'/x_k])=s(s(t')/x_k)(t)$.

When $t$ is a variable $x_k$ then,
$s(t[t'/x_k])=s(x_k[t'/x_k])=s(t')$ and
$$s(s(t')/x_k)(t)=s(s(t')/x_k)(x_k)=s(t')$$
Hence in this case $s(t[t'/x_k])=s(s(t')/x_k)(t)$.

Let the theorem hold for the terms $t_1,t_2,...,t_i$. We will show that the theorem holds for the term $f^i_jt_1t_2....t_i$ in the following way.
\begin{eqnarray*}
s((f^i_jt_1t_2...t_i)[t'/x_k]) & = & s(f^i_jt_1[t'/x_k]t_2[t'/x_k]...t_i[t'/x_k])\\
 &=& I(f^i_j)(s(t_1[t'/x_k]),s(t_2[t'/x_k]),...,s(t_i[t'/x_k]))\\
 &=& I(f^i_j)(s(s(t')/x_k)(t_1), s(s(t')/x_k)(t_2),...,s(s(t')/x_k)(t_i))\\
 &=& s(s(t')/x_k)(f^i_jt_1t_2...t_i).\end{eqnarray*}

2. \ We will prove it using mathematical induction on the complexity of $\phi(=n)$.

\underline{Base Case:} When $n=0$ then $\phi$ is any one of the following formulae: 
$$p^i_jt_1t_2...t_i,\ (t_i=t_j),\ \top, \ \bot.$$
When $\phi$ is $p^i_jt_1t_2...t_i$:
\begin{align*}
gr(s\ \text{sat}\ \phi[t/x_k])& = gr(s\ \text{sat}\ (p^i_jt_1t_2...t_i)[t/x_k])\\
& = gr(s\ \text{sat}\  p^i_jt_1[t/x_k]t_2[t/x_k]...t_i[t/x_k])\\
& = I(p^i_j)(s(t_1[t/x_k]),s(t_2[t/x_k]),...,s(t_i[t/x_k]))\\
& = I(p^i_j)(s(s(t)/x_k)(t_1),s(s(t)/x_k)(t_2),...,s(s(t)/x_k)(t_i))\\
& = gr(s(s(t)/x_k)\ \text{sat}\ p^i_jt_1t_2...t_i)\\ 
& = gr(s(s(t)/x_k)\ \text{sat}\ \phi).
\end{align*} 
When $\phi$ is $(t_i=t_j)$:
\begin{align*}
gr(s\ \text{sat}\ \phi[t/x_k]) & = gr(s\ \text{sat}\ (t_i=t_j)[t/x_k])\\
& = gr(s\ \text{sat}\ t_i[t/x_k]=t_j[t/x_k])\\
&  =\begin{cases}
        1 & \text{if}\ \ s(t_i[t/x_k])=s(t_j[t/x_k]) \\
        0 & \text{otherwise}
    \end{cases}\\
&  =\begin{cases}
        1 & \text{if}\ \ s(s(t)/x_k)(t_i)=s(s(t)/x_k)(t_j) \\
        0 & \text{otherwise}
    \end{cases}\\
& =  gr(s(s(t)/x_k)\ \text{sat}\ (t_i=t_j)) \\ 
& = gr(s(s(t)/x_k)\ \text{sat}\ \phi).    
\end{align*}
When $\phi$ is $\top$:
\begin{align*}
gr(s\ \text{sat}\ \phi[t/x_k]) & = gr(s\ \text{sat}\ \top[t/x_k])\\
& = gr(s\ \text{sat}\top)\\
& = 1\\
& = gr(s(s(t)/x_k)\ \text{sat}\ \top)\\
& = gr(s(s(t)/x_k)\ \text{sat}\ \phi).
\end{align*}
When $\phi$ is $\bot$:
\begin{align*}
gr(s\ \text{sat}\ \phi[t/x_k]) & = gr(s\ \text{sat}\ \bot[t/x_k])\\
& = gr(s\ \text{sat}\bot)\\
& = 0\\
& = gr(s(s(t)/x_k)\ \text{sat}\ \bot)\\
& = gr(s(s(t)/x_k)\ \text{sat}\ \phi).
\end{align*}
Let the theorem hold for $n\leq m$.
Let $\phi$ be a geometric formula of complexity $m+1$. Then, $\phi$ is either $\phi_1\wedge\phi_2$ or $\phi_1\vee\phi_2$ or $\exists x_i\psi$, where $\phi_1,\ \phi_2,\ \psi$ are of complexity less than equal to $m$.\\
When $\phi$ is $\phi_1\wedge\phi_2$:
\begin{align*}
gr(s\ \text{sat}\ \phi[t/x_k]) & = gr(s\ \text{sat}\ (\phi_1\wedge\phi_2)[t/x_k])\\
& = gr(s\ \text{sat}\ \phi_1[t/x_k])\wedge gr(s\ \text{sat}\ \phi_2[t/x_k])\\
& = gr(s(s(t)/x_k)\ \text{sat}\ \phi_1)\wedge gr(s(s(t)/x_k)\ \text{sat}\ \phi_2)\\
& = gr(s(s(t)/x_k)\ \text{sat}\ \phi_1\wedge\phi_2)\\
& = gr(s(s(t)/x_k)\ \text{sat}\ \phi).
\end{align*}
When $\phi$ is $\bigvee\{\phi_i\}_{i\in I}$:
\begin{align*}
gr(s\ \text{sat}\ \phi[t/x_k]) & = gr(s\ \text{sat}\ (\bigvee\{\phi_i\}_{i\in I})[t/x_k])\\
& = sup\{gr(s\ \text{sat}\ \phi_i[t/x_k])\}_{i\in I}\\
& =sup\{ gr(s(s(t)/x_k)\ \text{sat}\ \phi_i)\}_{i\in I}\\
& = gr(s(s(t)/x_k)\ \text{sat}\ \bigvee\{\phi_i\}_{i\in I})\\
& = gr(s(s(t)/x_k)\ \text{sat}\ \phi).
\end{align*}
When $\phi$ is $\exists x_i\psi$ then following two cases arises:\\
\underline{Case I: } $x_i=x_k$
\begin{align*}
gr(s\ \text{sat}\ \phi[t/x_k]) & = gr(s\ \text{sat}\ (\exists x_i\psi)[t/x_i])\\
& = gr(s\ \text{sat}\ \exists x_i\psi)\\
& = sup\{gr(s(d/x_i)\ \text{sat}\ \psi)\mid d\in D\}\\
& = sup\{gr(s(s(t)/x_k)(d/x_i)\ \text{sat}\ \psi)\mid d\in D\}\\
& = gr(s(s(t)/x_k)\ \text{sat}\ \exists x_i\psi)\\
& = gr(s(s(t)/x_k)\ \text{sat}\ \phi).
\end{align*}
\underline{Case II:} $x_i\neq x_k$
\begin{align*}
gr(s\ \text{sat}\ \phi[t/x_k]) & = gr(s\ \text{sat}\ (\exists x_i\psi)[t/x_k])\\
& = gr(s\ \text{sat}\ \exists x_i\psi[t/x_k])\\
& = sup\{gr(s(d/x_i)\ \text{sat}\ \psi[t/x_k])\mid d\in D\}\\
& = sup\{gr(s(d/x_i)(s(d/x_i)(t)/x_k)\ \text{sat}\ \psi)\mid d\in D\}\\
& = sup\{gr(s(s(t)/x_k)(d/x_i)\ \text{sat}\ \psi)\mid d\in D\}\\
& = gr(s(s(t)/x_k)\ \text{sat}\ \exists x_i\psi)\\
& = gr(s(s(t)/x_k)\ \text{sat}\ \phi).
\end{align*}
This completes the proof.
\end{proof}
\subsection{Rules of Inference}\index{rules of inference}
In this subsection the rules of inference for fuzzy geometric logic are given. A rule of inference for fuzzy geometric logic is of the form $\AxiomC{$\mathscr{S}_1,\ \mathscr{S}_2,\dots ,\ \mathscr{S}_i$}
\UnaryInfC{$\mathscr{S}$}
\DisplayProof$, where each of the $\mathscr{S}_1,\ \mathscr{S}_2,\dots,\ \mathscr{S}_i$ and $\mathscr{S}$ is a sequent. The sequents $\mathscr{S}_1,\ \mathscr{S}_2,\dots,\mathscr{S}_i$ are known as premises and the sequent $\mathscr{S}$ is called the conclusion. It should be noted that for a rule of inference the set of premises can be empty also.

The rules of inference for fuzzy geometric logic are as follows.
\begin{enumerate}
\item 
$\phi\vdash\phi$,
\item
\AxiomC{$\phi\vdash \psi$}
\AxiomC{$\psi\vdash\chi$}
\BinaryInfC{$\phi\vdash\chi$}
\DisplayProof ,
\item
(i) $\phi\vdash\top$,\ \ \ \  
(ii) $\phi\wedge\psi\vdash\phi$,\ \ \ \ 
(iii) $\phi\wedge\psi\vdash\psi$,\ \ \ \ 
(iv)\AxiomC{$\phi\vdash\psi$}
\AxiomC{$\phi\vdash\chi$}
\BinaryInfC{$\phi\vdash\psi\wedge\chi$}
\DisplayProof ,
\item 
(i) $\phi\vdash\bigvee S$ ($\phi\in S$), for any set $S$,
(ii)\AxiomC{$\phi\vdash\psi$}
\AxiomC{all $\phi\in S$}
\BinaryInfC{$\bigvee S\vdash \psi$}
\DisplayProof , for any $S$,
\item
$\phi\wedge\bigvee S\vdash\bigvee\{\phi\wedge\psi\mid\psi\in S\}$,
\item
$\top\vdash (x=x)$,
\item
$((x_1,\dots,x_n)=(y_1,\dots,y_n))\wedge\phi\vdash\phi[(y_1,\dots,y_n)\mid(x_1,\dots,x_n)]$,
\item
(i)\AxiomC{$\phi\vdash\psi[x\mid y]$}
\UnaryInfC{$\phi\vdash\exists y\psi$}
\DisplayProof ,
\hspace{24pt}
(ii)\AxiomC{$\exists y\phi\vdash\psi$}
\UnaryInfC{$\phi[x\mid y]\vdash\psi$}
\DisplayProof ,
\item
$\phi\wedge (\exists y)\psi\vdash(\exists y)(\phi\wedge\psi)$.
\end{enumerate}
\subsection{Soundness}\index{soundness}
The soundness of a rule means that if all the premises are valid in an interpretation $I$ then the conclusion must also valid in the same interpretation $I$. Satisfaction relation being many-valued, the validity of a sequent has a meaning different from that in the classical geometric logic. In this subsection we will show the soundness of the above rules of inference.
\begin{theorem}
The rules of inference for fuzzy geometric logic are universally valid.
\end{theorem}
\begin{proof}
Let us proceed in the following way.

1. \  $gr(s\ \text{sat}\ \phi)=gr(s\ \text{sat}\ \phi)$, for any $s$.
Hence $\phi\vdash\phi$ is valid.

2.\  Given $\phi\vdash \psi$ and $\psi\vdash\chi$ are valid.
So for any $s$, $$gr(s\ \text{sat}\ \phi)\leq gr(s\ \text{sat}\ \psi) \  \textrm{and }\ gr(s\ \text{sat}\ \psi)\leq gr(s\ \text{sat}\ \chi).$$ 

Therefore $gr(s\ \text{sat}\ \phi)\leq gr(s\ \text{sat}\ \chi)$, for any $s$. Hence $\phi\vdash\chi$ is valid when $\phi\vdash\psi$ and $\psi\vdash\chi$  are valid.

3.\  (i) $gr(s\ \text{sat}\ \phi)\leq 1=gr(s\ \text{sat}\ \top)$, for any $s$.
Hence $\phi\vdash \top$ is valid.

(ii) $gr(s\ \text{sat}\ \phi\wedge\psi)=gr(s\ \text{sat}\ \phi)\wedge gr(s\ \text{sat}\ \psi)\leq gr(s\ \text{sat}\ \phi)$, for any $s$.
Hence $\phi\wedge\psi\vdash\phi$ is valid.

(iii) $gr(s\ \text{sat}\ \phi\wedge\psi)=gr(s\ \text{sat}\ \phi)\wedge gr(s\ \text{sat}\ \psi)\leq gr(s\ \text{sat}\ \psi)$, for any $s$.
Hence $\phi\wedge\psi\vdash\psi$ is valid.

(iv) Given $\phi\vdash\psi$ and $\phi\vdash\chi$ are valid. So $gr(s\ \text{sat}\ \phi)\leq gr(s\ \text{sat}\ \psi)$ and $gr(s\ \text{sat}\ \phi)\leq gr(s\ \text{sat}\ \chi)$ for any $s$. Therefore for any $s$, $$gr(s\ \text{sat}\ \phi)\leq gr(s\ \text{sat}\ \psi)\wedge gr(s\ \text{sat}\ \chi)=gr(s\ \text{sat}\ \psi\wedge\chi).$$  Hence $\phi\vdash \psi\wedge \chi$ is valid when $\phi\vdash\psi$ and $\phi\vdash\chi$ are valid.

4. \  (i) $gr(s\ \text{sat}\ \phi)\leq gr(s\ \text{sat}\ \bigvee S(\phi\in S))$ for any $s$. Hence $\phi\vdash \bigvee S (\phi\in S)$ is valid.

(ii) Given $\phi\vdash \psi$ is valid for all $\phi\in S$.
So $gr(s\ \text{sat}\ \phi)\leq gr(s\ \text{sat}\ \psi)$ for all $\phi\in S$ and any $s$.
Therefore $sup_{\phi\in S}\{gr(s\ \text{sat}\ \phi)\}\leq gr(s\ \text{sat}\ \psi)$ for any $s$. Hence $gr(s\ \text{sat}\ \bigvee S)\leq gr(s\ \text{sat}\ \psi)$ for any $s$. So, $\bigvee S\vdash\psi$ is valid when $\phi\vdash \psi$ is valid for all $\phi\in S$.

5.\ Let us show the soundness of the 5th rule.
\begin{align*}
gr(s\ \text{sat}\ \phi\wedge\bigvee S) & = gr(s\ \text{sat}\ \phi)\wedge gr(s\ \text{sat}\ \bigvee S),\ \ \text{for any}\ s\\
& = gr(s\ \text{sat}\ \phi)\wedge sup\{gr(s\ \text{sat}\ \psi)\}_{\psi\in S},\ \ \text{for any}\ s\\
& = sup\{gr(s\ \text{sat}\ \phi)\wedge gr(s\ \text{sat}\ \psi)\}_{\psi\in S},\ \ \text{for any}\ s\\
& = sup\{gr(s\ \text{sat}\ \phi\wedge\psi)\mid\psi\in S\},\ \ \text{for any}\ s.
\end{align*}
Hence $\phi\wedge\bigvee S\vdash sup\{\phi\wedge\psi\mid\psi\in S\}$ is valid.

6.\ $gr(s\ \text{sat}\ \top)=1=gr(s\ \text{sat}\ x=x)$, for any $s$.
Hence $\top\vdash x=x$ is valid.

7.\ $gr(s\ \text{sat}\ ((x_1,\dots,x_n)=(y_1,\dots,y_n))\wedge\phi)$
$=gr(s\ \text{sat}\ ((x_1,\dots,x_n)=(y_1,\dots,y_n)))\wedge gr(s\ \text{sat}\ \phi)$.
Now, $$gr(s\ \text{sat}\ \phi[(y_1,\dots,y_n)/(x_1,\dots,x_n)])
=gr(s(s((y_1,\dots,y_n))/(x_1,\dots,x_n))\ \text{sat}\ \phi).$$
When $s((y_1,\dots,y_n))=s((x_1,\dots,x_n))$,
$$gr(s(s((y_1,\dots,y_n))/(x_1,\dots,x_n))\ \text{sat}\ \phi)=gr(s\ \text{sat}\ \phi).$$
Hence for any $s$, $$gr(s\ \text{sat}\ ((x_1,\dots,x_n)=(y_1,\dots,y_n))\wedge\phi)
\leq gr(s\ \text{sat}\ \phi[(y_1,\dots,y_n)/(x_1,\dots,x_n)]).$$
So, $((x_1,\dots,x_n)=(y_1,\dots,y_n))\wedge\phi\vdash\phi[(y_1,\dots,y_n)/(x_1,\dots,x_n)]$ is valid.

8.\  (i) $\phi\vdash \psi[x\mid y]$ is valid so, $gr(s\ \text{sat}\ \phi)\leq gr(s\ \text{sat}\ \psi[x\mid y])$, for any $s$. Using Theorem \ref{substitution}(2) $gr(s\ \text{sat}\ \phi)\leq gr(s(s(x)/y)\ \text{sat}\ \psi)$, for any $s$, which implies that for any $s$, $$gr(s\ \text{sat}\ \phi)\leq sup\{gr(s(d/y)\ \text{sat}\ \psi)\mid d\in D\}.$$ So, $gr(s\ \text{sat}\ \phi)\leq gr(s\ \text{sat}\ \exists y\psi)$ and hence $\phi\vdash\exists y\psi$ is valid.\\
(ii) $\exists y\phi\vdash \psi$ is valid if and only if $gr(s\ \text{sat}\ \exists y\phi)\leq gr(s\ \text{sat}\ \psi)$, for any $s$. Hence for any $s$, $$sup\{gr(s(d/y)\ \text{sat}\ \phi)\mid d\in D\}\leq gr(s\ \text{sat}\ \psi).$$ So, $gr(s(s(x)/y)\ \text{sat}\ \phi)\leq gr(s\ \text{sat}\ \psi)$, for any $s$, using Theorem \ref{substitution}(2). Therefore $gr(s\ \text{sat}\ \phi[x/y])\leq gr(s\ \text{sat}\ \psi)$, for any $s$ and hence $\phi[x/y]\vdash\psi$ is valid provided $\exists y\phi\vdash\psi$ is valid.

9. \ For any $s$,
\begin{eqnarray*}
 gr(s\ \text{sat}\ \phi\wedge (\exists y)\psi)
&=&gr(s\ \text{sat}\ \phi)\wedge gr(s\ \text{sat}\ \exists y\psi)\\\
&=&gr(s\ \text{sat}\ \phi)\wedge sup\{gr(s(d/y)\ \text{sat}\ \psi)\}_{d\in D}\\
&=&sup\{gr(s\ \text{sat}\ \phi)\wedge gr(s(d/y)\ \text{sat}\ \psi)\}_{d\in D}\\
&\leq& sup\{gr(s(d/y)\ \text{sat}\ \phi)\wedge gr(s(d/y)\ \text{sat}\ \psi)\}_{d\in D}\\
&=&sup\{gr(s\ \text{sat}\ \phi\wedge\psi)\}_{d\in D}\\
&=&gr(s\ \text{sat}\ (\exists y)\phi\wedge\psi).
\end{eqnarray*}
Hence $\phi\wedge(\exists y)\psi\vdash(\exists y)(\phi\wedge\psi)$ is valid.
This finishes the proof.
\end{proof}
Fuzzy geometric logic is sound in the sense that in every interpretation whenever the premise set of sequents is valid the conclusion sequent is also valid.
\section{Fuzzy Topological System via Fuzzy Geometric Logic}
Let us consider the triplet $(X,\models,A)$ where $X$ is the non empty set of assignments $s$, $A$ is the set of all geometric formulae and $\models$ defined as $gr(s\models \phi)=gr(s\ \text{sat}\ \phi)$.
\begin{theorem}\label{fts}
(i) $gr(s\models \phi\wedge\psi)=gr(s\models \phi)\wedge gr(s\models\psi)$.
\\(ii) $gr(s\models\bigvee\{\phi_i\}_{i\in I})=sup_{i\in I}\{gr(x\models \phi_i)\}$.
\end{theorem}
\begin{proof}
(i) $gr(s\models \phi\wedge\psi)=gr(s\ \text{sat}\ \phi\wedge\psi)=gr(s\ \text{sat}\ \phi)\wedge gr(s\ \text{sat}\ \psi)=gr(s\models \phi)\wedge gr(s\models \psi)$.\\
(ii) We have, $gr(s\models\bigvee\{\phi_i\}_{i\in I})=gr(s\ \text{sat}\ \bigvee\{\phi_i\}_{i\in I})=sup_{i\in I}\{gr(s\ \text{sat}\ \phi_i)\}=sup_{i\in I}\{gr(s\models\phi_i)\}$.
\end{proof} 
\begin{definition}\label{equiv}
$\phi\approx\psi$ iff $gr(s\models \phi)=gr(s\models \psi)$ for any $s\in X$ and $\phi, \ \psi\in A$.
\end{definition}
The above defined ``$\approx$" is an equivalence relation. Thus we get $A/_{\approx}$. The operations $\wedge$, $\vee$, $\bigvee$ can be lifted in the following way so that they will become independent of choice: $$[a]\wedge [b]=^{def}[a\wedge b],\ [a]\vee [b]=^{def}[a\vee b],\ \bigvee\{[a_i]\}_i=^{def}[\bigvee\{a_i\}_i].$$

\begin{theorem}\label{ftopsys}
$(X,\models ',A/_{\approx})$ is a fuzzy topological system, where $\models '$ is defined by $gr(s\models '[\phi])=gr(s\models \phi)$.
\end{theorem}
\begin{proof}
$X$ is a non empty set of assignments $s$.
Let us first prove that $A/_{\approx}$ is a frame in the following way. Here  for any $s$, we define $[\phi]\leq [\psi]$ as follows: $$[\phi]\leq [\psi]\hspace{3mm}  \text{iff}\ \ gr(s\models\phi)\leq gr(s\models \psi)$$ i.e., $\phi\vdash\psi$ is valid.
%\begin{align*}
%[\phi]\leq [\psi]\ \ & \text{iff}\ \ gr(s\models\phi)\leq gr(s\models \psi)\ \ \ \text{for any}\ s\\ 
%& \text{iff}\ \  gr(s\ \text{sat}\ \phi)\leq gr(s\ \text{sat}\ \psi)\ \ \ \text{for any}\ s\\
%& \text{iff}\ \  s\ \text{sat}\ \phi\vdash\psi\ \ \ \text{for any}\ s\\
%& \text{iff}\ \  \phi\vdash\psi \ \ \ \text{is valid}.
%\end{align*}
Now in fuzzy geometric logic $\phi\vdash\phi$ is valid and if $\phi\vdash\psi$ and $\psi\vdash\chi$ are valid then $\phi\vdash\chi$ is valid. Thus $\leq$ is reflexive and transitive. If $[\phi]\leq [\psi]$ and $[\psi]\leq [\phi]$ then $gr(s\models \phi)\leq gr(s\models \psi)$ and $gr(s\models \psi)\leq gr(s\models \phi)$ for any $s$. Therefore $gr(s\models \phi)=gr(s\models \psi)$ for any $s$. So $\phi\approx\psi$. Consequently $[\phi]=[\psi]$. Hence $A/_{\approx}$ is a poset. 

Now if $\phi,\ \psi\in A$ then $\phi\wedge\psi\in A$. So $[\phi],\ [\psi]\in A/_{\approx}$ and $[\phi\wedge\psi]\in A/_{\approx}$, i.e., $[\phi]\wedge [\psi]\in A/_{\approx}$. Similarly arbitrary join exists in $A/_{\approx}$. Also $$[\phi]\wedge\bigvee\{[\psi_i]\}_{i\in I}=[\phi]\wedge [\bigvee\{\psi_i\}_{i\in I}]=[\phi\wedge\bigvee\{\psi_i\}_{i\in I}].$$ Now we have $\phi\wedge\bigvee\{\psi_i\}_{i\in I}\vdash \bigvee_{i\in I}\{\phi\wedge\psi_i\}$ is valid. Hence for any $s$, $$gr(s\ \text{sat}\ \phi\wedge\bigvee\{\psi_i\}_{i\in I})\leq gr(s\ \text{sat}\ \bigvee\{\phi\wedge\psi_i\}_{i\in I}).$$ As $\bigvee_{i\in I}\{\phi\wedge\psi_i\}\vdash \phi\wedge\bigvee\{\psi_i\}_{i\in I}$ is derivable, so for any $s$, $$gr(s\ \text{sat}\ \bigvee\{\phi\wedge\psi_i\}_{i\in I})\leq gr(s\ \text{sat}\ \phi\wedge\bigvee\{\psi_i\}_{i\in I}).$$ Therefore for any $s$, $$gr(s\ \text{sat}\ \phi\wedge\bigvee\{\psi_i\}_{i\in I})=gr(s\ \text{sat}\ \bigvee\{\phi\wedge\psi_i\}_{i\in I}).$$ So, $$[\phi\wedge\bigvee\{\psi_i\}_{i\in I}]=[\bigvee\{\phi\wedge\psi_i\}_{i\in I}].$$ Hence $$[\phi]\wedge\bigvee\{[\psi_i]\}_{i\in I}=[\bigvee\{\phi\wedge\psi_i\}_{i\in I}]=\bigvee\{[\phi\wedge\psi_i]\}_{i\in I}=\bigvee\{([\phi]\wedge [\psi_i])\}_{i\in I}.$$ Therefore $A/_{\approx}$ is a frame.

Now it is left to show that (a) $gr(s\models ' [\phi]\wedge [\psi])=gr(s\models '[\phi])\wedge gr(s\models ' [\psi])$ and (b) $gr(s\models '\bigvee\{[\phi_i]\}_{i\in I})=sup_{i\in I}\{gr(s\models '[\phi_i])\}$.

Proof of the above follows easily using Theorem \ref{fts}. Hence $(X,\models ',A/_{\approx})$ is a fuzzy topological system.
\end{proof}
\begin{proposition}\label{spec}
In the fuzzy topological system $(X,\models ',A/_{\approx})$, defined as above, for all $s\in X$, $(gr(s\models '[\phi])=gr(s\models '[\psi]))\ \text{implies}\ ([\phi]=[\psi])$.
\end{proposition}
\begin{proof}
As $gr(s\models '[\phi])=gr(s\models '[\psi])$, for any $s$, we have $gr(s\models \phi)=gr(s\models \psi)$, for any $s$. Hence $\phi\approx\psi$ and consequently $[\phi]=[\psi]$.
\end{proof}
\section{Fuzzy Topology via Fuzzy Geometric Logic}
We first construct the fuzzy topological system $(X,\models ',A/_{\approx})$ from fuzzy geometric logic. Then $(X,ext(A/_{\approx}))$ is constructed as follows:
\\ $ext(A/_{\approx})=\{ext([\phi])\}_{[\phi]\in A/_{\approx}}$ where
$ext([\phi]):X\longrightarrow [0,1]$ is such that, for each $[\phi]\in A/_{\approx}$, $$ext([\phi])(s)=gr(s\models '[\phi])=gr(s\models \phi).$$
It can be shown that $ext(A/_{\approx})$ forms a fuzzy topology on $X$ as follows.
Let $ext([\phi]),\ ext([\psi])\in ext(A/_{\approx})$. Then we have the following. 
\begin{align*}
(ext([\phi ])\cap ext([\psi ]))(s) & = (ext([\phi ]))(s)\wedge (ext([\psi ]))(s)\\
& = gr(s\models ' [\phi ])\wedge gr(s\models ' [\psi ])\\
& = gr(s\models \phi )\wedge gr(s\models \psi )\\
& = gr(s\models \phi\wedge \psi )\\
& = gr(s\models '[\phi \wedge \psi ])\\
& = (ext([\phi \wedge \psi ]))(s).
\end{align*}
Hence $$ext([\phi])\cap ext([\psi])=ext([\phi\wedge \psi])\in ext(A/_{\approx}).$$

Similarly it can be shown that $ext(A/_{\approx})$ is closed under arbitrary union. Hence $(X,ext(A/_{\approx}))$ is a fuzzy topological space obtained via fuzzy geometric logic.\\
\textbf{Note:} In Chapter 2, the way to construct a fuzzy topological space from a given fuzzy topological system is explained. Here we obtain the space $(X,ext(A/_{\approx}))$ from the system $(X,\models',A/_{\approx})$ following the same method as in Lemma \ref{3.9_1}.
\section{Fuzzy Geometric Logic with Graded Consequence}\index{fuzzy geometric logic!with graded consequence}
Here the alphabet, terms, formulae, interpretation and satisfiability of a formula are the same as in fuzzy geometric logic. The difference lies in the definition of satisfiability of a sequent. This notion is also graded now.
\begin{definition}\label{validg}
1. $gr(s\ \emph{sat}\ \phi\vdash\psi)$=$gr(s\ \emph{sat}\ \phi)\rightarrow gr(s\ \emph{sat}\ \psi)$, where $\rightarrow\colon [0,1]\times [0,1]\longrightarrow [0,1]$ is the G$\ddot{o}$del arrow \index{G$\ddot{o}$del arrow} defined as follows:\  For $a,b\in [0,1]$
\begin{center}
$a\rightarrow b$ $=\begin{cases}
        1 & \emph{if $a\leq b$} \\
        b & \emph{if $a>b$}.
    \end{cases}$
\end{center}    
 
2. $gr(\phi\vdash\psi)=inf_s\{ gr(s\ \emph{sat}\ \phi\vdash \psi)\}$.
\end{definition}
Before proceeding on to the rules of inference let us enlist below some properties of G$\ddot{o}$del arrow \cite{KY} that would be used in the sequel. 
\subsection{Properties of G$\ddot{o}$del arrow}\label{pga}
In this subsection some required properties of G$\ddot{o}$del arrow are listed with verifications.

1.\  $a\rightarrow a=1$, for any $a\in [0,1]$.
\begin{proof}
Follows from the definition.
\end{proof}
2.\  $(a\rightarrow b)\wedge (b\rightarrow c)\leq (a\rightarrow c)$, for any $a,\ b,\ c\in [0,1]$.
\begin{proof}
Let us consider the following cases:\\
\underline{Case 1:  $a=b=c$}\\
In this case $a\rightarrow b=1$, $b\rightarrow c=1$ and $a\rightarrow c=1$.
Hence $(a\rightarrow b)\wedge (b\rightarrow c)=1=(a\rightarrow c)$.\\
\underline{Case 2: $a<b<c$}\\
In this case $a\rightarrow b=1$, $b\rightarrow c=1$ and $a\rightarrow c=1$.
Hence $(a\rightarrow b)\wedge (b\rightarrow c)=1=(a\rightarrow c)$.\\
\underline{Case 3: $b<a<c$}\\
In this case $a\rightarrow b=b$, $b\rightarrow c=1$ and $a\rightarrow c=1$.
Hence $(a\rightarrow b)\wedge (b\rightarrow c)=b<1=(a\rightarrow c)$.\\
\underline{Case 4: $a<c<b$}\\
In this case $a\rightarrow b=1$, $b\rightarrow c=c$ and $a\rightarrow c=1$.
Hence $(a\rightarrow b)\wedge (b\rightarrow c)=c<1=(a\rightarrow c)$.\\
\underline{Case 5: $c<a<b$}\\
In this case $a\rightarrow b=1$, $b\rightarrow c=c$ and $a\rightarrow c=c$.
Hence $(a\rightarrow b)\wedge (b\rightarrow c)=c=(a\rightarrow c)$.\\
\underline{Case 6: $b<c<a$}\\
In this case $a\rightarrow b=b$, $b\rightarrow c=1$ and $a\rightarrow c=c$.
Hence $(a\rightarrow b)\wedge (b\rightarrow c)=b<c=(a\rightarrow c)$.\\
\underline{Case 7: $c<b<a$}\\
In this case $a\rightarrow b=b$, $b\rightarrow c=c$ and $a\rightarrow c=c$.
Hence $(a\rightarrow b)\wedge (b\rightarrow c)=c=(a\rightarrow c)$.
\end{proof}
3.\ $a\leq b$ implies $(a\rightarrow x)\geq (b\rightarrow x)$, for any $a,\ b,\ x\in [0,1]$.
\begin{proof}
Follows by routine check.
\end{proof}
4.\  $a\leq b$ implies $(x\rightarrow a)\leq (x\rightarrow b)$, for any $a,\ b,\ x\in [0,1]$.
\begin{proof}
Follows by routine check.
\end{proof}
5.\ $(a\rightarrow b)\wedge (a\rightarrow c)=a\rightarrow (b\wedge c)$, for any $a,\ b,\ c\in [0,1]$.
\begin{proof}
Follows by routine check.
\end{proof}
6.\  $inf\{(a_i\rightarrow b)\}_i=sup\{a_{i}\}_i\rightarrow b$, for any $a_i,\ b\in [0,1]$.
\begin{proof}
$sup\{a_i\}_{i}\rightarrow b=\begin{cases}
        1 & \text{if}\ \ sup\{a_i\}_i\leq b, \\
        b & \text{if}\ \ sup\{a_i\}_1>b.
    \end{cases}$\\
Now for $sup\{a_i\}_i\leq b$ we have $a_i\leq sup\{a_i\}_i\leq b$.
Hence for this case $(a_i\rightarrow b)=1$, for each $i$ and consequently $inf\{a_i\rightarrow b\}_i=1$.

If $sup\{a_i\}_i>b$ then there exist atleast one $a_i$ such that $a_i>b$ and rest will be either less than $b$ or equal to $b$. Now for the case $a_i>b$, $a_i\rightarrow b=b$ and for all other cases $a_i\rightarrow b=1$. As $b\leq 1$, $inf\{a_i\rightarrow b\}_i=b$.
\end{proof}
7.\  $a\leq b$ iff $a\rightarrow b=1$.
\begin{proof}
Follows from the definition of G$\ddot{o}$del arrow.
\end{proof}
8.\ $a\wedge (a\rightarrow b)\leq b$.
\begin{proof}
Follows by routine check.
\end{proof}

\begin{theorem}\label{5.1g}
Graded sequents satisfy the following properties
\begin{enumerate}
\item 
$gr(\phi\vdash\phi)=1$;
\item
$gr(\phi\vdash\psi)\wedge gr(\psi\vdash\chi)\leq gr(\phi\vdash\chi)$;
\item
(i) $gr(\phi\vdash\top)=1$;\hspace{29pt}  
(ii) $gr(\phi\wedge\psi\vdash\phi)=1$;\\
(iii) $gr(\phi\wedge\psi\vdash\psi)=1$;\hspace{4pt}
(iv) $gr(\phi\vdash\psi)\wedge gr(\phi\vdash\chi)=gr(\phi\vdash\psi\wedge\chi)$;
\item 
(i) $gr(\phi\vdash\bigvee S)=1$ if $\phi\in S$;\\
(ii) $inf_{\phi\in S} \{gr(\phi\vdash\psi)\}\leq gr(\bigvee S\vdash\psi)$;
\item
$gr(\phi\wedge\bigvee S\vdash\bigvee\{\phi\wedge\psi\mid\psi\in S\})=1$;
\item
$gr(\top\vdash (x=x))=1$;
\item
$gr(((x_1,\dots,x_n)=(y_1,\dots,y_n))\wedge\phi\vdash\phi[(y_1,\dots,y_n)/(x_1,\dots,x_n)])$=1;
\item
(i) $gr(\phi\vdash\psi[x/ y])\leq gr(\phi\vdash\exists y\psi)$;\hspace{4pt}
(ii) $gr(\exists y\phi\vdash\psi)\leq gr(\phi[x/ y]\vdash\psi)$;
\item
$gr(\phi\wedge (\exists y)\psi\vdash(\exists y)(\phi\wedge\psi))=1$.
\end{enumerate}
\end{theorem}
\begin{proof}
Let us proceed in the following way.

1.\  $gr(s\ \text{sat}\ \phi\vdash\phi) = gr(s\ \text{sat}\ \phi)\rightarrow gr(s\ \text{sat}\ \phi) =1$ [by Property \ref{pga}(1)].

2. We have,
\begin{multline*}
gr(s\ \text{sat} \phi\vdash \psi)\wedge gr(s\ \text{sat}\ \psi\vdash\chi)\\
=( gr(s\ \text{sat}\ \phi)\rightarrow gr(s\ \text{sat}\ \psi))\wedge (gr(s\ \text{sat}\ \psi)\rightarrow gr(s\ \text{sat}\ \chi))\\
\leq  gr(s\ \text{sat}\ \phi)\rightarrow gr(s\ \text{sat}\ \chi) = gr(s\ \text{sat}\ \phi\vdash\chi).
 \end{multline*}
 
3.\ (i) $gr(s\ \text{sat}\ \phi\vdash \top)=gr(s\ \text{sat}\ \phi)\rightarrow gr(s\ \text{sat}\ \top)=gr(s\ \text{sat}\ \phi)\rightarrow 1=1$, as $gr(s\ \text{sat}\ \phi)\leq 1$.\\
(ii) $gr(s\ \text{sat}\ \phi\wedge\psi\vdash \phi)=gr(s\ \text{sat}\ \phi\wedge\psi)\rightarrow gr(s\ \text{sat}\ \phi)=(gr(s\ \text{sat}\ \phi)\wedge gr(s\ \text{sat}\ \psi))\rightarrow gr(s\ \text{sat}\ \phi)=1$, as $gr(s\ \text{sat}\ \phi)\wedge gr(s\ \text{sat}\ \psi)\leq gr(s\ \text{sat}\ \phi)$.\\
(iii) $gr(s\ \text{sat}\ \phi\wedge\psi\vdash \psi)=gr(s\ \text{sat}\ \phi\wedge\psi)\rightarrow gr(s\ \text{sat}\ \psi)=(gr(s\ \text{sat}\ \phi)\wedge gr(s\ \text{sat}\ \psi))\rightarrow gr(s\ \text{sat}\ \psi)=1$, as $gr(s\ \text{sat}\ \phi)\wedge gr(s\ \text{sat}\ \psi)\leq gr(s\ \text{sat}\ \psi)$.\\
(iv) $gr(s\ \text{sat}\ \phi\vdash\psi)\wedge gr(s\ \text{sat} \phi\vdash \chi)\\
=(gr(s\ \text{sat}\ \phi)\rightarrow gr(s\ \text{sat}\ \psi))\wedge (gr(s\ \text{sat}\ \phi)\rightarrow gr(s\ \text{sat}\ \chi))\\
=gr(s\ \text{sat}\ \phi)\rightarrow gr(s\ \text{sat}\ \psi\wedge\chi)\ [\text{by\ Property \ref{pga}(5)}]\\
=gr(s\ \text{sat}\ \phi\vdash\psi\wedge\chi)$.

4.\ (i) $gr(s\ \text{sat}\ \phi\vdash \bigvee S \ (\phi\in S))\\
=gr(s\ \text{sat}\ \phi)\rightarrow gr(s\ \text{sat}\ \bigvee S \ (\phi\in S ))\\
=gr(s\ \text{sat}\ \phi)\rightarrow sup\{gr(s\ \text{sat}\ \phi)\}_{\phi\in S}\\
=1$ [as $gr(s\ \text{sat}\ \phi)\leq sup \{gr(s\ \text{sat}\ \phi)\}_{\phi\in S}$].\\
(ii) $inf\{gr(s\ \text{sat}\ \phi\vdash\psi)\}_{\phi\in S}\\
=inf\{gr(s\ \text{sat}\ \phi)\rightarrow gr(s\ \text{sat}\ \psi)\}_{\phi\in S}\\
=sup\{gr(s\ \text{sat}\ \phi)\}_{\phi\in S}\rightarrow gr(s\ \text{sat}\ \psi)\ [\text{by\ Property \ref{pga}(6)}]\\
=gr(s\ \text{sat}\ \bigvee S\ (\phi\in S))\rightarrow gr(s\ \text{sat}\ \psi)\\
=gr(s\ \text{sat}\ \bigvee S\vdash \psi\ (\phi\in S))$.

5.\ $gr(s\ \text{sat}\ \phi\wedge\bigvee S\vdash sup\{\phi\wedge\psi\mid\psi\in S\})\\
=gr(s\ \text{sat}\ \phi\wedge\bigvee S)\rightarrow gr(s\ \text{sat}\ \bigvee\{\phi\wedge\psi\mid\psi\in S\})\\
=(gr(s\ \text{sat}\ \phi)\wedge gr(s\ \text{sat}\ \bigvee S))\rightarrow sup_{\psi\in S }\{gr(s\ \text{sat}\ \phi\wedge\psi)\}\\
%=(gr(s\ \text{sat}\ \phi)\wedge gr(s\ \text{sat}\ \bigvee S))\rightarrow %sup_{\psi\in S}\{gr(s\ \text{sat}\ \phi) \wedge gr(s\ \text{sat}\ \psi)\}\\
%=(gr(s\ \text{sat}\ \phi)\wedge gr(s\ \text{sat}\ \bigvee S))\rightarrow %(gr(s\ \text{sat}\ \phi) \wedge sup_{\psi\in S}\{gr(s\ \text{sat}\ \psi)\})\\
=(gr(s\ \text{sat}\ \phi)\wedge gr(s\ \text{sat}\ \bigvee S))\rightarrow (gr(s\ \text{sat}\ \phi)\wedge gr(s\ \text{sat}\ \bigvee S))\\
=1$ [by Property \ref{pga}(1)].

6.\ $gr(s\ \text{sat}\ \top\vdash x=x)=gr(s\ \text{sat}\ \top)\rightarrow gr(s\ \text{sat}\ x=x)=1$.

7.\ $gr(s\ \text{sat}\ ((x_1,x_2,\dots,x_n)=(y_1,y_2,\dots,y_n))\wedge\phi)$\\
$=gr(s\ \text{sat}\ ((x_1,x_2,\dots,x_n)=(y_1,y_2,\dots,y_n)))\wedge gr(s\ \text{sat}\ \phi)$.\\
Now $gr(s\ \text{sat}\ \phi[(y_1,y_2,\dots,y_n)/(x_1,x_2,\dots,x_n)])$\\
$=gr(s(s((y_1,y_2,\dots,y_n))/(x_1,x_2,\dots,x_n))\ \text{sat}\ \phi)$.\\
When $s((y_1,y_2,\dots,y_n))=s((x_1,x_2,\dots,x_n))$ \\
then $gr(s(s((y_1,y_2,\dots,y_n))/(x_1,x_2,\dots,x_n))\ \text{sat}\ \phi)=gr(s\ \text{sat}\ \phi)$.\\
So, $gr(s\ \text{sat}\ ((x_1,x_2,\dots,x_n)=(y_1,y_2,\dots,y_n))\wedge\phi)$\\
$\leq gr(s\ \text{sat}\ \phi[(y_1,y_2,\dots,y_n)/(x_1,x_2,\dots,x_n)])$, for any $s$.\\
Hence $gr(s\ \text{sat}\ ((x_1,\dots,x_n)=(y_1,\dots,y_n))\wedge\phi)$\\
$\rightarrow gr(s\ \text{sat}\ \phi[(y_1,\dots,y_n)/(x_1,\dots,x_n)])=1$.
So,\\ $gr(s\ \text{sat}\ ((x_1,\dots,x_n)=(y_1,\dots,y_n))\wedge\phi\vdash \phi[(y_1,\dots,y_n)/(x_1,\dots,x_n)])\\=1$.

8.\ (i) $gr(s\ \text{sat}\ \phi\vdash \psi[x/y])=gr(s\ \text{sat}\ \phi)\rightarrow gr(s\ \text{sat}\ \psi[x/y])=gr(s\ \text{sat}\ \phi)\rightarrow gr(s(s(x)/y)\ \text{sat}\ \psi)$ and $gr(s\ \text{sat}\ \phi\vdash\exists y\psi)=gr(s\ \text{sat}\ \phi)\rightarrow gr(s\ \text{sat}\ \exists y\psi)\\=gr(s\ \text{sat}\ \phi)\rightarrow sup_{d\in D}\{gr(s(d/y)\ \text{sat}\ \psi)\}$.

Now $gr(s(s(x)/y)\ \text{sat}\ \psi)\leq sup_{d\in D}\{gr(s(d/y)\ \text{sat}\ \psi)\}$, as $s(x)\in D$.
So, by Property \ref{pga}(4)
$$gr(s\ \text{sat}\ \phi)\rightarrow gr(s(s(x)/y)\ \text{sat}\ \psi) \leq gr(s\ \text{sat}\ \phi)\rightarrow sup_{d\in D}\{gr(s(d/y)\ \text{sat}\ \psi)\}$$ and consequently 
$$gr(s\ \text{sat}\ \phi\vdash \psi[x/y])\leq gr(s\ \text{sat}\ \phi\vdash\exists y\psi).$$
(ii) $gr(s\ \text{sat}\ \exists y\phi\vdash \psi)=gr(s\ \text{sat}\ \exists y\phi)\rightarrow gr(s\ \text{sat}\ \psi)\\=sup_{d\in D}\{gr(s(d/y)\ \text{sat}\ \phi)\}\rightarrow gr(s\ \text{sat}\ \psi)$.\\
$gr(s\ \text{sat}\ \phi[x/y]\vdash\psi)=gr(s\ \text{sat}\ \phi[x/y])\rightarrow gr(s\ \text{sat}\ \psi)\\=gr(s(s(x)/y)\ \text{sat}\ \phi)\rightarrow gr(s\ \text{sat}\ \psi)$.\\
Now $gr(s(s(x)/y)\ \text{sat}\ \phi)\leq sup_{d\in D}\{gr(s(d/y)\ \text{sat}\ \phi)\}$, as $s(x)\in D$.
So, by Property \ref{pga}(3)\\
$sup_{d\in D}\{gr(s(d/y)\ \text{sat}\ \phi)\}\rightarrow gr(s\ \text{sat}\ \psi)\leq gr(s(s(x)/y)\ \text{sat}\ \phi)\rightarrow gr(s\ \text{sat}\ \psi)$ and consequently
$gr(s\ \text{sat}\ \exists y\phi\vdash \psi)\leq gr(s\ \text{sat}\ \phi[x/y]\vdash\psi)$.

9.\  $gr(s\ \text{sat}\ \phi\wedge (\exists y)\psi\vdash\exists y (\phi\wedge\psi))\\
=gr(s\ \text{sat}\ \phi\wedge\exists y\psi)\rightarrow gr(s\ \text{sat}\ \exists y (\phi\wedge\psi))\\
=(gr(s\ \text{sat}\ \phi)\wedge gr(s\ \text{sat}\ \exists y\psi))\rightarrow sup_{d\in D}\{gr(s(d/y)\ \text{sat}\ \phi\wedge\psi)\}\\
=(gr(s\ \text{sat}\ \phi)\wedge sup_{d\in D}\{gr(s(d/y)\ \text{sat}\ \psi)\})\rightarrow (sup_{d\in D}\{gr(s(d/y)\ \text{sat}\ \phi)\wedge gr(s(d/y)\ \text{sat}\ \psi)\})\\
=(gr(s\ \text{sat}\ \phi)\wedge sup_{d\in D}\{gr(s(d/y)\ \text{sat}\ \psi)\})\rightarrow\\ (sup_{d\in D}\{gr(s(d/y)\ \text{sat}\ \phi)\}\wedge sup_{d\in D}\{gr(s(d/y)\ \text{sat}\ \psi)\})$.\\
Now $gr(s\ \text{sat}\ \phi)\leq sup_{d\in D}\{gr(s(d/y)\ \text{sat}\ \phi)\}$.\\
So, $(gr(s\ \text{sat}\ \phi)\wedge  sup_{d\in D}\{gr(s(d/y)\ \text{sat}\ \psi)\}) \\ \leq ( sup_{d\in D}\{gr(s(d/y)\ \text{sat}\ \phi)\}\wedge  sup_{d\in D}\{gr(s(d/y)\ \text{sat}\ \psi)\})$.\\
Hence $(gr(s\ \text{sat}\ \phi)\wedge sup_{d\in D}\{gr(s(d/y)\ \text{sat}\ \psi)\})\rightarrow \\ ( sup_{d\in D}\{gr(s(d/y)\ \text{sat}\ \phi)\}\wedge  sup_{d\in D}\{gr(s(d/y)\ \text{sat}\ \psi)\})=1$ and consequently
$gr(s\ \text{sat}\ \phi\wedge (\exists y)\psi\vdash\exists y (\phi\wedge\psi))=1$.

This completes the proof.
\end{proof}
These properties may be considered as the counterparts of the rules of inference in the previous section. In fact in the derivation of a sequent from other sequents the same rules as in the previous section shall be employed but now the sequents are tagged with grades. Fuzzy geometric logic with graded consequence deals with the following questions:
\begin{enumerate}
\item If some sequents are given with grades then it can be investigated which sequents are derivable from them and they are sequents of what grade in the least.
\item From a set of graded sequents whether a given graded sequent is derivable or not by the rules of inference.
\end{enumerate}
Fuzzy geometric logic with graded consequence is sound in the following senses. Let, as premise, a finite number of sequents $\phi_i\vdash\psi_i$ with grades $a_i$, $i=1,2,\dots,n$ be taken. By Theorem \ref{5.1g}, any sequent derived from this premise by rules will be of a grade at least equal to $min\{a_1,a_2,\dots,a_n\}$. This answers question (1). To answer question (2) we need to check whether the sequent considered is derivable by syntactic rules from the sequents of the premise and whether the grade of the sequent is greater than or equal to the minimum of the grades of the premise sequents. In particular in every interpretation whenever the sequents of the premise are of grade 1, the conclusion sequent is also of grade 1.\\
\textbf{Note:} A graded sequent $\phi\vdash\psi$ may be read as $\psi$ is a consequence of $\phi$ and there is a grade (or strength) of the consequence relation. A graded consequence relation is a fuzzy relation \cite{LZ} between the power set of wffs and the set of wffs if and only if the following conditions are satisfied:
\begin{itemize}
\item if $\alpha\in \Gamma$ then $gr(\Gamma\vdash\alpha)=1$ (overlap);
\item If $\Gamma\subseteq \Delta$ then $gr(\Gamma\vdash\alpha)\leq gr(\Delta\vdash\alpha)$ (monotonicity);
\item $inf_{\delta\in \Delta}gr(\Gamma\vdash\delta)* gr(\Delta\vdash\alpha)\leq gr(\Gamma\vdash\alpha)$ (cut);
\end{itemize}
where $\Gamma$, $\Delta$ are sets of wffs and $\alpha$, $\delta$ are wffs, the grade of relatedness denoted by $gr(\Gamma\vdash\alpha)$ is an element of a residuated lattice and $*$ is the product operation relative to which the residuum is constructed \cite{MK, MSD}. In particular the residuated lattice can be taken as the unit interval $[0,1]$ and the product as the minimum ($\wedge$). In the present work this particular case has been taken. Secondly, in this work the set $\Gamma$ is always a singleton set. So overlap condition takes the form $gr(\alpha\vdash\alpha)=1$ and cut condition becomes $gr(\beta\vdash\delta)\wedge gr(\delta\vdash\alpha)\leq gr(\beta\vdash\alpha)$. A special case of monotonicity condition is taken here in which the set $\{\beta,\delta\}$ to the left of $\vdash$ is equated with the single wff $\beta\wedge\delta$. So the condition becomes $gr(\beta\vdash\alpha)\leq gr(\beta\wedge\delta\vdash\alpha)$ for all $\alpha,\ \beta,\ \delta$. It can be shown that this inequality is derivable from the above rules of inference.

In fuzzy geometric logic with graded consequence a sequent with grade is derived, from a given set of sequents with grades. 
\section{Fuzzy Topological Space with graded inclusion, Graded Frame and Graded Fuzzy Topological System}
\begin{definition}[Fuzzy topological space with graded inclusion]\label{gfts} \index{fuzzy topological space!with!graded inclusion}
Let $X$ be a set, $\tau$ be a collection of fuzzy subsets of $X$ s.t.
\begin{enumerate}
\item $\tilde\emptyset$, $\tilde X\in \tau$, where $\tilde\emptyset (x)=0$ and $\tilde X (x)=1$, for all $x\in X$;
\item $\tilde T_i\in\tau$ for $i\in I$ imply $\bigcup_{i\in I}\tilde T_i\in \tau$, where $\bigcup_{i\in I}\tilde{T_i}(x)=sup_{i\in I}\{\tilde{T_i}(x)\}$;
\item $\tilde T_1$, $\tilde T_2\in\tau$ imply $ \tilde T_1\cap\tilde T_2\in\tau$, where $(\tilde{T_1}\cap \tilde{T_2})(x)=\tilde{T_1}(x)\wedge \tilde{T_2}(x)$,
\end{enumerate}
and $\subseteq$ be a fuzzy inclusion relation for fuzzy sets is defined as $gr(\tilde{T_1}\subseteq{\tilde{T_2}})=inf_{x\in X}\{\tilde{T_1}(x)\rightarrow\tilde{T_2}(x)\}$, where $\tilde{T_1},\  \tilde{T_2}$ are fuzzy subsets of $X$ and $\rightarrow$ is the G$\ddot{o}$del arrow.

Then $(X,\tau,\subseteq)$ is called a \textbf{fuzzy topological space with graded inclusion}. $(\tau,\subseteq)$ is called a \textbf{fuzzy topology with graded inclusion}  \index{fuzzy topology!with!graded inclusion} over $X$. 
\end{definition}
It is to be noted that we preferred to use the traditional notation $\tilde{A}$ to denote a fuzzy set \cite{MA}.
We list the properties of the members of fuzzy topology with graded inclusion, as propositions, that would be used subsequently. The proof of the propositions are provided here.
\begin{proposition}\label{gt1}
$gr(\tilde{T}\subseteq\tilde{T})=1$.
\end{proposition}
\begin{proof}
$gr(\tilde{T}\subseteq\tilde{T})=inf_x\{\tilde{T}(x)\rightarrow \tilde{T}(x)\}=1$.
\end{proof}
\begin{proposition}\label{gt2}
$gr(\tilde{T_1}\subseteq\tilde{T_2})=1=gr(\tilde{T_2}\subseteq\tilde{T_1})\Rightarrow \tilde{T_1}=\tilde{T_2}$.
\end{proposition}
\begin{proof}We have
\begin{align*}
\ \ gr(\tilde{T_1}\subseteq\tilde{T_2})=1 & \ \text{iff} \hspace{3mm}  inf_x\{\tilde{T_1}(x)\rightarrow \tilde{T_2}(x)\}=1\\
& \hspace{3mm} \text{iff} \hspace{4mm}  \tilde{T_1}(x)\rightarrow \tilde{T_2}(x)=1, \ \text{for any}\ x\\
& \hspace{3mm} \text{iff} \hspace{4mm}   \tilde{T_1}(x)\leq \tilde{T_2}(x),\ \text{for any}\ x. \\
\ \ gr(\tilde{T_2}\subseteq\tilde{T_1})=1 & \hspace{3mm}  \text{iff} \hspace{4mm}  inf_x\{\tilde{T_2}(x)\rightarrow \tilde{T_1}(x)\}=1\\
& \hspace{3mm}  \text{iff}\hspace{4mm}  \tilde{T_2}(x)\rightarrow \tilde{T_1}(x)=1, \ \text{for any}\ x\\
& \hspace{3mm}  \text{iff} \hspace{4mm}  \tilde{T_2}(x)\leq \tilde{T_1}(x),\ \text{for any}\ x.
\end{align*}
Hence $\tilde{T_1}(x)= \tilde{T_2}(x)\ \text{for any}\ x \ \text{iff}\ \tilde{T_1}=\tilde{T_2}$.
\end{proof}
\begin{proposition}\label{gt3}
$gr(\tilde{T_1}\subseteq\tilde{T_2})\wedge gr(\tilde{T_2}\subseteq\tilde{T_3})\leq gr(\tilde{T_1}\subseteq\tilde{T_3})$.
\end{proposition}
\begin{proof}We have,
\begin{align*}
gr(\tilde{T_1}\subseteq\tilde{T_2})\wedge gr(\tilde{T_2}\subseteq\tilde{T_3}) 
& = inf_x\{\tilde{T_1}(x)\rightarrow \tilde{T_2}(x)\}_{x} \wedge inf_x\{ \tilde{T_2}(x)\rightarrow \tilde{T_3}(x)\} \\ 
& =inf_x\{(\tilde{T_1}(x)\rightarrow \tilde{T_2}(x)) \wedge (\tilde{T_2}(x)\rightarrow \tilde{T_3}(x))\}\\
& \leq inf_x \{ \tilde{T_1}(x)\rightarrow \tilde{T_3}(x)\} \\
& =gr(\tilde{T_1}\subseteq\tilde{T_3}).
\end{align*}
\end{proof}
\begin{proposition}\label{gt4}
$gr(\tilde{T_1}\cap\tilde{T_2}\subseteq\tilde{T_1})=1=gr(\tilde{T_1}\cap\tilde{T_2}\subseteq\tilde{T_2})$.
\end{proposition}
\begin{proof}
We have,
\begin{align*}
gr(\tilde{T_1}\cap\tilde{T_2}\subseteq\tilde{T_1})
& =inf_x\{(\tilde{T_1}\cap\tilde{T_2})(x)\rightarrow \tilde{T_1}(x)\}\\
& =inf_x\{\tilde{T_1}(x)\wedge\tilde{T_2}(x)\rightarrow \tilde{T_1}(x)\}\\
& =1.
\end{align*}
Similarly $gr(\tilde{T_1}\cap\tilde{T_2}\subseteq\tilde{T_2})=1$.
\end{proof}
\begin{proposition}\label{gt5}
$gr(\tilde{T}\subseteq\tilde{X})=1$.
\end{proposition}
\begin{proof}
$gr(\tilde{T}\subseteq\tilde{X})=inf_x\{\tilde{T}(x)\rightarrow\tilde{X}(x)\}=inf_x\{\tilde{T}(x)\rightarrow 1\}=1$.
\end{proof}
\begin{proposition}\label{gt6}
$gr(\tilde{T_1}\subseteq\tilde{T_2})\wedge gr(\tilde{T_1}\subseteq\tilde{T_3})=gr(\tilde{T_1}\subseteq\tilde{T_2}\cap\tilde{T_3})$.
\end{proposition}
\begin{proof}
One has,
\begin{align*}
gr(\tilde{T_1}\subseteq\tilde{T_2})\wedge gr(\tilde{T_1}\subseteq\tilde{T_3})
& =inf_x\{\tilde{T_1}(x)\rightarrow\tilde{T_2}(x)\}\wedge inf_x\{\tilde{T_1}(x)\rightarrow\tilde{T_3}(x)\}\\
& =inf_x\{(\tilde{T_1}(x)\rightarrow\tilde{T_2}(x))\wedge (\tilde{T_1}(x)\rightarrow\tilde{T_3}(x))\}\\
& =inf_x\{\tilde{T_1}(x)\rightarrow (\tilde{T_2}(x)\wedge\tilde{T_3}(x))\}\\
& =gr(\tilde{T_1}\subseteq\tilde{T_2}\cap\tilde{T_3}).
\end{align*}
This completes the proof of the proposition.
\end{proof}
\begin{proposition}\label{gt7}
$gr(\tilde{T_i}\subseteq\bigcup_i\tilde{T_i})=1$.
\end{proposition}
\begin{proof}
$gr(\tilde{T_i}\subseteq\bigcup_i\tilde{T_i})=inf_x\{\tilde{T_i}(x)\rightarrow \bigvee_i\{\tilde{T_i}(x)\}\}=1$.
\end{proof}
\begin{proposition}\label{gt8}
$inf_{ \tilde{T_i}\in \tilde{S}}\{gr(\tilde{T_i}\subseteq\tilde{T})\}=gr(\bigcup \tilde{S}\subseteq\tilde{T})$.
\end{proposition}
\begin{proof}We have,
\begin{align*}
inf_{ \tilde{T_i}\in \tilde{S}}\{gr(\tilde{T_i}\subseteq\tilde{T})\}
& =inf_{ \tilde{T_i}\in \tilde{S}}\{inf_x\{\tilde{T_i}(x)\rightarrow \tilde{T}(x)\}\}\\
& =inf_x\{inf_{ \tilde{T_i}\in \tilde{S}}\{\tilde{T_i}(x)\rightarrow \tilde{T}(x)\}\}\\
& =inf_x\{sup_{\tilde{T_i\in \tilde{S}}}\{\tilde{T_i}(x)\}\rightarrow \tilde{T}(x)\}\\
& =inf_x\{\bigcup\tilde{S}(x)\rightarrow \tilde{T}(x)\}\\
& =gr(\bigcup \tilde{S}\subseteq\tilde{T}).
\end{align*}
\end{proof}
\begin{proposition}\label{gt9}
$gr(\tilde{T}\cap\bigcup_i\tilde{T_i}\subseteq \bigcup_i(\tilde{T}\cap\tilde{T_i}))=1$.
\end{proposition}
\begin{proof}
Let us proceed in the following way.
\begin{align*}
gr(\tilde{T}\cap\bigcup_i\tilde{T_i}\subseteq \bigcup_i(\tilde{T}\cap\tilde{T_i}))
& =inf_x\{(\tilde{T}\cap\bigcup_i\tilde{T_i})(x)\rightarrow \bigcup_i(\tilde{T}\cap\tilde{T_i})(x)\}\\
& =inf_x\{\tilde{T}(x)\wedge (\bigcup_i\tilde{T_i})(x)\rightarrow sup_i\{(\tilde{T}\cap\tilde{T_i})(x)\}\}\\
& =inf_x\{\tilde{T}(x)\wedge sup_i\{\tilde{T_i}(x)\}\rightarrow sup_i\{\tilde{T}(x)\wedge\tilde{T_i}(x)\}\}\\
& =inf_x\{sup_i\{\tilde{T}(x)\wedge\tilde{T_i}(x)\}\rightarrow sup_i\{\tilde{T}(x)\wedge\tilde{T_i}(x)\}\}\\
& =1.
\end{align*}
This completes the proof.
\end{proof}
\begin{proposition}\label{gt10}
$\tilde{T_1}(x)\wedge gr(\tilde{T_1}\subseteq\tilde{T_2})\leq \tilde{T_2}(x)$, for each $x$.
\end{proposition}
\begin{proof}
$\tilde{T_1}(x)\wedge gr(\tilde{T_1}\subseteq\tilde{T_2})=\tilde{T_1}(x)\wedge inf_x\{ \tilde{T_1}(x)\rightarrow \tilde{T_2}(x)\}=inf_x\{\tilde{T_1}(x)\wedge(\tilde{T_1}(x)\rightarrow \tilde{T_2}(x))\}$.
Now if $\tilde{T_1}(x)\leq \tilde{T_2}(x)$, then $\tilde{T_1}(x)\rightarrow \tilde{T_2}(x)=1$. Hence in this case $\tilde{T_1}(x)\wedge (\tilde{T_1}(x)\rightarrow \tilde{T_2}(x)) = \tilde{T_1}(x)\wedge 1
= \tilde{T_1}(x)
\leq \tilde{T_2}(x).$
Now when $\tilde{T_1}(x)> \tilde{T_2}(x)$, then $\tilde{T_1}(x)\rightarrow \tilde{T_2}(x)=\tilde{T_2}(x)$. Hence in this case
$$\tilde{T_1}(x)\wedge (\tilde{T_1}(x)\rightarrow \tilde{T_2}(x))
 = \tilde{T_1}(x)\wedge \tilde{T_2}(x)
 = \tilde{T_2}(x).$$
So for any $x$, $$\tilde{T_1}(x)\wedge (\tilde{T_1}(x)\rightarrow \tilde{T_2}(x))\leq \tilde{T_2}(x).$$
Therefore $$\tilde{T_1}(x)\wedge gr(\tilde{T_1}\subseteq\tilde{T_2})\leq \tilde{T_2}(x),$$
as $inf\{\tilde{T_1}(x)\wedge (\tilde{T_1}(x)\rightarrow \tilde{T_2}(x))\}\leq \tilde{T_1}(x)\wedge (\tilde{T_1}(x)\rightarrow \tilde{T_2}(x))$.
\end{proof}
\begin{definition}[Graded Frame]\label{gf}\index{graded frame}
A \textbf{graded frame} is a 5-tuple \\$(A,\top,\wedge,\bigvee,R)$, where $A$ is a non-empty set $\top\in A$, $\wedge$ is a binary operation, $\bigvee$ is an operation on arbitrary subset of $A$, $R$ is a $[0,1]$-valued fuzzy binary relation on $A$ satisfying the following conditions:
\begin{enumerate}
\item $R(a,a)=1$;
\item $R(a,b)=1=R(b,a)\Rightarrow a=b$;
\item $R(a,b)\wedge R(b,c)\leq R(a,c)$;
\item $R(a\wedge b,a)=1=R(a\wedge b,b)$;
\item $R(a,\top)=1$;
\item $R(a,b)\wedge R(a,c)=R(a,b\wedge c)$;
\item $R(a,\bigvee S)=1$ if $a\in S$;
\item $inf\{R(a,b)\mid a\in S\}=R(\bigvee S,b)$;
\item $R(a\wedge\bigvee S,\bigvee\{a\wedge b\mid b\in S\})=1$;
\end{enumerate}
for any $a,\ b,\ c\in A$ and $S\subseteq A$. We will denote graded frame by $(A,R)$. 
\end{definition}
It is to be noted that when $S$ is two element set $\{a_1,a_2\}$ (say) then $\bigvee S$ will be written as $a_1\vee a_2$.

Defining $a\prec b$ if and only if $R(a,b)=1$, it can be checked that $(A,\prec)$ is a partial order relation with an upper bound $\top$ and with respect to this partial ordering, $a\wedge b$ is the greatest lower bound (g.l.b) of two element set $\{a,b\}$, $\bigvee S$ is the least upper bound (l.u.b) of $S$. Consequently $(A,\prec)$ forms a frame.
\begin{proposition}\label{r1}
$R(\bigvee\emptyset ,a)=1$, for any $a\in A$.
\end{proposition}
\begin{proof}
\begin{align*}
R(\bigvee\emptyset ,a)
 & = inf\{R(b,a)\mid b\in\emptyset\}\ \ \ \text{[by Definition \ref{gf} (8)]}\\
& = inf\ \emptyset\ \ \ \ \ \ \ \ \ \ \ \ \ \ \ \ \ \ \ \ \text{[as there is no element in $\emptyset$]}\\
& = 1.
\end{align*}
This completes the proof.
\end{proof}
It should be noted that $\bigvee\emptyset$ plays the role of the bottom element of $A$ and hence can be denoted by the symbol $\bot$. Hence by Proposition \ref{r1} $R(\bot,a)=1$ for any $a\in A$. 
\begin{proposition}\label{r2}
$R(a,b)=R(a\wedge b,a)\wedge R(a,a\wedge b)=R(a\vee b,b)\wedge R(b,a\vee b)$.
\end{proposition}
\begin{proof}
We have,
\begin{align*}
R(a\wedge b,a)\wedge R(a,a\wedge b) & = 1\wedge R(a,a\wedge b)\ \ \ \ \text{[by Proposition \ref{gf}(4)]}\\
& = R(a,a\wedge b)\\
& = R(a,a)\wedge R(a,b)\ \ \text{[by Proposition \ref{gf}(6)]}\\
& = 1\wedge R(a,b)\ \ \ \ \ \ \ \ \ \text{[by Proposition \ref{gf}(1)]}\\
& = R(a,b).
\end{align*}
\begin{align*}
R(a\vee b,b)\wedge R(b,a\vee b) & = R(a\vee b,b)\wedge 1\ \ \ \ \text{[by Proposition \ref{gf}(7)]}\\
& = R(a\vee b,b)\\
& = R(a,b)\wedge R(b,b)\ \ \text{[by Proposition \ref{gf}(8)]}\\
& = R(a,b)\wedge 1\ \ \ \ \ \ \ \ \ \text{[by Proposition \ref{gf}(1)]}\\
& = R(a,b).
\end{align*}
\end{proof}
It is to note that as a special case of the above definition $A$ with operations $\wedge$, $\bigvee$, and a relation $R:A\times A\longrightarrow \{0,1\}$ satisfying the conditions 1-9 forms a frame where $a\leq b$ iff $R(a,b)=1$.
\begin{example}
If $(X,\tau,\subseteq)$ is a fuzzy topological space with graded inclusion then $(\tau,\subseteq)$ with operations $\cap$, $\bigcup$ forms a graded frame as Propositions \ref{gt1}-\ref{gt9} hold. In this graded frame the top element is $X$ and the bottom is $\emptyset$.
\end{example} 
\begin{definition}[Graded fuzzy topological system]\label{gftsy}
A \textbf{graded fuzzy topological system} \index{graded fuzzy toopological system} is a quadruple $(X,\models,A,R)$ consisting of a nonempty set $X$, a graded frame $(A,R)$ and a fuzzy relation $\models$ from $X$ to $A$ such that
\begin{enumerate}
\item $gr(x\models a)\wedge R(a,b)\leq gr(x\models b)$;
\item for any finite subset including null set, $S$, of $A$, $gr(x\models\bigwedge S)=\\ inf\{gr(x\models a)\mid a\in S\}$;
\item for any subset $S$ of $A$, $gr(x\models\bigvee S)= sup\{gr(x\models a)\mid a\in S\}$.
\end{enumerate} 
\end{definition}
\begin{definition}\label{gspatial}
A graded fuzzy topological system $(X,\models,A,R)$ is said to be \textbf{spatial}\index{graded fuzzy topological system!spatial} if and only if (for any $x\in X$, $gr(x\models a)=gr(x\models b)$) imply ($a=b$), for any $a,\ b\in A$.
\end{definition}
\begin{theorem}\label{gfttsy}
For any fuzzy topological space with graded inclusion $(X,\tau,\subseteq)$, $(X,\in,\tau,\subseteq)$ forms a graded fuzzy topological system, where $gr(x\in \tilde{T})=\tilde{T}(x)$ for any $x\in X$ and $\tilde{T}\in \tau$.
\end{theorem}
\begin{proof}
As $(X,\tau,\subseteq)$ is a fuzzy topological space with graded inclusion so $X$ is a nonempty set, $(\tau,\subseteq)$ is a graded frame. Now $gr(x\in \tilde{T_1})\wedge gr(\tilde{T_1}\subseteq\tilde{T_2})\leq gr(x\in \tilde{T_2})$, as $gr(x\in\tilde{T})=\tilde{T}(x)$ by Proposition \ref{gt10}. The rest of the conditions, viz. $gr(x\in \tilde{T_1}\cap\tilde{T_2})=gr(x\in\tilde{T_1})\wedge gr(x\in\tilde{T_2})$ and for any $S\subseteq \tau$, $gr(x\in \bigcup S)=sup\{gr(x\in \tilde{T_i})\mid \tilde{T_i}\in \tau\}$, follow from the definition of fuzzy topological space with graded inclusion. Hence $(X,\in,\tau,\subseteq)$ is a graded fuzzy topological system.
\end{proof}
\begin{definition}\label{extent$_g$}
For any graded fuzzy topological system $(X,\models,A,R)$, the \textbf{extent$_g$} \index{extent$_g$} of each $a\in A$ denoted by $ext_g(a)$ is a mapping from $X$ to $[0,1]$ such that $ext_g(a)(x)=gr(x\models a)$, $x\in X$. Let us denote the set $\{ext_g(a)\}_{a\in A}$ by $ext_g(A)$.
\end{definition}
\begin{theorem}\label{gftsyt}
For any graded fuzzy topological system $(X,\models,A,R)$, \\$(X,ext_g(A),\subseteq)$ forms a fuzzy topological space with graded inclusion, where $\subseteq$ is the fuzzy inclusion relation.
\end{theorem}
\begin{proof}
To show that $(ext_g(A),\subseteq)$ forms a graded fuzzy topology on $X$ let us proceed as follows.
\begin{enumerate}
\item For any $x\in X$, $ext_g(\bot)(x)=gr(x\models \bot)=gr(x\models\bigvee\emptyset)=sup\{gr(x\models a)\mid a\in\emptyset\}=sup\{\emptyset\}=0$.

Similarly $ext_g(\top)(x)=1$, for any $x\in X$.
\item Taking $ext_g(a_i)$'s from $ext_g(A)$, we get,
\begin{align*} 
\bigcup_iext_g(a_i)(x) & = sup_i\{ext_g(a_i)(x)\}\\
& = sup_i\{gr(x\models a_i)\} \\
& = gr(x\models \bigvee\{a_i \}_i ) \\
& = ext_g(\bigvee\{a_i \}_i )(x).
\end{align*}
Hence $\bigcup_iext_g(a_i)=ext_g(\bigvee\{a_i\}_i)\in ext_g(A)$, as $\bigvee\{a_i\}_i\in A$ for $a_i$'s $\in A$.
\item Similarly $ext_g(a_1)$, $ext_g(a_2)\in ext_g(A)$ imply $ext_g(a_1)\cap ext_g(a_2)\in ext_g(A)$.
\end{enumerate}
\end{proof}
\begin{definition}\label{gequiv}
Let $(X,\models,A,R)$ be a graded fuzzy topological system. Define $a\approx b$ iff $gr(x\models a)=gr(x\models b)$ for any $x\in X$ and $a,b\in A$.
\end{definition}
\begin{proposition}\label{eq}
$\approx$ is an equivalence relation.
\end{proposition}
\begin{proof}
We have $gr(x\models a)=gr(x\models a)$ for any $x\in X$ and so $a\approx a$.\\
Thereby $\approx$ is \textbf{reflexive}.
\begin{align*}
a\approx b & \ \text{iff}\ gr(x\models a)=gr(x\models b)\ \text{for any}\ x\in X\\
& \ \text{iff}\ gr(x\models b)=gr(x\models a)\ \text{for any}\ x\in X\\
& \ \text{iff}\ b\approx a.
\end{align*}
So $a\approx b\ \text{iff}\ b\approx a$ and hence $\approx$ is \textbf{symmetric}.
\begin{align*}
& a\approx b\ \text{and}\ b\approx c \\& \ \text{iff} \ gr(x\models a)=gr(x\models b)\ \text{and}\ gr(x\models b)=gr(x\models c)\ \text{for any}\ x\in X\\
& \Rightarrow gr(x\models a)=gr(x\models c)\ \text{for any}\ x\in X\\
& \ \text{iff}\ a\approx c.
\end{align*}
Hence $a\approx b\ \text{and}\ b\approx c\Rightarrow a\approx c$ and consequently $\approx$ is \textbf{transitive}.
\end{proof}
So we get the quotient $A/_{\approx}$. The operations $\wedge$, $\bigvee$ can be lifted in the following way so that they will become independent of choice:
$$[a]\wedge [b]=^{def}[a\wedge b],\ [a]\vee [b]=^{def}[a\vee b],\ \bigvee\{[a_i]\}_{i}=^{def}[\bigvee\{a_i\}_{i}].$$
\begin{definition}\label{6.5}
Let us define the relation $\models '$ between $X$ and $A/_{\approx}$ in the following way : $gr(x\models '[a])=gr(x\models a)$.
\end{definition}
\begin{proposition}\label{propc}
$\forall x\in X,\ (gr(x\models '[a])=gr(x\models '[b]))\Rightarrow [a]=[b]$.
\end{proposition}
\begin{proof}
$\forall x\in X,\ (gr(x\models '[a])=gr(x\models '[b]))\ \text{if and only if}\ \forall x\in X,\ (gr(x\models a)=gr(x\models b))$. Hence $a\approx b$, and so $[a]=[b]$.
\end{proof}
\begin{theorem}\label{agf}
$(X,\models ',A/_{\approx},R)$ is a graded fuzzy topological system.
\end{theorem}
\begin{proof}
$X$ is a nonempty set.

Let us show that $(A/_{\approx},R)$ be a graded frame where $R$ is defined as\\ $R([a],[b])=inf_x\{(gr(x\models '[a])\rightarrow gr(x\models ' [b]))\}$, and $\rightarrow$ is the G$\ddot{o}$del arrow.
\begin{enumerate}
\item $(gr(x\models '[a])\rightarrow gr(x\models ' [a])) = 1$
$\text{iff}\ R([a],[a]) =1$.
\item Let $R([a],[b])=1=R([b],[a])$ is given.
Now $ R([a],[b])=1$ if and only if $inf_{x}\{(gr(x\models '[a])\rightarrow gr(x\models '[b]))\}=1$. Hence $gr(x\models '[a])\rightarrow gr(x\models '[b])=1 \ \text{for any}\ x$. Using Definition \ref{6.5} $gr(x\models a)\rightarrow gr(x\models b)=1\ \text{for any}\ x$. So by Property \ref{pga}(7) $gr(x\models a)\leq gr(x\models b)\ \text{for any}\ x$. \\
Similarly $R([b],[a])=1$ if and only if $gr(x\models b)\leq gr(x\models a)\ \text{for any}\ x$.

Hence $gr(x\models a)= gr(x\models b)\ \text{for any}\ x \ \text{iff}\ a\approx b
\ \text{iff}\ [a]=[b]$.
\item $R([a],[b])\wedge R([b],[c]) \\ 
= inf\{gr(x\models ' [a])\rightarrow gr(x\models ' [b])\}_{x} \wedge inf\{ gr(x\models ' [b])\rightarrow gr(x\models ' [c])\}_{x} \\ 
=inf\{(gr(x\models ' [a])\rightarrow gr(x\models ' [b])) \wedge (gr(x\models ' [b])\rightarrow gr(x\models ' [c]))\}_{x} \\
 \leq inf \{ gr(x\models ' [a])\rightarrow gr(x\models '[c])\}_{x} \\
 = R([a],[c])$.
 \item $R([a]\wedge [b],[a])\\
=inf\{gr(x\models '[a]\wedge [b])\rightarrow gr(x\models '[a])\}_x\\
=inf\{gr(x\models ' [a\wedge b])\rightarrow gr(x\models '[a])\}_x\\
=inf\{gr(x\models  a\wedge b)\rightarrow gr(x\models a)\}_x\\
=inf\{(gr(x\models  a)\wedge gr(x\models b))\rightarrow gr(x\models a)\}_x\\
=1$.\\
Similarly $R([a]\wedge [b],[b])=1$.
\item $R([a],[\top])\\
=inf\{gr(x\models '[a])\rightarrow gr(x\models '[\top])\}_x\\
=inf\{gr(x\models a)\rightarrow gr(x\models \top)\}_x\\
=inf\{gr(x\models a)\rightarrow 1\}_x\\
=1$.
\item $R([a],[b])\wedge R([a],[c])\\
=inf\{gr(x\models '[a])\rightarrow gr(x\models '[b])\}_x\wedge inf\{gr(x\models '[a])\rightarrow gr(x\models '[c])\}_x\\
=inf\{gr(x\models a)\rightarrow gr(x\models b)\}_x\wedge inf\{gr(x\models a)\rightarrow gr(x\models c)\}_x\\
=inf\{(gr(x\models a)\rightarrow gr(x\models b))\wedge (gr(x\models a)\rightarrow gr(x\models c))\}_x\\
=inf\{gr(x\models a)\rightarrow (gr(x\models b)\wedge gr(x\models c))\}\\
=inf\{gr(x\models' [a])\rightarrow (gr(x\models' [b])\wedge gr(x\models' [c]))\}\\
=inf\{gr(x\models' [a])\rightarrow (gr(x\models' [b]\wedge [c]))\}\\
=inf\{gr(x\models' [a])\rightarrow (gr(x\models' [b\wedge c]))\}\\
=R([a],[b\wedge c])\\
=R([a],[b]\wedge [c])$.
\item Let $a\in S$ and $S\subseteq A$.\\
$R([a],\bigvee [S])\\
=R([a],[\bigvee S])\\
=inf\{gr(x\models'[a])\rightarrow gr(x\models'[\bigvee S])\}_x\\
=inf\{gr(x\models a)\rightarrow gr(x\models\bigvee S)\}_x\\
=1$.
\item $inf\{R([a],[b])\}_{a\in S}\\
=inf\{inf\{gr(x\models'[a])\rightarrow gr(x\models'[b])\}_x\}_{a\in S}\\
=inf\{inf\{gr(x\models a)\rightarrow gr(x\models b)\}_x\}_{a\in S}\\
=inf\{inf\{gr(x\models a)\rightarrow gr(x\models b)\}_{a\in S}\}_x\\
=inf\{ sup\{gr(x\models a)\}_{a\in S}\rightarrow gr(x\models b)\}_x\\
=inf\{gr(x\models \bigvee S)\rightarrow gr(x\models b)\}_x\\
=inf\{gr(x\models' [\bigvee S])\rightarrow gr(x\models' [b])\}_x\\
=R([\bigvee S],[b])\\
=R(\bigvee [S],[b])$.
\item $R([a]\wedge\bigvee [S],\bigvee\{[a]\wedge [b]\}_{b\in S})\\
=R([a]\wedge [\bigvee S], \bigvee\{[a\wedge b]\}_{b\in S})\\
=R([a]\wedge\bigvee S],[\bigvee\{a\wedge b\}_{b\in S}])\\
=inf\{gr(x\models'[a\wedge\bigvee S])\rightarrow gr(x\models'[\bigvee\{a\wedge b\}_{b\in S}])\}_x\\
=inf\{gr(x\models a\wedge\bigvee S)\rightarrow gr(x\models\bigvee\{a\wedge b\}_{b\in S})\}_x\\
=inf\{gr(x\models a\wedge\bigvee S)\rightarrow  sup\{gr(x\models a\wedge b\}_{b\in S})\}_x\\
=inf\{gr(x\models a\wedge\bigvee S)\rightarrow  sup\{gr(x\models a)\wedge gr(x\models b)\}_{b\in S})\}_x\\
=inf\{gr(x\models a\wedge\bigvee S)\rightarrow gr(x\models a)\wedge  sup\{gr (x\models b)\}_{b\in S})\}_x\\
=inf\{gr(x\models a\wedge\bigvee S)\rightarrow gr(x\models a)\wedge gr(x\models\bigvee S)\}_x\\
=inf\{gr(x\models a\wedge\bigvee S)\rightarrow gr(x\models a\wedge \bigvee S)\}_x\\
=1$.
\end{enumerate}
It is now to show that the three conditions of Definition \ref{gftsy} hold.
\begin{enumerate}
\item $gr(x\models'[a])\wedge R([a],[b])
=gr(x\models'[a])\wedge inf_x\{ gr(x\models'[a])\rightarrow gr(x\models ' [b])\}$ $
=inf_x\{gr(x\models'[a])\wedge(gr(x\models'[a])\rightarrow gr(x\models'[b]))\}$.\\
Now if $gr(x\models'[a])\leq gr(x\models'[b])$, then $gr(x\models'[a])\rightarrow gr(x\models'[b])=1$.\\
Hence in this case\\
$gr(x\models'[a])\wedge (gr(x\models'[a])\rightarrow gr(x\models'[b]))\\
=gr(x\models'[a])\wedge 1\\
=gr(x\models'[a])\leq gr(x\models' [b])$.\\
Now when $gr(x\models'[a])> gr(x\models'[b])$, \\we get $gr(x\models'[a])\rightarrow gr(x\models'[b])=gr(x\models'[b])$.\\
Hence in this case\\
$gr(x\models'[a])\wedge (gr(x\models'[a])\rightarrow gr(x\models'[b]))\\
=gr(x\models'[a])\wedge gr(x\models'[b])\\
=gr(x\models' [b])$.\\
So for any $x$, $gr(x\models'[a])\wedge (gr(x\models'[a])\rightarrow gr(x\models'[b]))\leq gr(x\models'[b])$.
\\ Therefore $gr(x\models'[a])\wedge R([a],[b])\leq gr(x\models'[b])$,\\
as $inf_x\{gr(x\models'[a])\wedge (gr(x\models'[a])\rightarrow gr(x\models'[b]))\}\leq gr(x\models'[a])\wedge (gr(x\models'[a])\rightarrow gr(x\models'[b]))$.
\item $gr(x\models'\bigwedge [S])
=gr(x\models' [\bigwedge S])
=gr(x\models \bigwedge S)
=inf_{a\in S}\{gr(x\models a)\}
=inf_{a\in S}\{gr(x\models'[a])\}$.
\item $gr(x\models'\bigvee [S])
=gr(x\models' [\bigvee S])
=gr(x\models \bigvee S)
=sup_{a\in S}\{gr(x\models a)\}
=sup_{a\in S}\{gr(x\models'[a])\}$.
\end{enumerate}
This finishes the proof.
\end{proof}
Proposition \ref{propc} indicates that the graded fuzzy topological system $(X,\models ',A/_{\approx},R)$ is spatial.
\section{Interrelations}
\subsection{Fuzzy Geometric Logic with Graded Consequence to Graded Fuzzy Topological System}
Let us consider the quadruple $(X,\models,A,R)$ where $X$ is a non empty set of assignments $s$, $A$ is the set of all geometric formulae, $\models$ defined as $gr(s\models \phi)=gr(s\ \text{sat}\ \phi)$ and $R(\phi,\psi)=gr(\phi\vdash \psi)=inf_s\{gr(s\ \text{sat}\ \phi\vdash\psi)\}$.
\begin{theorem}\label{Gfts}
(i) $gr(s\models\phi)\wedge R(\phi,\psi)\leq gr(s\models \psi)$. 
\\(ii) $gr(s\models \phi\wedge\psi)=gr(s\models \phi)\wedge gr(s\models\psi)$.
\\(iii) $gr(s\models\bigvee\{\phi_i\}_{i\in I})= sup_{i\in I}\{gr(x\models \phi_i)\}$.
\end{theorem}
\begin{proof}
(i) For any $s$, $gr(s\models\phi)\wedge R(\phi,\psi)=gr(s\ \text{sat}\ \phi)\wedge gr(\phi\vdash\psi)=gr(s\ \text{sat}\ \phi)\wedge inf_s\{ gr(s\ \text{sat}\ \phi)\rightarrow gr(s\ \text{sat}\ \psi)\}=inf_s\{gr(s\ \text{sat}\ \phi)\wedge ( gr(s\ \text{sat}\ \phi)\\ \rightarrow gr(s\ \text{sat}\ \psi))\}\leq gr(s\ \text{sat}\ \phi)\wedge ( gr(s\ \text{sat}\ \phi)\rightarrow gr(s\ \text{sat}\ \psi))
\leq gr(s\ \text{sat}\ \psi)= gr(s\models\psi)$.
\\(ii) $gr(s\models \phi\wedge\psi)=gr(s\ \text{sat}\ \phi\wedge\psi)=gr(s\ \text{sat}\ \phi)\wedge gr(s\ \text{sat}\ \psi)=gr(s\models \phi)\wedge gr(s\models \psi)$.
\\(iii) $gr(s\models\bigvee\{\phi_i\}_{i\in I})=gr(s\ \text{sat}\ \bigvee\{\phi_i\}_{i\in I})= sup_{i\in I}\{gr(s\ \text{sat}\ \phi_i)\}= \\sup_{i\in I}\{gr(s\models\phi_i)\}$.
\end{proof} 
It is to note that $(A,R)$ is not a graded frame since $R$ does not satisfy antisymmetry viz. condition 2, Definition \ref{gf} and hence $(X,\models,A,R)$ is not a graded fuzzy topological system. Let us construct a graded fuzzy topological system using $(X,\models,A,R)$ in the following way.
\begin{definition}\label{Gequiv}
$\phi\approx\psi$ iff $gr(s\models \phi)=gr(s\models \psi)$ for any $s\in X$ and $\phi,\ \psi\in A$.
\end{definition}
It can be shown that the above defined ``$\approx$" is an equivalence relation. Thus we get $A/_{\approx}$.
\begin{theorem}\label{Gftopsys}
$(X,\models ',A/_{\approx},R)$ is a graded fuzzy topological system, where $\models '$ is defined by $gr(s\models '[\phi])=gr(s\models \phi)$ and $R([\phi],[\psi])=inf_s\{gr(s\models'[\phi])\rightarrow gr(s\models'[\psi])\}$.
\end{theorem}
\begin{proof}
$X$ is a non empty set of assignments $s$.\\
Using the fact
$R([\phi ],[\psi ]) = inf_s\{ gr(s\models' [\phi ])\rightarrow gr(s\models' [\psi ])\}=\\ inf_s\{ gr(s\models \phi )\rightarrow gr(s\models \psi )\}=inf_s\{ gr(s\ \text{sat}\ \phi )\rightarrow gr(s\ \text{sat}\ \psi )\}=inf_s\{ gr(s\ \text{sat}\ \phi\vdash\psi )\}=gr(\phi\vdash\psi )$, it can be shown that $(A/_{\approx},R)$ is a graded frame in the following way.
\begin{enumerate}
\item For fuzzy geometric logic with graded consequence $gr(\phi\vdash\phi)=1$, so $R([\phi],[\phi])=1$.
\item $R([\phi],[\psi])=1=R([\psi],[\phi])$, i.e., $gr(\phi\vdash\psi)=1=gr(\psi\vdash\phi)$.\\
Hence $inf\{gr(s\ \text{sat}\ \phi)\rightarrow gr(s\ \text{sat}\ \psi)\}_s=1=inf\{gr(s\ \text{sat}\ \psi)\rightarrow gr(s\ \text{sat}\ \phi)\}_s$ and so
$gr(s\ \text{sat}\ \phi)\rightarrow gr(s\ \text{sat}\ \psi)=1=gr(s\ \text{sat}\ \psi)\rightarrow gr(s\ \text{sat}\ \phi)$, for any $s$.\\
Hence $gr(s\ \text{sat}\ \phi)\leq gr(s\ \text{sat}\ \psi) $ and $gr(s\ \text{sat}\ \psi)\leq gr(s\ \text{sat}\ \phi)$, for any $s$.\\
So, $gr(s\ \text{sat}\ \phi)= gr(s\ \text{sat}\ \psi)$, for any $s$ and consequently $\phi\approx \psi$.\\
Therefore $[\phi]=[\psi]$.
\item We have $gr(\phi\vdash\psi)\wedge gr(\psi\vdash \chi)\leq gr(\phi\vdash\chi)$.\\
Hence $R([\phi],[\psi])\wedge R([\psi],[\chi])\leq R([\phi],[\chi])$.
\item $R([\phi]\wedge [\psi],[\phi])=R([\phi\wedge\psi],[\phi])=gr(\phi\wedge\psi\vdash\phi)=1$.\\
Similarly $R([\phi]\wedge[\psi],[\psi])=1$.
\item $R([\phi],[\top])=gr(\phi\vdash\top)=1$.
\item $R([\phi],[\psi])\wedge R([\phi],[\chi])\\
=gr(\phi\vdash\psi)\wedge gr(\phi\vdash\chi)\\
=gr(\phi\vdash\psi\wedge\chi)\\
=R([\phi],[\psi\wedge\chi])\\
=R([\phi],[\psi]\wedge [\chi])$.
\item Let $\phi\in S$, then
$R([\phi],\bigvee [S])=R([\phi],[\bigvee S])=gr(\phi\vdash\bigvee S)=1$.
\item $inf\{R([\phi],[\psi])\}_{\phi\in S}\\
=inf\{gr(\phi\vdash\psi)\}_{\phi\in S}\leq gr(\bigvee S\vdash\psi)\\
=R([\bigvee S],[\psi])\\
=R(\bigvee[S],[\psi])$.
\item $R([\phi]\wedge\bigvee[S],\bigvee\{[\phi]\wedge[\psi]\mid \psi\in S\})\\
=R([\phi]\wedge[\bigvee S],\bigvee\{[\phi\wedge\psi]\mid\psi\in S\})\\
=R([\phi\wedge\bigvee S],[\bigvee\{\phi\wedge\psi\mid\psi\in S\}])\\
=gr(\phi\wedge\bigvee S\vdash \bigvee\{\phi\wedge\psi\mid\psi\in S\})\\
=1$.
\end{enumerate}
Hence $A/_{\approx}$ is a graded frame.\\
Now it is left to show that  (a) $gr(s\models' [\phi])\wedge R([\phi],[\psi])\leq gr(s\models' [\psi])$, (b) $gr(s\models ' [\phi]\wedge [\psi])=gr(s\models '[\phi])\wedge gr(s\models ' [\psi])$ and (c) $gr(s\models '\bigvee\{[\phi_i]\}_{i\in I})= sup_{i\in I}\{gr(s\models '[\phi_i])\}$.
\\Proof of the above statement follows easily using Theorem \ref{Gfts}.\\
Hence $(X,\models ',A/_{\approx},R)$ is a graded fuzzy topological system.
\end{proof}
\begin{proposition}\label{Gspec}
If $(X,\models ',A/_{\approx},R)$ is a graded fuzzy topological system defined as above then for all $s\in X$, $(gr(s\models '[\phi])=gr(s\models '[\psi]))\Rightarrow ([\phi]=[\psi])$.
\end{proposition}
\begin{proof}
$gr(s\models '[\phi])=gr(s\models '[\psi])$, for any $s$
if and only if $gr(s\models \phi)=gr(s\models \psi)$, for any $s$.
Thereby $\phi\approx\psi$. So, $[\phi]=[\psi]$.
\end{proof}
\subsection{Fuzzy Topological Space with Graded Inclusion and Graded Fuzzy Topological System}
Using Theorem \ref{gftsyt} a fuzzy topological space with graded inclusion,\\ $(X,ext_g(A/_{\approx}),\subseteq)$, can be obtained from the graded fuzzy topological system $(X,\models',A/_{\approx},R)$.

Using Theorem \ref{gfttsy} it can be shown that if $(X,ext_g(A/_{\approx}),\subseteq)$ is a fuzzy topological space with graded inclusion then, $(X,\in ,ext_g(A/_{\approx}),\subseteq)$ forms a graded fuzzy topological system.

\subsection{Fuzzy Topological Space with Graded Inclusion to Fuzzy Geometric Logic with Graded Consequence}
Let $(X,\tau,\subseteq)$ be a fuzzy topological space with graded inclusion. For each $\tilde{T_i}\in \tau$ let $P_{\tilde{T_i}}$ be a propositional variable and let the following axioms for sequents $P_{\tilde{T_i}}\vdash P_{\tilde{T_j}}$ be assumed.
\begin{enumerate}
\item $gr(P_{\tilde{T_i}}\vdash P_{\tilde{T_j}})=gr(\tilde{T_i}\subseteq \tilde{T_j})$;
\item $gr(P_{\tilde{T_i}}\wedge P_{\tilde{T_j}}\vdash P_{\tilde{T_i}\cap\tilde{T_j}})=1=gr(P_{\tilde{T_i}\cap\tilde{T_j}}\vdash P_{\tilde{T_i}}\wedge P_{\tilde{T_j}})$, for any $\tilde{T_i},\tilde{T_j}\in \tau$;
\item $gr(P_{\bigcup_i\tilde{T_i}}\vdash \bigvee \{P_{\tilde{T_i}}\}_i)=1=gr(\bigvee \{P_{\tilde{T_i}}\}_i\vdash P_{\bigcup_i\tilde{T_i}})$, for any $\tilde{T_i}$'s $\in \tau$.
\end{enumerate}
We adopt binary meet ($\wedge$) and arbitrary join ($\bigvee$) to form other formulae.
\begin{definition}\label{equivsubst}
Let for wffs $\alpha$, $\beta$ the relation $\equiv$ be defined by $\alpha\equiv\beta$ if and only if $gr(\alpha\vdash\beta)=1=gr(\beta\vdash\alpha)$.
\end{definition}
From (2) and (3) it is clear that for each formula $\alpha$ there exists a propositional variable $P_{\tilde{T_{\alpha}}}$ such that $\alpha\equiv P_{\tilde{T_{\alpha}}}$.

We now extend the definition of assignment of grades to an arbitrary sequent by $gr(\alpha\vdash\beta)=gr(P_{\tilde{T_{\alpha}}}\vdash P_{\tilde{T_{\beta}}})=gr(\tilde{T_{\alpha}}\subseteq \tilde{T_{\beta}})$. Because of the above remark this extension is justified.

The following lemmas may be established.
\begin{lemma}\label{l1}
$gr(P_{\tilde{T_i}}\vdash P_{\tilde{T_i}})=1$.
\end{lemma}
\begin{proof}
$gr(P_{\tilde{T_i}}\vdash P_{\tilde{T_i}})=gr(\tilde{T_i}\subseteq \tilde{T_i})=1$, using Proposition \ref{gt1}.
\end{proof}
\begin{lemma}\label{l3}
$gr(P_{\tilde{T_i}}\vdash P_{\tilde{T_j}})\wedge gr(P_{\tilde{T_j}}\vdash P_{\tilde{T_k}})\leq gr(P_{\tilde{T_i}}\vdash P_{\tilde{T_k}})$.
\end{lemma}
\begin{proof}
$gr(P_{\tilde{T_i}}\vdash P_{\tilde{T_j}})\wedge gr(P_{\tilde{T_j}}\vdash P_{\tilde{T_k}})=gr(\tilde{T_i}\subseteq \tilde{T_j})\wedge gr(\tilde{T_j}\subseteq \tilde{T_k})\leq gr(\tilde{T_i}\subseteq \tilde{T_k})=gr(P_{\tilde{T_i}}\vdash P_{\tilde{T_k}}) $, using Proposition \ref{gt3}.
\end{proof}
\begin{lemma}\label{l6}
$gr(P_{\tilde{T_i}}\vdash\tilde{X})=1$.
\end{lemma}
\begin{proof}
$gr(P_{\tilde{T_i}}\vdash\tilde{X})=gr(\tilde{T_i}\subseteq \tilde{X})=1$, using Proposition \ref{gt5}.
\end{proof}
\begin{lemma}\label{l7}
$gr(P_{\tilde{T_i}}\wedge P_{\tilde{T_j}}\vdash P_{\tilde{T_i}})=1$.
\end{lemma}
\begin{proof}
$gr(P_{\tilde{T_i}}\wedge P_{\tilde{T_j}}\vdash P_{\tilde{T_i}}) = gr(P_{\tilde{T_i}\cap\tilde{T_j}}\vdash P_{\tilde{T_i}})
 = gr(\tilde{T_i}\cap\tilde{T_j}\subseteq \tilde{T_i})
 = 1$, using Definition \ref{equivsubst} and Proposition \ref{gt4}.
\end{proof}
\begin{lemma}\label{l8}
$gr(P_{\tilde{T_i}}\wedge P_{\tilde{T_j}}\vdash P_{\tilde{T_j}})=1$.
\end{lemma}
\begin{proof}
$gr(P_{\tilde{T_i}}\wedge P_{\tilde{T_j}}\vdash P_{\tilde{T_j}}) = gr(P_{\tilde{T_i}\cap\tilde{T_j}}\vdash P_{\tilde{T_j}}) = gr(\tilde{T_i}\cap\tilde{T_j}\subseteq \tilde{T_j}) = 1$, using Definition \ref{equivsubst} and Proposition \ref{gt4}.
\end{proof}
\begin{lemma}\label{l9}
$gr(P_{\tilde{T_i}}\vdash P_{\tilde{T_j}})\wedge gr(P_{\tilde{T_i}}\vdash P_{\tilde{T_k}})=gr(P_{\tilde{T_i}}\vdash P_{\tilde{T_j}}\wedge P_{\tilde{T_k}})$.
\end{lemma}
\begin{proof}
$gr(P_{\tilde{T_i}}\vdash P_{\tilde{T_j}})\wedge gr(P_{\tilde{T_i}}\vdash P_{\tilde{T_k}}) = gr(\tilde{T_i}\subseteq \tilde{T_j})\wedge gr(\tilde{T_i}\subseteq \tilde{T_k})= gr(\tilde{T_i}\subseteq \tilde{T_j}\cap \tilde{T_k}) = gr(P_{\tilde{T_i}}\vdash P_{\tilde{T_j}\cap \tilde{T_k}})= gr(P_{\tilde{T_i}}\vdash P_{\tilde{T_j}}\wedge P_{\tilde{T_k}})$, using Proposition \ref{gt6} and Definition \ref{equivsubst}.
\end{proof}
\begin{lemma}\label{l10}
$gr(P_{\tilde{T_i}}\vdash\bigvee \{P_{\tilde{T_i}}\}_i)=1$.
\end{lemma}
\begin{proof}
$gr(P_{\tilde{T_i}}\vdash\bigvee\{ P_{\tilde{T_i}}\}_i) = gr(P_{\tilde{T_i}}\vdash P_{\bigcup_i\tilde{T_i}}) = gr(\tilde{T_i}\subseteq \bigcup_i\tilde{T_i}) = 1$, using Definition \ref{equivsubst} and Proposition \ref{gt7}.
\end{proof}
\begin{lemma}\label{l11}
$inf_i\{gr(P_{\tilde{T_i}}\vdash P_{\tilde{T}})\}\leq gr(\bigvee\{P_{\tilde{T_i}}\}_i\vdash P_{\tilde{T}})$.
\end{lemma}
\begin{proof}
$inf_i\{gr(P_{\tilde{T_i}}\vdash P_{\tilde{T}})\} = inf_i\{gr(\tilde{T_i}\subseteq \tilde{T})\}\leq gr(\bigcup_i\tilde{T_i}\subseteq\tilde{T}) =gr(P_{\bigcup_i\tilde{T_i}}\vdash P_{\tilde{T}})=gr(\bigvee\{P_{\tilde{T_i}}\}_i\vdash P_{\tilde{T}})$, using Proposition \ref{gt8} and Definition \ref{equivsubst}.
\end{proof}
\begin{lemma}\label{l12}
$gr(P_{\tilde{T}}\wedge\bigvee \{P_{\tilde{T_i}}\}_i\vdash\bigvee\{P_{\tilde{T}}\wedge P_{\tilde{T_i}}\}_i)=1$.
\end{lemma}
\begin{proof}
$gr(P_{\tilde{T}}\wedge\bigvee \{P_{\tilde{T_i}}\}_i\vdash\bigvee\{P_{\tilde{T}}\wedge P_{\tilde{T_i}}\}_i)= gr(P_{\tilde{T}\cap\bigcup_i\tilde{T_i}}\vdash P_{\bigcup_i(\tilde{T}\cap\tilde{T_i})})
 = gr(\tilde{T}\cap\bigcup_i\tilde{T_i}\subseteq \bigcup_i(\tilde{T}\cap\tilde{T_i})) = 1$, using Definition \ref{equivsubst} and Proposition \ref{gt9}.
\end{proof}
With these lemmas and the extended definition of grade assignment to sequents of formulae the following theorem is obtained.
\begin{theorem}
For any fuzzy topological space with graded inclusion $(X,\tau,\subseteq)$, the logic defined as above forms a propositional fuzzy geometric logic with graded consequence. 
\end{theorem}
\textbf{Note:} The relationships between fuzzy geometric logic, fuzzy topological systems and fuzzy topological spaces can be shown as a special case of the above results.
\section{Some Concluding Remarks on Chapter 6}
Subsections 6.7.1, 6.7.2 and 6.7.3 show that there exists a close relationship between the three notions.

The logic that is developed here is not many-sorted. It is clear that the present formalism can be extended to include multiple sorts. An important issue, however, is to place things into categorical set up. It is well known that the research-culture in geometric logic is deeply involved with category theory. In the opinion of Vickers, of this field ``the full mathematical insights come only through category theory"\cite{SV1}. The categorical recast of the notions included in this chapter has already been developed by us which will be presented in the next chapter.

In summary, the chief contribution here lies in the definition and treatment of the notions fuzzy geometric logic, fuzzy geometric logic with graded consequence, graded fuzzy topological system, fuzzy topology with graded inclusion and graded frame. It seems that the last two notions may be interesting in themselves. The study of fuzzy topology may have an additional dimension by the incorporation of graded inclusion of fuzzy sets instead of usual crisp inclusion as is in vogue. The way in which graded frame has been defined may be used to define graded counterparts of other algebraic structures. The consequence is yet to be seen.

It may be marked that G$\ddot{o}$del arrow has been used while defining graded inclusion of fuzzy sets as well as fuzzy geometric logic with graded consequence. This fuzzy implication operator has been essential at some steps of proof. In our future papers we shall endeavour to get out of this restriction which is a bit unsatisfactory and see how much could be achieved by other important implications.
\chapter{Categorical Study of Graded Fuzzy Topological System}
\section{Introduction}
This chapter deals with the categorical framework of the previous chapter. The categories of the fuzzy topological spaces with graded inclusion, graded fuzzy topological systems and graded frames are introduced. On top of that this chapter depicts the transformation of morphisms between the objects which play an interesting role too. Through the categorical study it becomes more clear why the graded inclusion is important in the fuzzy topology to establish the desired interrelations. It should be noted that the inclusion relation is defined using G$\ddot{o}$del arrow. Hence there is a space to change the arrow to define graded inclusion and make appropriate graded algebraic structures to get the similar connections of this chapter. Although we do not delve into this matter in this thesis but hope to deal with this in near future. 
\section{Categories: Graded Fuzzy Top, Graded Fuzzy TopSys, Graded Frm and their interrelationships}
\subsection{Categories}
\subsection*{$\mathbf{Graded Fuzzy Top}$}
\begin{definition}[{$\mathbf{Graded\ Fuzzy\ Top}$}]\label{Grft}\index{category! Graded Fuzzy Top} 
The category $\mathbf{Graded\ Fuzzy\ Top}$ is defined thus.
\begin{itemize}
\item The objects are fuzzy topological spaces with graded inclusion $(X,\tau,\subseteq)$, $(Y,\tau ',\subseteq)$ etc (c.f. Chapter 6).
\item The morphisms\index{Graded Fuzzy Top!morphism} are functions satisfying the following continuity property: If $f:(X,\tau,\subseteq)\longrightarrow (Y,\tau ',\subseteq)$ and $\tilde{T'}\in\tau '$ then $f^{-1}(\tilde{T'})\in\tau$.
\item The identity on $(X,\tau,\subseteq)$ is the identity function.  It can be shown that the identity function is a $\mathbf{Graded\ Fuzzy\ Top}$ morphism\index{Graded Fuzzy Top!morphism}.
\item If $f:(X,\tau,\subseteq)\longrightarrow (Y,\tau ',\subseteq)$ and 
$g:(Y,\tau ',\subseteq)\longrightarrow (Z,\tau '',\subseteq)$ are morphisms in $\mathbf{Graded\ Fuzzy\ Top}$, their 
composition $g\circ f$ is the composition of functions between two sets. It can be verified that $g\circ f$ is a morphism in $\mathbf{Graded\ Fuzzy\ Top}$.
\end{itemize}
\end{definition}
\begin{proposition}\label{ftres}
$gr(\tilde{T_1}\subseteq\tilde{T_2})\leq gr(f(\tilde{T_1})\subseteq f(\tilde{T_2}))$.
\end{proposition}
\begin{proof}
It is known that $gr(\tilde{T_1}\subseteq\tilde{T_2})=inf_{x\in X}\{\tilde{T_1}(x)\rightarrow\tilde{T_2}(x)\}$ and 
\begin{multline*}
gr(f(\tilde{T_1})\subseteq f(\tilde{T_2}))=inf_{y\in Y}\{f(\tilde{T_1})(y)\rightarrow f(\tilde{T_2})(y)\}\\=inf_{y\in Y}\{sup_{x\in f^{-1}(y)}\{\tilde{T_1}(x)\} \rightarrow sup_{x\in f^{-1}(y)}\{\tilde{T_2}(x)\}\}.
\end{multline*}
Now, 
\begin{multline*}
sup_{x\in f^{-1}(y)}\{\tilde{T_1}(x)\}\rightarrow sup_{x\in f^{-1}(y)}\{\tilde{T_2}(x)\}=inf_{x\in f^{-1}(y)}\{\tilde{T_1}(x)\\ \rightarrow sup_{x\in f^{-1}(y)}\{\tilde{T_2}(x)\}\}
\end{multline*}
 and $\tilde{T_2}(x)\leq sup_{x\in f^{-1}(y)}\tilde{T_2}(x)$ for any $x\in f^{-1}(y)$.
Therefore for any $x\in f^{-1}(y)$ 
$$\tilde{T_1}(x)\rightarrow \tilde{T_2}(x)\leq\tilde{T_1}(x)\rightarrow sup_{x\in f^{-1}(y)}\{\tilde{T_2}(x)\}.$$
So, 
$$inf_{x\in f^{-1}(y)}\{\tilde{T_1}(x)\rightarrow \tilde{T_2}(x)\}\leq inf_{x\in f^{-1}(y)}\{\tilde{T_1}(x)\rightarrow sup_{x\in f^{-1}(y)}\{\tilde{T_2}(x)\}\}.$$
Now, 
$$inf_{x\in X}\{\tilde{T_1}(x)\rightarrow \tilde{T_2}(x)\}\leq inf_{x\in f^{-1}(y)}\{\tilde{T_1}(x)\rightarrow \tilde{T_2}(x)\}$$ as $f^{-1}(y)\subseteq X.$
So for any $y\in Y$, 
$$inf_{x\in X}\{\tilde{T_1}(x)\rightarrow \tilde{T_2}(x)\}\leq inf_{x\in f^{-1}(y)}\{\tilde{T_1}(x)\rightarrow sup_{x\in f^{-1}(y)}\{\tilde{T_2}(x)\}\}$$  and consequently, 
\begin{multline*}
inf_{x\in X}\{\tilde{T_1}(x)\rightarrow \tilde{T_2}(x)\}\leq inf_{y\in Y}\{sup_{x\in f^{-1}(y)}\{\tilde{T_1}(x)\}\rightarrow sup_{x\in f^{-1}(y)}\{\tilde{T_2}(x)\}\}.
\end{multline*}
\end{proof}
\begin{definition}[{$\mathbf{Graded\ Frm}$}]\label{Grfrm} 
The category $\mathbf{Graded\ Frm}$\index{category!Graded Frm} is defined thus.
\begin{itemize}
\item The objects are graded frames $(A,R)$, $(B,R')$ etc. (c.f. Chapter 6).
\item The morphisms\index{Graded Frm!morphism} are graded frame homomorphisms\index{graded!frame homomorphism} defined in the following way: If $f:(A,R)\longrightarrow (B,R')$ then\\
(i) $f(a_1\wedge a_2)=f(a_1)\wedge f(a_2)$;\\
(ii) $f(\bigvee_i a_i)=sup_i\{f(a_i)\}$;\\
(iii) $R(a_1,a_2)\leq R'(f(a_1),f(a_2))$.
\item The identity on $(A,R)$ is the identity morphism. It can be shown by routine check that the identity morphism is a $\mathbf{Graded\ Frm}$ morphism\index{Graded Frm!morphism}.
\item If $f:(A,R)\longrightarrow (B,R')$ and 
$g:(B,R')\longrightarrow (C,R'')$ are morphisms in $\mathbf{Graded\ Frm}$, their 
composition $g\circ f$ is the composition of graded homomorphisms between two graded frames. It can be verified that 
$g\circ f$ is a morphism in $\mathbf{Graded\ Frm}$ (vide Proposition \ref{next}).
\end{itemize}
\end{definition}
\begin{proposition}\label{next}
If $f:(A,R)\longrightarrow (B,R')$ and 
$g:(B,R')\longrightarrow (C,R'')$ are morphisms in $\mathbf{Graded\ Frm}$ then $g\circ f:(A,R)\longrightarrow (C,R'')$ is a morphism in $\mathbf{Graded\ Frm}$.
\end{proposition}
\begin{proof}
To show that $g\circ f:(A,R)\longrightarrow (C,R'')$ is a graded frame homomorphism we will proceed in the following way.
\begin{align*}
(i)\ g\circ f (a_1\wedge a_2) & = g(f(a_1\wedge a_2))\\
& = g(f(a_1)\wedge f(a_2))\\
& = g(f(a_1))\wedge g(f(a_2))\\
& = g\circ f(a_1)\wedge g\circ f(a_2).\\
(ii)\ g\circ f (\bigvee_i a_i) & = g(f(\bigvee_i a_i))\\
& = g(sup_i\{f(a_i)\})\\
& = sup_i\{g(f(a_i))\}\\
& = sup_i\{g\circ f(a_i)\}.\\
(iii)\ R(a_1,a_2) & \leq R'(f(a_1),f(a_2))\\
& \leq R''(g(f(a_1)),g(f(a_2)))\\
& = R''(g\circ f(a_1),g\circ f(a_2)).
\end{align*}
This completes the proof.
\end{proof}
\begin{definition}[{$\mathbf{Graded\ Fuzzy\ TopSys}$}]\label{Grftsy} 
The category of graded fuzzy topological systems, $\mathbf{Graded\ Fuzzy\ TopSys}$\index{category!Graded Fuzzy TopSys}, is defined thus.
\begin{itemize}
\item The objects are graded fuzzy topological systems $(X,\models,A,R)$, $(Y,\models,B,R')$ etc. (c.f. Chapter 6).
\item The morphisms are pair of maps satisfying the following continuity properties: If $(f_1,f_2):(X,\models ,A,R)\longrightarrow (Y,\models',B,R')$ then\\
(i) $f_1:X\longrightarrow Y$ is a set map;\\
(ii) $f_2:(B,R')\longrightarrow (A,R)$ is a graded frame homomorphism;\\
(iii) $gr(x\models f_2(b))=gr(f_1(x)\models' b)$.
\item The identity on $(X,\models,A,R)$ is the pair $(id_X,id_A)$, where $id_X$ is the identity map on $X$ and $id_A$ is the identity graded frame homomorphism on $A$. That this is an $\mathbf{Graded\ Fuzzy\ TopSys}$ morphism\index{Graded Fuzzy TopSys!morphism} can be proved.
\item If $(f_1,f_2):(X,\models,A,R)\longrightarrow (Y,\models',B,R')$ and 
$(g_1,g_2):(Y,\models',B,R')\longrightarrow (Z,\models'',C,R'')$ are morphisms in $\mathbf{Graded\ Fuzzy\ TopSys}$, their 
composition $(g_1,g_2)\circ (f_1,f_2)=(g_1\circ f_1,f_2\circ g_2)$ is the pair of composition of functions between two sets and composition of graded frame homomorphisms between two graded frames. It can be verified that 
$(g_1,g_2)\circ (f_1,f_2)$ is a morphism in $\mathbf{Graded\ Fuzzy\ TopSys}$ using Proposition \ref{next} and Lemma \ref{3.4_1}.
\end{itemize}
\end{definition}
\subsection{Functors}
In this subsection, we define various functors, which are required to prove our desired results.
\subsection*{Functor $Ext_g$ from Graded Fuzzy TopSys to Graded Fuzzy Top}\index{functor!$Ext_g$}
\begin{definition}\label{3.8g}
Let $(X,\models,A,R)$ be a graded fuzzy topological system, and $a\in A$. For each $a$, its \textbf{extent$_g$}\index{extent$_g$} in $(X,\models, A,R)$ is a mapping $ext_g(a)$ from $X$ to $[0,1]$ given by $ext_g(a)(x)=gr(x\models a)$.
Also $ext_g(A)=\{ext_g(a)\}_{a\in A}$ and $gr(ext_g(a_1)\subseteq ext_g(a_2))=inf\{ext_g(a_1)(x)\rightarrow ext_g(a_2)(x)\}_x$, for any $a_1,\ a_2\in A$.
\end{definition}
\begin{lemma}\label{3.9g}
$(ext_g(A),\subseteq)$ forms a graded fuzzy topology on $X$.
\end{lemma}
\begin{proof}
Proof of the lemma follows from Theorem \ref{gftsyt}.
\end{proof}
As a consequence $(X,ext_g(A),\subseteq)$ forms a graded fuzzy topological space.
\begin{lemma}\label{3.10g}
If $(f_1,f_2):(X,\models ',A,R)\longrightarrow (Y,\models '',B,R')$ is continuous then \\$f_1:(X,ext_g(A),\subseteq)\longrightarrow (Y,ext_g(B),\subseteq)$ is graded fuzzy continuous.
\end{lemma}
\begin{proof}
$(f_1,f_2):(X,\models ',A,R)\longrightarrow (Y,\models '',B,R')$ is continuous.
So we have for all $x\in X,\ b\in B$, $gr(x\models 'f_2(b))=gr(f_1(x)\models ''b)$. Now,
\begin{align*} 
(f_1^{-1}(ext_g(b)))(x) & =ext_g(b)(f_1(x))\\
& =gr(f_1(x)\models ''b)\\
& =gr(x\models 'f_2(b))\\
& =ext_g(f_2(b))(x).
\end{align*}
$So,$ $f_1^{-1}(ext_g(b))=ext_g(f_2(b))\in ext_g(A)$.
$So,$ $f_1$ is a fuzzy continuous map from $(X,ext_g(A),\subseteq)$ to $(Y,ext_g(B),\subseteq)$.
\end{proof}
\begin{definition}\label{3.11g}
$\mathbf{Ext_g}$\index{functor!$Ext_g$} is a functor from $\mathbf{Graded\ Fuzzy\ TopSys}$ to\\ $\mathbf{Graded\ Fuzzy\ Top}$ defined as follows.\\
$Ext_g$ acts on an object $(X,\models ',A,R)$ as $Ext_g(X,\models ',A,R)=(X,ext_g(A),\subseteq)$ and on a morphism $(f_1,f_2)$ as $Ext_g(f_1,f_2)=f_1$.
\end{definition}
The above two Lemmas \ref{3.9g} and \ref{3.10g} shows that $Ext_g$ is a functor.
\subsection*{Functor $J_g$ from the category $\mathbf{Graded\ Fuzzy\ Top}$ to the category $\mathbf{Graded\ Fuzzy\ TopSys}$}
\begin{definition}\label{3.12g}
$\mathbf{J_g}$\index{functor!$J_g$} is a functor from the category $\mathbf{Graded\ Fuzzy\ Top}$ to the category $\mathbf{Graded\ Fuzzy\ TopSys}$ defined as follows.\\
$J_g$ acts on an object $(X,\tau,\subseteq)$ as $J_g(X,\tau,\subseteq)=(X, \in ,\tau,\subseteq)$, where $gr(x\in \tilde{T})=\tilde{T}(x)$ for $\tilde{T} \in \tau$, and on a morphism $f$ as $J_g(f)=(f,f^{-1})$.
\end{definition}
\begin{lemma}\label{3.13g}
$(X,\in , \tau,\subseteq)$ is a graded fuzzy topological system.
\end{lemma}
\begin{proof}
$(\tau,\subseteq)$ forms a graded frame can be shown using Propositions \ref{gt1}-\ref{gt9}.\\
To show
(1) $gr(x\in \tilde{T_1})\wedge gr(\tilde{T_1}\subseteq\tilde{T_2})\leq gr(x\in \tilde{T_2})$,
(2) $gr(x\in \tilde{T_1}\cap \tilde{T_2})=inf\{ gr(x\in \tilde{T_1}),gr(x\in \tilde{T_2})\}$ and (3) $gr(x\in \bigcup_i \tilde{T_i})=sup_i\{ gr(x\in \tilde{T_i})\}$ let us proceed in the following way.\\
(1) $gr(x\in \tilde{T_1})\wedge gr(\tilde{T_1}\subseteq\tilde{T_2})\leq gr(x\in \tilde{T_2})$ follows by Proposition \ref{gt10}.
\begin{align*}
(2)\ gr(x\in \tilde{T_1}\cap \tilde{T_2}) & =(\tilde{T_1}\cap \tilde{T_2})(x)\\
& =\tilde{T_1}(x)\wedge \tilde{T_2}(x)\\
& = gr(x\in \tilde{T_1})\wedge gr(x\in \tilde{T_2}).\\
(3)\ gr(x\in \bigcup_i \tilde{T_i}) & =(\bigcup_i \tilde{T_i})(x)\\
& =sup_i\{\tilde{T_i}(x)\}\\
& =sup_i\{ gr(x\in \tilde{T_i})\}.
\end{align*}
\end{proof}
\begin{lemma}\label{3.14g}
$J_g(f)=(f,f^{-1})$ is continuous provided $f$ is graded fuzzy continuous.
\end{lemma}
\begin{proof}
We have $f:(X,\tau_1)\longrightarrow (Y,\tau_2)$ and $(f,f^{-1}):(X,\in ,\tau_1)\longrightarrow (Y,\in ,\tau_2)$. It is enough to show that for $\tilde{T_2}\in \tau_2$, $gr(x\in f^{-1}(\tilde{T_2}))=gr(f(x)\in \tilde{T_2})$ as Proposition \ref{ftres} holds.

Now, $$gr(x\in f^{-1}(\tilde{T_2}))=(f^{-1}(\tilde{T_2}))(x)=\tilde{T_2}(f(x))=gr(f(x)\in \tilde{T_2}).$$ Hence $J_g(f)=(f,f^{-1})$ is continuous.
\end{proof}
So $ J_g$ is a functor from $\mathbf{Graded\ Fuzzy\ Top}$ to $\mathbf{Graded\ Fuzzy\ TopSys}$.
\subsection*{Functor $fm_g$ from the category $\mathbf{Graded\ Fuzzy\ TopSys}$ to the category $\mathbf{Graded\ Frm^{op}}$}
\begin{definition}\label{3.15g}
$\mathbf{fm_g}$\index{functor!$J_g$} is a functor from the category $\mathbf{Graded\ Fuzzy\ TopSys}$ to the category $\mathbf{Graded\ Frm^{op}}$ defined as follows.\\
$fm_g$ acts on an object $(X,\models ,A,R)$ as $fm_g(X,\models ,A,R)=(A,R)$ and on a morphism $(f_1,f_2)$ as $fm_g(f_1,f_2)=f_2$.
\end{definition}
 It is easy to see that $fm_g$ is a functor.
\subsection*{Functor $S_g$ from $\mathbf{Graded\ Frm^{op}}$ to $\mathbf{Graded\ Fuzzy\ TopSys}$}
\begin{definition}\label{3.16g}
Let $(A,R)$ be a graded frame, $Hom_g((A,R),([0,1],R^*))=\{ graded\ frame$ $ hom$ $v:(A,R)\longrightarrow ([0,1],R^*)\}$, where $R^*:[0,1]\times [0,1]\longrightarrow [0,1]$ such that $R^*(a,b)=1$ iff $a\leq b$ and $R^*(a,b)=b$ iff $a>b$.
\end{definition}
\begin{lemma}\label{3.17g}
$(Hom_g((A,R),([0,1],R^*)),\models_*,A,R)$, where $(A,R)$ is a graded frame and $gr(v\models_* a)=v(a)$, is a graded fuzzy topological system.
\end{lemma}
\begin{proof}
Let us show that $$gr(v\models_*a)\wedge R(a,b)\leq gr(v\models_*b).$$ That is, it will be enough to show that $v(a)\wedge R(a,b)\leq v(b)$. Now as $v$ is a graded frame homomorphism so $R(a,b)\leq R^*(v(a),v(b))$. Hence $$v(a)\wedge R(a,b)\leq R(a,b)\leq R^*(v(a),v(b)).$$ Now as $v(a)$, $v(b)\in [0,1]$ so either $v(a)\leq v(b)$ or $v(a)>v(b)$. For $v(a)\leq v(b)$, $$v(a)\wedge R(a,b)\leq v(a)\leq v(b)$$ and for $v(a)>v(b)$, $$R^*(v(a),v(b))=0.$$ So for the second case $$v(a)\wedge R(a,b)\leq R^*(v(a),v(b))=v(b).$$ Hence $$gr(v\models_*a)\wedge R(a,b)\leq gr(v\models_*b).$$ As $v$ is a graded frame homomorphism, so $v(a\wedge b)=v(a)\wedge v(b)$ and hence $$gr(v\models_*a\wedge b)=gr(v\models_*a)\wedge gr(v\models_*b).$$
Similarly $gr(v\models_*\bigvee S)=sup\{gr(v\models_* a)\mid a\in S\}$, for $S\subseteq A$.
\end{proof}
\begin{lemma}\label{3.18g}
If $f:(B,R')\longrightarrow (A,R)$ is a graded frame homomorphism then \\$(\_\circ f,f):(Hom((A,R),([0,1],R^*)),\models_*,A,R)\longrightarrow (Hom((B,R'),([0,1],R^*)),\models_*,B,R')$ is continuous.
\end{lemma}
\begin{definition}\label{3.19g}
$\mathbf{S_g}$\index{functor!$S_g$} is a functor from the category $\mathbf{Graded\ Frm^{op}}$ to the category $\mathbf{Graded\ Fuzzy\ TopSys}$ defined as follows.\\
$S_g$ acts on an object $(A,R)$ as $S_g((A,R))=(Hom_g((A,R),([0,1],R^*)),\models_*,A,R)$ and on a morphism $f$ as $S_g(f)=(\_\circ f,f)$.
\end{definition}
Previous two Lemmas \ref{3.17g} and \ref{3.18g} shows that $S_g$ is indeed a functor.
\begin{lemma}\label{3.20g}
$Ext_g$ is the right adjoint to the functor $J_g$.
\end{lemma}
\begin{proof}[Proof Sketch]
It is possible to prove the theorem by presenting the co-unit of the adjunction.
Recall that $J_g(X,\tau,\subseteq)=(X,\in, \tau,\subseteq)$ and $Ext_g(X,\models,A,R)=(X,ext_g(A),\subseteq)$.
So, $$J_g(Ext_g(X,\models,A,R))=(X,\in,ext(A),\subseteq).$$
Let us draw the diagram of co-unit.
\begin{center}
\begin{tabular}{ l | r } 
$\mathbf{Graded\ Fuzzy\ TopSys}$ & $\mathbf{Graded\ Fuzzy\ Top}$\\
\hline
 {\begin{tikzpicture}[description/.style={fill=white,inner sep=2pt}] 
    \matrix (m) [matrix of math nodes, row sep=2.5em, column sep=2.0em]
    { J_g(Ext_g(X,\models,A,R))&&(X,\models,A,R)  \\
         J_g(Y,\tau ',\subseteq) \\ }; 
    \path[->,font=\scriptsize] 
        (m-1-1) edge node[auto] {$\xi_X$} (m-1-3)
        (m-2-1) edge node[auto] {$J_g(f)(\equiv(f_1,f_1^{-1}))$} (m-1-1)
        (m-2-1) edge node[auto,swap] {$\hat f (\equiv(f_1,f_2))$} (m-1-3)
        %(m-1-1) edge node[auto,swap] {$f$} (m-2-3)
       % (m-1-5) edge node[auto,swap] {$j$} (m-1-3)
       % (m-1-5) edge node[auto] {$\psi$} (m-2-3)
        %(m-1-3) edge node[auto] {$G(\hat{f})$} (m-2-3)
         ;
\end{tikzpicture}} & {\begin{tikzpicture}[description/.style={fill=white,inner sep=2pt}] 
    \matrix (m) [matrix of math nodes, row sep=2.5em, column sep=2.0em]
    { Ext_g(X,\models,A,R)  \\
         (Y,\tau ',\subseteq) \\ }; 
    \path[->,font=\scriptsize] 
        (m-2-1) edge node[auto,swap] {$f(\equiv f_1)$} (m-1-1)
       
         ;
\end{tikzpicture}} \\ 
\end{tabular}
\end{center}
Hence co-unit is defined by $\xi_X=(id_X,ext_g^*)$. That is,
\begin{center}
\begin{tikzpicture}[description/.style={fill=white,inner sep=2pt}] 
    \matrix (m) [matrix of math nodes, row sep=2.5em, column sep=2.0em]
    { (X,\in ,ext_g(A),\subseteq )&&(X,\models,A,R )  \\
          }; 
    \path[->,font=\scriptsize] 
        (m-1-1) edge node[auto] {$\xi_X$} (m-1-3)
        (m-1-1) edge node[auto,swap] {$(id_X,ext_g^*)$} (m-1-3)
         ;
\end{tikzpicture}
\end{center}
 where $ext_g^*$ is a mapping from $A$ to $ext_g(A)$ such that $ext_g^*(a)=ext_g(a)$, for all $a\in A$.
 
Now by routine check one can conclude that $Ext_g$ is the right adjoint to the functor $J_g$.
\end{proof}
Diagram of the unit of the above adjunction is as follows.

\begin{center}
\begin{tabular}{ l | r  } 
 $\mathbf{Graded\ Fuzzy\ Top}$ & $\mathbf{Graded\ Fuzzy\ TopSys}$ \\
\hline
 {\begin{tikzpicture}[description/.style={fill=white,inner sep=2pt}] 
    \matrix (m) [matrix of math nodes, row sep=2.5em, column sep=2.0em]
    {(X,\tau,\subseteq) & &Ext(J(X, \tau,\subseteq))  \\
        & & Ext(Y,\models' ,B,R') \\ }; 
    \path[->,font=\scriptsize] 
        (m-1-1) edge node[auto] {$\eta_X(\equiv id_X)$} (m-1-3)
        (m-1-1) edge node[auto,swap] {$\hat{f}(\equiv f_1)$} (m-2-3)
       % (m-1-5) edge node[auto,swap] {$j$} (m-1-3)
        %(m-1-5) edge node[auto] {$\psi$} (m-2-3)
        (m-1-3) edge node[auto] {$ext(f)(\equiv f_1)$} (m-2-3);
\end{tikzpicture}} &  {\begin{tikzpicture}[description/.style={fill=white,inner sep=2pt}] 
    \matrix (m) [matrix of math nodes, row sep=2.5em, column sep=2.5em]
    {& &J(X,\tau,\subseteq)  \\
        & &(Y,\models' ,B,R') \\ }; 
    \path[->,font=\scriptsize]

        %(m-1-5) edge node[auto,swap] {$j$} (m-1-3)
        %(m-1-5) edge node[auto] {$\psi$} (m-2-3)
        (m-1-3) edge node[auto] {$f(\equiv(f_1,f_1^{-1}))$} (m-2-3);
\end{tikzpicture}}\\
\end{tabular}
\end{center}
%Hence we have $(\eta_X,\xi_X):J\dashv Ext:FBSy_n\longrightarrow FBS_n$.

\begin{observation}\label{o1_g}
If a graded fuzzy topological system $(X,\models,A,R)$ is spatial then the co-unit $\xi_X$ becomes a natural isomorphism.
\end{observation}
\begin{observation}\label{o2_g}
For any graded fuzzy topological space $(X,\tau,\subseteq)$, the unit $\eta_X$ is a natural isomorphism.  
\end{observation}
Observation \ref{o1_g} and Observation \ref{o2_g} gives the following theorem.
\begin{theorem}\label{equigft}
Category of spatial graded fuzzy topological systems is equivalent to the category $\mathbf{Graded\ Fuzzy\ Top}$.
\end{theorem}
\begin{lemma}\label{3.22g}
$fm_g$ is the left adjoint to the functor $S_g$.
\end{lemma}
\begin{proof}[Proof Sketch]
It is possible to prove the theorem by presenting the unit of the adjunction.
Recall that  $$S_g((B,R'))=(Hom_g((B,R'),([0,1],R^*)),\models_*,B,R')$$ where $gr(v\models_* a)=v(a)$.
Hence $$S_g(fm_g(X,\models ,A,R))=(Hom_g((A,R),([0,1],R^*)),\models_* ,A,R).$$
\begin{center}
\begin{tabular}{ l | r  } 
 $\mathbf{Graded\ Fuzzy\ TopSys}$ & $\mathbf{Graded\ Frm^{op}}$ \\
\hline
 {\begin{tikzpicture}[description/.style={fill=white,inner sep=2pt}] 
    \matrix (m) [matrix of math nodes, row sep=2.5em, column sep=1.5em]
    {(X,\models ,A,R) & &S_g(fm_g(X, \models ,A,R))  \\
        & & S_g((B,R')) \\ }; 
    \path[->,font=\scriptsize] 
        (m-1-1) edge node[auto] {$\eta_A$} (m-1-3)
        (m-1-1) edge node[auto,swap] {$f(\equiv(f_1,f_2))$} (m-2-3)
       % (m-1-5) edge node[auto,swap] {$j$} (m-1-3)
        %(m-1-5) edge node[auto] {$\psi$} (m-2-3)
        (m-1-3) edge node[auto] {$S_g\hat{f}$} (m-2-3);
\end{tikzpicture}} &  {\begin{tikzpicture}[description/.style={fill=white,inner sep=2pt}] 
    \matrix (m) [matrix of math nodes, row sep=2.5em, column sep=1.5em]
    {& &fm_g(X,\models ,A,R)  \\
        & &(B,R') \\ }; 
    \path[->,font=\scriptsize]

        %(m-1-5) edge node[auto,swap] {$j$} (m-1-3)
        %(m-1-5) edge node[auto] {$\psi$} (m-2-3)
        (m-1-3) edge node[auto] {$\hat{f}(\equiv f_2)$} (m-2-3);
\end{tikzpicture}}\\
\end{tabular}
\end{center}
Then unit is defined by $\eta_A =(p^*,id_A)$. That is,

\begin{center}
 \begin{tikzpicture}[description/.style={fill=white,inner sep=2pt}] 
    \matrix (m) [matrix of math nodes, row sep=2.5em, column sep=2.0em]
    { (X,\models ,A,R) & &S_g(fm_g(X, \models ,A,R))  \\
        }; 
    \path[->,font=\scriptsize] 
        (m-1-1) edge node[auto] {$\eta_A$} (m-1-3)
        (m-1-1) edge node[auto,swap] {$(p^*,id_A)$} (m-1-3)
        %(m-1-1) edge node[auto,swap] {} (m-2-3)
       % (m-1-5) edge node[auto,swap]
        ;
\end{tikzpicture}
\end{center}
where, 
\begin{align*}
  p^* \colon X &\longrightarrow Hom_g((A,R),([0,1],R^*))\\
  x &\longmapsto p_x:(A,R)\longrightarrow ([0,1],R^*)
\end{align*}
such that $p_x(a)=gr(x\models a)$ and $R(a,a')\leq R^*(p_x(a),p_x(a'))$ for any $a,\ a'\in A$.
Now by routine check one can conclude that $fm_g$ is the left adjoint to the functor $S_g$.
\end{proof}
\begin{theorem}\label{3.25g}
$Ext_g\circ S_g$ is the right adjoint to the functor $fm_g\circ J_g$.
\end{theorem}
\begin{proof}
Follows from the combination of the adjoint situations in Lemmas \ref{3.20g}, \ref{3.22g}.
\end{proof}
The obtained functorial relationships can be illustrated by the following diagram:  
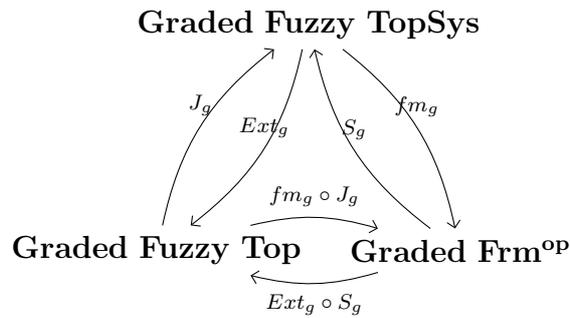
\begin{center}
\begin{tikzpicture}
\node (C) at (0,3) {$\mathbf{Graded\ Fuzzy\ TopSys}$};
\node (A) at (-2,0) {$\mathbf{Graded\ Fuzzy\ Top}$};
\node (B) at (2,0) {$\mathbf{Graded\ Frm^{op}}$};
%\node at (0,0) {\rotatebox{270}{$\Rightarrow$}};
\path[->,font=\scriptsize ,>=angle 90]
(A) edge [bend left=15] node[above] {$fm_g\circ J_g$} (B);
\path[<-,font=\scriptsize ,>=angle 90]
(A)edge [bend right=15] node[below] {$Ext_g \circ S_g$} (B);
\path[->,font=\scriptsize ,>=angle 90]
(A) edge [bend left=20] node[above] {$J_g$} (C);
\path[<-,font=\scriptsize ,>=angle 90]
(A)edge [bend right=20] node[above] {$Ext_g$} (C);
\path[->,font=\scriptsize, >=angle 90]
(C) edge [bend left=20] node[above] {$fm_g$} (B);
\path[<-,font=\scriptsize, >=angle 90]
(C)edge [bend right=20] node[above] {$S_g$} (B);
\end{tikzpicture}
\end{center}
This chapter indicates that the category of graded fuzzy topological systems constructed from fuzzy geometric logic with graded consequence is equivalent with the category of graded fuzzy topological spaces (c.f. Theorem \ref{agf} and Proposition \ref{propc} of Chapter 6, and Theorem \ref{equigft}).
\newpage
\section*{Concluding Remarks and Future Directions}
\begin{enumerate}
\item We have not raised the point of completeness of fuzzy geometric logic and fuzzy geometric logic with graded consequence. In both cases the logics would be incomplete since their restrictions to the crisp case, that is taking the values either 0 or 1, reduce to ordinary geometric logic and it is incomplete \cite{SVL}. But it may be an interesting question that under which sufficient and necessary conditions the following holds
$$gr(\phi\vdash\psi)=inf_s\{ gr(s\ \emph{sat}\ \phi)\rightarrow gr(s\ \emph{sat}\ \psi)\},\ \text{if}\ 0 <  gr (\phi\vdash\psi) < 1.$$
\item Proposing logics  with respect to the other topological systems studied in the work, such as the $\mathscr{L}$-valued fuzzy topological systems, variable basis fuzzy topological systems are pending.
\item It may be marked that G$\ddot{o}$del arrow has been used while defining graded inclusion of fuzzy sets as well as fuzzy geometric logic with graded consequence. This fuzzy implication operator has been essential at some steps of proof. We shall endeavour to get out of this restriction which is a bit unsatisfactory and see how much could be achieved by other important implications.
\item Possible dualities and equivalences among the appropriate subcategories of the categories of fuzzy topological systems over fuzzy sets, fuzzy topological spaces on fuzzy sets and frames need to be investigated.
\item Study of possible dualities and equivalences among the appropriate subcategories of the categories of variable basis fuzzy topological systems, variable basis fuzzy topological spaces and frames  has to be taken up.
\end{enumerate}

%\backmatter

\clearpage
\printindex
\end{document}